\documentclass[11pt, twoside]{Thesis}
\graphicspath{{Pictures/}} % Specifies the directory where pictures are stored

\hypersetup{urlcolor=blue, colorlinks=true} % Colors hyperlinks in blue - change to black if annoying
\title{\ttitle} % Defines the thesis title - don't touch this

\begin{document}
\selectlanguage{french} 
\frontmatter % Use roman page numbering style (i, ii, iii, iv...) for the pre-content pages

\setstretch{1.3} % Line spacing of 1.3

% Define the page headers using the FancyHdr package and set up for one-sided printing
%\fancyhead{} % Clears all page headers and footers
%\rhead{\thepage} % Sets the right side header to show the page number
%\lhead{} % Clears the left side page header

\pagestyle{fancy} % Finally, use the "fancy" page style to implement the FancyHdr headers

\newcommand{\HRule}{\rule{\linewidth}{0.5mm}} % New command to make the lines in the title page

% PDF meta-data
\hypersetup{pdftitle={\ttitle}}
\hypersetup{pdfsubject=\subjectname}
\hypersetup{pdfauthor=\authornames}
\hypersetup{pdfkeywords=\keywordnames}

%----------------------------------------------------------------------------------------
%	TITLE PAGE
%----------------------------------------------------------------------------------------

\begin{titlepage}
\begin{figure}%[!htbp]
      \includegraphics[width=0.3\textwidth]{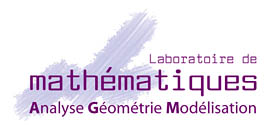}\q\q\q\q\q\q\q\q\q\q\q\q\q\q\q
      \includegraphics[width=0.3\textwidth]{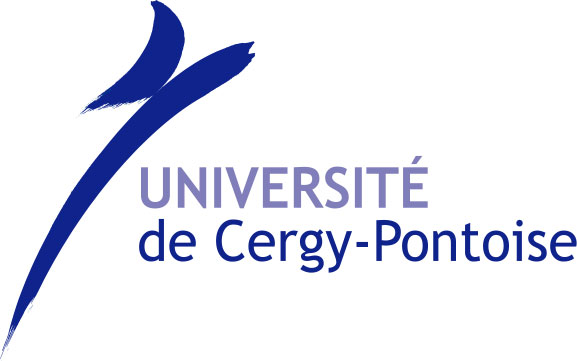}
\end{figure}
\begin{center}

\textsc{\LARGE Universit\'e de Cergy-Pontoise}\\[1.5cm] % University name
\textsc{\Large Th\`ese de Doctorat en Math\'ematiques}\\[0.5cm] % Thesis type

\HRule \\[0.4cm] % Horizontal line
{\huge \bfseries M\'elange et grandes d\'eviations pour \\[0.1cm] l'\'equation des ondes non lin\'eaire \\[0.4cm] avec bruit blanc}\\[0.4cm] % Thesis title
\HRule \\[1.5cm] % Horizontal line
 
\begin{minipage}{0.4\textwidth}
\begin{flushleft} \large
\emph{Pr\'esent\'ee par}\\ \textsc{\LARGE Davit Martirosyan}
% Author name - remove the \href bracket to remove the link
\end{flushleft}
\end{minipage}
\bigskip
\bigskip

%Rapporteurs:\q\q\q\q{\bf Sandra Cerrai}, Professeur, Universit\'e de Maryland
%\begin{minipage}{0.4\textwidth}
%\begin{flushright} \large
%\emph{Supervisor:} \\
%\href{http://www.jamessmith.com}{\supname} % Supervisor name - remove the \href bracket to remove the link  
\bigskip

\begin{center}
\noindent \large
\begin{tabular}{llcl}
{Rapporteurs :}    & {\bf Sandra Cerrai}, Universit\'e de Maryland     \\
          & {\bf Lorenzo Zambotti}, Universit\'e Pierre et Marie Curie    \\
           {Examinateurs :}    & {\bf Francis Comets}, Universit\'e Paris Diderot     \\
                & {\bf Annie Millet}, Universit\'e Paris 1 Panth\'eon-Sorbonne   \\
                 & {\bf Etienne Pardoux}, Universit\'e d'Aix-Marseille   \\
                   & {\bf Nikolay Tzvetkov}, Universit\'e de Cergy-Pontoise   \\
      {Directeur de th\`ese:}  & {\bf Armen Shirikyan}, Universit\'e de Cergy-Pontoise    \\

\end{tabular}
\end{center}
%\end{flushright}
%\end{minipage}\\[3cm]
 
%\large \textit{A thesis submitted in fulfilment of the requirements\\ for the degree of \degreename}\\[0.3cm] % University requirement text
%\textit{in the}\\[0.4cm]
%\groupname\\\deptname\\[2cm] % Research group name and department name
 
 \bigskip
{\large 12 Novembre, 2015}\\[4cm] % Date

\vfill
\end{center}

\end{titlepage}

%----------------------------------------------------------------------------------------
%	DECLARATION PAGE
%	Your institution may give you a different text to place here
%----------------------------------------------------------------------------------------

\pagestyle{empty} % Page style needs to be empty for this page
\mbox{}
\thispagestyle{empty}
\newpage

\dedicatory{\`a ma m\`ere Nune} % Dedication text

\clearpage % Start a new page

\begin{center}
{\bf R\'esum\'e}
\end{center}
Dans cette th\`ese on \'etudie l'ergodicit\'e et les grandes d\'eviations pour l'\'equation des ondes non lin\'eaire avec bruit blanc lisse en 3D. Sous certaines hypoth\`eses standards sur la nonlin\'earit\'e, on prouve que le processus de Markov associ\'e au flot de cette \'equation poss\`ede une unique mesure stationnaire qui attire toute autre solution avec une vitesse exponentielle. Ce r\'esultat implique, en particulier,  la loi forte des grands nombres ainsi que le th\'eor\`eme de la limite centrale pour les trajectoires. On s'int\'eresse ensuite au comportement asymptotique de la famille de mesures stationnaires correspondantes au bruit dont l'amplitude tend vers z\'ero et on prouve que cette famille ob\'eit au principe de grandes d\'eviations. Lorsque l'\'equation limite (i.e., sans bruit) poss\`ede un nombre fini de solutions stationnaires, parmi lesquelles une seule solution $\hat\uu$ asymptotiquement stable, ce r\'esultat implique que la famille de ces mesures converge faiblement vers la masse de Dirac concentr\'ee en $\hat\uu$. Finalement, on \'etudie le probl\`eme du comportement asymptotique en temps grands pour la famille de mesures d'occupation associ\'ees \`a l'\'equation des ondes non lin\'eaire et on \'etablit le r\'esultat de grandes d\'eviations locales pour cette famille. On prouve aussi que la bonne fonction de taux correspondante n'est pas triviale, ce qui implique qu'une forte concentration vers la mesure stationnaire est impossible.

\bigskip
\bigskip
\bigskip

\begin{center}
{\bf Abstract}
\end{center} This thesis is devoted to the study of ergodicity and large deviations for the stochastic  nonlinear wave (NLW) equation with smooth white noise in 3D. Under some standard growth and dissipativity assumptions on the nonlinearity, we show that the Markov process associated with the flow of NLW equation has a unique stationary measure that attracts the law of any solution with exponential rate. This result implies, in particular, the strong law of large numbers as well as the central limit theorem for the trajectories. We next consider the problem of small noise asymptotics for the family of stationary measures and prove that this family obeys the large deviations principle. When the limiting equation (i.e., without noise) possesses finitely many stationary solutions, among which only one asymptotically stable solution $\hat\uu$, this result implies that the family of these measures weakly converges to the Dirac measure concentrated at $\hat\uu$. Finally, we study the problem of large time asymptotics for the family of occupation measures corresponding to the NLW equation and show that it satisfies the local large deviations principle. We also show that a high concentration towards the stationary measure is impossible by proving that the corresponding rate function does not have the trivial form.

\clearpage % Start a new page
\newpage

%----------------------------------------------------------------------------------------
%	ABSTRACT PAGE
%----------------------------------------------------------------------------------------

%\addtotoc{Abstract} % Add the "Abstract" page entry to the Contents

%\abstract{\addtocontents{toc}{\vspace{1em}} % Add a gap in the Contents, for aesthetics

%The Thesis Abstract is written here (and usually kept to just this page). The page is kept centered %vertically so can expand into the blank space above the title too\ldots
%}

%\clearpage % Start a new page

%----------------------------------------------------------------------------------------
%	ACKNOWLEDGEMENTS
%----------------------------------------------------------------------------------------

\setstretch{1.3} % Reset the line-spacing to 1.3 for body text (if it has changed)

\acknowledgements{\addtocontents{toc}{\vspace{1em}} % Add a gap in the Contents, for aesthetics

Je tiens \`a remercier, en premier lieu, mon directeur de th\`ese, Armen Shirikyan.  Ses connaissances profondes et diverses en math\'ematiques, son exigence, l'intuition forte, la bonne humeur et les qualit\'es personnelles ont laiss\'e leur trace sur ma th\`ese.

\bigskip
Je remercie tous les membres du jury pour avoir accept\'e d'\'evaluer cette th\`ese. Merci \`a Sandra Cerrai et Lorenzo Zambotti pour l'avoir rapport\'e.  

\bigskip
Je remercie chaleureusement tous les membres du laboratoire AGM o\`u j'ai pass\'e ma th\`ese. Merci au directeur du labo Emmanuel, au directeur du d\'epartement Eva, aux secr\'etaires Caroline et Linda, et \`a la responsable informatique Am\'elie.
\bigskip

Merci \`a tous les th\'esards du labo: Julien, Mouhamadou, Pierre, Thibault, ... Merci \`a Bruno (d\'ej\`a docteur) et \`a Pierre-Damien avec qui j'ai pass\'e plus de temps car on \' etait dans le m\^eme bureau;) Je remercie \'egalement les anciens th\'esards et donc les docteurs Anne-Sophie, Hong Cam, Julien S., Lysianne, Nicolas.  Merci \`a Andrey pour nos discussions sur des sujets tr\`es divers, y compris la litt\'erature, la politique, le sport, des choses administratives et un peu la math\'ematique;) Merci \`a mon compatriote et un grand fr\`ere de th\`ese Hayk Nersisyan pour ses encouragements.

\bigskip
Je veux  remercier \'egalement le ma\^itre de conf\'erence Thierry Daud\'e. Gr\^ace \`a lui, j'ai d\'ecouvert les meilleurs vins de Bordeaux (et la ville elle-m\^eme), le Th\'e\^atre de la Temp\^ete, le groupe \'ethio-jazz Akal\'e Wub\'e, certaines subtilit\'es grammaticales et plein d'autres choses.

\bigskip
Merci \`a l'ex-directeur du labo, Vladimir Georgescu, une personnalit\'e dont la pr\'esence seule suffit pour assurer une ambiance positive dans la salle. La proposition suivante n'est pas loin de la v\'erit\'e. Vladimir, il connait tout. Je n'ai jamais rencontr\'e quelqu'un avec de si grandes connaissances en math\'ematiques, on dirait un wikip\'edia de math;)

\bigskip
Merci \`a Vahagn Nersesyan, avec qui, entre autres, j'ai eu la chance d'\'ecrire un article. Je le remercie \'egalement pour son soutien.

\medskip
Je remercie mes amis arm\'eniens Alexandr Mkrtchyan, Davit Avetisyan, Gor Poghosyan qui ont \'et\'e avec moi malgr\'e la distance.

\medskip
Je remercie finalement ma famille. Merci \`a mon fr\`ere Vahagn et \`a ma m\`ere Nune pour tout ce qu'ils ont fait pour moi.
}
\clearpage % Start a new page

%----------------------------------------------------------------------------------------
%	QUOTATION PAGE
%----------------------------------------------------------------------------------------

\pagestyle{empty} % No headers or footers for the following pages

%\null\vfill % Add some space to move the quote down the page a bit

%\textit{``Thanks to my solid academic training, today I can write hundreds of words on virtually any topic without possessing a shred of information, which is how I got a good job in journalism."}

%\begin{flushright}
%Dave Barry
%\end{flushright}

%\vfill\vfill\vfill\vfill\vfill\vfill\null % Add some space at the bottom to position the quote just right

%\clearpage % Start a new page

%----------------------------------------------------------------------------------------
%	LIST OF CONTENTS/FIGURES/TABLES PAGES
%----------------------------------------------------------------------------------------

\pagestyle{fancy} % The page style headers have been "empty" all this time, now use the "fancy" headers as defined before to bring them back

%\lhead{\emph{Table des mati\`eres}} % Set the left side page header to "Contents"
\tableofcontents % Write out the Table of Contents

%\lhead{\emph{List of Figures}} % Set the left side page header to "List of Figures"
%\listoffigures % Write out the List of Figures

%\lhead{\emph{List of Tables}} % Set the left side page header to "List of Tables"
%\listoftables % Write out the List of Tables

%----------------------------------------------------------------------------------------
%	ABBREVIATIONS
%----------------------------------------------------------------------------------------

\clearpage % Start a new page

\setstretch{1.5} % Set the line spacing to 1.5, this makes the following tables easier to read

\setstretch{1} % Return the line spacing back to 1.3

\addtocontents{toc}{\vspace{2em}} % Add a gap in the Contents, for aesthetics

%----------------------------------------------------------------------------------------
%	THESIS CONTENT - CHAPTERS
%----------------------------------------------------------------------------------------

\mainmatter % Begin numeric (1,2,3...) page numbering

\pagestyle{fancy} % Return the page headers back to the "fancy" style

% Include the chapters of the thesis as separate files from the Chapters folder
% Uncomment the lines as you write the chapters

% Chapter 1

\chapter{Introduction} % Main chapter title

\label{Chapter1} % For referencing the chapter elsewhere, use \ref{Chapter1} 

%{Chapitre 1. \emph{Introduction}}

 % This is for the header on each page - perhaps a shortened title

%----------------------------------------------------------------------------------------
Cette th\`ese est consacr\'ee \`a l'\'etude des propri\'et\'es ergodiques et des grandes d\'eviations pour l'\'equation des ondes non lin\'eaire avec bruit blanc. On commence ce chapitre par une description dans un cadre abstrait des probl\`emes qu'on a \'etudi\'es. On pr\'esentera ensuite les r\'esultats ant\'erieurs sur ces probl\`emes et on d\'ecrira le mod\`ele pr\'ecis consid\'er\'e.

\section{Les probl\`emes dans le cadre abstrait}
{\it Ergodicit\'e.} Consid\'erons une \'equation d'\'evolution abstraite 
\be\label{c0.1}
\dt \uu(t)+\bB(\uu(t))=\vartheta(t)
\ee
dont le flot $\uu(t)$ appartient \`a un espace $X$. Ici $\mathfrak{b}$ est un op\'erateur non lin\'eaire
et $\vartheta=\vartheta^\omega$ est une force al\'eatoire, c'est-\`a-dire une fonction qui d\'epend d'une variable al\'eatoire $\om$. Soit $\vv$ une condition initiale du flot $\uu(t)$ qui peut \^etre elle-m\^eme al\'eatoire et soit $\mu(t)$ la loi de solution $\uu(t)$. Le probl\`eme d'ergodicit\'e consiste en les trois questions suivantes. \mpp{Existence de mesure stationnaire}: Est-ce qu'il existe une solution $\uu(t)$ dont la loi $\mu=\mu(t)$ est ind\'ependante de $t$? \mpp{Unicit\'e}: Si une telle solution existe, est-elle unique, i.e., si $\uu'(t)$ est une autre solution dont la loi $\mu'=\mu'(t)$ est stationnaire, alors a-t-on $\mu=\mu'$?
\mpp{M\'elange}: Si la mesure stationnaire $\mu$ est unique, est-elle m\'elangeante, i.e., attirant la loi de toute solution de l'\'equation \ef{c0.1}?

\medskip
Ainsi, lorsque $\mu$ est une mesure m\'elangeante pour \ef{c0.1}, on a
$$
\e\psi(\uu(t))\to \int_X \psi(\vv)\mu(\Dd\vv)\q\text{ lorsque }t\to\iin
$$
pour une large classe de fonctions $\psi:X\to \rr$ h\"old\'eriennes,  o\`u $\e\ph$ d\'esigne l'esp\'erance de $\ph$. Une autre question importante est la vitesse de cette convergence. On peut montrer en particulier, que si cette convergence est ``suffisamment rapide" alors on a \mpp{la loi forte des grands nombres}: pour tout $\vv\in X$ et $\De>0$ on a la convergence
$$
\lim_{t\to\iin}t^{-\f{1}{2}+\De}\left(t^{-1}\int_0^t\psi(\uu(s))\dd s-\int_X \psi(\vv)\mu(\Dd\vv)\right)=0
$$
avec $\pp_\vv$-probabilit\'e 1. On a \'egalement \mpp{le th\'eor\`eme de la limite centrale}:  si $\psi$ est une fonction h\"old\'erienne non identiquement nulle telle que $(\psi, \mu)=0$, alors pour tout $\vv\in X$, on a 
$$
\ddd_\vv\left(\f {1} {\sqrt{t}}\int_0^t\psi(\uu(s))\dd s\right)\to \nnn(0, \sigma_\psi)\q\text{lorsque } t\to\iin,
$$
o\`u $\ddd_\vv$ d\'esigne la distribution d'une variable sous la loi $\pp_\vv$ et $\sigma_\psi>0$ est une constante qui ne d\'epend que de $\psi$.

\bigskip
{\it Grandes d\'eviations pour les mesures stationnaires.} Soit $\mu^\es$ une unique mesure stationnaire de l'\'equation
\be\label{c0.4}
\dt \uu(t)+\mathfrak{b}(\uu(t))=\sqrt{\es}\,\vartheta(t),
\ee
o\`u $\vartheta$ est un bruit blanc.
La question est de savoir si on peut d\'ecrire le comportement asymptotique de la famille $(\mu^\es)$ lorsque $\es$ tend vers z\'ero. Plus pr\'ecis\'ement, on s'int\'eresse \`a l'existence d'une fonction $\vvv:X\to [0, \iin]$ qui d\'ecrit cette asymptotique dans le sens suivant. Pour tout ensemble bor\'elien $\Gamma$ de $X$, on a 
\be\label{c0.2}
-\inf_{\uu\in \dot\Gamma}\vvv(\uu)\le\liminf_{\es\to 0}\es\ln \mu^\es(\Gamma)\le\limsup_{\es\to 0}\es\ln \mu^\es(\Gamma)\le -\inf_{\uu\in \bar\Gamma}\vvv(\uu),
\ee
o\`u on d\'esigne par $\dot\Gamma$ et $\bar\Gamma$ son int\'erieur et sa fermeture, respectivement. 

\medskip
Lorsque la fonction $\vvv(\uu)$ est \`a niveau compact, c'est-\`a-dire l'ensemble $\{\vvv\le M\}$ est compact dans $X$ pour tout $M\ge 0$, on dit que la famille $(\mu^\es)$ satisfait le principe de grandes d\'eviations avec bonne fonction de taux $\vvv$. Remarquons que si $(\mu^\es)$ satisfait le principe de grandes d\'eviations avec la fonction $\vvv$ dont le noyau est fini, $\text{ker}(\vvv)=\{\hat\uu_1, \ldots, \hat\uu_\ell\}$, alors chaque limite faible $\mu_*$ de la famille $(\mu^\es)$ est de forme
$$
\mu_*=\sum_{i=1}^\ell c_i\De_{\hat\uu_i},\q 0\le c_i\le 1.
$$
On verra dans la suite que c'est le cas lorsque l'\'equation $\bB(\uu)=0$ a un nombre fini de solutions. Plus pr\'ecis\'ement, on a $\text{ker}(\vvv)\subset \ker(\bB)$ et $\vvv(\hat\uu_i)=0$ implique que $\hat\uu_i$ est asymptotiquement stable pour l'\'equation limite
\be\label{c0.7}
\dt\uu(t)+\bB(\uu(t))=0,
\ee
\`a savoir, toute solution $\uu(t)$ de \ef{c0.7} issue d'un petit voisinage de $\hat\uu_i$ tend vers $\hat\uu_i$ lorsque $t\to\iin$. Il est important de noter que l'inverse n'est pas vrai et la situation suivante est possible: la famille $\mu^\es$ converge faiblement vers la masse de Dirac concentr\'ee en $\hat\uu\in\{\hat\uu_1, \ldots, \hat\uu_\ell\}$, alors que $\hat\uu$ n'est pas le seul \'equilibre qui est stable asymptotiquement. Voici un exemple: on consid\`ere \ef{c0.4} dans $\rr$ avec $\mathfrak{b}(\uu)=\uu(\uu-1)(\uu-3)$. Alors $\mu^\es$ converge faiblement vers la masse de Dirac concentr\'ee en $\uu=3$ alors que $\uu=0$ est un \'equilibre asymptotiquement stable pour l'\'equation limite \ef{c0.7}. La formule explicite de la bonne fonction de taux correspondante donn\'ee plus bas (voir \ef{c0.8}) permet de v\'erifier cet exemple.

\bigskip
{\it Grandes d\'eviations pour les mesures d'occupation.} Soit $\uu(t)$ une solution de l'\'equation \ef{c0.1} et soit $\zeta(t)$ la mesure d'occupation associ\'ee, c'est \`a dire, 
\be\label{c0.3}
\zeta(t)=\f{1}{t}\int_0^t \De_{\uu(\tau)}\dd \tau,\q t>0,
\ee
o\`u $\De_\vv$ d\'esigne la masse de Dirac en $\vv$. Le probl\`eme consiste \`a d\'ecrire le comportement asymptotique de la famille $(\zeta(t))$ lorsque $t\to\iin$. On veut savoir s'il existe une fonction $I$ d\'efinie sur l'ensemble $\ppp(X)$ des mesures de probabilit\'e sur $X$ et \`a valeurs dans $[0, \iin]$ telle qu'on a 
\be\label{c0.6}
-\inf_{\lm\in \dot\Gamma}I(\lm)\le\liminf_{t\to\iin}\f{1}{t}\ln \pp(\zeta(t)\in\Gamma)\le\limsup_{t\to\iin}\f{1}{t}\ln \pp(\zeta(t)\in\Gamma)\le -\inf_{\lm\in \bar\Gamma}I(\lm)
\ee
pour tout ensemble bor\'elien $\Gamma$ de $\ppp(X)$.

\section{R\'esultats ant\'erieurs}
{\it R\'esultats ant\'erieurs sur l'ergodicit\'e.}
Dans le cadre des EDP stochastiques, le probl\`eme d'existence d'une mesure stationnaire a \'et\'e \'etudi\'e par Vishik--Fursikov--Komech \cite{VKF-1979} pour le syst\`eme de Navier-Stokes, et a \'et\'e ensuite d\'evelopp\'e pour d'autres probl\`emes (voir les r\'ef\'erences dans  \cite{DZ1992}). Dans le cas d'EDP stochastique dissipatives, l'argument de Bogolyubov-Krylov permet de construire une telle mesure si on arrive \`a montrer qu'il existe un point initial $\vv\in X$ tel que la moyenne de solution $\uu(t)$ issue de $\vv$ est born\'ee uniform\'ement en temps dans un espace plus r\'egulier que l'espace de phase $X$. L'unicit\'e et le m\'elange sont des questions beaucoup plus d\'elicates car une description plus compl\`ete est n\'ecessaire sur le comportement asymptotique des solutions en temps grands. Les premiers r\'esultats dans cette direction ont \'et\'e \'etablis dans les papiers \cite{FM-1995, KS-cmp2000, EMS-2001, BKL-2002} consacr\'es \`a l'\'etude du syst\`eme de  Navier--Stokes et d'autres EDP provenant de la physique math\'ematique (voir aussi \cite{MR1245306, MR1641664} pour les \'equations paraboliques en 1D). Ces r\'esultats ont \'et\'e ensuite \'etendus aux \'equations avec un bruit multiplicatif et tr\`es d\'eg\'en\'er\'e \cite{odasso-2008,HM-2008}. Voir le livre \cite{KS-book} et l'article \cite{debussche2013ergodicity} pour un compte rendu d\'etaill\'e des r\'esultats obtenus dans cette direction. 

\medskip
Dans le cadre des \'equations dispersives il y a moins de r\'esultats connus.
Un des premiers r\'esultats sur l'ergodicit\'e des EDP dispersives a \'et\'e \'etabli dans l'article de E, Khanin, Mazel et Sinai \cite{EWKMS2000}, o\`u les auteurs prouvent l'existence et l'unicit\'e de mesure stationnaire pour l'\'equation de Burgers non visqueuse perturb\'ee par un bruit blanc p\'eriodique dans l'espace.
L'ergodicit\'e de l'\'equation des ondes a \'et\'e \'etudi\'e par Barbu et Da Prato \cite{BD-2002}. Ici, les auteurs prouvent l'existence d'une mesure stationnaire pour une nonlin\'earit\'e qui est une fonction non d\'ecroissante et satisfaisant aux conditions de dissipativit\'e standards. Ils \'etablissent aussi un r\'esultat d'unicit\'e mais sous une hypoth\`ese assez restrictive sur la nonlin\'earit\'e.  Dans l'article par Debussche et Odasso \cite{DO-2005}, les auteurs \'etablissent la convergence vers l'\'equilibre pour l'\'equation de Schr\"odinger non lin\'eaire 1D avec une dissipation d'ordre z\'ero. Dans \cite{DirSoug2005}, Dirr et Souganidis \'etudient les \'equations de  Hamilton-Jacobi perturb\'ees par un bruit additif. Ils prouvent, en particulier, que sous les hypoth\`eses appropri\'ees sur l'hamiltonien, l'\'equation stochastique poss\`ede  \`a une constante pr\`es une unique solution qui est p\'eriodique en espace et attirant toute autre solution, si l'\'equation sans bruit elle-m\^eme poss\`ede une telle solution.
Dans l'article relativement r\'ecent par Debussche et Vovelle \cite{debussche2013invariant}, l'existence et l'unicit\'e d'une mesure stationnaire est \'etudi\'ee pour les lois de conservations scalaires du premier ordre avec un bruit additif. C'est une g\'en\'eralisation des r\'esultats obtenus dans \cite{EWKMS2000} au cas d'une dimension quelconque.

\bigskip
{\it R\'esultats ant\'erieurs sur grandes d\'eviations pour les mesures stationnaires.} \`A notre connaissance, dans le cadre des EDP,  il n'y a que deux r\'esultats ant\'erieurs sur ce probl\`eme. Ce sont les articles de Sowers \cite{sowers-1992b} et Cerrai-R\"ockner \cite{CeRo2005} o\`u le principe de grandes d\'eviations est \'etabli pour les mesures stationnaires de l'\'equation de r\'eaction-diffusion.  Dans le premier, la force est une perturbation non gaussienne, alors que le deuxi\`eme traite le cas d'un bruit multiplicatif. Dans les deux articles, l'origine est le seul \'equilibre de l'\'equation limite, et le bruit est suffisamment irr\'egulier en espace.  On d\'ecrira avec plus de d\'etails plus loin les r\'esultats de grandes d\'eviations du type Freidlin-Wentzell dans le cadre des EDO.

\medskip
Tous les autres r\'esultats sur les grandes d\'eviations pour des EDP avec un bruit dont l'amplitude tend vers z\'ero concernent principalement les solutions du probl\`eme de Cauchy. C'est la direction la plus d\'evelopp\'ee. Les EDP stochastiques \'etudi\'ees dans ce contexte incluent l'\'equation de r\'eaction-diffusion \cite{sowers-1992a, CerRoc2004}, le syst\`eme de Navier-Stokes en 2D, \cite{Chang1996, SrSu2006}, l'\'equation de Schr\"odinger non lin\'eaire \cite{Gautier2005-2}, l'\'equation de Allen-Cahn  \cite{HairWeb2014},  les \'equations quasi-g\'eostrophiques \cite{LiuRocZhuChan2013}, les \'equations avec un drift monotone g\'en\'eral \cite{Liu2010}, et lois de conservations \cite{Mariani2010}. Voir aussi les articles \cite{KalXi1996, CheMil1997, CarWeb1999, CM-2010} pour des r\'esultats dans un cadre plus abstrait qui couvre une large classe d'EDP stochastiques, y compris les mod\`eles hydrodynamiques en 2D. Une autre direction bien d\'evelopp\'ee est l'\'etude du probl\`eme de sortie pour les trajectoires. Voir les articles \cite{Peszat1994, CheZhiFre2005, Gautier2008, FreKor2010, CerSal2014, BreCerFre2014}.

\bigskip
{\it R\'esultats ant\'erieurs sur grandes d\'eviations pour les mesures d'occupation.}
Dans ce cas aussi, il y a tr\`es peu de travaux. Les deux premiers r\'esultats dans cette direction ont \'et\'e obtenus par Gourcy \cite{gourcy-2007a,gourcy-2007b}, o\`u l'auteur prouve le principe de grandes d\'eviations pour les mesures d'occupation des \'equations de Burgers et de Navier-Stokes sous l'hypoth\`ese que la force est un bruit blanc qui est suffisamment irr\'egulier en espace. Il y a eu un progr\`es r\'ecent gr\^ace aux travaux de Jak$\check{\rm s}$i\'c, Nersesyan, Pillet et Shirikyan \cite{JNPS-2012, JNPS-2014} o\`u les auteurs ont \'etabli les grandes d\'eviations pour les \'equations avec forte dissipation du type Navier-Stokes et Ginzburg-Landau, avec une force discr\`ete born\'ee ou non born\'ee sans imposer une hypoth\`ese de faible r\'egularit\'e.

\bigskip
{\it Grandes d\'eviations pour les mesures stationnaires dans le cadre des EDO.} Consid\'erons l'\'equation \ef{c0.4} dans le cas de la dimension finie, c'est \`a dire lorsque $\uu(t)\in \rr^n$. Si la fonction $\bB(\uu)$ a un potentiel, i.e., si elle est de la forme $\bB(\uu)=\g \mathfrak{a} (\uu)$, alors la mesure stationnaire $\mu^\es$ existe si et seulement si 
$$
\int_{\rr^n}e^{-2\aA(\uu)}\dd\uu<\iin
$$
et sa densit\'e est donn\'ee par
$$
m^\es(\uu)=(\int_{\rr^n}e^{-2\aA(\vv)/\es}\dd\vv)^{-1}e^{-2\aA(\uu)/\es}.
$$
Gr\^ace au principe de Laplace, on voit que la famille $(\mu^\es)$ satisfait le principe de grandes d\'eviations avec la bonne fonction de taux $\vvv:\rr^n\to [0, \iin)$ donn\'ee par
\be\label{c0.8}
\vvv(\uu)=2(\aA(\uu)-\inf_{\rr^n}\aA).
\ee
Lorsque la fonction $\bB$ n'a pas de potentiel, la description du comportement asymptotique de la famille $(\mu^\es)$ n'est plus triviale. Dans le cas o\`u $\bB$ admet un seul point stationnaire, on peut utiliser les techniques d\'evelopp\'ees par Sowers \cite{sowers-1992b} et Cerrai-R\"ockner \cite{CeRo2005} pour d\'ecrire ces asymptotiques. Supposons que $\hat\uu\in\rr^n$ est l'unique solution de $\bB(\uu)=0$. Alors la famille $(\mu^\es)$ satisfait le principe de grandes d\'eviations avec une bonne fonction de taux $\vvv:\rr^n\to [0, \iin]$, o\`u $\vvv(\uu)$ repr\'esente l'\'energie minimale n\'ecessaire pour atteindre le point $\uu$ \`a partir de $\hat\uu$ en un temps fini. Remarquons que dans ce cas $\mu^\es$ converge faiblement vers la masse de Dirac concentr\'ee en $\hat\uu$. Si $\bB$ a plusieurs (mais en nombre fini) points stationnaires, on utilise la th\'eorie d\'evelopp\'ee par Freidlin et Wentzell pour r\'esoudre ce probl\`eme. Sans entrer dans les d\'etails, on va d\'ecrire ici la bonne fonction de taux $\vvv$ correspondante. Soient $\hat\uu_1, \ldots, \hat\uu_\ell$ les solutions de $\bB(\uu)=0$.
\'Etant donn\'es $\ell\in\nn$ et $i\leq\ell$, 
on notera $G_{\ell}(i)$ l'ensemble de graphes compos\'es de fl\`eches 
$$
(m_1\to m_2 \to\cdots \to m_{\ell-1}\to m_\ell)
$$
tel que
$$
\{m_1,\ldots, m_\ell\}=\{1,\ldots, \ell\}\q\text{ et } m_\ell=i.
$$
On introduit
\be\label{c0.5}
W_\ell(\hat\uu_i)=\min_{\mathfrak{g}\in G_{\ell}(i)}\sum_{(m\to n)\in \mathfrak{g}} V(\hat\uu_m, \hat\uu_n),
\ee
o\`u $V(\uu_1, \uu_2)$ est l'\'energie minimale n\'ecessaire pour atteindre $\uu_2$ \`a partir de $\uu_1$ en un temps fini. La fonction $\vvv:\rr^n\to [0, \iin]$ est d\'efinie par
\be\label{c0.6}
\vvv(\uu)=\min_{i\leq \ell}[W_\ell(\hat \uu_i)+V(\hat\uu_i, \uu)]-\min_{i\leq\ell}W_\ell(\hat\uu_i).
\ee
En d\'eveloppant l'approche introduite par Freidlin et Wentzell, on montrera que ce r\'esultat reste vrai dans le cadre des EDP.

\section{Le mod\`ele d\'eterministe en consid\'eration} \label{c1.13}
On commence cette section par la description du mod\`ele  d\'eterministe (i.e., sans bruit) qu'on \'etudie et on pr\'esente quelques propri\'et\'es importantes des solutions.
Le mod\`ele sans perturbation est donn\'e par l'\'equation des ondes amortie
\be\label{c1.1}
\p_t^2u+\gamma \p_tu-\de u+f(u)=h(x)
\ee
dans un domaine born\'e $D\subset \rr^3$, o\`u $\gamma>0$  est le param\`etre d'amortissement, $h(x)$ est une fonction dans $H^1_0(D)$, et $f(u)$ est la nonlin\'earit\'e. Ici et dans le chapitre suivant, on ne va pas imposer d'hypoth\`eses g\'en\'erales sur $f$. Nous nous limitons \`a consid\'erer le cas de fonctions 
\be\label{c1.2}
f(u)=|u|^\rho u-\lm u,
\ee
o\`u $\rho\in (0, 2)$ et $\lm\in \rr$, qui mod\'elise l'\'equation de Klein-Gordon et pour laquelle ces hypoth\`eses sont satisfaites. On va noter $F$ la primitive de $f$. On \'ecrira $\uu(t)$ pour le flot de l'\'equation \ef{c1.1} \`a l'instant $t$, \`a savoir, $\uu(t)=[u(t), \dt u(t)]$. Le point initial $\uu(0)$ appartient \`a l'espace de phase $\h=H^1_0(D)\times L^2(D)$ qui est muni de la norme
$$
|\uu|_{\h}^2=\|\g u_1\|^2+\|u_2+\al u_1\|^2\q \text{ pour } \uu=[u_1,u_2]\in\h,
$$
o\`u $\al=\al(\gamma)>0$ est un petit param\`etre et $\|\cdot\|$ d\'esigne la norme dans $L^2(D)$. 

\medskip
{\it Estimations \`a priori.} On introduit la fonctionnelle d'\'energie de l'\'equation \ef{c1.1} donn\'ee par
\be\label{c1.3}
\ees(\uu)=|\uu|_\h^2+2\int_D F(u_1)\dd x\q\text{ pour }\uu=[u_1, u_2]\in\h.
\ee
La proposition suivante permet de montrer, en particulier, que le probl\`eme de Cauchy pour l'\'equation \ef{c1.1} est bien pos\'e dans $\h$. La preuve est standard.
\begin{proposition-fr}\label{c1.7}
Sous les hypoth\`eses ci-dessus, pour tout $\uu_0\in \h$, le flot $\uu(t)$ de l'\'equation \ef{c1.1} issu de $\uu_0$ v\'erifie
\be\label{c1.4}
\ees(\uu(t))\le \ees(\uu_0)e^{-\al t}+C(\gamma, \|h\|)\q\text{ pour tout } t\ge 0. 
\ee
\end{proposition-fr}
Un autre r\'esultat important concerne la r\'egularit\'e des solutions. Il est bien connu que l'op\'erateur r\'esolvant de l'\'equation \ef{c1.1} n'a pas d'effet r\'egularisant. N\'eanmoins, il y a un effet r\'egularisant asymptotique qui est central dans l'\'etude. Pour \'enoncer le r\'esultat, on introduit un espace $\h^s$ plus r\'egulier que l'espace des phases $\h$, $\h^s=[H^{s+1}(D)\cap H^1_0(D)]\times H^s(D)$, o\`u $H^s$ est l'espace de Sobolev d'ordre $s$, muni de la norme naturelle $\|\cdot\|_s$. Ici $s$ est une constante fix\'ee dans l'intervalle $(0, 1-\rho/2)$.
\begin{proposition-fr}\label{c1.5}
Sous les hypoth\`eses ci-dessus, pour tout $\uu_0\in \h$, le flot $\uu(t)$ de \ef{c1.1} avec le point initial $\uu_0$ peut \^etre d\'ecompos\'e en la somme $\uu(t)=\vv(t)+\Zz(t)$, tel qu'on ait 
\be\label{c1.6}
|\vv(t)|_{\h}^2\le |\uu_0|_\h^2 e^{-\al t},\q \q |\Zz(t)|_{\h^s}^2\le C(|\uu_0|_\h, \|h\|).
\ee
\end{proposition-fr}

\medskip
{\it Esquisse de la preuve}. Soit $\uu_0\in\h$. On consid\`ere le flot $\vv(t)$ de l'\'equation lin\'eaire
\be\label{c1.8}
\p_t^2 v+\gamma \p_t v-\de v=0
\ee
issu du point $\uu_0$. En multipliant cette \'equation par $\dt v+\al v$ et en int\'egrant sur $D$, on obtient
$$
\p_t |\vv(t)|_\h^2\le -\al |\vv(t)|_\h^2\q\text{ pour } t\ge 0,
$$
d'o\`u la premi\`ere in\'egalit\'e dans \ef{c1.6}. On pose $\Zz(t)=\uu(t)-\vv(t)$. D'apr\`es \ef{c1.1} et \ef{c1.6}, cette fonction correspond au flot de l'\'equation
\be\label{c1.9}
\p_t^2 z+\gamma \p_t z-\de z+f(u)=h(x)
\ee
issu de l'origine. En d\'erivant une fois cette \'equation par rapport \`a $t$, et en utilisant quelques transformations standards ainsi que les injections de Sobolev et l'in\'egalit\'e \ef{c1.4}, on peut voir que la fonction $a(t)=\dt z(t)$ satisfait l'in\'egalit\'e $|[a(t), \dt a(t)]|_{\h^{s-1}}\le C(|\uu_0|_\h, \|h\|)$. Pour finir la preuve, il suffit de remarquer que gr\^ace \`a \ef{c1.9}, on a $\de z=\dt a+\gamma a+f(u)-h$ et donc $\|z(t)\|_{s+1}\le C'(|\uu_0|_\h, \|h\|)$.

\medskip
{\it Estimations du type Foia\c{s}-Prodi.} Le r\'esultat suivant montre que le comportement des solutions en temps grand est d\'etermin\'e par les basses fr\'equences. On consid\`ere les flots $\uu(t)$ et $\vv(t)$ des \'equations
\begin{align}
\p_t^2 u+\gamma \p_t u-\de u+f(u)&=h(x)+\ph(t,x),\label{FP"IN"1}\\
\p_t^2 v+\gamma \p_t v-\de v+f(v)+P_N[f(u)-f(v)]&=h(x)+\ph(t,x)\label{FP"IN"2},
\end{align}
sur un intervalle $[0, T]$ quelconque, o\`u $\varphi$ est soit une fonction dans l'espace $L^2(0, T; L^2(D))$, soit une fonction dont la primitive est dans $C(0, T; H^1_0(D))$. Ici $P_N$ d\'esigne la projection orthogonale de $L^2(D)$ sur l'espace engendr\'e par les $N$ premi\`eres fonctions propres $e_1(x), \ldots, e_N(x)$ de l'op\'erateur de Laplace. Supposons qu'il existe des constantes positives $l$ et $K$ telles que l'in\'egalit\'e
\be\label{c1.10}
\int_0^t|\Zz(s)|_\h^2\dd s\leq l+Kt \q \text{ pour tout } t\in [0, T]
\ee
soit vraie pour $\Zz=\uu$ et $\Zz=\vv$.
\begin{proposition-fr}\label{4.13}Sous les hypoth\`eses ci-dessus, pour tout $\es>0$ il existe un entier $N_*$ qui ne d\'epend que de $\es$ et $K$ tel que pour tout $N\ge N_*$, on a 
\be \label{4.16}
|\vv(t)-\uu(t) |^2_\h\leq e^{-\al t+\es l}|\vv(0)-\uu(0)|^2_\h \q \text{ pour tout } t\in [0, T].
\ee
\end{proposition-fr}
Pour la preuve voir Chapitre \ref{Chapter3}.
 \section{Mesures et couplage}
Soit $X$ un espace polonais, i.e. un espace m\'etrique s\'eparable complet, et soit $\bbb(X)$ la tribu des sous-ensembles bor\'eliens de $X$. 
On notera $\ppp(X)$ l'ensemble des mesures de probabilit\'e sur $X$. \'Etant donn\'ees deux mesures $\mu, \mu'\in \ppp(X)$, on d\'efinit la distance variationnelle entre $\mu$ et $\mu'$ donn\'ee par la formule
$$
|\mu-\mu'|_{var}=\sup_{\Gamma\in \bbb(X)}|\mu(\Gamma)-\mu'(\Gamma)|.
$$
Il est facile de voir que $(\ppp(X), |\cdot|_{var})$ est un espace m\'etrique complet.
Soit $C_b(X)$ l'espace des fonctions continues born\'ees sur $X$ et soit $\psi\in C_b(X)$. On introduit
$$
(\psi, \mu)=\int_X \psi(\uu)\mu(\Dd\uu).
$$
On peut montrer qu'on a l'\'egalit\'e
$$
|\mu-\mu'|_{var}=\f{1}{2}\sup_{|\psi|_{\iin}\le 1}|(\psi,\mu)-(\psi, \mu')|,
$$
o\`u $|\cdot|_\iin$ est la norme naturelle de $C_b(X)$. 
\begin{definition-fr} Soient $\mu_1, \mu_2\in\ppp(X)$. Une paire $(\zeta_1, \zeta_2)$ de variables al\'eatoires d\'efinies sur le m\^eme espace probabilis\'e est appel\'ee \mpp{un couplage} pour $(\mu_1, \mu_2)$ si 
$$
\ddd\zeta_i=\mu_i \q\text{ pour } i=1, 2,
$$
o\`u $\ddd\zeta$ d\'esigne la loi d'une variable $\zeta$.
\end{definition-fr}
Remarquons que gr\^ace \`a l'in\'egalit\'e
$$
\mu_1(\Gamma)-\mu_2(\Gamma)=\pp(\zeta_1\in\Gamma)-\pp(\zeta_2\in\Gamma)=\e(\ch_{\zeta_1\neq \zeta_2}[\ch_{\Gamma}(\zeta_1)-\ch_{\Gamma}(\zeta_2)])\le \pp(\zeta_1\neq\zeta_2)
$$
on a toujours
$$
|\mu_1-\mu_2|_{var}\le  \pp(\zeta_1\neq\zeta_2).
$$
\begin{definition-fr} Un couplage $(\zeta_1, \zeta_2)$ de $(\mu_1, \mu_2)$ est dit \mpp{maximal}, si
$$
|\mu_1-\mu_2|_{var}=  \pp(\zeta_1\neq\zeta_2)
$$
et si les variables al\'eatoires $\zeta_1$ et $\zeta_2$ conditionn\'ees sur l'\'ev\'enement $\nnn=\{\zeta_1\neq \zeta_2\}$ sont ind\'ependantes, c'est-\`a-dire
$$
\pp(\zeta_1\in\Gamma_1, \zeta_2\in\Gamma_2\,|\,\nnn)=\pp(\zeta_1\in\Gamma_1\,|\,\nnn)\,\pp(\zeta_2\in\Gamma_2\,|\,\nnn)
$$ 
pour tout $\Gamma_1, \Gamma_2\in \bbb(X)$.
\end{definition-fr}
Le th\'eor\`eme suivant est fondamental. Voir \cite{KS-book} pour la preuve.
\begin{theorem-fr}
Tout couple $(\mu_1, \mu_2)$ de mesures de probabilit\'e sur un espace polonais admet un couplage maximal $(\zeta_1, \zeta_2)$.
\end{theorem-fr}
\section{Convergence faible des mesures}
\begin{definition-fr}
On dit qu'une suite de mesures $(\mu_n)\subset\ppp(X)$ converge faiblement vers $\mu\in\ppp(X)$ si 
$$
(\psi, \mu_n)\to (\psi, \mu)\q\text{ lorsque }n\to\iin
$$
pour tout $\psi\in C_b(X)$. 
\end{definition-fr}
La topologie de la convergence faible est m\'etrisable. Elle est engendr\'ee par la m\'etrique
$$
|\mu-\mu'|_L^*=\sup_{|\psi|_L\le 1}|(\psi,\mu)-(\psi, \mu')|
$$
o\`u on pose
$$
|\psi|_L=|\psi|_\iin+\sup_{\uu\neq \vv}\f{|\psi(\uu)-\psi(\vv)|}{d_X(\uu, \vv)}.
$$
De toute \'evidence, cette topologie est plus faible que celle engendr\'ee par la convergence en variation. On peut montrer que $(\ppp(X), |\cdot|_L^*)$ est un espace polonais (voir \cite{KS-book}).

\section{Processus de Markov}
Soit $X$ un espace polonais. \mpp{Un espace probabilis\'e filtr\'e} est un espace mesurable $(\Omega, \fff)$ avec une famille $(\fff_t)_{t\in\ttt_+}\subset\fff$ croissante de $\sigma$-alg\`ebres. 
\begin{definition-fr}\label{c1.14}\mpp{Un processus de Markov dans} $X$ est une collection d'objets suivants. 
\bi
\item un espace mesurable $(\Omega, \fff)$ avec une filtration $(\fff_t)_{t\ge 0}$;
\item une famille de mesures de probabilit\'e $(\pp_\vv)_{\vv\in X}$ sur $(\Omega, \fff)$ telle que l'application $\vv\mapsto \pp_\vv(A)$ est universellement mesurable  \footnote{Une fonction $g:\Omega\to \rr$ est dite universellement mesurable, si pour tout $a\in\rr$, l'ensemble $\{g\le a\}$ appartient \`a la compl\'etion de $\fff$ par rapport \`a toute mesure de probabilit\'e sur $(\Omega, \fff)$.}  pour tout $A\in\fff$;
\item un processus $(\uu(t))_{t\in\ttt_+}$ \`a valeur dans $X$ adapt\'e \`a la filtration $\fff_t$ (i.e., $\uu(t)$ est $\fff_t$-mesurable pour tout $t\ge 0$) et qui v\'erifie
\begin{align*}
\pp_\vv\{\uu(0)=\vv\}&=1,\\
\pp_\vv\{\uu(t+s)\in \Gamma\,|\,\fff_s\}&=P_t(\uu(s), \Gamma) \q\q\pp_v\text{-presque partout}
\end{align*}
pour tout $\vv\in X$, $\Gamma\in\bbb(X)$ et $s, t\ge 0$. Ici $P_t$ d\'esigne la probabilit\'e de transition donn\'ee par $P_t(\vv, \Gamma)=\pp_\vv\{\uu(t)\in\Gamma\}$.
\ei
\end{definition-fr}
\'Etant donn\'e un processus de Markov, on peut lui associer deux familles $(\PPPP_t)$ et $(\PPPP^*_t)$ d'op\'erateurs (appel\'es \mpp{semigroupes de Markov}) agissant, respectivement, sur l'espace $b(X)$ des fonctions mesurables born\'ees et l'espace $\ppp(X)$ des mesures de probabilit\'e. On les d\'efinit par
\begin{align*}
\PPPP_t\psi(\uu)&=\int_\h\psi(\Zz)P_t(\uu,\Dd\Zz) \q \,\text{pour tout } \psi\in b(X),\\
\PPPP^*_t\lm(\Gamma)&=\int_\h P_t(\uu,\Gamma)\lm(\Dd\uu)\q \text{pour tout } \lm\in\ppp(X).
\end{align*}
On a $\PPPP_0=Id$, $\PPPP_{t+s}=\PPPP_t\circ\PPPP_s$, ainsi que
$$
(\PPPP_t\psi, \lm)=(\psi, \PPPP_t^*\lm)\q\text{ pour tout }\psi\in b(X) \text{ et } \lm\in\ppp(X).
$$
\section{Extension d'un processus de Markov}\label{EdpdM}
Soit $(\uu(t), \pp_\vv)$ un processus de Markov dans $X$. On consid\`ere le produit direct $\bar X=X\times X$ et les op\'erateurs $\Pi$ et $\Pi'$ de projections naturelles de $\bar X$ sur $X$, i.e., 
$$
\Pi(\bar\uu)=\uu, \q \Pi'\bar\uu=\uu'\q\text{ pour }\bar\uu=(\uu, \uu')\in X.
$$ 
Soit $(\bar\uu(t), \pp_{\bar\vv})$ un processus de Markov dans $\bar X$. 
\begin{definition-fr}On dira que $(\bar\uu(t), \pp_{\bar\vv})$ est \mpp{une extension de} $(\uu(t), \pp_\vv)$ si
$$
\Pi_*\bar P_t(\bar\uu, \cdot)=P_t(\uu, \cdot), \q \Pi'_* \bar P_t(\bar\uu, \cdot)=P_t(\uu', \cdot),
$$
o\`u $\bar P_t(\bar\uu, \cdot)$ est la probabilit\'e de transition du processus $(\bar\uu(t), \pp_{\bar\vv})$ et on d\'esigne par $\psi_*\lm$ l'image de $\lm$ sous $\psi$, i.e., $\psi_*\lm(\cdot)=\lm(\psi^{-1}(\cdot))$.
\end{definition-fr}
\section{Mesures stationnaires et m\'elange}
\begin{definition-fr}
Soit $(\uu(t), \pp_\vv)$ un processus de Markov dans un espace $X$. Une mesure $\mu\in\ppp(X)$ est dite \mpp{stationnaire} pour $(\uu(t), \pp_\vv)$ si $\PPPP^*_t\mu=\mu$ pour tout $t\ge 0$. Une mesure stationnaire $\mu$ est dite \mpp{m\'elangeante} si $\PPPP^*_t\lm$ converge faiblement vers $\mu$ pour tout $\lm\in \ppp(X)$, i.e.,
$$
(\PPPP_t\psi,\lm)\to (\psi, \mu)\q\text{ lorsque } t\to\iin 
$$
pour tout $\psi\in C_b(X)$.
\end{definition-fr}
\section{Principe de grandes d\'eviations}
Soit $\zzz$ un espace polonais. Une fonction $\mathfrak{I}$ d\'efinie sur $\zzz$ et \`a valeurs dans $[0,\iin]$ dont les ensembles de niveaux sont compacts, c'est-\`a-dire,  $\{\mathfrak{I}\leq M\}$ est compact dans $\zzz$ pour tout $M\geq 0$, est appel\'ee \mpp{une bonne fonction de taux}. Soit $(\mathfrak{m}^\es)_{\es>0}$ une famille de mesures de probabilit\'e sur $\zzz$. 
\begin{definition-fr}
On dit que la famille $(\mathfrak{m}^\es)_{\es>0}$ satisfait \mpp{le principe de grandes d\'eviations dans} $\zzz$ \mpp{avec une bonne fonction de taux} $\mathfrak{I}:\zzz\to [0,\iin]$ si on a
$$
-\inf_{z\in \dot\Gamma}\mathfrak{I}(z)\le\liminf_{\es\to 0}\es\ln \mathfrak{m}^{\es}(\dot\Gamma)\le \limsup_{\es\to 0}\es\ln \mathfrak{m}^{\es}(\bar\Gamma)\le-\inf_{z\in \bar\Gamma} \mathfrak{I}(z)
$$
pour tout ensemble mesurable $\Gamma\subset \zzz$, o\`u $\dot\Gamma$ et $\bar\Gamma$ d\'esignent son int\'erieur et sa fermeture, respectivement. 
\end{definition-fr}
\section{Le mod\`ele complet}\label{c1.15}
On consid\`ere l'\'equation des ondes non lin\'eaire amortie
\be\label{c1.11}
\p_t^2u+\gamma \p_tu-\de u+f(u)=h(x)+\vartheta(t,x)
\ee
dans un domaine born\'e $D\subset \rr^3$. Les hypoth\`eses sur $h$ et $f$ sont les m\^emes que dans la Section \ref{c1.13}. En ce qui concerne la force $\vartheta$, on suppose que c'est un bruit blanc de la forme
\be\label{c2.3}
\vartheta(t,x)=\sum_{j=1}^\infty b_j \dt\beta_j(t)e_j(x).
\ee
Ici, $\{\beta_j(t)\}$ est une suite de mouvements browniens standards ind\'ependants, $\{e_j\}$ est une base orthonorm\'ee de $L^2(D)$ form\'ee de fonctions de l'op\'erateur de Laplace. Les coefficients $b_j>0$ sont des nombres v\'erifiant
\be\label{c2.4}
\BBB_1=\sum_{j=1}^\iin \lm_j b_j^2<\iin,
\ee
o\`u $\lm_j$ est la valeur propre associ\'ee \`a la fonction $e_j(x)$. On commence par la d\'efinition de solution de l'\'equation \ef{c1.11}. On introduit
$$
\hat\zeta(t)=\sum_{j=1}^\iin\beta_j(t)[0, b_j e_j]\in C(\rr_+;\h).
$$
\begin{definition-fr}\label{definition 2.1}
Soit $\uu_0$ une variable al\'eatoire \`a valeurs dans $\h$ d\'efinie sur un espace probabilis\'e $(\omm,\fff,\pp)$ et qui est ind\'ependante de $\hat\zeta(t)$. Un processus al\'eatoire $\uu(t)=[u(t),\dt u(t)]$ d\'efini sur  $(\omm,\fff,\pp)$ est appel\'e la \emph{solution} (ou \emph{le flot}) de l'\'equation \eqref{c1.11} si les conditions suivantes sont satisfaites.
\bi 
\item Presque toute trajectoire $\uu(t)$ appartient \`a l'espace $C(\rr_+;\h)$, et le processus $\uu(t)$ est adapt\'e \`a la filtration $\fff_t$ engendr\'ee par $\uu_0$ et $\hat\zeta(t)$.

\item L'\'equation \eqref{c1.11} est satisfaite dans le sens o\`u, avec probabilit\'e 1, on a
\be\label{c1.12}
\uu(t)=\uu_0+\int_0^t g(s)\,ds+\hat\zeta(t),\q t\geq 0,
\ee
o\`u on pose
$$
g(t)=[\dt u,-\gamma \dt u+\de u-f(u)+h(x)],
$$
et la relation \ef{c1.12} est valable dans $L^2(D)\times H^{-1}(D)$.
\ei
\end{definition-fr}
\begin{theorem-fr}\label{2.12}
Sous les hypoth\`eses ci-dessus, soit $\uu_0$ une variable al\'eatoire \`a valeurs dans $\h$ qui est ind\'ependante de $\hat\zeta$ et qui v\'erifie $\e\ees(\uu_0)<\iin$. Alors l'\'equation \eqref{c1.11} poss\`ede une solution dans le sens de la D\'efinition \ref{definition 2.1}. De plus, cette solution est unique, \`a savoir, si $\tilde\uu(t)$ est une autre solution avec la condition initiale $\uu_0$, alors, avec $\pp$-probabilit\'e 1, on a  $\uu(t)=\tilde \uu(t)$ pour tout $t\geq 0$. Par ailleurs, on a l'estimation \`a priori suivante
\be\label{2.2}
\e\ees (\uu(t))\leq\e\ees(\uu_0)e^{-\al t}+C(\gamma,\BBB_1,\|h\|).
\ee
\end{theorem-fr}
Pour la preuve voir Chapitre \ref{Chapter3}.

\medskip
Soit $S_t(\vv,\cdot)$ le flot \`a l'instant $t$ de l'\'equation \eqref{c1.11} issu de $\vv\in\h$. On peut montrer que $S_t(\vv,\cdot)$ d\'efinit un processus de Markov (on le notera $(\uu(t), \pp_\vv)$) dans $\h$ (voir, par exemple \cite{DZ1992, KS-book}). 

 \medskip
Dans le chapitre suivant, on pr\'esente les trois principaux r\'esultats qu'on a obtenus sur ce mod\`ele. Le premier r\'esultat concerne l'existence, l'unicit\'e et le m\'elange de mesure stationnaire pour le processus  $(\uu(t), \pp_\vv)$. Le deuxi\`eme porte sur le principe de grandes d\'eviations pour les mesures stationnaires $(\mu^\es)$ associ\'ees \`a l'\'equation \ef{c1.11} o\`u $\vartheta$ doit \^etre remplac\'e par $\sqrt{\es}\,\vartheta$. Le dernier r\'esultat concerne le principe de grandes d\'eviations locales pour les mesures d'occupation associ\'ees \`a \ef{c1.11}.

% Chapter 1

\chapter{Pr\'esentation des travaux de th\`ese} % Main chapter title

\label{Chapter2} % For referencing the chapter elsewhere, use \ref{Chapter1} 
%\lhead[]{Chapitre 2. \emph{Pr\'esentation des travaux de th\`ese}}
 % This is for the header on each page - perhaps a shortened title
% \rhead[]{M\'ELANGE ET GRANDES D\'EVIATIONS}

%----------------------------------------------------------------------------------------

\section{M\'elange exponentiel pour l'\'equation des ondes}
Dans cette partie nous pr\'esentons le r\'esultat principal de l'article \cite{DM2014} consacr\'e \`a l'\'etude de l'ergodicit\'e de  l'\'equation des ondes non lin\'eaire stochastique. Pour toute terminologie et notation utilis\'ees ici, on renvoit au chapitre pr\'ec\'edent.
\subsection{L'\'enonc\'e du probl\`eme et le r\'esultat principal}
 On consid\`ere l'\'equation des ondes 
 \be\label{c2.24}
 \p_t^2u+\gamma \p_tu-\de u+f(u)=h(x)+\vartheta(t,x)
 \ee
dans un domaine born\'e $D\subset \rr^3$. Ici les hypoth\`eses sur la fonction $h(x)$, la nonlin\'earit\'e $f(u)$ et la force $\vartheta(t,x)$ sont les m\^emes que dans les Section \ref{c1.13} et \ref{c1.15}. On s'int\'eresse \`a l'existence, l'unicit\'e et le m\'elange de mesures stationnaires du processus de Markov $(\uu(t), \pp_\vv)$ associ\'e.

\begin{theorem-fr}\label{theorem c2}
Sous les hypoth\`eses ci-dessus, il existe un entier $N\ge 1$, tel que si les $N$ premiers coefficients $b_j$ dans \ef{c2.3} ne sont pas z\'ero, alors le processus de Markov associ\'e au flot $\uu(t)=[u(t), \dt u(t)]$ de l'\'equation \ef{c1.11} poss\`ede une unique mesure stationnaire $\mu\in \ppp(\h)$. De plus, cette mesure est m\'elangeante et il existe des constantes positives $C$ et $\kp$ telles que pour tout $\lm\in \ppp(\h)$, on a
\be\label{c2.5}
|\PPPP^*_t\lm-\mu|_L^*\leq C e^{-\kp t}\int_\h \exp(\kp|\uu|_\h^4)\,\lm(\Dd\uu).
\ee 
\end{theorem-fr}

\subsection{L'id\'ee de la preuve}
La preuve est bas\'ee sur la construction du couplage qui satisfait les hypoth\`eses du Th\'eor\`eme 3.1.7 dans \cite{KS-book}, fournissant un crit\`ere g\'en\'eral pour le m\'elange exponentiel. Nous nous limitons ici \`a d\'ecrire cette construction, car c'est l'id\'ee principale derri\`ere la preuve. La proc\'edure est la suivante. Soit $\bar\uu=(\uu, \uu')$ un point de l'espace $\h\times\h$ et soient $\uu(t)$ et $\uu'(t)$ les flots de l'\'equation \ef{c2.24} issus de $\uu$ et $\uu'$, respectivement. On consid\`ere un processus interm\'ediaire $\vv(t)=[v(t), \dt v(t)]$, qui est la solution de l'\'equation
$$
\p_t^2v+\gamma \p_t v-\de v+f(v)+P_N[f(u)-f(v)]=h(x)+\vartheta(t,x)
$$
issue de $\uu'$. Ici $N\geq 1$ est un entier qui est fix\'e dans la suite et $P_N$ est la projection orthogonale de $L^2(D)$ sur l'espace $N$-dimensionnel engendr\'e par les fonctions $e_1,e_2,\ldots,e_N$.
On note $\lm(\uu, \uu')$ et $\lm'(\uu, \uu')$ les lois des processus $\{\vv\}_T$ et $\{\uu'\}_T$, respectivement, o\`u $\{\Zz\}_T$ est la restriction de $\{\Zz(t); t\geq 0\}$ sur $[0,T]$. Ainsi, $\lm$ et $\lm'$ sont des mesures de probabilit\'e sur $C(0,T;\h)$. Soit $(\vvv(\uu,\uu'),\vvv'(\uu,\uu'))$ un couplage maximal pour $(\lm(\uu,\uu'),\lm'(\uu,\uu'))$. Pour tout $s\in [0,T]$, on \'ecrit $\vvv_s$ et $\vvv'_s$ pour les restrictions de $\vvv$ et $\vvv'$ \`a l'instant $s$. On consid\`ere les flots $\tilde\vv(t)=[\tilde v,\p_t \tilde v]$ et $\tilde\uu'(t)=[\tilde u',\p_t \tilde u']$ correspondants. Alors, on a
$$
\p_t^2 \tilde v+\gamma \p_t \tilde v-\de\tilde v+f(\tilde v)-P_N f(\tilde v)=h(x)+\psi(t),\q [\tilde v(0),\p_t\tilde v(0)]=\uu',
$$
o\`u $\psi$ satisfait
\be\label{c2.6}
\ddd\{\int_0^t \psi(s)\,ds\}_T=\ddd\{\zeta(t)-\int_0^t P_N f(u)\,ds\}_T.
\ee
On introduit un autre processus $\tilde \uu(t)$, qui r\'esout
$$
\p_t^2 \tilde u+\gamma \p_t \tilde u-\de\tilde u+f(\tilde u)-P_N f(\tilde u)=h(x)+\psi(t),\q \tilde\uu(0)=\uu.
$$
Remarquons maintenant que $\uu(t)$ satisfait la m\^eme \'equation, avec $\psi$ remplac\'e par $\vartheta(t)-P_N f(u)$. Gr\^ace \`a \eqref{c2.6}, on a
\begin{equation}\label{6.40}
\ddd\{\tilde \uu\}_T=\ddd\{\uu\}_T.
\end{equation}
On introduit
\be
\rrr_t(\uu, \uu')=\tilde\uu(t), \q \rrr_t'(\uu, \uu')=\tilde\uu'(t) \q \text{ pour } t\in[0,T].\label{6. Coupling relation 1}
\ee
Il est clair que $\RR_t=(\rrr_t,\rrr'_t)$ est une extension de $S_t(\uu)$ sur l'intervalle $[0,T]$. Soit $\SSS_t=(S_t(\uu), S'_t(\uu))$ l'extension de $S_t(\uu)$ construit par it\'eration de $\RR_t$ sur la demi-droite $t\ge 0$. Rappelons que $S_t(\uu)$ d\'esigne le flot \`a l'instant $t$ de l'\'equation \ef{c2.24} issu de $\uu$. Avec un l\'eger abus de notation, on continuera d'\'ecrire $\tilde\uu(t)$ et $\tilde\uu'(t)$ pour les extensions de ces deux processus, et on \'ecrira $\tilde\vv(t)=\vvv_s(\SSS_{kT}(\uu))$ pour $t=s+kT,\q 0\leq s<T$. 

\nt Pour tout processus continu $\uu(t)$ \`a valeurs dans $\h$, on introduit une fonctionnelle 
\be\label{2.13}
\fff^{\uu}(t)=|\ees(\uu(t))|+\al\int_0^t|\ees (\uu(s))|\dd s,
\ee
et le temps d'arr\^et
\be\label{6 In-6,1,1}
\tau^\uu=\inf\{t\geq 0:\fff^\uu(t)\geq \fff^\uu(0)+(L+M)t+r\},
\ee
o\`u $L,M$ et $r$ sont des constantes positives choisies apr\`es. 
On introduit aussi 
\begin{align*}
 \vo&=\inf\{t=s+kT: \vvv_s(\SSS_{kT}(\bar\uu))\neq \vvv'_s(\SSS_{kT}(\bar\uu))\}\equiv\inf\{t\geq 0: \tilde\vv(t)\neq \tilde\uu'(t)\},\\
 \tau&=\tau^{\tilde\uu}\wedge\tau^{\tilde \uu'},\q \sigma=\vo\wedge\tau.
 \end{align*}
Supposons qu'on peut \'etablir le r\'esultat suivant. 
\begin{theorem-fr}\label{2.14}
Sous les hypoth\`eses du Th\'eor\`eme \ref{theorem c2}, il existe des constantes positives $\al,\De,\kp,d$ et $C$ telles que les propri\'et\'es suivantes sont satisfaites.\\
\text{\textnormal{(R\'ecurrence):}} Pour tout $\bar\uu=(\uu,\uu')\in\hh$, on a
\begin{align*}
\e_\uu\exp(\kp\ees(\uu(t))&\leq\e_\uu\exp(\kp\ees (\uu(0))e^{-\al t}+C(\gamma,\BBB,\|h\|),\\
\e_{\uu}\exp(\kp\tau_d)&\leq C(1+|\uu|_{\hh}^4), 
\end{align*}
o\`u $\tau_d$ d\'esigne le premier instant d'entr\'ee dans la boule $B_{\hh}(d)$.\\
\text{\textnormal{(Serrage exponentiel):}} Pour tout $\bar\uu=(\uu, \uu')\in B_{\hh}(d)$, on a 
\begin{align*}
|S_t(\bar\uu)-S_t'(\bar\uu)|^2_\h&\leq Ce^{-\al t}|\uu-\uu'|_\h^2 \q\text{ pour } 0\leq t\leq\sigma,\\
\pp_{\bar\uu}\{\sigma=\infty\}&\geq\De,\q \e_{\bar\uu}[\ch_{\{\sigma<\infty\}}\exp(\De\sigma)]\leq C,\q\e_{\bar\uu}[\ch_{\{\sigma<\infty\}} |\bar\uu(\sigma)|_{\hh}^{8}]\leq C.
\end{align*}
\end{theorem-fr}
D'apr\`es le Th\'eor\`eme 3.1.7 de \cite{KS-book}, ce r\'esultat implique le Th\'eor\`eme \ref{theorem c2}. La preuve de \mpp{r\'ecurrence} repose sur la technique de fonction de Lyapunov, alors que la d\'emonstration du \mpp{serrage exponentiel} est bas\'ee sur l'estimation du type Foia\c{s}-Prodi pour l'\'equation \eqref{c2.24}, le th\'eor\`eme de Girsanov et un argument de temps d'arr\^et. Mentionnons que l'id\'ee de consid\'erer  un processus interm\'ediaire lorsqu'on applique le th\'eor\`eme de Girsanov a \'et\'e introduite par Odasso \cite{odasso-2008}.

\section{Grandes d\'eviations pour les mesures stationnaires}
Ici on pr\'esente le r\'esultat principal de l'article \cite{DM2015} consacr\'e \`a l'\'etude de grandes d\'eviations pour les mesures stationnaires de l'\'equation des ondes non lin\'eaire avec bruit blanc. 
\subsection{Le probl\`eme consid\'er\'e et les hypoth\`eses}
 On \'etudie l'\'equation des ondes 
\be\label{c2.7}
\p_t^2u+\gamma \p_tu-\de u+f(u)=h(x)+\sqrt{\es}\,\vartheta(t,x)
\ee
dans un domaine $D$ born\'e de $\rr^3$. On s'int\'eresse au comportement asymptotique de la famille $(\mu^\es)$ de mesures stationnaires du processus de Markov associ\'e, lorsque $\es$ tend vers z\'ero.

Les hypoth\`eses sur la nonlin\'earit\'e $f$ et la force $\vartheta$ sont les m\^emes que dans la section pr\'ec\'edente. Il y a une hypoth\`ese suppl\'ementaire sur la fonction $h(x)$: on suppose que l'\'equation \ef{c2.7} sans perturbation
\be\label{c2.9}
\p_t^2u+\gamma \p_tu-\de u+f(u)=h(x)
\ee
poss\`ede un nombre fini de solutions stationnaires. Cette condition est g\'en\'erique par rapport \`a $h$ dans la classe $H^1_0(D)$.
\subsection{Le r\'esultat principal}
Soit $\E=\{\hat\uu_1, \ldots, \hat\uu_\ell\}\subset\h$ l'ensemble des flots stationnaires $\uu=[u, 0]$ de l'\'equation \ef{c2.9}. Rappelons que le point d'\'equilibre $\hat \uu$ est dit \mpp{stable au sens de Lyapunov} si pour tout $\eta>0$ il existe $\De>0$ tel que chaque flot de \ef{c2.9} issue du $\De$-voisinage de $\hat\uu$ reste dans le $\eta$-voisinage de $\hat\uu$ pour tout temps. On \'ecrira $\E_s\subset \{\hat\uu_1, \ldots, \hat\uu_\ell\}$ pour l'ensemble des \'equilibres stables au sens de Lyapunov. 
\begin{theorem-fr}\label{c2.10}
Sous les hypoth\`eses ci-dessus, il existe une fonction $\vvv:\h\to [0,\iin]$ dont les ensembles de niveau sont compacts et telle que 
\be\label{c2.8}
-\inf_{\uu\in \dot\Gamma}\vvv(\uu)\le\liminf_{\es\to 0}\es\ln \mu^\es(\Gamma)\le\limsup_{\es\to 0}\es\ln \mu^\es(\Gamma)\le -\inf_{\uu\in \bar\Gamma}\vvv(\uu),
\ee
o\`u $\Gamma$ est un ensemble bor\'elien de $\h$ quelconque, et on note $\dot\Gamma$ et $\bar\Gamma$ son int\'erieur et sa fermeture, respectivement. Autrement dit, la famille $(\mu^\es)$ satisfait le principe de grandes d\'eviations avec une bonne fonction de taux $\vvv$. De plus, la fonction $\vvv$ ne peut s'annuler que sur l'ensemble $\E_s\subset \{\hat\uu_1, \ldots, \hat\uu_\ell\}$. En particulier, la famille $(\mu^\es)$ est exponentiellement tendue et chaque limite faible de cette famille est concentr\'ee sur $\E_s$.
\end{theorem-fr}
\subsection{Esquisse de la preuve}
{\it Construction de la fonction $\vvv$.} La fonction $\vvv$ est donn\'ee par \ef{c0.5}-\ef{c0.6}, o\`u $V(\uu_1, \uu_2)$ est l'\'energie minimale n\'ecessaire pour atteindre le voisinage arbitrairement petit de $\uu_2$ \`a partir de $\uu_1$ en temps fini. 

\bigskip
{\it Cha\^ine de Markov sur le bord.} Ce qui suit est une modification d'une construction introduite dans \cite{FW2012} (voir Chapitre 6) qui est elle-m\^eme une variation d'un argument utilis\'e dans \cite{Khas2011}. On se donne un point $\uu\in\h\backslash\{\hat\uu_1,\ldots,\hat\uu_{\ell}\}$ et on \'ecrit $\hat\uu_{\ell+1}=\uu$. Soit  $\rho_*$ un nombre positif quelconque.  Pour tous $0<\rho'_1<\rho_0'<\rho_1<\rho_0<\rho_*$ on introduit les objets suivants. Pour $i\leq \ell$, on note $g_i$ et $\tilde g_i$ le $\rho_1$- et $\rho_0$-voisinages ouverts de $\hat\uu_i$, respectivement. De m\^eme, on note  $g_{\ell+1}$ et $\tilde g_{\ell+1}$, respectivement, le $\rho_1'$- et $\rho_0'$-voisinages de $\hat\uu_{\ell+1}$. Ensuite, on note $g$ et $\tilde g$ l'union sur $i\leq \ell+1$ de $g_i$ et $\tilde g_i$, respectivement. Pour tout $\es>0$ et $\vv\in \h$ soit $S^\es(t;\vv)$ le flot \`a l'instant $t$ de \ef{c2.7} issu de $\vv$. Soit $\sigma_0^\es$ l'instant de la premi\`ere sortie de $\tilde g$ du processus $S^\es(t;\cdot)$, et soit $\tau_1^\es$ le premier instant apr\`es $\sigma_0^\es$ quand $S^\es(t;\cdot)$ touche le bord de $g$. De m\^eme, pour $n\geq 1$ on note $\sigma_n^\es$ le premier instant apr\`es $\tau_n^\es$ de sortie de $\tilde g$ et $\tau_{n+1}^\es$ le premier instant apr\`es $\sigma_n$ lorsque $S^\es(t;\cdot)$ touche $\p g$. On consid\`ere une cha\^ine de Markov sur le bord $\p g$ d\'efinie par $Z_n^\es (\cdot)= S^\es(\tau^\es_n,\cdot)$. On notera $\tilde P^\es(\vv,\Gamma)$ la probabilit\'e de transition de cha\^ine $(Z_n^\es)$, c'est-\`a-dire 
$$
\tilde P^\es(\vv,\Gamma)=\pp(S^\es(\tau_1^\es;\vv)\in \Gamma)\q\text{ pour tout } \vv\in\p g\,\text{ et } \Gamma\subset\p g.
$$
Soit $b(\p g)$ l'ensemble des fonctions bor\'eliennes born\'ees $\psi:\p g\to \rr$ muni de la topologie de convergence uniforme et soit $\lm$ une fonctionnelle continue sur $b(\p g)$ telle que $\lm(\psi)\ge 0$ pour $\psi\ge 0$ et $\lm (1)=1$. On dira que $\lm=\lm^\es$ est \mpp{une mesure finiment additive stationnaire} de $\tilde P^\es(\vv,\Gamma)$ si 
\be\label{c2.11}
\lm(\pP\psi)=\lm(\psi)\q\text{ pour tout }\psi\in b(\p g)
\ee
o\`u $\pP=\pP^\es:b(\p g)\to b (\p g)$ est d\'efini par
\be\label{c2.12}
\pP\psi(\vv)=\int_{\p g}\psi(\Zz)\tilde P_1^\es(\vv, \Dd\Zz)\equiv \e\psi(S^\es(\tau_1^\es;\vv)).
\ee
On montrera qu'une telle mesure existe. La preuve du th\'eor\`eme s'op\`ere en quatre \'etapes dont la premi\`ere est le r\'esultat suivant.
\begin{proposition-fr}\label{c2.13}
Pour tout $\beta>0$ et $\rho_*>0$ il existe $0<\rho'_1<\rho_0'<\rho_1<\rho_0<\rho_*$ tels que pour tout $\es>0$ suffisamment petit, on a
\be\label{c2.14}
\exp(-(\vvv(\hat\uu_j)+\beta)/\es)\leq \lm^\es(\ch_{g_j})\leq \exp(-(\vvv(\hat\uu_j)-\beta)/\es),
\ee
o\`u les in\'egalit\'es sont vraies pour tout $\vv\in \p g_i$ et tout $i, j\leq \ell+1, i\neq j$.
\end{proposition-fr}
Pour tout $\es>0$, on introduit une fonctionnelle continue $\tilde\mu=\tilde\mu^\es$ sur $b(\h)$ par 
\be\label{c2.15}
\tilde\mu(\psi)=\lm(\elll\psi),
\ee
o\`u $\lm=\lm^\es$ est donn\'e par \ef{c2.11}-\ef{c2.12}, et $\elll=\elll^\es:b(\h)\to b(\h)$ est d\'efini par
\be\label{c2.16}
\elll\psi(\vv)=\e\int_0^{\tau_1^\es}\psi(S^\es(t;\vv))\dd t.
\ee
On notera $\hat\mu=\hat\mu^\es$ la normalisation de $\tilde\mu$, \`a savoir $\hat\mu(\psi)=\tilde\mu(\psi)/\tilde \mu(1)$. Finalement, on introduit $\hat\mu(\Gamma)=\hat\mu(\ch_\Gamma)$ pour tout ensemble bor\'elien de $\h$. La deuxi\`eme \'etape consiste \`a montrer qu'on a 
\be\label{c2.18}
\mu(\dt\Gamma)\le\hat\mu(\dt\Gamma)\le\hat\mu(\bar\Gamma)\le\mu(\bar\Gamma)
\ee
pour tout $\Gamma$, o\`u $\mu=\mu^\es$. En trosi\`eme \'etape, on utilise \ef{c2.18} avec la proposition \ref{c2.13} pour d\'emontrer le r\'esultat suivant.
\begin{proposition-fr}\label{c2.17}
Pour tout $\beta>0$ et $\rho_*>0$ il existe $0<\rho'_1<\rho'_0<\rho_1<\rho_0<\rho_*$ tels que pour tout $j\leq \ell+1$ et $\es<<1$, on a 
\begin{align}
\mu^\es(g_j)&\leq \exp(-(\vvv(\hat\uu_j)-\beta)/\es)\label{9.79}\\
\mu^\es(\bar g_j)&\ge \exp(-(\vvv(\hat\uu_j)+\beta)/\es)\label{9.80}.
\end{align}
\end{proposition-fr}
La derni\`ere \'etape consiste \`a montrer que la famille $(\mu^\es)$ est exponentiellement tendue. Cette propri\'et\'e conjugu\'ee avec la proposition \ref{c2.17} implique le th\'eor\`eme \ref{c2.10}. 

\medskip
Sans entrer dans les aspects techniques, d\'emontrons ici l'existence d'une mesure finiment additive stationnaire (i.e., g\'en\'eralis\'ee) $\lm$ et prouvons l'in\'egalit\'e \ef{c2.18}. Soit $b^*(\p g)$ l'espace dual de $b(\p g)$. Consid\'erons l'espace $$
\mathfrak{F}=\{\lm\in b^*(\p g):  \lm(1)= 1 \text{ et } \lm(\psi)\ge 0 \text{ pour }\psi\ge 0\}
$$
muni de topologie faible*. Notons que si $\lm\in \mathfrak{F}$, alors $|\lm|_{b^*(\p g)}=1$. Gr\^ace au th\'eor\`eme de Banach-Alaoglu, $\mathfrak{F}$ est relativement compact. Par ailleurs, il est clair que $\mathfrak{F}$ est aussi ferm\'e et convexe. On consid\`ere le dual de $\pP$, i.e., l'application $\pP^*:b^*(\p g)\to b^*(\p g)$ d\'efinie par 
$$
\pP^*\lm(\psi)=\lm(\pP\psi).
$$
Comme $\pP1\equiv1$ et $\pP\psi\ge 0$ pour $\psi\ge 0$, $\pP^*$ envoie $\mathfrak{F}$ sur lui m\^eme. Il est facile de voir que $\pP^*$ est continue. Par le th\'eor\`eme de Leray-Schauder, $\pP^*$ admet un point fixe $\lm\in \mathfrak{F}$, ce qui signifie que $\lm$ est une mesure finiment additive stationnaire de $\pP$.

\medskip
Prouvons maintenant l'in\'egalit\'e \ef{c2.18}. En utilisant l'argument de Khasminskii, il n'est pas difficile de voir qu'on a 
\be\label{c2.26}
\hat\mu(P_t\psi)=\hat\mu(\psi)\q\text{ pour tout }\psi\in b(\h) \text{ et } t\ge 0,
\ee
o\`u $P_t=P_t^\es$ est la fonction de transition de $S^\es(t,\cdot)$. Soit $b_0(\h)$ l'espace des fonctions  mesurables born\'ees, muni de la topologie de convergence uniforme sur les ensembles born\'es de $\h$ pour les suites uniform\'ement born\'ees dans $\h$. On peut montrer que $\hat\mu$ est continue de $b_0(\h)$ sur $\rr$. D'apr\`es l'in\'egalit\'e \ef{c2.5}, $P_t\psi$ converge vers $(\psi, \mu)$ dans $b_0(\h)$ pour tout $\psi$ dans l'espace $L_b(\h)$ de fonctions Lipshitziennes born\'ees sur $\h$. 
Avec \ef{c2.26}, cette propri\'et\'e implique
\be\label{c2.27}
\hat\mu(\psi)=\hat\mu(P_t\psi)\to \hat\mu((\psi, \mu))=(\psi,\mu)\q\text{ pour tout }\psi\in L_b(\h).
\ee
Remarquons que l'in\'egalit\'e \ef{c2.18} sera \'etablie, si on arrive \`a montrer  
$$
\hat\mu(F)\le\mu(F)
$$
pour tout $F\subset \h$ ferm\'e.
Supposons que cette in\'egalit\'e n'est pas vraie, et soient $F\subset \h$ ferm\'e et $\eta>0$ tels que
\be\label{c2.28}
\hat\mu(F)\ge \mu(F)+\eta.
\ee
Soit $\ch_F\le\psi_n\le 1$ une suite de fonctions dans $L_b(\h)$ qui converge ponctuellement vers $\ch_F$ lorsque $n\to \iin$. On peut prendre, par exemple, 
$$
\psi_n(\uu)=\f{d_\h(\uu, F^c_{1/n})}{d_\h(\uu, F^c_{1/n})+d_\h(\uu,F)},
$$
o\`u $F_r$ d\'esigne le $r$-voisinage ouvert de $F$. Gr\^ace \`a \ef{c2.27}, l'in\'egalit\'e \ef{c2.28} et la monotonie de $\hat\mu$, on a
$$
(\psi_n,\mu)=\hat\mu(\psi_n)\ge\hat\mu(\ch_F)=\hat\mu(F)\ge \mu(F)+\eta.
$$
Pourtant, ce n'est pas possible, car $(\psi_n, \mu)$ converge vers $\mu(F)$ par le th\'eor\`eme de convergence domin\'ee. L'in\'egalit\'e \ef{c2.18} est \'etablie.

%----------------------------------------------------------------------------------------
  
\section{Grandes d\'eviations pour les mesures d'occupation}
Ici on pr\'esente le r\'esultat principal de l'article \cite{DM-VN2015} consacr\'e \`a l'\'etude de grandes d\'eviations pour les mesures d'occupation de l'\'equation des ondes non lin\'eaire avec bruit blanc. 
\subsection{Le probl\`eme \'etudi\'e et le  r\'esultat principal}
 On consid\`ere l'\'equation des ondes \ef{c1.11}
\be\label{c2.19}
\p_t^2u+\gamma \p_tu-\de u+f(u)=h(x)+\vartheta(t,x)
\ee
dans un domaine born\'e de $\rr^3$. Les hypoth\`eses sur la fonction $h(x)$, la nonlin\'earit\'e $f(u)$ et la force $\vartheta(t,x)$ sont les m\^emes que dans les Section \ref{c1.13} et \ref{c1.15}. On s'int\'eresse au comportement asymptotique dans la limite $t\to\iin$ de la famille $(\zeta(t))$ de mesures d'occupation d\'efinie par
$$
\zeta(t)=\frac1t\int_{0}^{t}\De_{\uu(\tau;\vv)}\dd \tau, \quad t>0
$$
sur l'espace probabilis\'e $(\Omega,\fff,\pp)$, o\`u $\uu(\tau;\vv)$ est la solution \`a l'instant $\tau$ de l'\'equation \ef{c2.19} issue de point $\vv\in\h$. Rappelons qu'on note $\h^s$ l'espace $[H^{s+1}(D)\cap H^1_0(D)]\times H^s(D)$ avec $s\in (0, 1-\rho/2)$ fix\'e. 
  \begin{theorem-fr}  \label{c2.20}
Sous les hypoth\`eses ci-dessus, pour toute fonction $\psi\in C_b(\h)$ non constante, il existe $\es=\es(\psi)>0$ et une fonction convexe $I^\psi:\R\to \R_+$
tels que pour tout $\vv\in \h^s$ et tout ouvert $O$ de l'intervalle $((\psi, \mu)-\es, (\psi, \mu)+\es))$, on a 
\be\label{Llimit}
\lim_{t\to\infty} \frac1t\ln   \pp\left\{\f{1}{t}\int_0^t \psi(\uu(\tau;\vv))\dd\tau\in O\right\}= -\inf_{\alpha\in O} I^\psi(\alpha),
\ee
o\`u $\mu$ est la mesure stationnaire de $(\uu(t),\pp_\vv)$. De plus, cette limite est uniforme par rapport \`a $\vv$ sur les ensembles born\'es de $\h^s$. 
\end{theorem-fr} 
On montrera, en fait, un r\'esultat plus g\'en\'eral du type niveau-2 et en d\'eduira ce th\'eor\`eme.

 \subsection{Certains ingr\'edients de la preuve}
On commence par introduire quelques notations. Pour toute fonction $V\in C_b(\h)$ on introduit le semigroupe $(\PPPP_t^V)_{t\ge 0}$ de Fynman-Kac agissant sur $C_b(\h)$ par la formule
$$
\PPPP_t^V\psi(\vv)=\e_\vv \left[\psi(\uu(t))\exp(\int_0^t V(\uu(\tau))\dd\tau)\right].
$$
Soit $\mmm_+(\h)$ l'ensemble des mesures bor\'eliennes positives sur $\h$ muni de la topologie de convergence faible. On consid\`ere le dual du semigroupe $(\PPPP_t^V)_{t\ge 0}$, \`a savoir, le semigroupe  $(\PPPP_t^{V*})_{t\ge 0}$ agissant sur $\mmm_+(\h)$ par la formule
$$
(\psi, \PPPP_t^{V*}\mu)=(\PPPP_t^V\psi, \mu)\q\text{ pour }\mu\in \mmm_+(\h) \text{ et }\psi\in C_b(\h).
$$
On dira que $\mu_V\in\ppp(\h)$ est \mpp{un vecteur propre du} semigroupe $(\PPPP_t^{V*})$ s'il existe $\lm\in\rr$ tel que $\PPPP_t^{V*}\mu_V=\lm^t\mu_V$ pour tout $t\ge 0$. De m\^eme, on dira que $h_V\in C(\h^s)$ est un vecteur propre de $(\PPPP_t^V)$ s'il existe $\lm\in\rr$ tel que $\PPPP_t^V h_V=\lm^t h_V$ dans $\h^{s}$ pour tout $t\ge 0$. On introduit la fonction de poids $\we:\h^s\to [1, \iin)$ donn\'ee par
$$
\we(\vv)=1+|\vv|_{\h^s}^2+\ees^4(\vv)
$$
et on d\'esigne $C_\we(\h^s)$ l'espace des fonctions continues $\psi:\h^s\to\rr$\ telles que 
$$
|\psi|_{C_\we}=\sup_{\vv\in\h^s}\f{|\psi(\vv)|}{\we(\vv)}<\iin.
$$
Finalement, on note $\ppp_\we(\h)$ l'espace de mesures de probabilit\'e $\lm$ sur $\h$ telles que $(\we, \lm)<\iin$, muni de topologie de convergence faible. On peut montrer que le th\'eor\`eme \ref{c2.20} est une cons\'equence du r\'esultat suivant.
\begin{proposition-fr}
Sous les hypoth\`eses ci-dessus, il existe une constante $\De>0$ telle que les propri\'et\'es suivantes sont satisfaites pour toute fonction $V$ lipschitzienne born\'ee sur $\h$ dont l'oscillation $\Osc(V)=(\sup V-\inf V)$ ne d\'epasse pas $\De$.  Le semigrope $(\PPPP_t^{V*})$ poss\`ede un unique vecteur propre $\mu_V\in\ppp_\we(\h)$ correspondant \`a une valeur propre $\lm_V$ positive. De plus, $\mu_V$ satisfait
\be\label{c2.22}
\int_\h [|\vv|_{\h^s}^m+\exp(\kp\ees(\vv))]\mu_V(\Dd \vv)<\iin
\ee
pour tout entier $m\ge 1$, o\`u $\kp=(2\al)^{-1}\sum b_j^2$. En outre, le semigroupe $\PPPP_t^V$ admet un unique vecteur propre $h_V\in C_\we(\h^s)$, $h_V\ge 0$, correspondant \`a la valeur propre $\lm_V$ normalis\'e par la condition $(h_V, \mu_V)=1$. Finalement, on a les convergences  
\begin{align*}
\lm_V^{-t}\PPPP_t^V\psi&\to (\psi, \mu_V)h_V \q\text{ dans } C_b(\h^s_R)\cap L^1(\h, \mu_V),\\
\lm_V^{-t}\PPPP_t^{V*}\nu&\to (h_V, \nu)\mu_V \q\text{ dans } \mmm_+(\h)
\end{align*}
lorsque $t\to\iin$, pour tout $\psi\in C_\we(\h^s)$, $\nu\in \ppp_\we(\h)$ et $R>0$, o\`u $\h^s_R$ est la boule de rayon $R$ dans $\h^s$.
\end{proposition-fr}
Sans entrer dans les d\'etails, on fait l'esquisse de la preuve de la premi\`ere partie de cette proposition. Montrons en particulier, le r\'esultat suivant.
 \begin{proposition-fr}\label{e23}
 Pour tout $t>0$, $V\in C_b(\h)$ et $m\geq 1$, l'op\'erateur $\PPPP_t^{V*}$ admet un vecteur propre $\mu_V$ avec une valeur propre positive. De plus, on a  \ef{c2.22}.
 \end{proposition-fr}
On introduit 
$$
\tilde\we_m(\vv)=1+|\vv|_{\h^\sS}^{2m}+\ees^{4m}(\vv)+\exp(\kp\ees(\vv))
$$
et on admet l'in\'egalit\'e suivante
\be\label{c2.21}
\e_\vv\tilde\we_m(\uu(t) )\leq 2e^{-\al m t}\tilde\we_m(\vv)+C_m,
\ee
qui est vraie pour tout $\vv\in\h^\sS$, $m\ge1$, et $t\ge0$.
Soient $t>0$ et $V\in C_b(\h)$ fix\'es. Pour tout $A>0$ et $m\geq 1$, on introduit l'ensemble convexe
$$
D_{A, m}=\{\sigma\in \ppp(\h): (\tilde \we_m, \sigma)\leq A\},
$$
et on consid\`ere l'application continue de $D_{A, m}$ sur $\ppp(\h)$ donn\'ee par
$$
G(\sigma)=\PPPP_t^{V*}\sigma/\PPPP_t^{V*}\sigma(\h).
$$
Gr\^ace \`a l'in\'egalit\'e \ef{c2.21},  on a 
\begin{align}
(\tilde\we_{m}, G(\sigma))&\leq \exp\left(t \Osc(V)\right)(\tilde\we_{m}, \PPPP_t^*\sigma)\notag\\
&\leq 2\exp\left(t (\Osc(V)-\al m)\right)(\tilde\we_{m}, \sigma)+C_m \exp\left(t  \Osc(V)\right)\label{c2.23}.
\end{align}
Soit $m$ assez grand tel que l'on ait
$$
\Osc(V)\leq \al m/2\q\text{ et }\q \exp(-\al m t/2)\leq 1/4,
$$
et soit $A=2 C_m e^{\al m t}$. 
Alors, compte tenu de l'in\'egalit\'e \ef{c2.23}, on a $(\tilde\we_{m}, G(\sigma))\leq A$ pour tout $\sigma\in D_{A, m}$, i.e.,       $G(D_{A, m})\subset D_{A, m}$. En plus, il n'est pas difficile de voir que l'ensemble $D_{A, m}$ est compact dans $\ppp(\h)$ (on peut utiliser le crit\`ere de compacit\'e de Prokhorov 
pour montrer qu'il est relativement compact et le lemme de Fatou pour montrer qu'il est ferm\'e). Par le th\'eor\`eme de Leray-Schauder, $G$ a un point fixe $\mu_V \in D_{A, m}$. Pour finir la preuve, il suffit de remarquer que d'apr\`es les constructions de $D_{A, m}$ et $G$, la mesure $\mu_V$ est un vecteur propre de l'op\'erateur $\PPPP_t^{V*}$ avec valeur propre positive $\PPPP_t^{V*}\mu (\h)$ et elle v\'erifie \ef{c2.22}.  On peut montrer que, en fait, une telle mesure $\mu_V$ est unique, et elle ne d\'epend pas de $t$ et $m$.

% Chapter 3

\chapter{M\'elange} % Main chapter title

\label{Chapter3} % For referencing the chapter elsewhere, use \ref{Chapter1} 

%\lhead[]{Chapitre 4. \emph{M\'elange de}} % This is for the header on each page - perhaps a shortened title

%----------------------------------------------------------------------------------------

%\author{Davit Martirosyan\footnote{Department of Mathematics, University of Cergy-Pontoise, CNRS %UMR 8088, 2 avenue
%Adolphe Chauvin, 95300 Cergy-Pontoise, France;e-mail: \href{mailto:Davit.Martirosyan@u-cergy.fr}{Davit.Martirosyan@u-cergy.fr}}}
\selectlanguage{english}
\section*{Exponential mixing for the white\,-\,forced damped nonlinear wave equation} 

{\bf Abstract}.
The paper is devoted to studying the stochastic nonlinear wave (NLW) equation 
$$
\p_t^2 u+\gamma \p_t u-\de u+f(u)=h(x)+\eta(t,x)
$$
in a bounded domain $D\subset\rr^3$. The equation is supplemented with the Dirichlet boundary condition. Here $f$ is a nonlinear term, $h(x)$ is a function in $H^1_0(D)$ and $\eta(t,x)$ is a non-degenerate white noise. We show that the Markov process associated with the flow $\xi_u(t)=[u(t),\dt u(t)]$ has a unique stationary measure $\mu$, and the law of any solution converges to $\mu$ with exponential rate in the dual-Lipschitz norm.

\medskip

\nt 
\bigskip
\section{Introduction}
We consider the stochastic NLW equation
\be\label{1.1.3}
\p_t^2u+\gamma \p_tu-\de u+f(u)=h(x)+\eta(t,x),\q [u(0),\dt u(0)]=[u_0, u_1]
\ee
in a bounded domain $D\subset\rr^3$ with a smooth boundary. The equation is supplemented with the Dirichlet boundary condition. The nonlinear term $f$ satisfies the dissipativity and growth conditions that are given in the next section (see \eqref{1.8.3}-\eqref{1.6.3}). Here we only mention that they hold for functions  $f(u)=\sin u$ and $f(u)=|u|^{\rho}u-\lm u$, where $\lm$ and $\rho\in(0,2)$ are some constants. These functions correspond to the damped sine-Gordon and Klein-Gordon equations, respectively.
The force $\eta(t)$ is a white noise of the form
\be\label{1.9.3}
\eta(t,x)=\sum_{j=1}^\infty b_j\dt\beta_j(t)e_j(x).
\ee
Here $\{\beta_j(t)\}$ is a sequence of independent standard Brownian motions, $\{e_j\}$ is an orthonormal basis in $L^2(D)$ composed of the eigenfunctions of the Dirichlet Laplacian, and $\{b_j\}$ is a sequence of positive numbers that goes to zero sufficiently fast (see \eqref{2.24.3}). 
The initial point $[u_0,u_1]$ belongs to the phase space $\h=H^1_0(D)\times L^2(D)$. Finally, $h(x)$ is a function in $H^1_0(D)$. 
The following theorem is the main result of this paper.
\begin{mt}\label{1.4.3}
Under the above hypotheses, the Markov process associated with the flow $y(t)=[u(t),\dt u(t)]$ of equation \eqref{1.1.3} possesses a unique stationary measure $\mu\in\ppp(\h)$. Moreover, there are positive constants $C$ and $\kp$ such that
\be\label{6.53.3}
|\e \psi(y(t))-\int_{\h}\psi(z)\mu(dz)|\leq Ce^{-\kp t}\exp(\kp|y|_\h^4),\q t\geq 0, 
\ee
for any 1-Lipschitz function $\psi:\h\to\rr$, and any initial point $y\in\h$. 
\end{mt}

\nt
Thus, the limit of the average of $\psi(y(t))$ is a quantity that does not depend on the initial point.

Before outlining the main ideas of the proof of this result, let us discuss some of the earlier works concerning the ergodicity of the stochastic nonlinear PDE's and the main difficulties that occur in our case. In the context of stochastic PDE's, the initial value problem and existence of a stationary measure was studied by Vishik--Fursikov--Komech \cite{VKF-1979} for the stochastic Navier--Stokes system and later developed for many other problems (see the references in \cite{DZ1992}). The uniqueness of stationary measure and its ergodicity are much more delicate questions. First results in this direction were obtained in the papers \cite{FM-1995, KS-cmp2000, EMS-2001, BKL-2002} devoted to the Navier--Stokes system and other PDE's arising in mathematical physics (see also \cite{MR1245306, MR1641664} and Part III in \cite{DZ1996} for some 1D parabolic equations). They were later extended to equations with multiplicative and very degenerate noises \cite{odasso-2008,HM-2008}. We refer the reader to the recent book \cite{KS-book} and the review paper \cite{debussche2013ergodicity} for a detailed account of the main results obtained so far. 

We now discuss in more details the case of dispersive equations, for which fewer results are known.
One of the first results on the ergodicity of dispersive PDE's was stablished in the paper of E, Khanin, Mazel and Sinai \cite{EWKMS2000}, where the authors prove the existence and uniqueness of  stationary measure for the one dimensional inviscid Burgers equation perturbed by a space-periodic white noise. The qualitative study of stationary solutions is also carried out, and the analysis relies on the Lax-Oleinik variational principle.
The ergodicity of a white-forced NLW equation was studied by Barbu and Da Prato \cite{BD-2002}, where the authors prove the existence of stationary distribution for a nonlinearity which is a non-decreasing function satisfying the growth restriction $|f''(u)|\leq C(|u|+1)$, and some standard dissipativity conditions. Uniqueness is established under the additional hypotheses, that $f$ satisfies \eqref{1.8.3} with $\rho<2$, and $\sup\{|f'(u)|\cdot |u|^{-\rho}, u\in\rr\}$ is sufficiently small. In the paper by Debussche and Odasso \cite{DO-2005}, the authors establish the convergence to the equilibrium with polynomial speed at any order (\mpp{polynomial mixing}) for weakly damped nonlinear Schr\"odinger equation. The proof of this result relies on the coupling argument. The main difficulty in establishing the exponential rate of convergence is due to the complicated Lyapunov structure and the fact that the Foa\c{s}-Prodi estimates hold in average and not path-wise. 
In \cite{DirSoug2005}, Dirr and Souganidis study the Hamilton-Jacobi equations perturbed by additive noise. They show, in particular, that under suitable assumptions on the Hamiltonian, the stochastic equation has a unique up to constants space-periodic global attracting solution, provided the unperturbed equation possesses such solution.  
In the recent paper by Debussche and Vovelle \cite{debussche2013invariant} the existence and uniqueness of stationary measure is studied for scalar periodic first-order conservation laws with additive noise in any space dimension. It generalizes to higher dimensions the results established in \cite{EWKMS2000} (see also \cite{IturK2003}). In another recent paper \cite{BCK2014} by Bakhtin, Cator and Khanin, the authors study the ergodicity of the Burgers equation perturbed by a space-time stationary random force. It is proved, in particular, that the equation possesses space-time stationary global solutions, and that they attract all other solutions. The proof uses the Aubry-Mather theory for action-minimizing trajectories, and weak KAM theory for the Hamilton-Jacobi equations. 

In the present paper we extend the results established in \cite{BD-2002}, proving that the hypotheses  $f'\geq 0$ and $\sup\{|f'(u)|\cdot |u|^{-\rho}, u\in\rr\}$ is small are not needed, and that the convergence to the equilibrium has exponential rate. We also show that the conclusion of the Main Theorem remains true for a force that is non-degenerate only in the low Fourier modes (see Theorem \ref{6.10.3}). The proof mainly relies on the coupling argument.
 
Of course, one of the main difficulties when dealing with dispersive PDE's comes from the lack of the regularizing property, and with it, of some well-known compactness arguments. As a consequence, this changes the approach when showing the stability of solutions. In particular, this is the case, when establishing the Foia\c{s}-Prodi estimate for NLW (Proposition \ref{4.13.3}). Moreover, this estimate (which shows that the large time behavior of solutions is determined by finitely many modes and enables one to use the Girsanov theorem) differs from the classical one, since the growth of the intermediate process should be controlled (see inequality \eqref{4.16.3}). Due to the last fact, the  coupling constructed through the projections of solutions (cf. \cite{shirikyan-bf2008, odasso-2008}) does not ensure exponential rate of convergence. We therefore introduce a new type of coupling constructed via the intermediate process (see \eqref{2.25.3}-\eqref{6. Coupling relation 1}).
The same difficulty occurs when showing the recurrence of solutions, i.e. that the trajectory of the solution enters arbitrarily small ball with positive probability in a finite time (Proposition \ref{6.12.3}). The standard argument to show this property is the use of the portmanteau theorem. However, due to the lack of the smoothing effect, the portmanteau technique is not applicable, and another approach is proposed. 

Without going into details, we give an informal description of our approach. The proof of the existence of stationary measure is rather standard and relies on the Bogolyubov-Krylov argument, which ensures the existence, provided the process $y(t)=[u(t),\dt u(t)]$ has a uniformly bounded moment in some $\h$-compact space. To obtain such a bound, we follow a well-known argument coming from the theory of attractors (e.g., see \cite{BV1992, Har85}). Namely, we split the function $u$ to the sum $u=v+z$, where, roughly speaking, $v$ takes the Brownian of equation, and $z$-nonlinearity. We then show that the corresponding flows have uniformly bounded moments in $\h^s=H^{1+s}(D)\times H^s(D)$ for $s>0$ sufficiently small (Proposition \ref{2.11.3}). The bound for $|[v(t),\dt v(t)]|_{\h^s}$ follows from the It\^o formula, while that of $|[z(t),\dt z(t)]|_{\h^s}$ is based on the argument similar to the one used in \cite{zelik2004}.
The proof of exponential mixing relies on Theorem 3.1.7 in \cite{KS-book}, which gives a general criterion that ensures the convergence to the equilibrium with exponential rate. Construction of a coupling that satisfies the hypotheses of the mentioned theorem is based on four key ingredients: the Foia\c{s}-Prodi estimate for NLW, the Girsanov theorem, the recurrence property of solutions, and the stopping time technique. 

Finally, we make some comments on the hypotheses imposed on the nonlinear term $f$ and the coefficients $b_j$ entering the definition of the force $\eta$. Inequalities \eqref{1.5.3}-\eqref{1.6.3} are standard in the study of NLW equation, they ensure that the Cauchy problem is well-posed (e.g., see \cite{CV2002} and \cite{Lions1969} for deterministic cases). The hypothesis $\rho<2$ is needed to prove the stability of solutions. The fact that the coefficients $b_j$ are not zero ensures that $\eta$ is non-degenerate in all Fourier modes, which is used to establish the recurrence of solutions and exponential squeezing. As was mentioned above, we show that this condition could be relaxed.

The paper is organized as follows. In Section \ref{2.26.3} we announce the main result and outline the scheme of its proof. Next, the large time behavior and stability of solutions are studied in Sections \ref{3.0.3} and \ref{4.0.3}, respectively. Finally, the complete proof of the main result is presented in Section \ref{6.31.3}.

\subsection*{Notation}
For an open set $D$ of a Euclidean space and separable Banach spaces $X$ and $Y$, we introduce the following function spaces:

\nt
$L^p=L^p(D)$ is the Lebesgue space of measurable functions whose $p^{\text{th}}$ power is integrable. In the case $p=2$ the corresponding norm is denoted by $\|\cdot\|$.

\nt
$H^s=H^s(D)$ is the Sobolev space of order $s$ with the usual norm $\|\cdot\|_s$.

\nt
$H^s_0=H^s_0(D)$ is the closure in $H^s$ of infinitely smooth functions with compact support.

\nt
$H^{1,p}=H^{1,p}(D)$ is the Sobolev space of order $1$ with exponent $p$, that is, the space of $L^p$ functions whose first order derivatives remain in $L^p$.

\nt
$L(X,Y)$ stands for the space of linear continuous operators from $X$ to $Y$ endowed with the natural norm.

\nt
$C_b(X)$ is the space of continuous bounded functions $\psi:X\to\rr$ endowed with the norm of uniform convergence:
$$
|\psi|_\infty=\sup_{x\in X}|\psi(x)|.
$$
$L_b(X)$ is the space of bounded Lipschitz functions, i.e. of functions $\psi\in C_b(X)$ such that
$$
|\psi|_L:=|\psi|_\infty+\sup_{x\neq y}\f{|\psi(x)-\psi(y)|}{|x-y|_X}<\infty.
$$

\nt
$B_X(R)$ stands for the ball in $X$ of radius $R$ and centered at the origin.

\nt
$\bbb(X)$ is the Borel $\sigma$-algebra of subsets of $X$.

\nt
$\ppp(X)$ denotes the space of probability Borel measures on $X$. Two metrics are defined on the space $\ppp(X)$: the metric of total variation
$$
|\mu_1-\mu_2|_{var}=\sup_{\Gamma\in \bbb(X)}|\mu_1(\Gamma)-\mu_2(\Gamma)|,
$$
and the dual Lipschitz metric
$$
|\mu_1-\mu_2|_L^*=\sup_{|\psi|_L\leq 1}|(f,\mu_1)-(f,\mu_2)|,
$$
where $(\psi,\mu)$ denotes the integral of $\psi$ over $X$ with respect to $\mu$. 

\nt
Finally, by $C_1,C_2,\ldots$, we shall denote unessential positive constants.

\section{Exponential mixing}\label{2.26.3}
We start this section by a short discussion of the well-posedness of the Cauchy problem for equation \eqref{1.1.3}. We then state the main result and outline the scheme of its proof.
\subsection{Existence and uniqueness of solutions}
Before giving the definition of a solution of equation \eqref{1.1.3}, let us make the precise hypotheses on the nonlinearity and the coefficients entering the definition of $\eta(t)$. We suppose that the function $f$ satisfies the growth restriction
\be\label{1.8.3}
|f''(u)|\leq C(|u|^{\rho-1}+1),\q u\in\rr,
\ee  
where $C$ and $\rho<2$ are positive constants, and the dissipativity conditions
\begin{align}
F(u)&\geq -\nu u^2-C,\q u\in\rr\label{1.5.3},\\
f(u)u- F(u)&\geq-\nu u^2-C, \q u\in\rr\label{1.6.3},
\end{align}
where $F$ is the primitive of $f$, $\nu\leq (\lm_1\wedge\gamma)/8$ is a positive constant, and $\lm_j$ stands for the eigenvalue corresponding to $e_j$. The coefficients $b_j$ are supposed to be positive numbers satisfying
\be \label{2.24.3}
\BBB=\sum_{j=1}^\iin b_j^2<\iin,\q\BBB_1=\sum_{j=1}^\iin\lm_j b_j^2<\infty.
\ee
Let us introduce the functions
$$
g_j=[0,b_j e_j],\q \hat\zeta(t)=\sum_{j=1}^\iin\beta_j(t)g_j.
$$
\begin{definition}\label{definition 2.1}
Let $y_0=[u_0,u_1]$ be a $\h$-valued random variable defined on a complete probability space $(\omm,\fff,\pp)$ that is independent of $\hat\zeta(t)$. A random process $y(t)=[u(t),\dt u(t)]$ defined on  $(\omm,\fff,\pp)$ is called \emph{a solution} (or \emph{a flow}) of equation \eqref{1.1.3} if the following two  conditions hold:
\bi 
\item Almost every trajectory of $y(t)$ belongs to the space $C(\rr_+;\h)$, and the process $y(t)$ is adapted to the filtration $\fff_t$ generated by $y_0$ and $\hat\zeta(t)$.

\item Equation \eqref{1.1.3} is satisfied in the sense that, with probability 1,
\be\label{2.1.3}
y(t)=y_0+\int_0^t g(s)\,ds+\hat\zeta(t),\q t\geq 0,
\ee
where we set
$$
g(t)=[\dt u,-\gamma \dt u+\de u-f(u)+h(x)],
$$
and relation \eqref{2.1.3} holds in $L^2\times H^{-1}$.
\ei
\end{definition}

\nt
Let us endow the space $\h$ with the norm
$$
|y|_{\h}^2=\|\g y_1\|^2+\|y_2+\al y_1\|^2\q \text{ for } y=[y_1,y_2]\in\h,
$$
where $\al>0$ is a small parameter.
Introduce the energy functional
\be 
\ees(y)=|y|_\h^2+2\int_D F(y_1)\,dx, \q y=[y_1,y_2]\in\h,
\ee 
and let $\ees_u(t)=\ees(y(t))$. We have the following theorem.
\begin{theorem}\label{2.12.3}
Under the above hypotheses, let $y_0$ be an $\h-$valued random variable that is independent of $\hat\zeta$ and satisfies $\e\ees(y_0)<\iin$. Then equation \eqref{1.1.3} possesses a solution in the sense of Definition \ref{definition 2.1}. Moreover, it is unique, in the sense that if $\tilde y(t)$ is another solution, then with $\pp$-probability 1 we have $y(t)=\tilde y(t)$ for all $t\geq 0$. In addition, we have the a priori estimate
\be\label{2.2.3}
\e\ees_u(t)\leq\e\ees_u(0)e^{-\al t}+C(\gamma,\BBB,\|h\|).
\ee
\end{theorem}
We refer the reader to the book \cite{DZ1992} for proofs of similar results. We confine ourselves to the formal derivation of inequality \eqref{2.2.3} in the next section.
\subsection{Main result and scheme of its proof}\label{Main result and scheme of its proof}
Let us denote by $S_t(y,\cdot)$ the flow of equation \eqref{1.1.3} issued from the initial point $y\in\h$. A standard argument shows that $S_t(y,\cdot)$ defines a Markov process in $\h$ (e.g., see \cite{DZ1992, KS-book}). We shall denote by $(y(t),\pp_y)$ the corresponding Markov family. In this case, the Markov operators have the form

\begin{align*}
\PPPP_t\psi(y)&=\int_\h\psi(z)P_t(y,dz) \q \text{ for any } \psi\in C_b(\h),\\
\PPPP^*_t\lm(\Gamma)&=\int_\h P_t(y,\Gamma)\lm(dy)\q \text{ for any } \lm\in\ppp(\h),
\end{align*}
where $P_t(y,\Gamma)=\pp_y(S_t(y,\cdot)\in\Gamma)$ is the transition function.
The following theorem on exponential mixing is the main result of this paper.
\begin{theorem}\label{2.15.3}
Under the hypotheses of Theorem \ref{2.12.3}, the Markov process associated with the flow of equation \eqref{1.1.3} has a unique stationary measure $\mu\in\ppp(\h)$. Moreover, there exist positive constants $C$ and $\kp$ such that for any $\lm\in\ppp(\h)$ we have
\be\label{6. In-main inequality}
|\PPPP^*_t\lm-\mu|_L^*\leq C e^{-\kp t}\int_\h \exp(\kp|y|_\h^4)\,\lm(dy).
\ee
\end{theorem}

{\it Scheme of the proof.}
We shall construct an extension for the family $(y(t),\pp_y)$ that satisfies the hypotheses of Theorem 3.1.7 in \cite{KS-book}, providing a general criterion for exponential mixing. To this end, let us fix an initial point $\yy=(y,y')$ in $\hh=\h\times\h$, and let $\xi_u=[u,\p_t u]$ and $\xi_{u'}=[u',\p_t u']$ be the flows of equation \eqref{1.1.3} that are issued from $y$ and $y'$, respectively. Consider an intermediate process $v$, which is the solution of 
\be\label{2.25.3}
\p_t^2v+\gamma \p_t v-\de v+f(v)+P_N[f(u)-f(v)]=h(x)+\eta(t,x), \q\xi_v(0)=y'.
\ee
Here $N\geq 1$ is integer that will be chosen later, and $P_N$ stands for the orthogonal projection from $L^2(D)$ to  its $N$-dimensional subspace spanned by the functions $e_1,e_2,\ldots,e_N$.
Let us denote by $\lm(y,y')$ and $\lm'(y,y')$ the laws of the processes $\{\xi_v\}_T$ and $\{\xi_{u'}\}_T$, respectively, where $\{z\}_T$ stands for the restriction of $\{z(t); t\geq 0\}$ to $[0,T]$. Thus, $\lm$ and $\lm'$ are probability measures on $C(0,T;\h)$. Let $(\vvv(y,y'),\vvv'(y,y'))$ be a maximal coupling for $(\lm(y,y'),\lm'(y,y'))$. By Proposition 1.2.28 in \cite{KS-book}, such a pair exists and can be chosen to be a measurable function of its arguments. For any $s\in [0,T]$, we shall denote by $\vvv_s$ and $\vvv'_s$ the restrictions of $\vvv$ and $\vvv'$ to the time $s$. Denote by $[\tilde v,\p_t \tilde v]$ and $[\tilde u',\p_t \tilde u']$ the corresponding flows. Then we have
\be\label{7.1.3}
\p_t^2 \tilde v+\gamma \p_t \tilde v-\de\tilde v+f(\tilde v)-P_N f(\tilde v)=h(x)+\psi(t),\q \xi_{\tilde v}(0)=y',
\ee
where $\psi$ satisfies
\be\label{6.30.3}
\ddd\{\int_0^t \psi(s)\,ds\}_T=\ddd\{\zeta(t)-\int_0^t P_N f(u)\,ds\}_T.
\ee
Introduce an auxiliary process $\tilde u$, which is the solution of 
\be 
\p_t^2 \tilde u+\gamma \p_t \tilde u-\de\tilde u+f(\tilde u)-P_N f(\tilde u)=h(x)+\psi(t),\q \xi_{\tilde u}(0)=y.
\ee
Let us note that $u$ satisfies the same equation, where $\psi$ should be replaced by $\eta(t)-P_N f(u)$. In view of \eqref{6.30.3}, we have (see the appendix for the proof)
\begin{equation}\label{6.40.3}
\ddd\{\xi_{\tilde u}\}_T=\ddd\{\xi_{u}\}_T.
\end{equation}
Introduce 
\be
\rrr_t(y,y')=\xi_{\tilde u}(t), \q \rrr_t'(y,y')=\xi_{\tilde u'}(t) \q \text{ for } t\in[0,T].\label{6. Coupling relation 1}
\ee
It is clear that $\RR_t=(\rrr_t,\rrr'_t)$ is an extension of $S_t(y)$ on the interval $[0,T]$. Let $\SSS_t=(S_t(\yy), S'_t(\yy))$ be the extension of $S_t(y)$ constructed by iteration of $\RR_t=(\rrr_t,\rrr'_t)$ on the half-line $t\geq 0$ (we do not recall here the procedure of construction, see the paper \cite{shirikyan-bf2008} for the details). With a slight abuse of notation, we shall keep writing $[\tilde u, \p_t\tilde u]$ and $[\tilde u', \p_t\tilde u']$ for the extensions of these two processes, and write $\xi_{\tilde v}(t)=\vvv_s(\SSS_{kT}(\yy))$ for $t=s+kT,\q 0\leq s<T$. This will not lead to a confusion.

\nt For any continuous process $y(t)$ with range in $\h$, we introduce the functional
\be\label{2.13.3}
\fff_y(t)=|\ees(y(t))|+\al\int_0^t|\ees (y(s))|\,ds,
\ee
and the stopping time
\be\label{6 In-6,1,1}
\tau_y=\inf\{t\geq 0:\fff_y(t)\geq \fff_y(0)+(L+M)t+r\},
\ee
where $L,M$ and $r$ are some positive constants to be chosen later.
In the case when $y$ is a process of the form $y=[z,\dt z]$, we shall write, $\fff^z$ and $\tau^z$ instead of $\fff_{[z,\dt z]}$ and $\tau_{[z,\dt z]}$, respectively.
Introduce the stopping times:
\begin{align*}
 \vo&=\inf\{t=s+kT: \vvv_s(\SSS_{kT}(\yy))\neq \vvv'_s(\SSS_{kT}(\yy))\}\equiv\inf\{t\geq 0: \xi_{\tilde v}(t)\neq \xi_{\tilde u'}(t)\},\\
 \tau&=\tau^{\tilde u}\wedge\tau^{\tilde u'},\q \sigma=\vo\wedge\tau.
 \end{align*}
Suppose that we are able to prove the following.
\begin{theorem}\label{2.14.3}
Under the hypotheses of Theorem \ref{2.15.3}, there are positive constants $\al,\De,\kp,d$ and $C$ such that the following properties hold.\\
\text{\textnormal{(Recurrence):}} For any $\yy=(y,y')\in\hh$, we have
\begin{align}
\e_y\exp(\kp\ees(y(t))&\leq\e_y\exp(\kp\ees (y(0))e^{-\al t}+C(\gamma,\BBB,\|h\|), \label{2.17.3}\\
\e_{\yy}\exp(\kp\tau_d)&\leq C(1+|\yy|_{\hh}^4), \label{2.16.3}
\end{align}
where $\tau_d$ stands for the first hitting time of the set $B_{\hh}(d)$.\\
\text{\textnormal{(Exponential squeezing):}} For any $\yy\in B_{\hh}(d)$, we have
\begin{align}
|S_t(\yy)-S_t'(\yy)|^2_\h&\leq Ce^{-\al t}|y-y'|_\h^2 \q\text{ for } 0\leq t\leq\sigma\label{6.33.3},\\
\pp_\yy\{\sigma=\infty\}&\geq\De\label{6.3.3},\\
\e_\yy[\ch_{\{\sigma<\infty\}}\exp(\De\sigma)]&\leq C\label{6.4.3},\\
\e_\yy[\ch_{\{\sigma<\infty\}} |\yy(\sigma)|_{\hh}^{8}]&\leq C\label{6.5.3}.
\end{align}
\end{theorem}
In view of Theorem 3.1.7 in \cite{KS-book}, this will imply Theorem \ref{2.15.3}. We establish Theorem \ref{2.14.3} in Section \ref{6.31.3}. The proof of \mpp{recurrence} relies on the Lyapunov function technique, while the proof of \mpp{exponential squeezing} is based on the Foia\c{s}-Prodi type estimate for equation \eqref{1.1.3}, the Girsanov theorem and the stopping time argument.

\subsection{Law of large numbers and central limit theorem}\label{Law of large numbers and central limit theorem}
Theorem \ref{2.15.3} implies the following result, which follows from inequality \eqref{6. In-main inequality} and some results established in Section 2 of \cite{shirikyan-ptrf2006}.
\begin{theorem}
Under the hypotheses of Theorem \ref{2.15.3}, for any Lipschitz bounded functional $\psi:\h\to\rr$ and any solution $y(t)=[u(t),\dt u(t)]$ of equation \eqref{1.1.3} issued from a non-random point $y_0\in\h$, the following statements hold.\\
\text{\textnormal{Strong law of large numbers.}} For any $\es>0$ there is an almost surely finite random constant $l\geq 1$ such that
\be\label{6.51.3}
|t^{-1}\int_0^t \psi(y(s))\,ds-(\psi,\mu)|\leq C(y_0,\psi) t^{-\f{1}{2}+\es} \q \text{ for } t\geq l.
\ee
\text{\textnormal{Central limit theorem.}} If $(\psi,\mu)=0$, there is a constant $a\geq 0$ depending only on $\psi$, such that for any $\es>0$, we have
\be\label{6.52.3}
\sup_{z\in\rr}(\theta_a(z)\cdot |\pp\{t^{-\f{1}{2}}\int_0^t\psi(y(s))\,ds\leq z\}-\Phi_a(z)|)\leq C(y_0,\psi)t^{-\f{1}{4}+\es},
\ee
where we set
$$
\theta_a(z)\equiv 1,\q \Phi_a(z)=\f{1}{a\sqrt{2\pi}}\int_{-\iin}^z e^{-\f{s^2}{2 a^2}}\,ds \q\text{ for } a>0,
$$
and 
$$
\theta_0(z)=1\wedge|z|,\q \Phi_0(z)=\ch_{\rr_+}(z).
$$
\end{theorem}

\nt
The proof of inequalities \eqref{6.51.3} and  \eqref{6.52.3} follow, respectively, from Corollary 3.4 and Theorem 2.8 in \cite{shirikyan-ptrf2006}, combined with inequalities \eqref{6.53.3} and \eqref{2.17.3}. \section{Large time estimates of solutions}\label{3.0.3}
The goal of this section is to analyze the dynamics of solutions and to obtain some a priori estimates for them.
\subsection{Proof of inequality \eqref{2.2.3}}

Let us apply the It\^o formula to the function $\GG(y)=|y|_\h^2$. Recall that for the process of the form \eqref{2.1.3}, the It\^o formula gives 
\be\label{2.3.3}
\GG(y(t))=\GG(y(0))+\int_0^t A(s)\,ds+\sum_{j=1}^\iin\int_0^t B_j(s)d\beta_j(s),
\ee
where we set
$$
A(t)=(\p_y \GG)(y(t);g(t))+\f{1}{2}\sum_{j=1}^\iin(\p_y^2 \GG)(y(t);g_j,g_j),\q
B_j(t)=(\p_y \GG)(y(t);g_j).
$$
Here $(\p_y\GG)(y;v)$ and $(\p_y^2\GG)(y;v,v)$ stand for the values of the first- and second-order derivatives of $\GG$ on the vector $v$.  
Since for $\GG(y)=|y|_\h^2$ we have
$$
\p_y\GG(y;\bar y)=2(y,\bar y)_\h,\q \p_y^2\GG(y;\bar y,\bar y)=2|\bar y|_\h^2,
$$
relation \eqref{2.3.3} takes the form
\be\label{2.3,2}
|y(t)|_{\h}^2=|y(0)|_\h^2+2\int_0^t(y,g)_\h\,ds+t\cdot\sum_{j=1}^\infty|g_j|_\h^2 +2\sum_{j=1}^\infty\int_0^t(y,g_j)_\h\,d\beta_j(s).
\ee
Let us note that
\begin{align}
(y,g)_\h&=(\g u,\g \dt u)+(\dt u+\al u,-\gamma \dt u+\de u-f(u)+h(x)+\al \dt u)\notag\\
&=-\al\|\g u\|^2-(\gamma-\al)\| \dt u\|^2+(\al^2-\al\gamma)(u,\dt u)+(\dt u+\al u,h)\notag\\
&\q\,-(\dt u+\al u,f(u))\label{2.4.3}.
\end{align}
By the Young and Poincar\'e inequalities, we have
\begin{align}
|(\al^2-\al\gamma)(u,\dt u)|&\leq \f{\al}{16}\|\g u\|^2+\f{4\al(\gamma-\al)^2}{\lm_1}\|\dt u\|^2\label{6.65.3},\\
|(\al u,h)|&\leq \f{\al}{16}\|\g u\|^2+\f{4\al}{\lm_1}\|h\|^2,\\
|(\dt u,h)|&\leq \f{\gamma-\al}{4}\|\dt u\|^2+(\gamma-\al)^{-1}\|h\|^2.\label{6.66.3}
\end{align}
Note also that, thanks to inequality \eqref{1.6.3}, we have 
\be\label{2.3,3}
-\al f(u)u\leq - \al F(u)+\al \nu u^2+\al C\leq - \al F(u)+\f{\al \lm_1}{8} u^2+\al C,
\ee
so that
\be\label{6.67.3}
-(\al u,f(u))\leq  - \al\int_D F(u)+\f{\al}{8} \|\g u\|^2+\al C\cdot\text{Vol}(D)
\ee
Now, by substituting \eqref{2.4.3} into \eqref{2.3,2}, using inequalities \eqref{6.65.3}-\eqref{6.67.3}, and noting that
$$
\int_0^t(\dt u,f(u))\,ds=\int_0^t\f{d}{ds}F(u(s))\,ds=F(u(t))-F(u(0)),
$$
we obtain that for $\al>0$ sufficiently small
\be\label{2.19.3}
\ees_u(t)\leq \ees_u(0)+\int_0^t\left(-\al\ees_u(s)+\kkk\right)\,ds-\f{\al}{2}\int_0^t|y(s)|_\h^2\,ds+M(t),
\ee
where $\kkk>0$ depends only on $\gamma, \BBB$ and $\|h\|$, and $M(t)$ is the stochastic integral
\be\label{2.20.3}
M(t)=2\sum_{j=1}^\infty b_j\int_0^t (\dt u+\al u,e_j)\,d\beta_j(s).
\ee
Taking the mean value in inequality \eqref{2.19.3} and using the Gronwall comparison principle, we arrive at \eqref{2.2.3}.
\subsection{Exponential moment of the flow}
In the following proposition we establish the uniform boundedness of exponential moment of $|\xi_u(t)|_\h$.
\begin{proposition}\label{proposition_exp}
Under the hypotheses of Theorem \ref{2.12.3}, there exists $\kp>0$ such that if the random variable $\ees_u(0)$ satisfies
$$
\e\exp(\kp\ees_u(0))<\infty,
$$
then
\be\label{2. In-exponential-moment-of-the-flow} 
\e\exp(\kp\ees_u(t))\leq\e\exp(\kp\ees_u(0))e^{-\al t}+C(\gamma,\BBB,\|h\|).
\ee
\end{proposition}
\bp
We represent $\xi_u(t)$ in the form \eqref{2.1.3}, and apply the It\^o formula \eqref{2.3.3} to the function  
$$
\GG(y)=\exp(\kp\ees(y)).
$$
Since 
\begin{align*}
\p_y\GG(y,\bar y)&=2\kp\,\GG(y)((y,\bar y)_\h+(f(y_1),\bar y_1)),\\
\p_y^2\GG(y;\bar y,\bar y)&=2\kp\,\GG(y)(2\kp\,((y,\bar y)_\h+(f(y_1),\bar y_1)^2+|\bar y|_\h^2+(f'(y_1),\bar y_1^2),
\end{align*}
we have
\begin{align*}
\p_y\GG(y;g)&=2\kp\,\GG(y)((y,g)_\h+(f(u),\dt u)),\\
\p_y^2\GG(y;g_j,g_j)&=2\kp\,\GG(y)(2\kp(y,g_j)_\h^2+|g_j|_\h^2).
\end{align*}
Hence, relation \eqref{2.3.3}, after taking the mean value, takes the form
$$
\e\GG(y(t))=\e\GG(y(0))+\kp\,\e\int_0^t \GG(y(s))\mmm(s)\,ds,
$$
where
$$
\mmm(t)=2((y,g)_\h+(f(u),\dt u))+2\kp\sum_{j=1}^\iin(y,g_j)_\h^2+\sum_{j=1}^\iin|g_j|_\h^2.
$$
Now note that by developing the expression $(y,g)_\h+(f(u),\dt u)$, the term $(f(u),\dt u)$ will disappear (see \eqref{2.4.3}). There remains another term containing $f$, namely the term $(-\al u,f(u))$, but this can be estimated using inequality \eqref{6.67.3}. Let us choose $\kp>0$ so small that $\kp\,\BBB\leq\al/2$. It follows that $\GG(y)$ satisfies
$$
\e\GG(y(t))\leq\e\GG(y(0))+\kp\,\e\int_0^t \GG(y(s))(-\al\ees(y(s))+C(\gamma,\BBB,\|h\|)\,ds.
$$
It remains to use the inequality
$$
\kp e^v(-\al v+C_1)\leq -\al e^v+C_2\q \text{ for all } v\geq -C,
$$
and the Gronwall lemma, to conclude.
\ep 
\subsection{Exponential supermartingale-type inequality}
The following result provides an estimate for the rate of growth of solutions.
\begin{proposition}\label{supermartingale}
Under the hypotheses of Theorem \ref{2.12.3}, the following a priori estimate holds for solutions of equation \eqref{1.1.3}
\be 
\pp \bg\{\sup_{t\geq 0}(\ees_u(t)+\int_0^t(\al \ees_u(s)-\kkk)\,ds)\geq\ees_u(0)+r\bg\}\leq e^{-\beta r} \q\tx{ for any } r>0,
\ee
where $\kkk$ is the constant from inequality \eqref{2.19.3}, and $\beta=\al/8\cdot(\sup b_j^2)^{-1}$.
\end{proposition}
\bp
Let us first note that
\begin{align*}
\e\sum_{j=1}^\infty b_j^2\int_0^t(\dt u+\al u,e_j)^2\,ds\leq (\sup_{j\geq 1}b_j^2)\int_0^t\e\|\dt u+\al u\|^2\,ds<\infty,\,\quad  t\geq 0.
\end{align*}
It follows that the stochastic integral $M(t)$ defined in \eqref{2.20.3} is a martingale, and its quadratic variation $\langle M\rangle(t)$ equals 
$$
\langle M\rangle(t)=4\sum_{j=1}^\infty b_j^2\int_0^t(\dt u+\al u,e_j)^2\,ds\leq 4\sup_j b_j^2\int_0^t\|\dt u+\al u\|^2\,ds.
$$
Combining this with inequality inequality \eqref{2.19.3}, we obtain
$$
\ees_u(t)+\int_0^t(\al \ees_u(s)-\kkk)\,ds\leq\ees_u(0)+\left(M(t)-\f{1}{2}\beta\langle M\rangle(t)\right).
$$
We conclude that
\begin{align*}
&\pp \bg\{\sup_{t\geq 0}(\ees_u(t)+\int_0^t(\al \ees_u(s)-\kkk)\,ds)\geq\ees_u(0)+r)\bg\}\\
&\leq\pp \bg\{\sup_{t\geq 0}(M(t)-\f{1}{2}\beta\langle M\rangle(t))\geq r\bg\}=\pp \bg\{\sup_{t\geq 0}\exp(\beta M(t)-\f{1}{2}\langle \beta M\rangle(t))\geq e^{\beta r}\bg\}\\
&\leq e^{-\beta r},
\end{align*}
where we used the exponential supermartingale inequality.
\ep 

\medskip
\nt
We recall that for a process of the form $y(t)=[u(t),\dt u(t)]$, $\fff^u\equiv\fff_y$ stands for the functional defined by \eqref{2.13.3}, and $\tau^u\equiv\tau_y$ stands for the stopping time defined by \eqref{6 In-6,1,1}.
\begin{corollary}\label{supermartingale lemma}
Suppose that the hypotheses of Theorem \ref{2.12.3} are fulfilled. Then for any solution $u(t)$ of equation \eqref{1.1.3}, we have
\begin{align*}
\pp \bg\{\sup_{t\geq 0}(\fff^u(t)-Lt)\geq\fff^u(0)+r\bg\}&\leq  \exp(4\beta C-\beta r) \q\q\q\q\tx{ for any } r>0,\nt\\
\pp\{l\leq\tau^u<\iin\}&\leq \exp(4\beta C-\beta r-\beta l M)  \q\tx{ for any } l\geq 0,
\end{align*}
where $L=\kkk+4\al C$, $\kkk$ and $\beta$ are the constants from the previous proposition and $C$ is the constant from inequalities \eqref{1.5.3}-\eqref{1.6.3}.
\end{corollary}

\nt
This result follows from Proposition \ref{supermartingale} and the fact that, due to inequality \eqref{1.5.3}, we have
$$
\ees_u(t)\leq|\ees_u(t)|\leq\ees_u(t)+4 C.
$$

\subsection{Existence of stationary measure}
In this subsection we show that the process $y(t)=[u(t),\dt u(t)]$ has a bounded second moment in the more regular space $\h^s=H^{s+1}(D)\times H^s(D)$, with $s=s(\rho)>0$ sufficiently small. By the Bogolyubov-Krylov argument, this immediately implies the existence of stationary distribution for the corresponding Markov process. 
\begin{proposition}\label{2.11.3}
Under the hypotheses of Theorem \ref{2.12.3}, there is an increasing function $Q$ such that, for any $s\in(0,1-\rho/2)$, and any solution of equation \eqref{1.1.3}, we have
$$
\e|y(t)|_{\h^s}^2\leq Q(|y(0)|_{\h})+|y(0)|_{\h^s}^2e^{-\al t}.
$$
\end{proposition}
\bp
Let us split $u$ to the sum $u=v+z$, where $v$ solves
\be \label{2.6.3}
\p_t^2 v+\gamma \p_t v-\de v=h(x)+\eta(t), \q \xi_v(0)=\xi_u(0).
\ee 
The standard argument shows that for any $s\in[0,1]$, we have
\be 
\e|\xi_v(t)|_{\h^s}^2\leq C(\gamma, \|h\|_1)+|y(0)|_{\h^s}^2 e^{-\al t},
\ee
so that it remains to bound the average of $|\xi_z(t)|_{\h^s}^2$. In view of \eqref{1.1.3} and \eqref{2.6.3}, $z(t)$ is the solution of
\be\label{2.7.3} 
\p_t^2 z+\gamma \p_t z-\de z+f(u)=0,\q \xi_z(0)=0.
\ee
We now follow the argument used in \cite{zelik2004}.
Let us differentiate \eqref{2.7.3} in time, and set $\theta=\p_t z$. Then $\theta$ solves
\be\label{2.8.3} 
\p_t^2 \theta+\gamma \p_t \theta-\de \theta+f'(u)\p_t u=0,\q [\theta(0),\dt \theta(0)]=[0,-f(u(0))].
\ee
Let us fix $s\in(0,1-\rho/2)$, multiply this equation by $(-\de)^{s-1}(\dt \theta+\al\theta)$ and integrate over $D$. We obtain
\be\label{2.9.3}
\f{d}{dt}\tilde\ees_{\theta}(t)+\f{3\al}{2}\tilde\ees_{\theta}(t)\leq 2\int_D |f'(u)\dt u||(-\de)^{s-1}(\dt \theta+\al\theta)|\,dx=:\elll,
\ee
where we set
$$
\tilde\ees_{\theta}(t)=|\xi_\theta|_{\h^{s-1}}^2+\al\gamma|\theta|_{H^{s-1}}^2+2\al(\theta,\dt \theta)_{H^{s-1}}.
$$
By the H\"older and Sobolev inequalities
\begin{align*}
\elll&\leq C_1\int_D (|u|^\rho+1)|\dt u||(-\de)^{s-1}(\dt \theta+\al\theta)|\,dx\\
&\leq C_1(|u|_{L^6}^\rho+1)|\dt u|_{L^2}|(-\de)^{s-1}(\dt \theta+\al\theta)|_{L^{6/(3-\rho)}}\\
&\leq C_2(|u|_{L^6}^2+1)|\dt u|_{L^2}|(-\de)^{s-1}(\dt \theta+\al\theta)|_{H^{1-s}}\\
&\leq C_3(\|\g u\|^2+1)\|\dt u\||\dt \theta+\al\theta|_{H^{s-1}}\leq \f{\al}{2}\tilde\ees_{\theta}(t)+C_4(\|\g u\|^4+1)\|\dt u\|^2,
\end{align*}
where we used the embedding $H^{1-s}\hookrightarrow L^{6/(3-\rho)}$.
Substituting this estimate in \eqref{2.9.3} and taking the mean value we obtain
$$
\f{d}{dt}\e\tilde\ees_{\theta}(t)\leq -\al\e\tilde\ees_{\theta}(t)+C_4\e(\|\g u\|^4+1)\|\dt u\|^2.
$$
Applying the Gronwall lemma and using Proposition \ref{proposition_exp}, we see that
$$
\e\tilde\ees_{\theta}(t)\leq\e\tilde\ees_{\theta}(0)+C_5,
$$
where the constant $C_5$ depends only on $\al$ and $|y(0)|_\h$.  Moreover, by \eqref{2.8.3} we have
$$
\tilde\ees_{\theta}(0)=|f(u(0))|_{H^{s-1}}^2\leq |f(u(0))|_{L^2}^2\leq C(1+|y(0)|_\h^6), 
$$
so that 
\be\label{2.10.3}
\e\tilde\ees_{\theta}(t)\leq Q_1(|y(0)|_\h).
\ee
In view of \eqref{2.7.3}
\begin{align*}
|z|_{H^{s+1}}=|\de z|_{H^{s-1}}=|\ddot z+\gamma \dt z+f(u)|_{H^{s-1}}&\leq |\ddot z+\gamma \dt z|_{H^{s-1}}+|f(u)|_{L^2}\\
&\leq |\dt \theta+\gamma \theta|_{H^{s-1}}+C(1+\|\g u\|^3),
\end{align*}
whence 
$$
|z(t)|_{H^{s+1}}^2\leq 2\tilde\ees_{\theta}(t)+C_6(1+|y(t)|_\h^{6}).
$$
Taking the mean value in this inequality and using \eqref{2. In-exponential-moment-of-the-flow} together with \eqref{2.10.3}, we obtain
$$
\e|\xi_z(t)|_{\h^s}^2\leq Q_2(|y(0)|_\h).
$$
This completes the proof of Proposition \ref{2.11.3}.
\ep 

\section{Stability of solutions}\label{4.0.3}
In this section we establish the stability and the recurrence property of solutions of equation \eqref{1.1.3}.
\subsection{The Foia\c{s}-Prodi estimate}
Here we establish an estimate which will allow us to use the Girsanov theorem.
Let us consider the following two equations:
\begin{align}
\p_t^2 u+\gamma \p_t u-\de u+f(u)&=h(x)+\p_t g(t,x),\label{FP"IN"1}\\
\p_t^2 v+\gamma \p_t v-\de v+f(v)+P_N[f(u)-f(v)]&=h(x)+\p_t g(t,x)\label{FP"IN"2},
\end{align}
where $g(t)$ is a function in $C(\rr_+;H^1_0(D))$. We recall that $P_N$ stands for the orthogonal projection from $L^2(D)$ to its subspace spanned by the functions $e_1,e_2,\ldots,e_N$.

\begin{proposition}\label{4.13.3}
Suppose that for some non-negative constants $K,l,s$ and $T$ the inequality
\be\label{4.17.3}
\int_s^t\|\g z\|^2\,d\tau\leq l+K(t-s) \q \text{ for } s\leq t\leq s+T,
\ee
holds for $z=u$ and $z=v$, where $u$ and $v$ are solutions of \eqref{FP"IN"1} and \eqref{FP"IN"2}, respectively. Then, for any $\es>0$ there is an integer $N_*\geq 1$ depending only on $\es$ and $K$ such that for all $N\geq N_*$ we have
\be \label{4.16.3}
|\xi_v(t)-\xi_u(t)|^2_\h\leq e^{-\al(t-s)+\es l}|\xi_v(s)-\xi_u(s)|^2_\h \q \text{ for } s\leq t\leq s+T.
\ee
\end{proposition}

\bp 
 Let us set $w=v-u$. Then $w(t)$ solves
\be\label{6.50.3}
\p_t^2 w+\gamma \p_t w-\de w+(I-P_N)[f(v)-f(u)]=0,
\ee
and we need to show that the flow $y(t)=\xi_w(t)$ satisfies
\be\label{4.15.3}
|y(t)|^2_\h\leq e^{-\al (t-s)+\es l}|y(s)|^2_\h \q \text{ for } s\leq t\leq s+T.
\ee
The function $y(t)$ satisfies
\begin{align}
\p_t|y|_{\h}^2&=2[(\g w,\g \dt w)+(-\gamma \dt w+\de w-(I-P_N)[f(v)-f(u)]+\al \dt w,\dt w+\al w)]\notag\\
&\leq 2[(\g w,\g \dt w)
+(-\gamma \dt w+\de w+\al \dt w,\dt w+\al w)]\notag\\
&\q+2\|(I-P_N)[f(v)-f(u)]\|(\|\dt w\|+\al\|w\|)\notag\\
&\leq -\f{3\al}{2} |y|_\h^2+4\|(I-P_N)[f(v)-f(u)]\||y|_\h\label{4.8.3}.
\end{align}
We first note that
\be\label{4.7.3}
\|(I-P_N)[f(v)-f(u)]\|\leq|I-P_N|_{L(H^{1,p}\to L^2)}|f(v)-f(u)|_{H^{1,p}},
\ee
and that
\be\label{FP"IN"3}
|f(v)-f(u)|_{H^{1,p}}\leq C\sum_{j=1}^3|\p_j[f(v)-f(u)]|_{L^p},
\ee
where $p\in(6/5,2)$ will be chosen later.
Further,
\begin{align}
|\p_j[f(v)-f(u)]|_{L^p}&=|f'(v)\p_j v-f'(u)\p_j u|_{L^p}\notag\\
&\leq |(f'(v)-f'(u))\p_j v|_{L^p}+|f'(u)(\p_j v-\p_j u)|_{L^p}=J_1+J_2\label{FP"IN"4}.
\end{align}
For $J_1$ we have 
\begin{align}
J_1&=\left(\int_D |(f'(v)-f'(u))\p_j v|^p\,dx\right)^{\f{1}{p}}\notag\\
&\leq C_4\left(\int_D |w|^p|\p_j v|^p(|v|^{p(\rho-1)}+|u|^{p(\rho-1)}+1)\,dx\right)^{\f{1}{p}}\notag\\
&\leq C_5|w|_{L^6}|\g v|_{L^2}(|v|_{L^6}^{\rho-1}+|u|_{L^6}^{\rho-1}+1)\leq C_6\|w\|_1(\|v\|_1^2+\|u\|_1^2+1)\label{FP"IN"5},
\end{align}
where we used the H\"older and Sobolev inequalities and chose $p=6(3+\rho)^{-1}$.
And finally, for $J_2$ we have
\begin{align}
J_2=\left(\int|f'(u)|^p|\p_j v-\p_j u|^p\,dx\right)^{\f{1}{p}}&\leq C_7\left(\int|\p_j w|^p (|u|^{p\rho}+1)\,dx\right)^{\f{1}{p}}\notag\\
&\leq C_8\|w\|_1(\|v\|_1^2+\|u\|_1^2+1)\label{FP"IN"6},
\end{align}
where we once again used the H\"older inequality.
Combining inequalities \eqref{4.7.3}-\eqref{FP"IN"6} together, we obtain
\be \label{4.10.3}
\|(I-P_N)[f(v)-f(u)]\|\leq C'_1|I-P_N|_{L(H^{1,p}\to L^2)}(\|v\|_1^2+\|u\|_1^2+1)|y|_\h.
\ee
Substituting this inequality in \eqref{4.8.3}, we see that
\be\label{4.14.3} 
\p_t|y|_{\h}^2\leq (-3\al/2+C|I-P_N|_{L(H^{1,p}\to L^2)}(\|v\|_1^2+\|u\|_1^2+1))|y|_\h^2.
\ee
By the Sobolev embedding theorem, the space $H^{1,p}(D)$ is compactly embedded in $L^2(D)$ for $p>6/5$. This implies that the sequence $|I-P_N|_{L(H^{1,p}\to L^2)}$ goes to zero as $N$ goes to infinity. Combining this fact with the Gronwall lemma applied to \eqref{4.14.3} and using \eqref{4.17.3}, we arrive at \eqref{4.15.3}.
\ep
\subsection{Controlling the growth of intermediate process}
The goal of this subsection is to show that inequality \eqref{4.17.3} (and therefore \eqref{4.16.3}) holds with high probability, for $g(t)=\zeta(t)$.

\nt
For any $\h$-valued continuous process $y(t)$, let $\tau_y$ be the stopping time defined in \eqref{6 In-6,1,1},
where $L$ is the constant constructed in Corollary \ref{supermartingale lemma}, and $M,r$ are some positive constants. We recall that for the process of the form $y=[z,\dt z]$ we shall write $\tau^z$ instead of $\tau_{[z,\dt z]}$.
\begin{proposition}\label{4.23.3}
Let $u$ and $v$ be solutions of \eqref{FP"IN"1} and \eqref{FP"IN"2} where $g(t)=\zeta(t)$, that are issued from initial points $y,y'\in B_1$, respectively. Then
\be \label{4.18.3}
\pp\{\tau^v<\iin\}\leq 3 \exp(4\beta C-\beta r)+C_{M,r}|y-y'|_\h,
\ee
where $\beta$ is the constant from Proposition \ref{supermartingale}.
\end{proposition}
\bp 
To prove this result, we follow the arguments presented in Section 3.3 of \cite{KS-book} and Section 4 of \cite{KN-2013}. First, note that since inequality \eqref{4.18.3} concerns only the law of $v$ and not the solution itself, we are free to choose the underlying probability space $(\omm,\fff,\pp)$. We assume that it coincides with the canonical space of the Wiener process $\{\hat\zeta (t)\}_{t\geq 0}$. More precisely, $\omm$ is the space of continuous functions $\om:\rr_+\to\h$ endowed with the metric of uniform convergence on bounded intervals, $\pp$ is the law of $\hat\zeta$ and $\fff$ is the completion of the Borel $\sigma$-algebra with respect to $\pp$.

\nt
Let us define vectors $\hat e_j=[0,e_j]$ and their vector span
$$
 \h_N=\text{span}\{\hat e_1,\hat e_2,\ldots,\hat e_N\},
 $$
 which is an $N$-dimensional subspace of $\h$. The space $\Omega=C(\rr_+,\h)$ can be represented in the form
 $$
 \Omega=\Omega_N\dt+\Omega_N^{\perp},
 $$
 where $\Omega_N=C(\rr_+, \h_N)$ and $\Omega_N^\perp=C(\rr_+, {\h^\perp_N})$.
We shall write $\om=(\om^{(1)},\om^{(2)})$ for $\om=\om^{(1)}\dt+\om^{(2)}$.

Let $u'$ be a solution of equation \eqref{FP"IN"1} that has the same initial data as $v$. Introduce the stopping time
\be \label{4.24.3}
\tilde\tau=\tau^u\wedge\tau^{u'}\wedge\tau^v,
\ee
and a transformation $\Phi:\Omega\to\Omega$ given by
\be\label{4.22.3}
 \Phi(\om)(t)=\om(t)-\int_0^t a(s)\,ds,\q a(t)=\ch_{t\leq\tilde\tau}\PP_N(0,[f(u)-f(v)]),
\ee
 where $\PP_N$ is the orthogonal projection from $\h$ to $\h_N$.
\begin{lemma}\label{4.21.3}
For any initial points
 $y$ and $y'$ in $B_1$, we have
\be \label{4.19.3}
|\pp-\Phi_*\pp|_{var}\leq C_{M,r} |y-y'|_\h,
\ee
where $\Phi_*\pp$ stands for the image of $\pp$ under $\Phi$.
\end{lemma}

\nt
{\it Proof of lemma \ref{4.21.3}}.

\nt
{\it Step 1.} 
Let us note that by the definition of $\tilde\tau$ we have
\be\label{6.54.3}
\fff^u(t)\leq \fff^u(0)+(L+M)t+r,\q
\fff^v(t)\leq\fff^{u'}(0)+(L+M)t+r,
\ee
for all $t\leq\tilde\tau$.
We claim that there is an integer $N=N(\al,L,M)$ such that for all $t\leq\tilde\tau$ we have
\be\label{7.8.3}
|\xi_{v}(t)-\xi_{u}(t)|_\h^2\leq e^{-\al t+\theta}|\xi_u(0)-\xi_{u'}(0)|_\h^2,
\ee
where $\theta=|\ees_u(0)|\vee|\ees_{u'}(0)|+r$. Indeed, in view of inequality \eqref{1.5.3}, for any $y=[y_1,y_2]$ in $\h$, we have
\begin{align*}
|y(t)|_{\h}^2&\leq |y(t)|_{\h}^2+2\int_D F(y_1)\,dx|-2\int_D F(y_1)\,dx
\leq |\ees(y)|+2\nu \|y_1\|^2+2C\\
&\leq |\ees(y)|+\f{\lm_1}{2}\|y_1\|^2+2C\leq |\ees(y)|+\f{1}{2}|y|_\h^2+2C,
\end{align*}
so that
\be\label{7.9.3}
|y|_{\h}^2\leq 2 |\ees(y)|+4C.
\ee
Combining this inequality with \eqref{6.54.3}, we see that for all $t\leq\tilde\tau$
$$
\al\int_0^t\|\g z(s)\|^2\,ds \leq 2(|\ees_u(0)|\vee |\ees_u'(0)|+r)+2(L+M+2C)t,
$$
for $z=u$ and $z=v$. Using this inequality and applying Proposition \ref{4.13.3} with $\es=\al/2$ we arrive at \eqref{7.8.3}.

\nt
{\it Step 2.} 
Let us note that the transformation  $\Phi$ can be represented in the form
$$
\Phi(\om)=(\Psi(\om),\om^{(2)}),
$$
where $\Psi:\omm\to\omm_N$ is given by
$$
\Psi(\om)(t)=\om^{(1)}(t)+\int_0^t a(s;\om)\,ds.
$$
It is straightforward to see that 
 \be \label{4.20.3}
 |\pp-\Phi_*\pp|_{var}\leq\int_{\omm^\perp_N}|\Psi_*(\pp_N,\om^{(2)})-\pp_N|_{var}\pp^\perp_N(d\om^{(2)}),
 \ee
 where $\pp_N$ and $\pp^\perp_N$ are the images of $\pp$ under the natural projections $\PP_N:\omm\to\omm_N$ and $\QQ_N:\omm\to\omm^\perp_N$, respectively.
Define the processes 
$$
 z(t)=\om^{(1)}(t),\q
 \tilde z(t)=\om^{(1)}(t)+\int_0^t a(s;\om)\,ds.
$$
 It follows that $\pp_N=\ddd z$ and $\Psi_*(\pp,\om^{(2)})=\ddd\tilde z$.
 By Theorem A.10.1 in \cite{KS-book}, we have
 \be\label{Dz-D tilde z}
 |\ddd z-\ddd \tilde z|_{var}\leq\f{1}{2}\left(\left(\e\exp\left[6\max_{1\leq j\leq N} b_j^{-1}\int_0^\infty |a(t)|^2\,dt\right]\right)^{\f{1}{2}}-1\right)^{\f{1}{2}},
 \ee
 provided the Novikov condition
 $$
 \e\exp\left(C\int_0^\infty |a(t)|^2\,dt\right)<\infty,\q \text{ for any } C>0,
 $$
 holds. In view of inequalities \eqref{6.54.3} and \eqref{7.8.3} we have
\begin{align*}
&\e\exp(C\int_0^\infty |a(t)|^2\,dt)=\e\exp(C\int_0^{\tilde\tau} |a(t)|^2\,dt)\\
&\leq\e\exp(C\int_0^{\tilde\tau} \|f(v)-f(u)\|^2\,dt)\leq\e\exp(C_1\int_0^{\tilde\tau} \|v-u\|_1^2(1+\|u\|_1^4+\|v\|_1^4)\,dt)\\
&\leq\e\exp(C_2 |\xi_u(0)-\xi_{u'}(0)|_\h^2\int_0^\infty e^{-\al t+\theta}K(t)\,dt),
 \end{align*}
 where
 $$
 K(t)=(1+|\ees_u(0)|\vee|\ees_{u'}(0)|+(L+M)t+r)^2.
 $$
 So not only the Novikov condition holds, but also there is a positive constant $C_{M,r}=C(\al,L,M,r)$ such that the term on the right-hand side of inequality \eqref{Dz-D tilde z} does not exceed $C_{M,r} |y-y'|_\h$. Combining this with inequality \eqref{4.20.3}, we arrive at \eqref{4.19.3}.
 
 \nt
Now we are ready to establish \eqref{4.18.3}. Introduce auxiliary $\h$-continuous processes $y^u, y^{u'}$ and $y^v$ defined as follows: for $t\leq \tilde\tau$ they coincide with processes $\xi_u, \xi_{u'}$ and $\xi_v$, respectively, while for $t\geq\tilde\tau$ they solve 
$$ 
\p_t y=-\lm y,
$$
where $\lm>0$ is a large parameter.
By construction, with probability 1, we have
\be\label{4.29.3}
y^v(t,\om)=y^{u'}(t,\Phi(\om)) \q\text{ for all } t\geq 0.
\ee
Let us note that
\begin{align}\label{4.25.3}
\pp(\tau^v<\iin)&=\pp(\tau^v<\iin, \tau^u\wedge\tau^{u'}<\iin)+\pp(\tau^v<\iin, \tau^u\wedge \tau^{u'}=\iin)\notag\\
&\leq \pp(\tau^u<\iin)+\pp(\tau^{u'}<\iin)+\pp(\tau^v<\iin, \tau^u\wedge \tau^{u'}=\iin).
\end{align}
Moreover, in view of \eqref{4.29.3}
\begin{align}
&\pp(\tau^v<\iin, \tau^u\wedge\tau^{u'}=\iin)\leq\pp(\tau_{y^v}<\iin)=\Phi_*\pp(\tau_{y^{u'}}<\iin)\notag\\
&\q\leq \pp(\tau_{y^{u'}}<\iin)+|\pp-\Phi_*\pp|_{var}\leq \pp(\tau^{u'}<\iin)+|\pp-\Phi_*\pp|_{var}\label{6.43.3},
\end{align}
where we used the fact that for $t\geq \tilde\tau$ the norms of auxiliary processes decay exponentially. Combining these two inequalities we obtain
\begin{align}\label{4.28.3}
\pp(\tau^v<\iin)\leq \pp(\tau^u<\iin)+2\pp(\tau^{u'}<\iin)+|\pp-\Phi_*(\pp)|_{var}.
\end{align}
It remains to use Corollary \ref{supermartingale lemma} and Lemma \ref{4.21.3}  to conclude.
\ep

\subsection{Hitting a non-degenerate ball}
Here we show that the trajectory of the process $y(t)=[u(t),\dt u(t)]$ issued from arbitrarily large ball hits any non-degenerate ball centered at the origin, with positive probability, at a finite non-random time. We denote by $B_d$ the ball of radius $d$ in $\h$, centered at the origin. 
\begin{proposition}\label{6.12.3}
For any $R>0$ and $d>0$ there is $T_*=T_*(R,d)>0$ such that for all $T\geq T_*$, we have
\be\label{6.13.3}
\inf_{y\in B_R}P_T(y, B_d)>0,
\ee
where $P_t(y,\Gamma)=\pp_y(S_t(y,\cdot)\in\Gamma)$ is the transition function of the Markov process corresponding to \eqref{1.1.3}.
\end{proposition}
\bp
{\it Step~1.} For any $y\in B_R$ let $[u^y(t),\dt u^y(t)]$ be the flow of equation
\be\label{8.3.3}
\p_t^2 u^y+\gamma\p_t u^y-\de u^y+f(u^y)=h(x)+\dt \zeta^y(t),\q [u^y(0),\dt u^y(0)]=y,
\ee
where $\zeta^y(t)$ is given by
\be\label{8.2.3}
\zeta^y(t)=-th+\int_0^t P_N f(u^y(s))\,ds.
\ee
We claim that there is an integer $N_*$ depending only on $R$ such that for all $N\geq N_*$ we have
\be\label{8.4.3}
|[u^y(t),\dt u^y(t)]|_\h^2\leq R^2 e^{-\al t},\q \text{ for all } t\geq 0 \text{ and } y\in B_R.
\ee
Indeed, let us first show that there is $N_*(R)$ such that for all $N\geq N_*$ we have
\be\label{8.5.3}
|[u^y(t),\dt u^y(t)]|_\h\leq R+1, \q \text{ for all } t\geq 0 \text{ and } y\in B_R.
\ee
To this end, for any $y\in B_R$, let $l^y$ be the time of the first exit of the flow $[u^y(t),\dt u^y(t)]$ from the ball of radius $R+1$. If there is no such time, we set $l^y=+\infty$. Inequality \eqref{8.5.3} will be proved, if we show that for $N\geq N_*(R)$ we have $l^y=+\iin$. Let us note that $l^y>0$ and for $0\leq t\leq l^y$, we have $|[u^y(t),\dt u^y(t)]|_\h\leq R+1.$ Moreover, the equation \eqref{8.3.3} can be written in the form
$$
\p_t^2 u^y+\gamma\p_t u^y-\de u^y+(I-P_N)[f(u^y)-f(0)]=0,\q [u^y(0),\dt u^y(0)]=y,
$$
where we supposed for the simplicity that $f(0)=0$.
Applying the Foia\c{s}-Prodi estimate established in Proposition \ref{4.13.3} to the interval $[0, l^y]$, we see that there is $N_*(R)$ such that for all $N\geq N_*(R)$ we have
$$
|[u^y(t),\dt u^y(t)]|_\h^2\leq e^{-\al t} |y|_\h^2\leq e^{-\al t} R^2,
$$
where inequality holds for all $t\in[0, l^y]$. In particular $[u^y(l^y),\dt u^y(l^y)]\in B_R$, which is impossible in the case $l^y<\iin$. Inequality \eqref{8.5.3} is thus established. Once again using the Foia\c{s}-Prodi estimate, but this time on the half-line $t\geq 0$ and using \eqref{8.5.3}, we derive \eqref{8.4.3}.

\nt
{\it Step~2.} In view of the previous step, we can find $T_*(d,R)>0$ such that for all $y\in B_R$, we have 
\be\label{8.1.3}
[u^y(T),\dt u^y(T)]\in B_{d/2},\q\text{ for all } T\geq T_*.
\ee
We claim that \eqref{6.13.3} holds with this time $T_*$.
Indeed, if this is not true, then there is $T\geq T_*$ such that 
\be\label{6.58.3}
\inf_{y\in B_R}P_T(y, B_d)=0.
\ee
Let us find $\es(d, T, R)>0$ so small that we have 
\be
|S_T(y,\om)-[u^y(T), \dt u^y(T)]|_\h\leq d/2, \q\text{ for }y\in B_R,
\ee
provided $|\zeta(\om)-\zeta^y|_{C(0,T; H^1_0)}\leq\es$. Combining this with \eqref{8.1.3}, we see that
\be \label{6.18.3}
P_T(y,B_d)\geq \pp(|\zeta-\zeta^y|_{C(0,T;H^1_0)}\leq\es).
\ee

We need the following lemma. It is established in the appendix.
\begin{lemma}\label{6.20.3}
For any $\rho<2$ there exists $s=s(\rho)>0$ such that if 
\be \label{5.17.3}
|f'(u)|\leq C(|u|^{\rho}+1),
\ee
then
\be\label{6.39.3}
|f(u)-f(v)|_{L^2}\leq C_1(|u|_{H^{1-s}}^\rho+|v|_{H^{1-s}}^\rho+1)|u-v|_{H^{1-s}},
\ee
where $C_1>0$ depends only on $C>0$.
\end{lemma}
Let us suppose that we have \eqref{6.58.3}, and let $y_j$ be a minimizing sequence. This sequence is bounded in $H^1\times L^2$, so up to extracting a subsequence, we can suppose that it is a Cauchy sequence in $H^{1-s}\times H^{-s}$ ($s$ is the constant from the previous lemma). Let us show that so is the sequence $\zeta_j=\zeta^{y_j}$ in $C(0, T;H^1_0)$. Indeed, in view of \eqref{8.2.3} and Lemma \ref{6.20.3} , we have
\begin{align}
|\zeta_j-\zeta_k|_{C(0,T;H^1_0)}&=|\int_0^t P_N(f(u^{y_j})-f(u^{y_k}))\,dr|_{C(0,T;H^1_0)}\notag\\
&\leq \int_0^T |P_N(f(u^{y_j})-f(u^{y_k}))|_{H^1_0}\,dr\notag\\
&\leq C_1 C_N \int_0^T (|u^{y_j}|_{H^{1-s}}^\rho+|u^{y_k}|_{H^{1-s}}^\rho+1)|u^{y_j}-u^{y_k}|_{H^{1-s}}\,dr.\label{8.6.3}
\end{align}
Now let us notice that in view of \eqref{8.3.3}, the difference $w=u^{y_j}-u^{y_k}$ solves
$$
\p_t^2 w+\gamma\p_t w-\de w+(I-P_N)[f(u^{y_k}+w)-f(u^{y_k})]=0, \q [w(0), \dt w(0)]=y_j-y_k.
$$
Applying the operator $(-\de)^{-s/2}$ to this equation and repeating the argument used in derivation of Proposition \ref{4.13.3}, together with inequality \eqref{8.4.3}, we see that for all $N\geq N_*(R)$, we have
$$
|[w(t),\dt w(t)]|_{H^{1-s}\times H^{-s}}^2\leq e^{-\al t}|y_j-y_k|^2_{H^{1-s}\times H^{-s}}.
$$
Finally, combining this inequality with \eqref{8.6.3}, we see that $|\zeta_j-\zeta_k|_{C(0,T;H^1_0)}\to 0$ as $j, k\to \iin$. We shall denote by $\tilde \zeta$ the corresponding limit. 

\nt
{\it Step~3.}
It follows from inequality \eqref{6.18.3} implies that
\be \label{6.22.3}
\pp(|\zeta-\zeta_j|_{C(0,T;H^1_0)}\leq\es)\to 0 \q \text{ as } j\to\iin.
\ee
Let us fix $j_0\geq 1$ so large that for all $j\geq j_0$
$$
|\zeta_j-\tilde\zeta|_{C(0,T;H^1_0)}\leq \es/2.
$$
Then by the triangle inequality, for all $j\geq j_0$
\begin{align*}
\pp(|\zeta-\tilde\zeta|_{C(0,T;H^1_0)}\leq \es/2)&=\pp(|\zeta-\zeta_j+\zeta_j-\tilde\zeta|_{C(0,T;H^1_0)}\leq\es/2)\\
&\leq \pp(|\zeta-\zeta_j|_{C(0,T;H^1_0)}\leq |\zeta_j-\tilde\zeta|_{C(0,T;H^1_0)}+\es/2)\\
&\leq \pp(|\zeta-\zeta_j|_{C(0,T;H^1_0)}\leq \es).
\end{align*}
Letting $j$ go to $\iin$ and using inequality \eqref{6.22.3}, we obtain
$$
\pp(|\zeta-\tilde\zeta|_{C(0,T;H^1_0)}\leq \es/2)=0,
$$
which is impossible, since the support of $\zeta$ restricted to $[0,T]$ coincides with ${C(0,T;H^1_0)}$. The proof of Proposition \ref{6.12.3} is complete.
\ep

\section{Proof of Theorem \ref{2.14.3}}\label{6.31.3}
In this section we establish Theorem \ref{2.14.3}. As it was already mentioned, this will imply Theorem \ref{2.15.3}. We then show that the non-degeneracy condition imposed
on the force can be relaxed to allow forces that are non-degenerate only in the
low Fourier modes (see Theorem \ref{6.10.3}).

\medskip
\subsection{Recurrence: verification of \eqref{2.17.3}-\eqref{2.16.3}}\label{6.32.3}
In view of Proposition \ref{proposition_exp}, it is sufficient to establish inequality \eqref{2.16.3}. To this end, we shall use the existence of a Lyapunov function, combined with an auxiliary result established in \cite{shirikyan-bf2008}.

Let $S_t(y,\om)$ be a Markov process in a separable Banach space $\h$ and let $\RR_t(\yy,\om)$ be its extension on an interval $[0, T]$. Consider a continuous functional $\gi(y)\geq 1$ on $\h$ such that
$$
\lim_{|y|_\h\to\infty}\gi(y)=\infty.
$$ 
Suppose that there are positive constants $d, R, t_*, C_*$ and $a<1$, such that
\begin{align}
\e_y\gi(S_{t_*})&\leq a\, \gi(y)\,\,\q\tx{ for } |y|_\h\geq R,\label{6.55.3}\\
\e_y\gi(S_{t})&\leq C_*\q\q\,\,\,\,\,\tx{ for } |y|_\h\leq R,\, t\geq 0,\label{6.56.3}
\end{align}
\be
\inf_{y,y'\in B_R}\pp_\yy\{|\rrr_T(y,y',\cdot)|_{\h}\vee |\rrr'_T(y,y',\cdot)|_{\h}\leq d\}>0.\label{6.57.3}
\ee

\medskip\nt
We shall denote by $\tau_d$ the first hitting time of the set $B_{\hh}(d)$. The following proposition is a weaker version of the result proved in \cite{shirikyan-bf2008} (see Proposition 3.3).
\begin{proposition}\label{6. Proposition-Lyapunov-dissipation}
Under the above hypotheses there are positive constants $C$ and $\kp$ such that the inequality
\be 
\e_\yy\exp(\kp\tau_d)\leq C(\gi(y)+\gi(y')), \q\tx{ for any } \yy=(y,y')\in\hh,
\ee
holds for the extension $\SSS_t$ constructed by iteration of $\RR_t$ on the half-line $t\geq 0$.
\end{proposition}

\nt
It follows from estimate \eqref{2.2.3} that inequalities \eqref{6.55.3} and \eqref{6.56.3} are satisfied for the functional
$$
\gi(y)=1+|\ees(y)|.
$$
We now show that for any $d>0$ we can find an integer $k\geq 1$ and $T_*\geq 1$ sufficiently large, such that we have \eqref{6.57.3} for any $T\in\{kT_*,(k+1)T_*,\ldots\}$. In what follows, we shall drop the subscript and write $|y|$ instead of $|y|_\h$. So let us fix any $d>0$, and consider the events
\begin{align*}
G_d&=\{|\rrr_T(y,y')|\leq d\},\q G'_d=\{|\rrr'_T(y,y')|\leq d\},\\
E_r&=\{\fff_\rrr(t)\leq \fff_\rrr(0)+Lt+r\}\cap \{\fff_\rrr'(t)\leq \fff_\rrr'(0)+Lt+r\},
\end{align*}
where $\fff_y(t)$ is defined in \eqref{2.13.3}, and $L$ is the constant from Corollary \ref{supermartingale lemma}.

\nt
{\it Step 1.} 
First, let us note that by Proposition \ref{6.12.3}, there is $T_*=T_*(R,d)\geq 1$, such that
\be\label{7.5.3}
\pp_y\{|S_{T_*}(y,\cdot)|\leq d/2\}\geq c_d \q\text{ for any } y\in B_R,
\ee
where $c_d$ is a positive constant depending on $d,R$ and $T_*$.
We claim that this implies 
\be\label{7.3.3}
\pp_y\{|S_{kT_*}(y,\cdot)|\leq d/2\}\geq c_d \q\text{ for all } k\geq 1 \text{ and } y\in B_R.
\ee
Indeed, let us fix any integer $k\geq 1$ and introduce the stopping times
$$
\bar\tau(y)=\min\{nT_*, n\geq 1 : |S_{nT_*}(y,\cdot)|>d/2\},\q \bar\sigma=\bar\tau\wedge kT_*.
$$
Let us note that if $\bar\tau$ is finite, then we have
\be\label{7.6.3}
|S_{\bar\tau-T_*}(y,\cdot)|\leq R \q\text{ and }\q |S_{\bar\tau}(y,\cdot)|>d/2,
\ee
where inequalities hold for any $y$ in $B_R$.
Moreover
\be\label{7.4.3}
\pp_y\{|S_{kT_*}(y,\cdot)|> d/2\}\leq \pp_y\{\bar\tau=\bar\sigma\},
\ee
where we used that for $\bar\tau>kT_*$, we have $|S_{kT_*}(y,\cdot)|{_\h}\leq d/2$. In view of \eqref{7.6.3}
\be\label{7.7.3}
\pp_y\{\bar\tau=\bar\sigma\}\leq \pp_y\{|S_{\bar\sigma-T_*}(y,\cdot)|\leq R,\,|S_{\bar\sigma}(y,\cdot)|>d/2\}:=p.
\ee
Since $\bar\sigma$ is a.s. finite, we can use the strong Markov property, and obtain
\begin{align*}
p&=\e_y[\e_y(\ch_{|S_{\bar\sigma-T_*}(y,\cdot)|\leq R}\cdot\ch_{|S_{\bar\sigma}(y,\cdot)|>d/2}|\fff_{\bar\sigma-T_*})]\\
&=\e_y[\ch_{|v|\leq R}\cdot\e_v(\ch_{|S_{T_*}(v,\cdot)|>d/2})]\\
&=\e_y[\ch_{|v|\leq R}\cdot\pp_v(|S_{T_*}(v,\cdot)|>d/2)]\leq \sup_{\bar v\in B_R}\pp_{\bar v}(|S_{T_*}(\bar v,\cdot)|>d/2),
\end{align*}
where $v=S_{\bar\sigma-T_*}(y,\cdot)$, and $\fff_t$ is the filtration generated by $S_t$. In view of \eqref{7.5.3}, the last term in this inequality does not exceed $1-c_d$. Combining this with inequalities \eqref{7.4.3} and \eqref{7.7.3}, we arrive at \eqref{7.3.3}.

\nt
{\it Step 2.} 
It follows from the previous step that for any $T\in\{T_*, 2T_*, \ldots\}$
\be\label{6.59.3}
\pp_\yy(G_{d/2})\wedge \pp_\yy(G'_{d/2})\geq c_d,\q \text{ for any } y,y'\in B_R,
\ee
where we used that $\RR_t$ is an extension of $S_t$.
Further, by Corollary \ref{supermartingale lemma} we have
\be\label{6.62.3}
\pp_\yy(E_r)\geq 1-2\exp(4\beta C-\beta r):=1-o(r).
\ee
Let us fix $r=r(d,R,T_*)>0$ so large that
\be\label{6.60.3}
o(r)\leq c_d^2/8.
\ee
By the symmetry, we can assume that
\be\label{6.63.3}
\pp_\yy(G'_{d/2}\nnn^c)\leq\pp_\yy(G_{d/2}\nnn^c),
\ee
where we set $\nnn=\{\vvv(y,y')\neq \vvv'(y,y')\}$. We claim that
\be\label{inclusion G-Delta}
G_{d/2}E_r\,\nnn^c\subset G_d G'_d,
\ee
for any $T\in\{kT_*,(k+1)T_*,\ldots\}$ with $k\geq 1$ sufficiently large.
To prove this, let us fix any $\om$ in $G_{d/2}E_r\,\nnn^c$, and note that it is sufficient to establish
\be\label{6.49.3}
|\rrr_T(y,y',\om)-\rrr'_T(y,y',\om)|\leq d/2,\q\text{ for any } y,y' \text{ in } B_R.
\ee
Since $\om\in\nnn^c$, we have that $\vvv=\vvv'$, and therefore, in view of \eqref{7.1.3}-\eqref{6. Coupling relation 1}, $\R_t(y,y')$ and $\R'_t(y,y')$ are, respectively, the flows of equations
\be 
\p_t^2 \tilde u+\gamma \p_t \tilde u-\de\tilde u+f(\tilde u)-P_N f(\tilde u)=h(x)+\psi(t),\q \xi_{\tilde u}(0)=y,
\ee
and 
\be 
\p_t^2 \tilde v+\gamma \p_t \tilde v-\de\tilde v+f(\tilde v)-P_N f(\tilde v)=h(x)+\psi(t),\q \xi_{\tilde u}(0)=y'.
\ee
It follows that their difference $w=\tilde v-\tilde u$ solves 
$$
\p_t^2 w+\gamma \p_t w-\de w+(I-P_N)[f(\tilde v)-f(\tilde u)]=0,\q [w(0),\dt w(0)]=y'-y.
$$
Using the Foia\c{s}-Prodi estimate established in Proposition \ref{4.13.3} (see \eqref{6.50.3}-\eqref{4.15.3}) together with the fact that $\om\in E_r$, we can find an integer $N\geq 1$ depending only on $L$ such that 
$$
 |\rrr_T(y,y',\om)-\rrr'_T(y,y',\om)|^2\leq C(r,R)e^{-\al T}|y-y'|^2\leq 4R^2 C(r,R)e^{-\al T}.
 $$
 Since $r$ is fixed, we can find $k\geq 1$ sufficiently large, such that the right-hand side of this inequality is less than $d^2/4$ for any $T\in\{kT_*,(k+1)T_*,\ldots\}$, so that we have \eqref{inclusion G-Delta}.
 
 \nt
 {\it Step 3.} 
We now follow the argument presented in \cite{shirikyan-bf2008}. In view of \eqref{inclusion G-Delta}
\begin{align*}
\pp_\yy(G_d G'_d)&=\pp_\yy(G_d G'_d \nnn^c)+\pp_\yy(G_d G'_d \nnn)\\
&\geq \pp_\yy(G_d G'_d E_r\nnn^c)+\pp_\yy(G_d|\nnn) \pp_\yy(G'_d|\nnn)\pp_\yy(\nnn)\\
&\geq \pp_\yy(G_{d/2} E_r \nnn^c)+\pp_\yy(G_d \nnn)\pp_\yy(G'_d\nnn),
\end{align*}
where we used the independence of $\vvv$ and $\vvv'$ conditioned on the event $\nnn$. Combining this inequality with \eqref{6.62.3}, we obtain
$$
\pp_\yy(G_d G'_d)\geq\pp_\yy(G_{d/2}\nnn^c)+\pp_\yy(G_{d} \nnn)\pp_\yy(G'_{d} \nnn)-o(r).
$$
We claim that the right-hand side of this inequality is no less than $c_d^2/8$.
Indeed, if $\pp_\yy(G_{d/2}\nnn^c)\geq c_d^2/4$, then the required result follows from inequality \eqref{6.60.3}. If not, then by inequalities \eqref{6.59.3} and \eqref{6.63.3}, we have
$$
c_d^2\leq \pp_\yy(G_{d/2})\pp_\yy(G'_{d/2})\leq \pp_\yy(G_{d/2}\nnn)\pp_\yy(G'_{d/2}\nnn)+3c_d^2/4,
$$
so that
$$
\pp_\yy(G_{d}\nnn)\pp_\yy(G'_{d}\nnn)\geq\pp_\yy(G_{d/2}\nnn)\pp_\yy(G'_{d/2}\nnn)\geq c_d^2/4.
$$
We have thus shown that for any $y,y'$ in $B_R$
$$
\pp_\yy\{|\rrr_T(y,y',\cdot)|\vee |\rrr'_T(y,y',\cdot)|\leq d\}\equiv \pp_\yy(G_d G'_d)\geq c_d^2/8,
$$
and therefore we have \eqref{6.57.3}. The hypotheses of Proposition \ref{6. Proposition-Lyapunov-dissipation} are thus satisfied, so that inequality \eqref{2.16.3} holds.\subsection{Exponential squeezing: verification of \eqref{6.33.3}-\eqref{6.5.3}}
Let $u,u',v,\tilde u, \tilde u', \tilde v$ and $\vo,\tau,\sigma$ be the processes and stopping times constructed in Subsection \ref{Main result and scheme of its proof}. Consider the following events:
$$
\qqq'_k=\{kT\leq\tau\leq (k+1)T,\tau\leq\vo\},\q\qqq''_k=\{kT\leq\vo\leq (k+1)T,\vo<\tau\}.
$$
\begin{lemma}\label{6.27.3}
There exist positive constants $d,r, L$ and $M$ such that for any initial point $\yy\in B_{\hh}(d)$ and any $T\geq1$ sufficiently large 
$$
\pp_{\yy}(\qqq'_k)\vee\pp_\yy(\qqq_k'')\leq e^{-2(k+1)}\q\text{ for all }k\geq 0.
$$
\end{lemma}
\bp

\nt\\
{\it~Step1.} (Probability of $\qqq_k'$).
Let $L$ be the constant from Corollary \ref{supermartingale lemma}. Then using second inequality of this corollary, we obtain
\be\label{6.48.3}
\pp_\yy(\qqq'_k)\leq\pp_\yy(kT\leq\tau<\iin)\leq 2\exp(4\beta C-\beta r-\beta k T M)\leq e^{-2(k+2)},
\ee
for $M\geq 2\beta^{-1}, r\geq 5\beta^{-1}+4C$. From now on, the constants $L,M$ and $r$ will be fixed.

\smallskip\nt
{\it~Step2.} (Probability of $\qqq_k''$). Let us first note that by the Markov property we have
\begin{align} 
\pp_\yy(\qqq''_k)&=\pp_\yy(\qqq''_k,\sigma\geq kT)=\e_\yy(\ch_{\qqq''_k}\cdot\ch_{\sigma\geq kT})=\e_\yy[\e_\yy(\ch_{\qqq''_k}\cdot\ch_{\sigma\geq kT}|\FFF_{kT})]\notag\\
&=\e_\yy[\ch_{\sigma\geq kT}\cdot\e_\yy(\ch_{\qqq''_k}|\FFF_{kT})]\leq \e_\yy[\ch_{\sigma\geq kT}\cdot\e_{\bar\yy}\ch_{0\leq \vo\leq T}]\notag\\
&=\e_\yy[\ch_{\sigma\geq kT}\cdot\pp_{\bar \yy}(0\leq \vo\leq T)],\label{6.28.3}
\end{align}
where $\bar\yy(\cdot)=\yy(kT,\cdot)$, and $\FFF_t$ stands for the filtration corresponding to the process $\SSS_t$.
Moreover, it follows from the definition of maximal coupling, that for any $\yy$ in $\hh$, we have
$$
\pp_{\yy}(0\leq \vo\leq T)= |\pp_{\yy}\{\xi_v\}_T-\pp_{\yy}\{\xi_{u'}\}_T|_{var}.
$$
Combining this with inequality \eqref{6.28.3}, we obtain 
\be\label{6.45.3}
\pp_\yy(\qqq''_k)\leq \e_\yy(\ch_{\sigma\geq kT}\cdot|\pp_{\bar\yy}\{\xi_v\}_T-\pp_{\bar\yy}\{\xi_{u'}\}_T|_{var})
\ee
Further, let us note that
\begin{align}
&|\pp_{\bar\yy}\{\xi_v\}_T-\pp_{\bar\yy}\{\xi_{u'}\}_T|_{var}=\sup_{\Gamma}|\pp_{\bar\yy}(\{\xi_v\}_T\in\Gamma)-\pp_{\bar\yy}(\{\xi_{u'}\}_T\in\Gamma)\notag|\\
&\q\leq \pp_{\bar\yy}(\tilde\tau<\iin)+ \sup_{\Gamma}|\pp_{\bar\yy}(\{\xi_v\}_T\in\Gamma, \tilde\tau=\iin)
-\pp_{\bar\yy}(\{\xi_{u'}\}_T\in\Gamma, \tilde\tau=\iin)|\notag\\
&\q:=\elll_1+\elll_2\label{6.41.3},
\end{align}
where $\tilde\tau=\tau^u\wedge\tau^{u'}\wedge\tau^v$, and the supremum is taken over all $\Gamma\in\bbb(C(0,T;\h))$. In view of \eqref{4.29.3} we have
\be\label{6.42.3}
\elll_2\leq|\pp_{\bar\yy}-\Phi_*\pp_{\bar\yy}|_{var},
\ee
where $\Phi$ is the transformation constructed in Subsection 4.2, and we used the fact that for $\tilde \tau=\iin$ we have $y^v\equiv \xi_v$ and $y^{u'}\equiv \xi_{u'}$
Further, in view \eqref{6.43.3} we have
\be\label{6.44.3}
\elll_1\leq \pp_{\bar\yy}(\tau^u\wedge\tau^{u'}<\iin)+\pp_{\bar\yy}(\tau^{u'}<\iin)+|\pp_{\bar\yy}-\Phi_*\pp_{\bar\yy}|_{var}.
\ee
Combining inequalities \eqref{6.45.3}-\eqref{6.44.3}, we get
\be\label{6.46.3}
\pp_\yy(\qqq''_k)\leq 2\,\e_\yy[\ch_{\sigma\geq kT}\cdot (\pp_{\bar\yy}(\tau<\iin)+|\pp_{\bar\yy}-\Phi_*\pp_{\bar\yy}|_{var})].
\ee
Let us note that for any $\om\in\{\sigma\geq kT\}$ we have 
\be\label{6. In-6.4,1}
|\ees_{\tilde u}(kT)|\vee|\ees_{\tilde u'}(kT)|\leq |\ees_{\tilde u}(0)|\vee|\ees_{\tilde u'}(0)|+(L+M)kT+r.
\ee
Moreover, it follows from Proposition \ref{4.13.3} (see the derivation of \eqref{7.8.3}) that for any $\es>0$ there is $N$ depending only on $\es, \al, L$ and $M$, such that for all $kT\leq t\leq\tau\wedge\tau^{\tilde v}$, on the set $\sigma\geq kT$, we have 
\begin{align} 
|\xi_{\tilde v}(t)-\xi_{\tilde u}(t)|_\h^2&\leq\exp(-\al(t-kT)+\theta)|\xi_{\tilde u}(kT)-\xi_{\tilde u'}(kT)|_\h^2\notag\\
&\leq\exp(-\al(t-kT)/2+\theta)|\xi_{\tilde u}(kT)-\xi_{\tilde u'}(kT)|_\h^2\label{6. In-6.5},
\end{align}
where we set
$$
\theta=\es\cdot(|\ees_{\tilde u}(kT)|\vee|\ees_{\tilde u'}(kT)|+r).
$$
By the same argument as in the derivation of \eqref{4.19.3}, we have 
\begin{align}\label{6. In-6.6}
\e_\yy(\ch_{\sigma\geq kT}\cdot|\pp_{\bar\yy}-\Phi_*\pp_{\bar\yy}|_{var})&\equiv \e_\yy(\ch_{\sigma\geq kT}\cdot|\pp_{\yy(kT)}-\Phi_*\pp_{\yy(kT)}|_{var})\notag\\
&\leq\f{1}{2}\left(\left(\e_\yy\exp\left[6\max_{1\leq j\leq N}b_j^{-1}\kkk\right]\ch_{\sigma\geq kT}\right)^{\f{1}{2}}-1\right)^{\f{1}{2}},
\end{align}
where 
\begin{align}
\kkk&=C_1\int_0^\infty \{\exp({-\al(t-kT)/2+\theta})|\xi_{\tilde u}(kT)-\xi_{\tilde u'}(kT)|_\h^2\notag\\
&\q\cdot(1+|\ees_{\tilde u}(kT)|\vee|\ees_{\tilde u'}(kT)|+(L+M)t+r)^2\}\,dt\notag\\
&\leq C_2\int_0^\infty \{\exp({-\al(t-kT)/2+\theta})e^{-\al kT}|y-y'|_\h^2\notag\\
&\q\cdot(1+|\ees_{\tilde u}(0)|\vee|\ees_{\tilde u'}(0)|+(L+M)kT+(L+M)t+r)^2\}\,dt\label{6. In-6.7},
\end{align}
and we used inequalities \eqref{6. In-6.4,1}-\eqref{6. In-6.5} combined with the fact that the mean value is taken along the characteristic of the set $\{\sigma\geq kT\}$.
Now let us fix $\es=\es(\al,L,M)>0$ such that $\al/4\geq\es\cdot(L+M)$, and let $C(\al,L, M)>0$ be so large that for any $k\geq 0$ and any $T\geq 1$
$$
\exp(-\al kT/4)(1+(L+M)kT)^2\leq C(\al,L,M).
$$
Combining this inequality with \eqref{6. In-6.7}, we obtain 
\begin{align}
\kkk&\leq C_3\cdot C(\al,L,M)e^{-\f{\al}{4} kT}|y-y'|_\h^2\int_0^\infty e^{-\al t+\es r}(1+(L+M)t+r)^2\,dt\notag\\
&=C(\al,r,L,M)e^{-\f{\al}{4} kT}|y-y'|_\h^2\label{4.27.3}.
\end{align}
Now recall that $N$ depends only on $\es, L$ and $M$, and $\es$ depends on $\al, L$ and $M$. It follows that $N$ depends only on $\al,L$ and $M$. Let us choose $d=d(\al,r,L,M)>0$ so small that 
\be\label{4.26.3}
6\max_{1\leq j\leq N}b_j^{-1}C(\al,r,L,M)d\leq 1.
\ee
Then, by inequalities \eqref{6. In-6.6} and \eqref{4.27.3}, we have
\be\label{6.47.3}
\e_\yy(\ch_{\sigma\geq kT}\cdot|\pp_{\bar\yy}-\Phi_*\pp_{\bar\yy}|_{var})\leq e^{-\f{\al}{8}kT}|y-y'|_\h\leq e^{-2(k+2)},
\ee
for $T\geq 16\al^{-1}$ and $d\leq e^{-4}$.
Further, by the Markov property and inequality \eqref{6.48.3} we have
\begin{align*}
e^{-2(k+2)}&\geq \pp_\yy\{kT\leq \tau<\infty\}=\e_\yy[\e_\yy(\ch_{kT\leq\tau<\iin}|\FFF_{kT})]
=\e_\yy[\pp_{\yy(kT)}(\tau<\iin)]\notag\\
&\geq \e_\yy[\ch_{\sigma\geq kT}\cdot\pp_{\yy(kT)}(\tau<\iin)]\equiv \e_\yy[\ch_{\sigma\geq kT}\cdot\pp_{\bar\yy}(\tau<\iin)].
\end{align*}
Combining this inequality with \eqref{6.46.3} and \eqref{6.47.3} we obtain
$$
\pp_\yy(\qqq''_k)\leq 4 e^{-2(k+2)}\leq e^{-2(k+1)}.
$$
\ep

\medskip
\nt
Now we are ready to establish \eqref{6.3.3}-\eqref{6.5.3}. We have
$$
\pp_{\yy}\{\sigma=\infty\}\geq1-\sum_{k=0}^\infty\pp_\yy\{kT\leq\sigma\leq(k+1)T\}\geq\f{1}{2},
$$
where used Lemma \ref{6.27.3} to show that
\begin{align*}
&\pp_\yy\{kT\leq\sigma\leq(k+1)T\}=\pp_\yy\{kT\leq\tau\leq (k+1)T,\tau\leq\vo\}\\
&\q+\pp_\yy\{kT\leq\vo\leq (k+1)T,\vo<\tau\}=\pp_\yy(\qqq_k')+\pp_\yy(\qqq_k'')\leq e^{-2(k+1)}.
\end{align*}
By the same argument,
\begin{align*}
&\e_\yy[\ch_{\{\sigma<\infty\}} e^{\De\sigma}]=\e_\yy[\ch_{\{\sigma<\infty,\tau\leq\vo\}} e^{\De\sigma}]+\e_\yy[\ch_{\{\sigma<\infty,\vo<\tau\}} e^{\De\sigma}]\\
&\leq \e_\yy[\ch_{\{\tau<\infty,\tau\leq\vo\}} e^{\De\tau}]+\e_\yy[\ch_{\{\vo<\infty,\vo<\tau\}} e^{\De\vo}]\leq 2\sum_{k=0}^\infty e^{-2(k+1)}e^{\De k(T+1)}\leq 2,
\end{align*}
for $\De<(1+T)^{-1}$.
So, inequalities \eqref{6.3.3} and \eqref{6.4.3} are established. To prove \eqref{6.5.3}, note that in view of \eqref{6 In-6,1,1}, for $\sigma<\infty$ we have 
\be\label{6.64.3}
|\ees_{\tilde u}(\sigma)|\leq |\ees_{\tilde u}(0)|+(L+M)\sigma+r.
\ee
Combining this inequality with \eqref{7.9.3} we obtain
$$
|S_{\sigma}|_{\h}^8\leq   (2|\ees_{\tilde u}(\sigma)|+4C)^4\leq(2(|\ees_{\tilde u}(0)|+(L+M)\sigma+r+2C))^4 \leq C(r,L,M)(1+\sigma^4).
$$
It is clear that the above inequality is satisfied also for $S$ replaced by $S'$, so that
$$
|\SSS_{\sigma}|_{\hh}^{8}\leq 2C(r,L,M)(1+\sigma^4).
$$
Multiplying this inequality by $\ch_{\{\sigma<\infty\}}$, taking the $\e_\yy$-mean value, and using inequality \eqref{6.4.3}, we arrive at \eqref{6.5.3}. The proof of Theorem \ref{2.14.3} (and with it of Theorem \ref{2.15.3})  is complete.

\subsection{Relaxed non-degeneracy condition}
We finish this section with the following result that allows to relax the non-degeneracy condition imposed on the force. 
\begin{theorem}\label{6.10.3}
There exists $N$ depending only on $\gamma, f, \|h\|_1$ and $\BBB_1$ such that the conclusion of Theorem \ref{2.15.3} remains true for any random force of the form \eqref{1.9.3}, whose first $N$ coefficients $b_j$ are not zero.
\end{theorem}
Let us fix an integer $N_1$ such that inequality \eqref{4.27.3} holds for any $N\geq N_1$ and let $d=d(N_1)$ be so small that we have \eqref{4.26.3}, where $N$ should be replaced by $N_1$. 
Theorem \ref{6.10.3} will be established, if we show that there is an $N=N(d)\geq N_1$ such
that inequality \eqref{6.13.3} holds, provided $b_j\neq 0$ for $j=1,\ldots, N$. 
Let the constants $T,\es$ and $N_*(R)$, together with the function $\zeta^y$ be the same as in the proof of Proposition \ref{6.12.3}. Let us find $N_1\geq N_*(R)$ such that for all $N\geq N_1$, we have
\begin{align*}\label{6.37.3}
\sup_{y\in B_R}|(I-P_N)\zeta^y|_{C(0,T;H^1_0)}&=|(I-P_N)th|_{C(0,T;H^1_0)}=T|(I-P_N)h|_{H^1_0}\leq\es/4.
\end{align*}
We claim that Theorem \ref{6.10.3} holds with this $N$. Indeed, let us suppose that inequality \eqref{6.13.3} does not hold, and we have \eqref{6.58.3}. Let $\zeta_j$ and $\tilde\zeta$ be the functions constructed in Proposition \ref{6.12.3}. Denote $\ccc=C(0,T;H^1_0)$. Then
\begin{align*}
\pp(|\zeta-P_N\tilde\zeta|_{\ccc}\leq \es/4)&=\pp(|\zeta-P_N\zeta_j+P_N\zeta_j-P_N\tilde\zeta|_{\ccc}\leq\es/4)\\
&\leq \pp(|\zeta-P_N\zeta_j|_{\ccc}\leq |P_N\zeta_j-P_N\tilde\zeta|_{\ccc}+\es/4)\\
&\leq \pp(|\zeta-\zeta_j+(I-P_N)\zeta_j|_{\ccc}\leq |\zeta_j-\tilde\zeta|_{\ccc}+\es/4)\\
&\leq \pp(|\zeta-\zeta_j|_{\ccc}\leq |(I-P_N)\zeta_j|_{\ccc}+3\es/4)\leq \pp(|\zeta-\zeta_j|_{\ccc}\leq \es).
\end{align*}
Letting $j$ go to infinity, and using inequality \eqref{6.22.3} we obtain
$$
\pp(|\zeta-P_N\tilde\zeta|_{\ccc}\leq \es/4)=0,
$$
which is impossible, since the support of $\zeta$ restricted to $[0,T]$ contains $C(0,T;P_N H^1_0)$. The proof Theorem \ref{6.10.3} is complete.

\bigskip
  \section{Appendix}
  
\medskip
\subsection{Proof of \eqref{6.40.3}} 
Let us consider the continuous map $\gi$ from $C(0,T;H^1_0(D))$ to $C(0,T;\h)$ defined by $\gi(\ph)=\tilde y$, where $\tilde y$ is the flow of equation
$$
\p_t^2 z+\gamma \p_t z-\de z+f(z)-P_N f(z)=h(x)+\p_t \ph,\q [z(0),\dt z(0)]=y.
$$
Then for any $\Gamma\in\bbb(C(0,T;\h))$, we have
\begin{align*}
\pp\{\xi_{\tilde u}(t)\in\Gamma\}&=\pp\{\gi({\int_0^t \psi(s)\,ds})\in\Gamma\}=\pp\{\int_0^t \psi(s)\,ds\in\gi^{-1}(\Gamma)\}\\
&=\pp\{\zeta(t)-\int_0^t P_N f(u)\,ds\in\gi^{-1}(\Gamma)\}\\
&=\pp\{\gi({\zeta(t)-\int_0^t P_N f(u)\,ds})\in\Gamma\}=\pp\{\xi_{u}(t)\in\Gamma\}.
\end{align*}

\medskip
\subsection{Proof of lemma \ref{6.20.3}} 
Let $f$ be a function that satisfies the growth restriction \eqref{5.17.3} with $\rho<2$. We claim that inequality \eqref{6.39.3} holds with $s=(2-\rho)/(2(\rho+1))$. Indeed, by the H\"older and Sobolev inequalities, we have 
\begin{align*}
|f(u)-f(v)|^2_{L^2}&=\int |f(u)-f(v)|^2\leq C\int (|u|^{2\rho}+|v|^{2\rho}+1)|u-v|^2\\
&\leq C||u|^{2\rho}+|v|^{2\rho}+1|_{L^{3\rho/(1-s)}}|u-v|^2_{L^{6/(1+2s)}}\\
&\leq C'(|u|_{H^{1-s}}^{2\rho}+|v|_{H^{1-s}}^{2\rho}+1)|u-v|^2_{H^{1-s}}.
\end{align*}
\selectlanguage{french}
% Chapter 3

\chapter{Grandes d\'eviations} % Main chapter title

\label{Chapter4} % For referencing the chapter elsewhere, use \ref{Chapter1} 

%\lhead[]{Chapitre 4. \emph{M\'elange de}} % This is for the header on each page - perhaps a shortened title

%----------------------------------------------------------------------------------------

%\author{Davit Martirosyan\footnote{Department of Mathematics, University of Cergy-Pontoise, CNRS %UMR 8088, 2 avenue
%Adolphe Chauvin, 95300 Cergy-Pontoise, France;e-mail: \href{mailto:Davit.Martirosyan@u-cergy.fr}{Davit.Martirosyan@u-cergy.fr}}}
\selectlanguage{english}
\section*{Large deviations for stationary measures of stochastic nonlinear wave equation with smooth white noise} 

{\bf Abstract}.
The paper is devoted to the derivation of large deviations principle for the family $(\mu^\es)_{\es>0}$ of stationary measures of the Markov process generated by the flow of equation
$$
\p_t^2u+\gamma \p_tu-\de u+f(u)=h(x)+\sqrt{\es}\,\vartheta(t,x).
$$
The equation is considered in a bounded domain $D\subset\rr^3$ with a smooth boundary and is supplemented with the Dirichlet boundary condition. Here $f$ is a nonlinear term satisfying some standard dissipativity and growth conditions, the force $\vartheta$ is a non-degenerate white noise, and $h$ is a function in $H^1_0(D)$. The main novelty here is that we do not assume that the limiting equation (i.e., when $\es=0)$ possesses a unique equilibrium and that we do not impose roughness on the noise. Our proof is based on a development of the approach introduced by Freidlin and Wentzell for the study of large deviations for stationary measures of stochastic ODEs on a compact manifold, and some ideas introduced by Sowers. Some ingredients of the proof rely on rather nonstandard techniques.

\setcounter{section}{-1}

\bigskip
\section{Introduction}
We study the large deviations for the family of probability measures $(\mu^\es)_{\es>0}$, where $\mu^\es$ stands for the invariant measure of the Markov process generated by the flow of equation
\be\label{0.1.4}
\p_t^2u+\gamma \p_tu-\de u+f(u)=h(x)+\sqrt{\es}\, \vartheta(t,x),\q [u(0),\dt u(0)]=[u_0, u_1].
\ee
The space variable $x$ belongs to a bounded domain $D\subset\rr^3$ with a smooth boundary, and the equation is supplemented with the Dirichlet boundary condition. The nonlinear term $f$ satisfies the dissipativity and growth conditions that are given in the next section. The force $\vartheta(t,x)$ is a colored white noise of the form
\be\label{1.63}
\vartheta(t,x)=\sum_{j=1}^\infty b_j \dt\beta_j(t)e_j(x).
\ee
Here $\{\beta_j(t)\}$ is a sequence of independent standard Brownian motions, $\{e_j\}$ is an orthonormal basis in $L^2(D)$ composed of the eigenfunctions of the Dirichlet Laplacian, and $\{b_j\}$ is a sequence of positive numbers that goes to zero sufficiently fast (see \eqref{1.57}). 
The initial point $[u_0,u_1]$ belongs to the phase space $\h=H^1_0(D)\times L^2(D)$. Finally, $h(x)$ is a function in $H^1_0(D)$ and satisfies a genericity assumption given in next section.  As it was shown in \cite{DM2014}, under the above hypotheses, the Markov process corresponding to \eqref{0.1.4} has a unique stationary measure $\mu^\es$ which exponentially attracts the law of any solution.

\medskip
Here we are interested in the asymptotic behavior of the family $(\mu^\es)$ as $\es$ goes to zero. We show that this family satisfies the large deviations principle (LDP), which means that there is a function that  describes precisely the logarithmic asymptotics  of $(\mu^\es)$ as the amplitude of the noise tends to zero. More formally, we have the following theorem which is part of the main result of this paper.
\begin{mt}
Let the above conditions be satisfied. Then there is a function $\vvv:\h\to [0,\iin]$ with compact level sets such that we have
\be\label{9.43}
-\inf_{\uu\in \dot\Gamma}\vvv(\uu)\le\liminf_{\es\to 0}\es\ln \mu^\es(\Gamma)\le\limsup_{\es\to 0}\es\ln \mu^\es(\Gamma)\le -\inf_{\uu\in \bar\Gamma}\vvv(\uu),
\ee
where $\Gamma$ is any Borel subset of $\h$, and we denote by $\dot\Gamma$ and $\bar\Gamma$ its interior and closure, respectively.
\end{mt}

\medskip
Before outlining the main ideas behind the proof of this result, we discuss some of the earlier works concerning the large deviations of stochastic PDEs. There is now a vast literature on this subject and the theory is developed in several directions. The most studied among them are the large deviations for the laws of trajectories of stochastic PDEs with vanishing noise. The SPDEs considered in this context include the reaction-diffusion equation \cite{sowers-1992a, CerRoc2004}, the 2D Navier-Stokes equations \cite{Chang1996, SrSu2006}, the nonlinear Schr\"odinger equation \cite{Gautier2005-2}, the Allen-Cahn equation \cite{HairWeb2014}, the quasi-geostrophic equations \cite{LiuRocZhuChan2013}, equations with general monotone drift \cite{Liu2010}, and scalar conservation laws \cite{Mariani2010}. See also the papers \cite{KalXi1996, CheMil1997, CarWeb1999, CM-2010} for results in a more abstract setting that cover a wide class of SPDEs including 2D hydrodynamical type models. Another direction is the study of exit problems for trajectories of stochastic PDEs. The results include \cite{Peszat1994, CheZhiFre2005, Gautier2008, FreKor2010, CerSal2014, BreCerFre2014}.

\medskip
The situation is completely different if we restrict our attention to the results devoted to the small-noise large deviations for stationary measures of stochastic PDEs. To the best of our knowledge, the only papers where the LDP is derived in this context are those by Sowers \cite{sowers-1992b} and Cerrai-R\"ockner \cite{CeRo2005}. These two important works are devoted to the LDP for stationary measures of the reaction-diffusion equation. In the first of them, the force is a non-Gaussian perturbation, while the second one deals with a multiplicative noise. In both papers, the origin is a unique equilibrium of the unperturbed equation and the noise is assumed to be sufficiently irregular with respect to the space variable. To the best of our knowledge, the present paper provides the first result of large deviations  for stationary measures of stochastic PDEs in the case of nontrivial limiting dynamics. Moreover, the random force $\vartheta(t,x)$ is spatially regular in our case. Both these facts create substantial additional problems which are discussed below.

\medskip
We now turn to outlining some ideas of the proof of our main result and describing the main novelty of this paper. 
Our proof relies on a development of Freidlin-Wentzell's approach.  In order to explain it, we briefly recall the original method, which relies on three main steps. The first one consists of establishing some large deviations estimates for the family of discrete-time Markov chains $(Z_n^\es)$ on the boundary. Next, one considers the family $(\lm^\es)$ of stationary measures of these chains and shows that similar estimates hold for $(\lm^\es)$. The final step is to use the Khasminskii formula to reconstruct the measure $\mu^\es$ through $\lm^\es$ and use the estimates derived for the latter in the second step, to get the LDP for $(\mu^\es)$. It turns out that in the PDE setting, this method breaks down already in the second step. Indeed, the existence of stationary measure $\lm^\es$ for the chain on the boundary is a highly nontrivial fact in this case, since on the one hand the Doob theorem cannot be applied, on the other hand this chain does not  possess the Feller property in case of a smooth random force. Moreover, even if we assume that the stationary measure exists, the classical argument does not allow to derive the LDP in this case, since the compactness of the phase space is needed.

\medskip
 To overcome these problems, we introduce a notion of \mpp{generalized stationary measure}, which is, informally speaking, a measure that is stationary but is not supposed to be $\sigma$-additive. We show that any discrete-time Markov chain possesses such a state, thus ensuring existence of stationary measure $\lm^\es$ for the chain on the boundary in this weaker sense. It turns out that at this point (this corresponds to the second step mentioned above) the argument developed by Freidlin and Wentzell does not use the $\sigma$-additivity of $\lm^\es$, and once the necessary estimates for $(Z_n^\es)$ are obtained, they imply similar bounds for $(\lm^\es)$. Here our use of the classical technique ends, and the proof goes in a completely different direction. The reason for this is that the initial measure $\mu^\es$ cannot be reconstructed through $\lm^\es$, since, unlike the previous step, here we do need the $\sigma$-additivity of the measure $\lm^\es$. To handle this new problem,  we use the estimates obtained for $(\lm^\es)$ together with the mixing property of $\mu^\es$ established in \cite{DM2014}, to construct an auxiliary finitely additive measure $\hat\mu^\es$ defined on Borel subsets of $\h$ that satisfies
\be\label{9.65''}
\mu^\es(\dt\Gamma)\le\hat\mu^\es(\dt\Gamma)\le\hat\mu^\es(\bar\Gamma)\le\mu^\es(\bar\Gamma)\q\text{ for any }\Gamma\subset \h
\ee
and such that the family $(\hat\mu^\es)$ obeys some large deviations estimates on the balls.  The proof of the upper bound in these estimates is not a problem. The lower bound relies on an additional new ingredient, namely the notion of \mpp{stochastic stability} of a set.

\medskip
We say that a set $E\subset\h$ is \mpp{stochastically stable} if we have \,\footnote{Let us note that in the case when it is known a priori that a family $(\mu^\es)$ satisfies the LDP with a rate function $\vvv$, then a set $E$ is stochastically stable if and only if its closure has a nonempty intersection with the kernel of $\vvv$.}
$$
\lim_{\es\to 0}\es\ln\mu^{\es}(E_\eta)=0 \q\q \text{ for any }\eta>0,
$$
where $E_\eta$ stands for the open $\eta$-neighborhood of $E$ in $\h$. 

\medskip
We use it in the following context. Let us denote by $\E\subset\h$ the set of stationary flows $\uu=[u, 0]$ of the unperturbed equation
\be\label{1.51}
\p_t^2u+\gamma \p_tu-\de u+f(u)=h(x).
\ee
\begin{lemma}\label{9.47}
The set $\E$ of equilibria of \ef{1.51} is stochastically stable.
\end{lemma}
This result allows to prove the above mentioned lower bound and to complete the proof of large deviations on balls for the family $(\hat\mu^\es)$. Inequality \ef{9.65''} implies that similar result holds for the family $(\mu^\es)$ of stationary measures. The final step is to prove that this family is exponentially tight and to show that this combined with the above large deviations estimates implies the LDP.

\medskip
 We now present another essential component of the proof which allows, in particular, to get exponential tightness and also prove Lemma \ref{9.47}. Let us consider the semigroup $S(t):\h\to \h$ corresponding to \ef{1.51} and denote by $\aaa$ its global attractor.

\begin{proposition}[A priori upper bound]\label{9.18}
Under the above hypotheses, there is a function $V_\aaa:\h\to [0, \iin]$ with compact level sets and vanishing only on the attractor $\aaa$ that provides the large deviations upper bound for the family $(\mu^\es)$, that is we have
\be\label{9.19}
\limsup_{\es\to 0}\es\ln \mu^{\es}(F)\leq-\inf_{\uu\in F} V_\aaa(\uu)\q \text{ for any } F\subset\h \text{ closed}.
\ee
In particular, the family $(\mu^\es)$ is exponentially tight and any of its weak limits is concentrated on the set $\aaa$. 
\end{proposition}
Let us mention that function $V_\aaa$ has an explicit interpretation in terms of the quasipotential. Namely, for any $\uu\in \h$, $V_\aaa(\uu)$ represents the minimal energy needed to reach arbitrarily small neighborhood of $\uu$ from the global attractor in a finite time. It should be emphasized that once the main result of the paper is established, this proposition will lose its interest, since, in general, $V_\aaa$ is not the function that governs the LDP of the family $(\mu^\es)$, and that much more is proved concerning weak limits of $(\mu^\es)$. Let us mention also that some ideas of the proof of Proposition \ref{9.18} are inspired by \cite{sowers-1992b}.

\medskip
At the end of this section, let us point out that when equation \ef{1.51} has a unique equilibrium, Proposition \ref{9.18} is sufficient to derive the LDP, and in this particular case there is no need to use the Freidlin-Wentzell theory and the above scheme. Indeed, we first note that in this case the attractor $\aaa$ is a singleton $\{\hat\uu\}$, where $\hat\uu=[\hat u, 0]$ is the equilibrium position. In view of Proposition \ref{9.18}, the family $(\mu^\es)$ is tight and any weak limit of it is concentrated on $\aaa=\{\hat\uu\}$. Therefore, $\mu^\es$ weakly converges to the Dirac measure concentrated at $\hat\uu$. A simple argument (see Section \ref{9.20}) shows that this convergence and the fact that $\aaa=\{\hat\uu\}$ imply that the function $V_\aaa$ provides also the large deviations lower bound for $(\mu^\es)$. Thus, in the case of the trivial dynamics, the function $V_\aaa$ governs the LDP of the family $(\mu^\es)$. We note also that this is the only case when that happens.

\medskip
The paper is organized as follows. In Section \ref{1.76}, we state the main result and present the scheme of its proof. In Section \ref{9.45}, we establish bounds for one-step transition probabilities for the chain on the boundary. The next two sections are devoted to the proof of large deviations estimates on the balls for $(\mu^\es)$. In Section \ref{1.77}, we establish Proposition \ref{9.18}. Finally, the appendix contains some auxiliary results used in the main text.

\bigskip
\section{Main result and scheme of its proof}\label{1.76}
In this section we state the main result of the paper and outline its proof. We start by recalling the notion of large deviations.
\subsection{Large deviations: equivalent formulations}
Let $\zzz$ be a Polish space. A functional $\mathfrak{I}$ defined on $\zzz$ and with range in $[0,\iin]$ is  called a (good) rate function if it has compact level sets, which means that the set $\{\mathfrak{I}\leq M\}$ is compact in $\zzz$ for any $M\geq 0$. Let $(\mathfrak{m}^\es)_{\es>0}$ be a family of probability measures on $\zzz$. The family $(\mathfrak{m}^\es)_{\es>0}$ is said to satisfy the large deviations principle in $\zzz$ with rate function $\mathfrak{I}:\zzz\to [0,\iin]$ if the following two conditions hold.

\bi
\item{\it Upper bound}
\ei
\be\label{7.1}
\limsup_{\es\to 0}\es\ln \mathfrak{m}^{\es}(F)\leq-\inf_{z\in F} \mathfrak{I}(z)\q \text{ for any } F\subset\zzz \text{ closed}.
\ee
This inequality is equivalent to the following (e.g., see Chapter 12 of \cite{DZ1992}). For any positive numbers $\De, \De'$ and $M$ there is $\es_*>0$ such that
\be\label{7.2}
\mathfrak{m}^\es(z\in\zzz: d_{\zzz}(z,\{\mathfrak{I}\leq M\})\geq\De)\leq\exp(-(M-\De')/\es)\q\text{ for } \es\leq\es_*.
\ee

\bi
\item{\it Lower bound}
\ei
\be\label{7.3}
\liminf_{\es\to 0}\es\ln\mathfrak{m}^{\es}(G)\geq-\inf_{z\in G} \mathfrak{I}(z)\q \text{ for any } G\subset\zzz \text{ open}.
\ee
This is equivalent to the following. For any $z_*\in\zzz$ and any positive numbers $\eta$ and $\eta'$ there is $\es_*>0$ such that
\be\label{7.4}
\mathfrak{m}^\es(z\in\zzz: d_\zzz(z, z_*)\le\eta)\geq\exp(-(\mathfrak{I}(z_*)+\eta')/\es)\q\text{ for } \es\leq\es_*.
\ee
The family of random variables $(\mathfrak{X}^{\es})_{\es>0}$ in $\zzz$ is said to satisfy the LDP with rate function $\mathfrak{I}$, if so does the family of their laws.

\subsection{Main result}
Before stating the main result, let us make the precise hypotheses on the nonlinearity and the coefficients entering the definition of $\vartheta(t)$. We suppose that function $f$ satisfies the growth restriction
\be\label{1.54}
|f''(u)|\leq C(|u|^{\rho-1}+1)\q u\in\rr,
\ee  
where $C$ and $\rho<2$ are positive constants, and the dissipativity conditions
\be\label{1.56}
F(u)\geq -\nu u^2-C,\q\q f(u)u- F(u)\geq-\nu u^2-C \q u\in\rr\,,
\ee
where $F$ is the primitive of $f$, $\nu\leq (\lm_1\wedge\gamma)/8$ is a positive constant, and $\lm_j$ stands for the eigenvalue corresponding to $e_j$. The coefficients $b_j$ are positive numbers satisfying
\be \label{1.57}
\BBB_1=\sum_{j=1}^\iin\lm_j b_j^2<\infty. 
\ee
Recall that we denote by $\E\subset\h$ the set of stationary flows $\uu=[u, 0]$ of equation \ef{1.51}.
It is well known that generically with respect to $h(x)$, the set $\E$ is finite (see Section \ref{9.21} for more details). We assume that $h(x)$ belongs to this generic set, so that there are finitely many equilibria, and we write $\E=\{\hat\uu_1, \ldots, \hat\uu_\ell\}$. Recall that the equilibrium $\hat \uu$ is called Lyapunov stable if for any $\eta>0$ there is $\De>0$ such that any flow of \ef{1.51} issued from the $\De$-neighborhood of $\hat\uu$ remains in the $\eta$-neighborhood of $\hat\uu$ for all time. We shall denote by $\E_s\subset \E$ the set of Lyapunov stable equilibria.

\medskip
The following theorem is the main result of this paper.
\begin{theorem}\label{1.58}
Let the above conditions be satisfied. Then the family $(\mu^\es)$ satisfies the large deviations principle in $\h$. Moreover, the corresponding rate function can vanish only on the set $\E_s\subset\{\hat\uu_1, \ldots, \hat\uu_\ell\}$ of Lyapunov stable equilibria of \eqref{1.51}. In particular, $(\mu^\es)$ is exponentially tight and any weak limit of this family is concentrated on $\E_s$. 
\end{theorem}

Let us mention that in the case when there is only one stable equilibrium $\hat\uu$ among $\{\hat \uu_1, \ldots, \hat\uu_\ell\}$ (which is the case, for example, when $\ell\leq 2$) the description of the rate function $\vvv:\h\to[0, \iin]$ that governs the LDP is quite explicit in terms of energy function (quasipotential). Namely, given $\uu$ in $\h$, $\vvv(\uu)$ represents the minimal energy needed to reach arbitrarily small neighborhood of point $\uu$ from $\hat\uu$ in a finite time. In the particular case, when the limiting equation of a stochastic PDE possesses a unique equilibrium that is globally asymptotically stable, this type of description was obtained for stochastic reaction-diffusion equation in papers \cite{sowers-1992b} and \cite{CeRo2005}.

\subsection {Scheme of the proof} 

In what follows we admit Proposition \ref{9.18}, whose proof is given in Section \ref{1.77}.

\bigskip
{\it Construction of the rate function.} We first introduce some notation following \cite{FW2012}; see Section 2 of Chapter 6 of that book. Given $\ell\in\nn$ and $i\leq\ell$, 
denote by $G_{\ell}(i)$ the set of graphs consisting of arrows 
$$
(m_1\to m_2 \to\cdots \to m_{\ell-1}\to m_\ell)
$$
such that
$$
\{m_1,\ldots, m_\ell\}=\{1,\ldots, \ell\}\q\text{ and } m_\ell=i.
$$
Further, let us introduce 
\be\label{8.47}
W_\ell(\hat\uu_i)=\min_{\mathfrak{g}\in G_{\ell}(i)}\sum_{(m\to n)\in \mathfrak{g}} V(\hat\uu_m, \hat\uu_n),
\ee
where $V(\uu_1, \uu_2)$ is the minimal energy needed to reach arbitrarily small neighborhood of $\uu_2$ from $\uu_1$ in a finite time (see \eqref{8.49} for the precise definition). The rate function $\vvv:\h\to [0, \iin]$ that governs the LDP of the family $(\mu^\es)$ is given by
\be\label{8.45}
\vvv(\uu)=\min_{i\leq \ell}[W_\ell(\hat \uu_i)+V(\hat\uu_i, \uu)]-\min_{i\leq\ell}W_\ell(\hat\uu_i).
\ee

Let us mention that when calculating these minima, we can restrict ourselves to considering only those $i$ for which $\hat\uu_i$ is Lyapunov stable.

\bigskip
{\it Markov chain on the boundary.}  What follows is a modification of a construction introduced in \cite{FW2012} (see Chapter 6) which itself is a variation of an argument used in \cite{Khas2011}. Let $\hat\uu_1,\ldots, \hat\uu_\ell$ be the stationary points of $S(t)$. Let us fix any $\uu\in\h\backslash\{\hat\uu_1,\ldots,\hat\uu_{\ell}\}$ and write $\hat\uu_{\ell+1}=\uu$. Given any $\rho_*>0$ and $0<\rho'_1<\rho_0'<\rho_1<\rho_0<\rho_*$, we use the following construction. For $i\leq \ell$, we denote by $g_i$ and $\tilde g_i$ the open $\rho_1$- and $\rho_0$-neighborhoods of $\hat\uu_i$, respectively. Similarly, we denote by $g_{\ell+1}$ and $\tilde g_{\ell+1}$, respectively, the $\rho_1'$- and $\rho_0'$-neighborhoods of $\hat\uu_{\ell+1}$. Further, we denote by $g$ and $\tilde g$ the union over $i\leq \ell+1$ of $g_i$ and $\tilde g_i$, respectively. For any $\es>0$ and $\vv\in \h$ let $S^\es(t;\vv)$ be the flow at time $t$ of \ef{0.1.4} issued from $\vv$. Let $\sigma_0^\es$ be the time of the first exit of the process $S^\es(t;\cdot)$ from $\tilde g$, and let $\tau_1^\es$ be the first instant after $\sigma_0^\es$ when $S^\es(t;\cdot)$ hits the boundary of $g$. Similarly, for $n\geq 1$ we denote by $\sigma_n^\es$ the first instant after $\tau_n^\es$ of  exit from $\tilde g$ and by $\tau_{n+1}^\es$ the first instant after $\sigma_n$ when $S^\es(t;\cdot)$ hits $\p g$. Let us mention that all these Markov times are almost surely finite and, moreover, have finite exponential moments (see \ef{9.22}). We consider the Markov chain on the boundary $\p g$ defined by $Z_n^\es (\cdot)= S^\es(\tau^\es_n,\cdot)$. We shall denote by $\tilde P^\es(\vv,\Gamma)$ the one-step transition probability of the chain $(Z_n^\es)$, that is 
$$
\tilde P^\es(\vv,\Gamma)=\pp(S^\es(\tau_1^\es;\vv)\in \Gamma)\q\text{ for any } \vv\in\p g\,\text{ and } \Gamma\subset\p g.
$$
The first step is a result for generalized stationary measure $\lm^\es$ of $\tilde P_\es(\vv, \Gamma)$. We confine ourselves to announcing the result and refer the reader to Section \ref{9.82} for the definition of this concept.
\begin{proposition}\label{9.5'}
For any $\beta>0$ and $\rho_*>0$ there exist $0<\rho'_1<\rho_0'<\rho_1<\rho_0<\rho_*$ such that for all $\es<<1$ (i.e., sufficiently small), we have
\be\label{8.27'}
\exp(-(\vvv(\hat\uu_j)+\beta)/\es)\leq \lm^\es({g_j})\leq \exp(-(\vvv(\hat\uu_j)-\beta)/\es),
\ee
where inequalities hold for any $i, j\leq \ell+1, i\neq j$.
\end{proposition}
This allows to show that for $\es<<1$, there is a finitely additive measure on $\h$ satisfying \ef{9.65''} and such that
\be\label{9.81'}
\exp(-(\vvv(\hat\uu_j)+\beta)/\es)\leq \hat\mu^\es(g_j)\leq \exp(-(\vvv(\hat\uu_j)-\beta)/\es).
\ee
 As a direct corollary of these relations, we get the following result. 
\begin{proposition}\label{9.4} For any $\beta>0$ and $\rho_*>0$ there exist $0<\rho'_1<\rho'_0<\rho_1<\rho_0<\rho_*$ such that for any $j\leq \ell+1$ and $\es<<1$, we have
\begin{align}
\mu^\es(g_j)&\leq \exp(-(\vvv(\hat\uu_j)-\beta)/\es)\label{9.79},\\
\mu^\es(\bar g_j)&\ge \exp(-(\vvv(\hat\uu_j)+\beta)/\es)\label{9.80}.
\end{align}
\end{proposition}
The passage from Proposition \ref{9.5'} to \ref{9.4} is the most involved part of the paper and construction of $\hat\mu^\es$ is the main idea behind its proof. Without going into details, we describe in few words another key ingredient of the proof, namely Lemma \ref{9.47}.
\begin{definition}
We shall say that a set $E\subset\h$ is \mpp{stochastically stable} or \mpp{stable with respect to}  $(\mu^\es)$ if we have 

$$
\lim_{\es\to 0}\es\ln\mu^{\es}(E_\eta)=0 \q\q \text{ for any }\eta>0,
$$
where $E_\eta$ stands for the open $\eta$-neighborhood of $E$ in $\h$. If the above relation holds only along some sequence $\es_j\to 0$ (that is $\es$ replaced by $\es_j$), we shall say that $E$ is \mpp{stable with respect to} $(\mu^{\es_j})$.
\end{definition}
Let us show how to derive Lemma \ref{9.47} using Proposition \ref{9.18}. We then show how the same proposition combined with \eqref{9.79}-\ef{9.80} implies the LDP.
 We admit the following result established in the appendix.
\begin{lemma}\label{1.201}
Let $\ooo$ be a heteroclinic orbit of $S(t)$ and let $\uu_1\in\ooo$. Suppose that $\uu_1$ is stable with respect to $(\mu^{\es_j})$ for some sequence $\es_j\to 0$. Then so is any point $\uu_2$ lying on $\ooo$ after $\uu_1$ (in the direction of the orbit).
\end{lemma}
Let us mention that we consider the endpoints of an orbit as its elements, and when saying $\uu$ is stable with respect to $(\mu^{\es_j})$ we mean that so is the set $\{\uu\}$. Now let us assume that Lemma \ref{9.47} is not true. Then we can find two positive constants $a$ and $\eta$, and a sequence $\es_j$ going to zero such that
\be\label{8.20}
\mu^{\es_j}(\E_\eta)=\sum_{j=1}^\ell \mu^{\es_j}(B(\hat\uu_j,\eta)) \leq\exp(-a/\es_j)\q\text{ for all }j\geq 1,
\ee
where $B(\uu, r)$ stands for the open ball in $\h$ of radius $r$ and centered at $\uu$. 
By Proposition \ref{9.18}, the sequence $(\mu^{\es_j})$ is tight and any weak limit of it is concentrated on $\aaa$. So, up to extracting a subsequence, we can assume that $\mu^{\es_j}\rightharpoonup \mu_*$, and $\mu_*$ is concentrated on $\aaa$. By Theorem \ref{Th-attractor}, the global attractor $\aaa$ consists of points $(\hat\uu_i)_{i=1}^n$ and joining them heteroclinic orbits. Let $\uu_*$ be a point lying on such an orbit that belongs to the support of $\mu_*$. By the portmanteau theorem, we have
$$
\liminf_{j\to \iin}\mu^{\es_j}(B(\uu_*,r))\ge\mu_*(B(\uu_*,r))>0\q\q\text{ for any }r>0.
$$
Therefore, the point $\uu_*$ is stable with respect to $(\mu^{\es_j})$. On the other hand, it follows from the previous lemma that so are all points of the attractor that lie on that orbit  after $\uu_*$. In particular so is the endpoint of $\ooo$, which is in contradiction with \eqref{8.20}. The proof of Lemma \ref{9.47} is complete.

\bigskip
{\it Derivation of the LDP.} 
We claim that the hypotheses of Lemma \ref{9.3} are satisfied for the family $(\mu^\es)_{\es>0}$ and rate function $\vvv$. Indeed, let $\beta$ and $\rho_*$ be two positive constants and let $\uu$ be any point in $\h$. If $\uu$ is not a stationary point, we denote $\hat\uu_{\ell+1}=\uu$ and use Proposition \ref{9.4} to find $\rho_1'<\rho_1<\rho_*$ such that we have \eqref{9.79}-\ef{9.80} and we set $\tilde\rho(\uu)=\rho_1'$. Otherwise ($\uu$ is stationary), we take any non stationary point $\uu'$ and denote $\hat\uu_{\ell+1}=\uu'$. We once again use Proposition \ref{9.4} to find $\rho_1'<\rho_1<\rho_*$ such that we have \eqref{9.79}-\ef{9.80} and we set $\tilde\rho(\uu)=\rho_1$. Let us note in this case ($\uu$ is stationary) the choice of $\uu'$ is not important due to the fact that we are interested in the asymptotic behavior of $(\mu^\es)$  only in the neighborhood of $\uu$, and we add a new point $\hat\uu_{\ell+1}=\uu'$ only to be consistent with Proposition \ref{9.4}. Thus, the hypotheses of Lemma \ref{9.3} are satisfied and the family $(\mu^\es)$ satisfies the LDP in $\h$ with rate function $\vvv$.

\bigskip
\section{Proof of Theorem \ref{1.58}}\label{9.45}
The present section is devoted to the proof of the main result of this paper. We admit Proposition \ref{9.18}, which is proved in the next section, and following the scheme presented above establish Theorem \ref{1.58}. We shall always assume that the hypotheses of this theorem are satisfied. 
 
\subsection{Construction of the rate function}\label{1.64}
Here we define the function $V:\h\times\h\to [0, \iin]$ entering relation \ef{8.45} and function $V_\aaa:\h\to [0, \iin]$ from Proposition \ref{9.18}. We first introduce some notation. 
For any $t\geq 0$, $\vv\in\h$ and $\ph\in L^2(0,T;L^2(D))$, let us denote by $S^\ph(t;\vv)$ the flow at time $t$ of equation 
\be\label{7.10}
\p_t^2 u+\gamma \p_t u-\de u+f(u)=h(x)+\ph(t,x)
\ee
issued from $\vv$.
Let $H_{\vartheta}$ be the Hilbert space defined by \be\label{1.34}
H_{\vartheta}=\{v\in L^2(D): |v|_{H_\vartheta}^2=\sum_{j=1}^\iin b_j^{-2}\,(v,e_j)^2<\iin\}.
\ee
For a trajectory $\uu_\cdot\in C(0,T;\h)$ we introduce
\be\label{8.50}
I_T(\uu_\cdot)=J_T(\ph):=\f{1}{2}\int_0^T |\ph(s)|_{H_\vartheta}^2\dd s,
\ee
if there is $\ph\in L^2(0,T;H_\vartheta)$ such that $\uu_\cdot=S^\ph(\cdot;\uu(0))$, and $I_T(\uu_\cdot)=\iin$ otherwise.
We now define $V:\h\times\h\to [0, \iin]$ by
\be\label{8.49}
V(\uu_1,\uu_2)=\lim_{\eta\to 0}\inf\{I_T(\uu_\cdot); T>0, \uu_\cdot\in C(0,T;\h): \uu(0)=\uu_1, \uu(T)\in B(\uu_2,\eta) \}.
\ee

\medskip
Let us note that this limit (finite or infinite) exists, since the expression written after the limit sign is monotone in $\eta>0$. As we mentioned in previous section, $V(\uu_1, \uu_2)$ represents the minimal energy needed to reach arbitrarily small neighborhood of $\uu_2$ from $\uu_1$ in a finite time.

\bigskip{\it Remark.} The definition of the quasipotential $V$ using this filtration in $\eta$ rather than  taking directly $\eta=0$ is explained by the lack of the exact controllability of the NLW equation by a regular force, and \ef{8.49} ensures the lower semicontinuity of function $\vvv$ given by \ef{8.45}.

\medskip
The function $V_\aaa:\h\to [0, \iin]$ entering Proposition \ref{9.18} is defined by
\be\label{9.6}
V_\aaa(\uu_*)=\inf_{\uu_1\in \aaa} V(\uu_1, \uu_*).
\ee

\medskip
 Notice that the compactness of level sets of $V_\aaa$ implies that $\vvv$ has relatively compact level sets. Combining this with lower semicontinuity of $\vvv$ (the proof of this fact is identical to that of $V_\aaa$, see Section \ref{1.77}), we see that $\vvv$ is a rate function in $\h$.
In what follows, the space $\h$ will be endowed with the norm
\be\label{9.11}
|\uu|_\h^2=\|\g u_1\|^2+\|u_2+\al u_1\|^2 \q\text{ for }\uu=[u_1, u_2]\in\h, \,\footnote{\, We denote by $\|\cdot\|_s$ the $H^s$-norm of a vector, and $\|\cdot\|$ stands for the $L^2$-norm.}
\ee
where $\al>0$ is a small parameter.

\subsection{Markov chain on the boundary}
In this section we establish a result that implies Proposition \ref{9.5'}. For the proof of this implication, see Chapter 6 of \cite{FW2012}; the only difference here is that $\lm^\es$ is not necessarily $\sigma$-additive, which does not affect the proof. 

\medskip
Recall that we denote by $V(\hat\uu_i, \hat\uu_j)$ the minimal energy needed to reach any neighborhood of $\hat\uu_j$ from $\hat\uu_i$ in a finite time. In what follows, we shall denote by $\tilde V(\hat\uu_i, \hat\uu_j)$ the energy needed to reach any neighborhood $\hat\uu_j$ from $\hat\uu_i$ in a finite time without intersecting any $\hat\uu_k$, for $k\le\ell+1$ different from $i$ and $j$.
\begin{proposition}\label{9.5}
For any positive constants $\beta$ and $\rho_*$ there exist $0<\rho'_1<\rho_0'<\rho_1<\rho_0<\rho_*$ such that for all $\es<<1$, we have
\be\label{8.27}
\exp(-(\tilde V(\hat\uu_i,\hat\uu_j)+\beta)/\es)\leq \tilde P^\es(\vv,\p g_j)\leq \exp(-(\tilde V(\hat\uu_i,\hat\uu_j)-\beta)/\es),
\ee
where inequalities hold for any $\vv\in \p g_i$ and any $i, j\leq \ell+1, i\neq j$.
\end{proposition}

\medskip
{\it Comment.} In what follows, when proving this type of inequalities, we shall sometimes derive them with $\beta$ replaced by $C\beta$, where $C\ge 1$ is an absolute constant. Since $\beta>0$ can be taken arbitrarily small, these bounds are equivalent and we shall use this without further stipulation.

\medskip
{\it Derivation of the lower bound.} We assume that $\tilde V(\hat\uu_i,\hat\uu_j)<\iin$, since otherwise there is nothing to prove. We shall first establish the bound for $i\leq \ell$.
We need the following result, whose proof is given at the end of this section.
\begin{lemma}\label{8.28}
There exists $\tilde\rho>0$ such that for any $0<\rho_2<\rho_1<\tilde\rho$ we can find a finite time $T>0$ depending only on $\rho_1$ and $\rho_2$ such that for any point $\vv\in \bar B(\hat\uu_i,\rho_1)$, $i\leq \ell$, there is an action $\ph_{\vv}$ defined on the interval $[0, T]$ with energy not greater than $\beta$ such that we have
\be\label{8.25}
S^{\ph_\vv}(t;\vv)\in \bar B(\hat\uu_i,\rho_1)\q\text{ for } t\in[0,T] \q\text{ and }\q S^{\ph_\vv}(T;\vv)\in \bar B(\hat\uu_i,\rho_2/2).
\ee
\end{lemma}
By definition of $\tilde V$, for $\rho_*>0$ sufficiently small and $\rho_1'<\rho_*$, we can find a finite time $\tilde T>0$ and an action $\tilde \ph$ defined on $[0, \tilde T]$ with energy smaller than $\tilde V(\hat\uu_i,\hat\uu_j)+\beta$ such that $S^{\tilde\ph}(\tilde T;\hat\uu_i)\in B(\hat\uu_j, \rho_1'/4)$ and the curve $S^{\tilde\ph}(\tilde T;\hat\uu_i)$ does not intersect $\rho_*$-neighborhood of $\hat\uu_k$ for $k\neq i, j$ (note that if a trajectory does not intersect $\hat\uu_k$ then it also does not intersect some small neighborhood of $\hat\uu_k$). Since $\rho_*>0$ can be taken arbitrarily small, we may assume that $\rho_*\le\tilde\rho$, where $\tilde\rho$ is the constant from the above lemma. Let $\rho_2<\rho_*$ be so small that for any $\vv\in \bar B(\hat\uu_i, \rho_2)$ we have $S^{\tilde\ph}(\tilde T;\vv)\in B(\hat\uu_j,\rho'_1/2)$. 
We take any $\rho_1\in(\rho_2, \rho_*)$ and use the following construction. For any $\vv\in \bar g_i$, we denote by $\tilde\ph_{\vv}$ the action defined on $[0, T+\tilde T]$ that coincides with $\ph_\vv$ on $[0, T]$ and with $\tilde\ph$ on $[T, T+\tilde T]$. Let us note that for any $\vv\in\bar g_i$, we have
\be\label{8.29}
I_{T+\tilde T}(S^{\tilde\ph_\vv}(\cdot;\vv))=J_{T+\tilde T}(\tilde\ph_\vv)=J_T(\ph_\vv)+J_{\tilde T}(\tilde\ph)\leq \tilde V(\hat\uu_i,\hat\uu_j)+2\beta.
\ee
Now let us take any $\rho_0\in (\rho_1,\rho_*)$ and denote by $\De$ any positive number that is smaller than $\min\{\rho_0-\rho_1, \rho_2/2, \rho_1'/2\}$. Then we have the following: if the trajectory $S^\es(t;\vv)$ is in the $\De$-neighborhood of $S^{\tilde\ph_\vv}(t;\vv)$ in $C(0, T+\tilde T;\h)$ distance, then $\tau_1^\es(\vv)\leq T+\tilde T$ and $S^\es(\tau_1^\es;\vv)\in \p g_j$. Therefore, we have
$$
\inf_{\vv\in\p g_i}\tilde P^\es(\vv,\p g_j)\geq \pp(A),
$$
where we set 
$$
A=\{\om\in\Omega:\sup_{\vv\in \bar g_i} d_{C(0,T+\tilde T;\h)}(S^\es(\cdot;\vv), S^{\tilde\ph_\vv}(\cdot;\vv))<\De\}.
$$
Combining this with inequality \eqref{8.29} and Theorem \ref{6.1}, we derive the lower bound of \eqref{8.27} in the case $i\leq \ell$.

\medskip
We now show that if $\rho_0'<\rho_1$ is sufficiently small, then the lower bound is also true for $i=\ell+1$. Indeed, let $\tilde V(\hat\uu_{\ell+1},\hat\uu_j)<\iin$ and let $T>0$ and $\ph$ be such that $S^{\ph}(T;\hat\uu_{\ell+1})\in B(\hat\uu_j,\rho_1/4)$ and 
\be\label{8.30}
J_{T}(\ph)\leq \tilde V(\hat\uu_{\ell+1},\hat\uu_j)+\beta.
\ee
We assume that $\rho_0'$ is so small that $S^{\ph}(T;\vv)\in B(\hat\uu_j,\rho_1/2)$ for any $\vv\in \tilde g_{\ell+1}$. Let us take any $\De<\min\{\rho_0'-\rho_1', \rho_1/2\}$. Then for any $\vv\in \bar g_{\ell+1}$ if the trajectory $S^\es(t;\vv)$ is in the $\De$-neighborhood of $S^{\ph}(t;\vv)$ in $C(0, T;\h)$ distance, then we have $\tau_1^\es(\vv)\leq T$ and $S^\es(\tau_1^\es;\vv)\in \p g_j$. Therefore
$$
\inf_{\vv\in\p g_{\ell+1}}\tilde P^\es(\vv,\p g_j)\geq \pp(A'),
$$
where we set 
$$
A'=\{\om\in\Omega:\sup_{\vv\in \bar g_{\ell+1}} d_{C(0,T;\h)}(S^\es(\cdot;\vv), S^{\ph}(\cdot;\vv))<\De\}.
$$
Combining this with inequality \eqref{8.30} and Theorem \ref{6.1}, we derive the lower bound in the case $i=\ell+1$. 

\bigskip
{\it Proof of the upper bound.} We assume that $\rho_*>0$ is so small that the energy needed to move the point from $\rho_*$-neighborhood of $\hat\uu_i$ to $\rho_*$-neighborhood of $\hat\uu_j$ without intersecting any other $\hat\uu_k$ is no less than $\tilde V(\hat\uu_i,\hat\uu_j)-\beta.$ Let us denote by $\tau_g^\es$ the first instant when the process $S^\es(t,\cdot)$ hits the set $\bar g$.  Then, by the strong Markov property, we have 
\be\label{8.33}
\sup_{\vv\in\p g_i}\pp(S^\es(\tau_1^\es;\vv)\in \p g_j)\leq \sup_{\vv\in\p \tilde g_i}\pp(S^\es(\tau_g^\es;\vv)\in \p g_j).
\ee
In what follows we shall denote by $g'$ the set $g\backslash g_{\ell+1}$, i.e. the union over $i\leq\ell$ of $\rho_1$-neighborhoods of $\hat\uu_i$. We need the following result.
\begin{lemma}\label{8.31}
For any positive constants $\rho_1, R$ and $M$ there is $T>0$ such that
\be\label{8.38}
\sup_{\vv\in B_R}\pp(\tau^\es_{g'}(\vv)\geq T)\leq\exp(-M/\es)\q\text{ for }\es<<1,
\ee
where $\tau^\es_{g'}$ stands for the first hitting time of the set $\bar g'$, and $B_R$ is the closed ball in $\h$ of radius $R$ centered at the origin. 
\end{lemma}
Note that for any $\vv\in\p\tilde g_i$, we have
$$
\pp(S^\es(\tau_g^\es;\vv)\in \p g_j)\leq\pp(S^\es(\tau_g^\es;\vv)\in \p g_j, \tau_g^\es(\vv)< T)+\pp(\tau^\es_{g'}(\vv)\geq T),
$$
where we used the fact that $\tau^\es_{g}\leq\tau^\es_{g'}$. Now notice that the event under the probability sign of the first term of this sum means that the trajectory $S^\es(t;\cdot)$ issued from $\p\tilde g_i$ hits the set $\bar g_j$ over time $T$ and does not intersect any $\hat\uu_k$ for $k$ different from $i$ and $j$. It follows from Theorem \ref{6.1} that this event has a probability no greater than $\exp(-(\tilde V(\hat\uu_i,\hat\uu_j)-2\beta)/\es)$. Combining this with Lemma \ref{8.31} and inequality \eqref{8.33}, we infer
$$
\sup_{\vv\in\p g_i}\tilde P^\es(\vv,\p g_j)\leq \exp(-(\tilde V(\hat\uu_i,\hat\uu_j)-2\beta)/\es)+\exp(-M/\es)\q\text{ for }\es<<1.
$$
Since $M>0$ can be chosen arbitrarily large, we derive the upper bound.

\bigskip
{\it Proof of Lemma \ref{8.28}.}
For any $\vv\in \bar B(\hat\uu_i,\rho_1)$, let $\tilde\vv(t;\vv)$ be the flow issued from $\vv$ corresponding to the solution of
$$
\p_t^2 \tilde v+\gamma\p_t \tilde v-\de \tilde v+f(\tilde v)=h(x)+P_N[f(\tilde v)-f(\hat u_i)],
$$
where $\hat\uu_i=[\hat u_i, 0]$ and $P_N$ stands for the orthogonal projection from $L^2$ to its $N$ dimensional subspace spanned by vectors $e_1,\ldots, e_N$. Let us define $\ph_{\vv}=P_N[f(\tilde v)-f(\hat u_i)]$. Then we have $\tilde\vv(t;\vv)=S^{\ph_\vv}(t;\vv)$. Moreover, it follows from Proposition \ref{1.88} that for $N\geq N_*(|\hat\uu_i|_\h)$ we have
\be\label{8.24}
|S^{\ph_\vv}(t;\vv)-\hat\uu_i|_\h^2\leq e^{-\al t}|\vv-\hat\uu_i|_\h^2,
\ee
where $\al>0$ is the constant entering \ef{9.11}. In particular, if we take $T=2(\ln\rho_1-\ln\rho_2)/\al$ then we get \eqref{8.25}. Moreover, we have
\begin{align}
J_{T}(\ph_\vv)&=\f{1}{2}\int_0^{T}|P_N[f(\tilde v)-f(\hat u_i)]|_{H_\vartheta}^2\dd s\leq C(N)\int_0^{T}|f(\tilde v)-f(\hat u_i)|_{L^1}^2\dd s\notag\\
&\leq C_1\, C(N)\int_0^{T}(\|\tilde v\|_1^2+\|\hat u_i\|_1^2+1)\|\tilde v-\hat u_i\|_1^2\dd s\notag\\
&\leq C_2\,C(N) \int_0^{T}|\vv-\hat\uu_i|^2 e^{-\al s}\dd s\leq C_3\, C(N)\, \tilde\rho^2\leq \beta\label{5.14},
\end{align}
provided $\tilde\rho>0$ is sufficiently small.

\bigskip
{\it Proof of Lemma \ref{8.31}. Step~1: Reduction}. Let the positive constants $\rho_1, R$ and $M$ be fixed. We claim that it is sufficient to prove that for any $R'> R$ we can find positive constants $T_*$ and $a$ such that 
\be\label{8.34}
\sup_{\vv\in B_{R'}}\pp(\tau^\es_{g'}(\vv)\geq T_*)\leq\exp(-a/\es)\q\text{ for }\es<<1.
\ee
Indeed, taking this inequality for granted, let us derive \eqref{8.38}.
To this end, let us use Proposition \ref{1.96} to find $R'> R$ so large that
\be\label{8.35}
\sup_{t\geq 0}\sup_{\vv\in B_R}\pp(S^\es(t;\vv)\notin B_{R'})\leq\exp(-(M+1)/\es)\q\text{ for }\es<1.
\ee
Once such $R'$ is fixed we find $T_*>0$ and $a>0$ such that we have \eqref{8.34}.
Let us take $n\geq 1$ so large that $an>(M+1)$.
For any $k\leq n$ we introduce
$$
p_k=\sup_{\vv\in B_R}\pp(\tau^\es_{g'}(\vv)\geq k\,T_*).
$$
Then, we have
\begin{align}\label{8.37}
p_n&\leq\sup_{\vv\in B_R}\pp(\tau^\es_{g'}(\vv)\geq n\,T_*, S^\es((n-1)T_*;\vv)\in B_{R'})\notag\\
&\q+\sup_{\vv\in B_R}\pp(S^\es((n-1)T_*;\vv)\notin B_{R'})\leq q_n+\exp(-(M+1)/\es),
\end{align}
where we denote by $q_n$ the first term of this sum and we used inequality \eqref{8.35} to estimate the second one. Now note that by the Markov property, we have
\begin{align*}
q_n&=\sup_{\vv\in B_R}\e_\vv[\e(\ch_{\tau^\es_{g'}\geq n\,T_*}\ch_{S^\es((n-1)T_*;\vv)\in B_{R'}}|\fff^\es_{(n-1) T_*})]\\
&=\sup_{\vv\in B_R}\e_\vv[\ch_{\tau^\es_{g'}\geq (n-1)\,T_*} \ch_{S^\es((n-1)T_*;\vv)\in B_{R'}}\pp(\tau^\es_{g'}(S^\es((n-1)T_*;\vv))\geq T_*)]\\
&\leq \sup_{\tilde\vv\in B_{R'}}\pp(\tau^\es_{g'}(\tilde\vv)\geq T_*)\,p_{n-1}\leq\exp(-a/\es)\,p_{n-1},
\end{align*}
where $\fff^\es_t$ is the filtration corresponding to $S^\es(t;\cdot)$, and we used inequality \eqref{8.34}. Combining this with \eqref{8.37}, we derive
$$
p_n\leq \exp(-a/\es)\,p_{n-1}+\exp(-(M+1)/\es).
$$
Iterating this inequality, we infer 
$$
p_n\leq \exp(-an/\es)+(1-\exp(-a/\es))^{-1}\exp(-(M+1)/\es)\leq\exp(-M/\es).
$$
It follows that inequality \eqref{8.38} holds with $T=n\, T_*$.

\bigskip
{\it Step~2: Derivation of \eqref{8.34}}. We first show that for any positive constants $\tilde R$ and $\eta$, we have
\be\label{8.36}
\sup_{\vv\in B_{\tilde R}}l(\vv)<\iin,
\ee
where $\l(\vv)$ stands for the first instant when the deterministic flow $S(t)\vv$ hits the set $\bar\ooo(\eta)$ and where $\ooo(\eta)$ is the union over $i\leq\ell$ of $\eta$-neighborhoods of $\hat\uu_i$. Indeed, let us suppose that this is not true, and let us find $\tilde R>0$ and $\eta>0$ for which this inequality fails. 
Then, there exists a sequence $(\vv_j)\subset B_{\tilde R}$ such that 
\be\label{8.39}
l(\vv_j)\geq 2j.
\ee
For each $j\geq 1$, let us split the flow $S(t)\vv_j$ to the sum $\uu^1_j(t)+\uu^2_j(t)$, where $\uu^1_j(t)$ stands for the flow issued from $\vv_j$ of equation \eqref{1.51} with $f=h=0$. Then, for all $t\geq 0$, we have
\be\label{8.40}
|\uu^1_j(t)|_\h^2\leq e^{-\al t}|\vv_j|_\h^2, \q\q |\uu^2_j(t)|_{H^{s+1}\times H^s}\leq C_s(\tilde R),
\ee
where $s<1-\rho/2$ is any constant (e.g., see \cite{BV1992, Har85}). Using second of these inequalities, let us find $(j_k)\subset\nn$ such that the sequence $\uu^2_{j_k}(j_k)$ converges in $\h$ and denote by $\tilde \uu$ its limit. Then, in view of first inequality of \eqref{8.40}, we have 
\be\label{8.41}
S(j_k)\vv_{j_k}\to\tilde \uu \q\text{ in } \h \q\text{ as } k\to\iin.
\ee
Now let us find $t_*>0$ so large that
\be\label{8.42}
S(t_*)\tilde\uu\in \bar\ooo(\eta/2).
\ee
Note that thanks to \eqref{8.41} and continuity of $S(t)$, we have
\be\label{8.43}
S(j_k+t_*)\vv_{j_k}\to S(t_*)\tilde\uu.
\ee
Further, notice that by \eqref{8.39} we have $S(j_k+t_*)\vv_{j_k}\notin \bar\ooo(\eta)$ for $k\geq 1$ sufficiently large. This is clearly in contradiction with \eqref{8.42}-\eqref{8.43}. Inequality \eqref{8.36} is thus established.

\medskip
We are now ready to prove \eqref{8.34}. Indeed, let us use inequality \eqref{8.36} with $\tilde R=R'+1$ and $\eta=\rho_1/2$, and let us set 
\be\label{8.44}
T_*=\sup_{\vv\in B_{R'+1}} l(\vv)+1.
\ee
Let us consider the trajectories $\uu_\cdot\in C(0,T_*;\h)$ issued from $B_{R'+1}\backslash \ooo(\rho_1/2)$ and assuming their values outside $\ooo(\rho_1/2)$. Note that the family $\elll$ of such trajectories is closed in $C(0, T_*;\h)$. Therefore, the infimum 
$$
a'=\inf_{\uu_\cdot\in\elll}I_{T_*}(\uu_\cdot)
$$
is attained and is positive, since in view of \eqref{8.44} there is no deterministic trajectory $S(t)\vv$ under consideration. Now using Theorem \ref{6.1}, we see that \eqref{8.34} holds with $a=a'/2$.
The proof of Lemma \ref{8.31} is complete.

\section{Generalized stationarity and auxiliary measure}\label{9.82}
In this section we introduce a notion of generalized stationary measure and show that any discrete-time Markov chain possesses such a state. This will be used to construct a finitely additive measure $\hat\mu^\es$ satisfying relation \ef{9.65''} and such that the family $(\hat\mu^\es)$ satisfies some large deviations estimates (see \ef{9.78}). This in turn will imply Proposition \ref{9.4}.
\subsection{Generalized stationary measure}
 Let $X$ be a metric space and let $b(X)$ the space of bounded Borel measurable functions on $X$ equipped with the topology of uniform convergence. We shall denote by $b^*(X)$ the dual of $b(X)$ \,\footnote{\, $b^*(X)$ can be identified  with the space $ba(X)$ of finite, finitely additive signed measures on $X$ (e.g., see Theorem IV.5.1 in \cite{Dunford-Schwartz1}).} . A linear continuous map $\pP$ from $b(X)$ into itself is called a Markov operator on $X$, if $\pP\psi\ge 0$ for any $\psi\ge 0$ and $\pP1\equiv 1$. Let $\pP^*:b^*(X)\to b^*(X)$ be the dual of $\pP$, that is  
 $$
\pP^*\lm(\psi)=\lm(\pP\psi)
$$
for any $\lm\in b^*(X)$ and $\psi\in b(X)$. We shall say that $\lm$ is \mpp{a generalized stationary state} (or \mpp {measure)}  for $\pP$ if it satisfies the following properties: $\lm(\psi)\ge 0$ for any $\psi\ge 0$, $\lm(1)=1$, and $\pP^*\lm=\lm$, that is we have
 \be\label{9.57}
\lm(\pP\psi)=\lm(\psi)\q\text{ for any }\psi\in b(X).
\ee
To any such $\lm$ we associate a finitely additive measure defined on Borel subsets of $X$ by $\lm(\Gamma)=\lm(\ch_\Gamma)$ for $\Gamma\subset X$.
\begin{lemma}\label{lem-stationary}
Any Markov operator possesses a generalized stationary measure.
\end{lemma}
\bp
Let $\pP$ be a Markov operator defined on a space $X$. Consider the space
$$
\mathfrak{F}=\{\lm\in b^*(X): \lm(1)= 1 \text{ and } \lm(\psi)\ge 0 \text{ for }\psi\ge 0\}
$$
endowed with weak* topology. Note that if $\lm\in \mathfrak{F}$ then its norm is equal to 1. In view of the Banach-Alaoglu theorem, $\mathfrak{F}$ is relatively compact. Moreover, it is easy to see that $\mathfrak{F}$ is also closed and convex. Since $\pP$ is a Markov operator, its dual $\pP^*$ maps $\mathfrak{F}$ into itself. Thanks to the Leray-Schauder theorem (e.g., see Chapter 14 in \cite{taylor1996}), $\pP^*$ has a fixed point $\lm\in \mathfrak{F}$, which means that $\lm$ is a generalized stationary state for $\pP$.  It should be emphasized that $\lm$ is not stationary in the classical sense, since it is not necessarily $\sigma$-additive. 
\ep

\bigskip
In what follows, given $\es>0$, we shall denote by $\lm=\lm^\es$ any of generalized stationary states of the operator $\pP=\pP^\es:b(\p g)\to b (\p g)$ defined by
\be\label{9.75}
\pP\psi(\vv)=\int_{\p g}\psi(\Zz)\tilde P_1(\vv, \Dd\Zz)\equiv \e\psi(S^\es(\tau_1^\es;\vv)).
\ee
So we have
\be\label{9.74}
\lm(\pP\psi)=\lm(\psi)\q\text{ for any }\psi\in b(\p g)
\ee
and $\lm\in b^*(\p g)$ satisfies $\lm(\psi)\ge 0$ for $\psi\ge 0$, $\lm (1)=1$. We shall always assume that $\es>0$ is so small that we have \ef{8.27}.

\subsection{Khasminskii type relation}
For each $\es>0$, let us define a continuous map $\tilde\mu=\tilde\mu^\es$ from $b(\h)$ to $\rr$ by 
\be\label{9.59}
\tilde\mu(\psi)=\lm(\elll\psi),
\ee
where $\lm=\lm^\es$ is given by \ef{9.75}-\ef{9.74}, and $\elll=\elll^\es:b(\h)\to b(\h)$ is defined by
\be\label{9.60}
\elll\psi(\vv)=\e\int_0^{\tau_1^\es}\psi(S^\es(t;\vv))\dd t.
\ee
\begin{lemma} For any $\psi\in b(\h)$, we have
\end{lemma}
\be\label{9.61}
\tilde\mu(P_s\psi)=\tilde\mu(\psi)\q\text{ for any } s\ge 0,
\ee
where $P_s=P_s^\es: b(\h)\to b(\h)$ stands for the Markov operator of the process $S^\es(\cdot)$.
\bp
We use the classical argument (see Chapter 4 in \cite{Khas2011}). Let us fix $s\ge 0$. By the Markov property, for any $\vv\in\h$, we have
$$
\e\int_0^{\tau_1^\es}\psi(S^\es(t+s;\vv))\,dt=\e\int_0^{\tau_1^\es}P_s\psi(S^\es(t;\vv))\dd t.
$$
It follows that
\begin{align*}
\tilde\mu(P_s\psi)&=\lm(\elll P_s\psi)=\lm(\e\int_0^{\tau_1^\es}P_s\psi(S^\es(t;\cdot))\dd t)=\lm(\e\int_0^{\tau_1^\es}\psi(S^\es(t+s;\cdot))\dd t)\notag\\
&=\lm(\e\int_s^{s+\tau_1^\es}\psi(S^\es(t;\cdot))\dd t)\notag\\
&=\lm(\e\int_{\tau_1^\es}^{s+\tau_1^\es}\psi(S^\es(t;\cdot))\dd t)-\lm(\e\int_0^{s}\psi(S^\es(t;\cdot))\dd t)+\tilde\mu(\psi).
\end{align*}
Conditioning with respect to $\fff_{\tau_1^\es}$ and using the strong Markov property together with \ef{9.75}-\ef{9.74}, we see that the first two terms are equal. Since $s\ge 0$ was arbitrary, we arrive at \ef{9.61}.
\ep

\subsection{Auxiliary finitely additive measure}
Let us denote by $b_0(\h)$ the space $b(\h)$ endowed with the following convergence: given a point $\psi\in b(\h)$ and a sequence $(\psi_n)\subset b(\h)$, we shall say that $\psi_n$ converges to $\psi$ in $b_0(\h)$  $n\to \iin$, if 
$$
\sup_n\sup_{\uu\in \h}|\psi_n(\uu)|<\iin
$$ 
and for any $a>0$ we have
$$
\sup_{\uu\in B_a}|\psi_n(\uu)-\psi(\uu)|\to 0\q\text{ as }n\to \iin.
$$ 
\begin{lemma}\label{9.58}
The map $\tilde\mu$ given by \ef{9.59}-\ef{9.60} is continuous from $b_0(\h)$ to $\rr$.
\end{lemma}
\bp
For the simplicity, we shall write $\tau_1=\tau_1^\es$ and $\uu(t;\vv)=S^\es(t;\vv)$. We need to show that $\tilde\mu(\psi_n)\to 0$ for any $\psi_n$ converging to zero in $b_0(\h)$.
Since $\lm$ is continuous from $b(\p g)$ to $\rr$, it is sufficient to show that $\elll\psi_n$ goes to zero uniformly in $B_R$, where $R>0$ is so large that $g\subset B_R$. Let us fix any $\eta>0$.
Clearly, we may assume that $|\psi_n(\uu)|\le 1$ for any $n\ge 1$ and $\uu\in \h$. 
It follows from the Cauchy-Schwarz inequality and \ef{9.22} that for $R_1>0$ sufficiently large, we have 
\be\label{9.62}
\e_\vv(\tau_1\cdot\ch_{\tau_1\ge R_1})\le \eta\q\text{ for any }\vv\in B_R.
\ee
Once $R_1$ is fixed, let us use Proposition 3.2 from \cite{DM2014} to find $R_2>R_1$ such that 
\be\label{9.63}
\pp\left(\sup_{t\in [0, R_1]}|\uu(t;\vv)|\ge R_2\right)\le\eta/{R_1}\q\text{ for any }\vv\in B_R.
\ee
Now we have
\begin{align}
|\elll\psi_n(\vv)|\le\e\int_0^{\tau_1}|\psi_n(\uu(t;\vv))|\dd t &\leq \e(\ch_{\tau_1\ge R_1}\int_0^{\tau_1}|\psi_n(\uu(t;\vv))|\dd t)\notag\\
&\q +\e\int_0^{R_1}|\psi_n(\uu(t;\vv))|\dd t:=I_1+I_2.\label{9.64}
\end{align}
Let us note that in view of \ef{9.62}, we have $I_1\le\eta$. Further, since $\psi_n$ converges to zero in $b_0(\h)$, we can find $n_*(\eta)\ge 1$ such that for all $n\ge n_*(\eta)$, we have
$$
\sup_{\uu\in B_{R_2}}|\psi_n(\uu)|\le \eta/{R_1}.
$$
Let us denote by $A_\vv$ the event under the probability sign in \ef{9.63}. Then
$$
I_2\le R_1 \pp(A_\vv)+\e(\ch_{A_\vv^c}\int_0^{R_1}|\psi_n(\uu(t;\vv))|\dd t).
$$
Combining \ef{9.63} with last two inequalities, we get $I_2\le 2\eta$, so that we have $I_1+I_2\le 3\eta$. Using this with \ef{9.64}, we see that
$$
\sup_{\vv\in B_R}|\elll\psi_n(\vv)|\le 3\eta\q\text{ for any }n\ge n_*(\eta).
$$
Since $\eta>0$ was arbitrary, the proof is complete.
\ep

\medskip
Let us note that in view of inequalities \ef{8.13} and \ef{9.22}, $\tilde\mu=\tilde\mu^\es$ satisfies $0<\tilde\mu(1)<\iin$.
We shall denote by $\hat\mu$ the normalization of $\tilde \mu$, that is
\be\label{9.68}
\hat\mu(\psi)=\tilde\mu(\psi)/\tilde\mu(1).
\ee 
Thanks to Lemma \ref{9.58}, $\hat\mu$ is continuous from $b_0(\h)$ to $\rr$.
For any Borel subset $\Gamma\subset \h$, we shall write
\be\label{9.66}
\hat\mu(\Gamma)=\hat\mu(\ch_\Gamma).
\ee
This notation will not lead to a confusion. 
\begin{lemma}\label{9.67}For any $\Gamma\subset \h$, we have
\be\label{9.65}
\mu(\dt\Gamma)\le\hat\mu(\dt\Gamma)\le\hat\mu(\bar\Gamma)\le\mu(\bar\Gamma),
\ee
where $\mu=\mu^\es$ is the stationary measure of the process $S^\es(\cdot)$.
\end{lemma}
\bp
We note that it is sufficient to show that for any closed set $F\subset\h$, we have
\be\label{9.71}
\hat\mu(F)\le\mu(F).
\ee

\medskip
{\it Step~1.}
Let us first show that for any bounded Lipschitz continuous function $\psi:\h\to\rr$, we have
\be\label{9.70}
\hat\mu(\psi)=(\psi, \mu).
\ee
Indeed, in view of \ef{9.61}, we have
\be\label{9.78'''}
\hat\mu(P_s\psi)=\hat\mu(\psi)\q\text{ for any }\psi\in b_0(\h).
\ee
In particular, this relation holds for any bounded Lipschitz function $\psi$ in $\h$. Moreover, it follows from inequality (1.3) in \cite{DM2014}, that for any such $\psi$, $P_s\psi$ converges to $(\psi,\mu)$ as $s\to\iin$ in the space $b_0(\h)$. Since $\hat\mu$ is continuous from $b_0(\h)$ to $\rr$, 
this implies 
$$
\hat\mu(\psi)=\hat\mu(P_s\psi)\to \hat\mu((\psi, \mu))=(\psi,\mu).
$$

\medskip
{\it Step~2.}
Now assume that inequality \ef{9.71} is not true, and let $F\subset \h$ closed and $\eta>0$ be such that
\be\label{9.72}
\hat\mu(F)\ge \mu(F)+\eta.
\ee
Let $\ch_F\le\psi_n\le 1$ be a sequence of Lipschitz continuous functions that converges pointwise to $\ch_F$ as $n\to \iin$. For example, one can take
$$
\psi_n(\uu)=\f{d_\h(\uu, F^c_{1/n})}{d_\h(\uu, F^c_{1/n})+d_\h(\uu,F)},
$$
where $F_r$ stands for the open $r$-neighborhood of $F$. Thanks to relation \ef{9.70}, inequality \ef{9.72} and monotonicity of $\hat\mu$, we have
$$
(\psi_n,\mu)=\hat\mu(\psi_n)\ge\hat\mu(\ch_F)=\hat\mu(F)\ge \mu(F)+\eta.
$$
However, this is impossible, since $(\psi_n, \mu)$ tends to $\mu(F)$ in view of the Lebesgue theorem on dominated convergence. The proof is complete.
\ep

\section{Proof of Proposition \ref{9.4}}
In view of Lemma \ref{9.67}, it is sufficient to prove that $\hat\mu=\hat\mu^\es$ satisfies
\be\label{9.78}
\exp(-(\vvv(\hat\uu_j)+\beta)/\es)\leq\hat\mu^\es(g_j)\leq \exp(-(\vvv(\hat\uu_j)-\beta)/\es).
\ee
\subsection{Upper bound} 
First note that by \eqref{9.59}-\ef{9.60}, for all $j\leq \ell+1$, we have
\be\label{8.2}
\tilde\mu^\es(\ch_{g_j})=\lm^\es(\e\int_0^{\tau_1^\es}\ch_{g_j}(S^{\es}(t;\cdot))\dd t)=\lm^\es(\ch_{\p g_j}(\cdot)\e\int_0^{\sigma_0^\es}\ch_{g_j}(S^{\es}(t;\cdot))\dd t).
\ee
In particular
\be\label{8.3}
\tilde\mu^\es(\ch_{g_j})\leq\lm^\es(\ch_{\p g_j}) \sup_{\vv\in\p g_j}\e_{\vv}\sigma_0^\es\leq\lm^\es(\ch_{\p g_j}) \sup_{\vv\in\tilde g}\e_{\vv}\sigma_0^\es.
\ee
On the other hand, we have
\be\label{8.4}
\tilde\mu^\es(1)\geq \tilde\mu^\es(\ch_{g'})\geq (1-\lm^\es(\ch_{\p g_{\ell+1}}))\min_{j\leq \ell} \inf_{\vv\in \bar g_j}\e\int_0^{\sigma_0^\es}\ch_{g_j}(S^{\es}(t;\vv))\dd t.
\ee
We recall that $g'$ stands for the set $g\backslash g_{\ell+1}$. We need the following result proved at the end of this section.
\begin{lemma}\label{8.21}
For any $\rho_*>0$ there exist $0<\rho'_1<\rho'_0<\rho_1<\rho_0<\rho_*$ such that for $\es<<1$ we have
\begin{align}
 \sup_{\vv\in\tilde g}\e_{\vv}\sigma_0^\es&\leq \exp(\beta/\es)\label{8.5},\\
\inf_{\vv\in \bar g_j}\e\int_0^{\sigma_0^\es}\ch_{g_j}(S^{\es}(t;\vv))\dd t&\geq \exp(-\beta/\es)\q\text{ for any } j\leq \ell.\label{8.6}
\end{align}
\end{lemma}
We first  note that $\vvv(\hat\uu_{\ell+1})$ is positive. Indeed, for any Lyapunov stable $\hat\uu_i$, the quantity $V(\hat\uu_i, \hat\uu_{\ell+1})$ is positive, and in view of \ef{8.45}, we have $\vvv(\hat\uu_{\ell+1})\ge \min V(\hat\uu_i, \hat\uu_{\ell+1})$, where the minimum is taken over $i\le \ell$ such that $\hat\uu_i$ is stable. Therefore, decreasing $\beta>0$ if needed, we may assume that $\vvv(\hat\uu_{\ell+1})\ge 2\beta$. In view of \ef{8.27'}, we have
\be\label{8.9}
\lm^\es(\ch_{\p g_{\ell+1}})\leq \exp (-(\vvv(\hat\uu_{\ell+1})-\beta)/\es)\leq\exp(-\beta/\es)\leq\f{1}{2}.
\ee
Combining this with inequalities \eqref{8.4} and \eqref{8.6}, we infer
\be\label{8.13}
\tilde\mu^\es(1)\geq \f{1}{2}\exp(-\beta/\es).
\ee
Further, using inequalities \ef{8.27'}, \eqref{8.3} and \eqref{8.5}, we get
\be\label{8.12}
\tilde\mu^\es(\ch_{g_j})\leq \exp(\beta/\es)\lm^\es(\ch_{\p g_j})\leq \exp(-(\vvv(\hat \uu_j)-2\beta)/\es).
\ee
Finally, combining this with \eqref{8.13}, we derive
$$
\hat\mu^\es(g_j)=\hat\mu^\es(\ch_{g_j})\leq 2\exp(\beta/\es) \exp(-(\vvv(\hat \uu_j)-2\beta)/\es)\leq \exp(-(\vvv(\hat \uu_j)-4\beta)/\es),
$$
where inequality holds for all $j\leq \ell+1$ and $\es<<1$.

\subsection{Lower bound}
We shall first establish the bound for $g_j$, $j\leq \ell$, and show that this implies the necessary bound for $g_{\ell+1}$. In view of \eqref{8.27'}, \eqref{8.2} and \eqref{8.6}, we have
\be\label{8.11}
\tilde\mu^\es(\ch_{g_j})\geq \exp(-(\vvv(\hat\uu_j)+2\beta)/\es).
\ee
On the other hand, by \eqref{8.12} we have
$$
\tilde\mu^\es(\ch_{g'})\leq \ell\exp(2\beta/\es).
$$
Note also that thanks to Lemmas \ref{9.47} and \ref{9.67}, we have that
$$
\hat\mu^\es(g')\ge\mu^\es(g')\geq \exp(-\beta/\es)\q\text{ for } \es<<1.
$$
Indeed, by definition, $g'$ represents the $\rho_1$-neighborhood of the set $\E=\{\hat\uu_1,\ldots, \hat\uu_\ell\}$.
It follows from the last two inequalities that
$$
\tilde\mu^\es(1)=\tilde\mu^\es(\ch_{g'})/\hat\mu^\es(g')\leq \ell\exp(3\beta/\es)\leq \exp(4\beta/\es)\q\text{ for } \es<<1.
$$
Finally, combing this inequality with \eqref{8.11}, we infer
\be\label{8.14}
\hat\mu^\es(g_j)\geq \exp(-(\vvv(\hat\uu_j)+6\beta)/\es),
\ee
where inequality holds for all $j\leq \ell$ and $\es<<1$.

\medskip
 We now show that inequality \eqref{8.14} implies 
\be\label{8.15}
\hat\mu^\es(g_{\ell+1})\geq \exp(-(\vvv(\hat\uu_{\ell+1})+8\beta)/\es)\q\text{ for }\es<<1.
\ee
We assume that $\vvv(\hat \uu_{\ell+1})<\iin$. First note that in view of \eqref{8.45}, we have 
\be\label{8.16}
\vvv(\hat \uu_{\ell+1})=\min_{i\leq \ell}[W_\ell(\hat \uu_i)+V(\hat\uu_i, \hat\uu_{\ell+1})]-\min_{i\leq \ell}W_\ell(\hat\uu_i).
\ee
Let us find $m\leq \ell$ such that
\be\label{8.17}
W_\ell(\hat \uu_m)+V(\hat\uu_m, \hat\uu_{\ell+1})=\min_{i\leq \ell}[W_\ell(\hat \uu_i)+V(\hat\uu_i, \hat\uu_{\ell+1})].
\ee

By definition of $V$, there is a finite time $T>0$ and an action $\ph\in L^2(0,T;H_\vartheta)$ such that
\be\label{5.3}
J_T(\ph)\leq V(\hat\uu_m, \hat\uu_{\ell+1})+\beta,\q |S^\ph(T;\hat\uu_m)-\hat\uu_{\ell+1}|_\h<\rho'_1/4.
\ee
Since the operator $S^\ph$ continuously depends on the initial point, there is $\kp>0$ such that 
$$
|S^\ph(T;\uu)-\hat\uu_{\ell+1}|_\h<\rho'_1/2,
$$
provided $|\uu-\hat\uu_m|_\h\leq\kp$. It follows from this inequality, relation \ef{9.78'''} and monotonicity of $\hat\mu^\es$ that
\begin{align}
\hat\mu^\es(g_{\ell+1})&=\hat\mu^\es(\ch_{g_{\ell+1}})=\hat\mu^\es(P_T\ch_{g_{\ell+1}})\ge \hat\mu^\es(\ch_{\bar B(\hat\uu_m, \kp)} P_T\ch_{g_{\ell+1}})\notag\\
&=\hat\mu^\es(\ch_{\bar B(\hat\uu_m, \kp)}(\cdot)\pp (|S^{\es}(T;\cdot)-\hat\uu_{\ell+1}|_\h<\rho'_1))\notag\\
&\geq  \hat\mu^\es(\ch_{\bar B(\hat\uu_m, \kp)}(\cdot)\pp(|S^{\es}(T;\cdot)-S^\ph(T;\cdot)|_\h<\rho'_1/2)\notag\\
&\ge\inf_{\uu\in \bar B(\hat\uu_m, \kp)}\pp(|S^{\es}(T;\uu)-S^\ph(T;\uu)|_\h<\rho'_1/2)\hat\mu^\es(B(\hat\uu_m, \kp))\label{5.6}.
\end{align}
By Theorem \ref{6.1} (applied to the set $B= \bar B(\hat\uu_m, \kp))$, we can find $\es_1=\es_1(\hat\uu_{\ell+1},\kp,\rho'_1,T)>0$ such that for all $\uu\in\bar B(\hat\uu_m, \kp)$, we have
$$
\pp(|S^{\es}(T;\uu)-S^\ph(T;\uu)|_\h<\rho'_1/2)\geq\exp(-(J_T(\ph)+\beta)/\es) \q\text {for }\es\leq \es_1.
$$
It follows that
\be\label{9.17}
\hat\mu^\es(g_{\ell+1})\geq\exp(-(J_T(\ph)+\beta)/\es) \hat\mu^\es(B(\hat\uu_m, \kp)) .
\ee
Combining this with first inequality of \eqref{5.3} and \eqref{5.6}, we get
$$
\hat\mu^\es(g_{\ell+1})\geq \hat\mu^\es  (B(\hat\uu_m, \kp))\exp(-(V(\hat\uu_m,\hat\uu_{\ell+1})+2\beta)/\es).
$$
Further, using this inequality and \eqref{8.14} \,\footnote{\, Recall that this inequality is true for any neighborhood of $\hat\uu_j,$ for $j\leq \ell$.} with $j=m$, we infer
$$
\hat\mu^\es(g_{\ell+1})\geq \exp(-(\vvv(\hat\uu_m)+V(\hat\uu_m,\hat\uu_{\ell+1})+8\beta)/\es).
$$
To complete the proof, it remains to note that thanks to \ef{8.17}, we have
\begin{align*}
\vvv(\hat\uu_m)+V(\hat\uu_m,\hat\uu_{\ell+1})&=W_\ell(\hat\uu_m)+V(\hat\uu_m,\hat\uu_{\ell+1})-\min_{i\leq \ell}W_\ell(\hat\uu_i)\\
&=\min_{i\leq \ell}[W_\ell(\hat \uu_i)+V(\hat\uu_i, \hat\uu_{\ell+1})]-\min_{i\leq \ell}W_\ell(\hat\uu_i)=\vvv(\hat\uu_{\ell+1}).\end{align*}

\subsection{Proof of Lemma \ref{8.21}} 
{\it Step~1: Derivation of \eqref{8.5}.} Let us fix $i\leq \ell+1$ and let $j\leq \ell$ be an integer different from $i$. Using Proposition \ref{1.88}, it is not difficult to show that $V(\hat\uu_i,\hat\uu_j)<\iin$ (see the derivation of \ef{9.12}). Therefore, if we denote by $d$ the distance between $\hat\uu_i$ and $\hat\uu_j$, we can find a finite time $T>0$ and an action $\ph$ defined on $[0, T]$ such that $S^\ph(T, \hat\uu_i)\in B(\hat\uu_j, d/2)$ and
$$
\int_0^T|\ph(s)|_{H_\vartheta}^2\dd s\leq V(\hat\uu_i,\hat\uu_j)+1.
$$
Let us find $t_*>0$ so small that 
\be\label{8.22}
\int_0^{t_*}|\ph(s)|_{H_\vartheta}^2\dd s\leq \beta\q\text{ and }\q \tilde \uu_i\neq\hat\uu_i,
\ee
where we set $\tilde \uu_i=S^\ph(t_*,\hat\uu_i)$. Further, let $\tilde\rho>0$ be so small, that $|\hat\uu_i-\tilde\uu_i|_\h\geq 4\tilde\rho$. And finally, let $0<\rho_*<\tilde\rho$ be such that
\be\label{8.23}
S^\ph(t_*;\uu)\in B(\tilde\uu_i,\tilde\rho)\q\text{ for any }\uu\in B(\hat\uu_i, \rho_*).
\ee
Now notice that if the trajectory $S^\es(t;\vv)$ issued from $\vv\in B(\hat\uu_i,\rho_*)$ is in the $\tilde\rho$-neighborhood of $S^\ph(t,\vv)$ in $C(0,t_*;\h)$ distance, then $S^\es(t_*,\vv)\notin B(\hat\uu_i,\rho_*)$. Moreover, if we denote by $\ph_*$ the restriction of $\ph$ on $[0, t_*]$, then for any $\vv$ in $B(\hat\uu_i,\rho_*)$, we have
$$
I_{t_*}(S^{\ph_*}(\cdot;\vv))=J_{t_*}({\ph_*})=\int_0^{t_*}|\ph(s)|_{H_\vartheta}^2\dd s\leq \beta.
$$
Applying Theorem \ref{6.1}, we derive
$$
\sup_{\vv\in B(\hat\uu_i,\rho_*)}\pp_\vv(\tau^\es_{\text{exit}}>t_*)\leq 1-\exp(-2\beta/\es),
$$
where we denote by $\tau^\es_{\text{exit}}(\vv)$ the time of the first exit of the process $S^\es(\cdot;\vv)$ from $B(\hat\uu_i,\rho_*)$. Now using the Markov property, we infer
\be\label{9.27}
\sup_{\vv\in B(\hat\uu_i,\rho_*)}\pp_\vv(\tau^\es_{\text{exit}}>n\,t_*)\leq (1-\exp(-2\beta/\es))^n,
\ee
which implies \eqref{8.5}.

 \bigskip
{\it Step~2: Proof of \eqref{8.6}.} Let us fix any stationary point $\hat\uu_i$ and let $\rho_*>0$. Given any $0<\rho_1<\rho_*$ let us find  $0<\rho_2<\rho_1$ such that for any $\vv\in \bar B(\hat\uu_i, \rho_2)$ we have $S(t;\vv)\in B(\hat\uu_i,\rho_1/2)$ for all $t\in [0,1]$. We assume that $\rho_*>0$ is so small that the conclusion of Lemma \ref{8.28} holds. We use the following construction: given any point $\vv\in \bar B(\hat\uu_i,\rho_1)$ we denote by $\tilde\ph_\vv$ the action defined on the time interval $[0, T+1]$ that coincides with $\ph_\vv$ on $[0,T]$ and vanishes on $[T, T+1]$. Then we have
\be\label{8.26}
I_{T+1}(S^{\tilde\ph_\vv}(\cdot;\vv))=J_{T+1}(\tilde\ph_\vv)=J_T(\ph_\vv)\leq\beta\q\text{ for any }\vv\in \bar B(\hat\uu_i,\rho_1).
\ee
Now let us take any $\rho_0\in (\rho_1, \rho_*)$ and let $\De<\min\{(\rho_0-\rho_1),\rho_2/2)\}$ be any positive number. Then by construction we have the following: if the trajectory $S^\es(t;\vv)$ is in the $\De$-neighborhood of $S^{\tilde\ph_\vv}(t;\vv)$ in the $C(0, T+1;\h)$ distance, then it remains in $\tilde g_i\equiv B(\hat\uu_i,\rho_0)$ for all $t\in [0,T+1]$ and moreover, it belongs to $g_i\equiv B(\hat\uu_i,\rho_1)$ for all $t\in[T,T+1]$. Therefore, we have
$$
\inf_{\vv\in \bar g_i}\e\int_0^{\sigma_0^\es}\ch_{g_j}(S^{\es}(t;\vv))\dd t\geq \pp(A), 
$$
where
$$
A=\{\om\in\Omega:\sup_{\vv\in \bar g_i} d_{C(0,T+1;\h)}(S^\es(\cdot;\vv), S^{\tilde\ph_\vv}(\cdot;\vv))<\De\}.
$$
Combining this with inequality \eqref{8.26} and Theorem \ref{6.1}, we arrive at \eqref{8.6}. The proof of Lemma \ref{8.21} is complete.

\section{A priori upper bound}\label{1.77}
In this section we establish Proposition \ref{9.18}. To simplify presentation, we first outline the main ideas.

\subsection{Scheme of the proof of Proposition \ref{9.18}}
\nt

\bigskip
{\it Compactness of level sets.} Let us suppose that we can prove the following: there is a constant $s\in (0, 1/2)$ such that
\be\label{1.75} 
|\uu_*|_{H^{s+1}(D)\times H^s(D)}\leq C(M)\q \text{ for any } \uu_*\in\{V_\aaa\leq M\},
\ee
where $H^s$ stands for the scale of Hilbert spaces associated with $-\de$.
Then the compactness of the embedding $H^{s+1}(D)\times H^s(D)\hookrightarrow\h$ implies that the level sets of $V_\aaa$ are relatively compact in $\h$. Thus, if inequality \eqref{1.75} is true, we only need to prove that the level sets  of $V_\aaa$ are closed.
Let $\uu_{*}^j$ be a sequence in $\{V_\aaa\leq M\}$ that converges to $\uu_*$ in $\h$. We need to show that $V_\aaa(\uu_*)\leq M$. By definition of $V_\aaa$, we have to show that for any positive constants $\eta$ and $\eta'$ there is an initial point $\uu_0\in\aaa$, a finite time $T=T_\eta>0$, and an action $\ph$ such that
\be\label{1.42}
J_{T}(\ph)\leq M+\eta'\q \text{ and }\q |S^{\ph}(T;\uu_{0})-\uu_*|_\h\leq \eta.
\ee
Let us fix $j$ so large that
\be\label{1.43}
|\uu_{*}^j-\uu_*|_\h\leq \eta/2.
\ee

Since $V_\aaa(\uu_*^j)\leq M$, there is a point $\uu_0\in\aaa$, a time $T=T_\eta>0$ and an action $\ph$ such that
$$
J_{T}(\ph)\leq M+\eta'\q \text{ and }\q |S^{\ph}(T;\uu_{0})-\uu_*^j|_\h\leq \eta/2.
$$
Combining this with inequality \eqref{1.43}, we derive \eqref{1.42}. The proof of inequality \eqref{1.75} relies on some estimates of the limiting equation and is carried out in the appendix.

\bigskip
{\it The bound \ef{9.19}}.
Due to the equivalence of \eqref{7.1} and \eqref{7.2}, we need to show that for any positive numbers $\De, \De'$ and $M$ there is $\es_*>0$ such that 
\be\label{1.99}
\mu^\es(\uu\in\h: d_{\h}(\uu,\{V_\aaa\leq M\})\geq\De)\leq\exp(-(M-\De')/\es)\q\text{ for } \es\leq\es_*.
\ee
From now on, we shall suppose that the constants $\De, \De'$ and $M$ are fixed.

\bigskip
{\it Reduction.} To prove \ef{1.99}, we first show that there is $\eta>0$ such that
\be\label{1.69}
\{\uu(t): \uu(0)\in \aaa_\eta, I_t(\uu_\cdot)\leq M-\De'\}\subset K_{\De/2}(M), \q t> 0,
\ee
where $\aaa_\eta$ stands for the open $\eta$-neighborhood of the set $\aaa$ and $K_\De(M)$ is  the open $\De$-neighborhood of the level set $\{V_\aaa\leq M\}$.
We then show that there is $R>0$ such that 
\be
\pP_1:=\mu^\es(B_R^c)\leq \exp(-M/\es)\q\text{ for }\es\le1.\label{1.30}
\ee
Once the constants $\eta$ and $R$ are fixed, we prove that there is $T_*>0$ such that
\be\label{2.16}
a:=\inf\{I_{T_*}(\uu_\cdot); \, \uu_\cdot\in C(0,T_*;\h), \, \uu(0)\in B_R, \,\uu(T_*)\in \aaa_\eta^c\}>0.
\ee

Taking inclusion \ef{1.69} and inequalities \ef{1.30}-\ef{2.16} for granted, let us show how to derive \ef{1.99}.

\bigskip
{\it Auxiliary construction.}
For any $n\geq 1$ introduce the set 
$$
E_n=\{\uu_\cdot\in C(0, nT_*;\h): \uu(0)\in B_R; \,\,\uu(kT_*)\in B_R\cap \aaa_\eta^c, \, k=1,\ldots, n\}.
$$
Let us mention that the idea of this construction is inspired by \cite{sowers-1992b} and $E_n$ is a modification of a set introduced by Sowers in that paper.
We claim that inequality \ef{1.69} and the structure of the set $E_n$ imply that for $n$ sufficiently large we have
\be\label{1.27}
\pP_2:=\sup_{\vv\in B_R}\pp(S^{\es}(\cdot;\vv)\in E_n)\leq\exp(-M/\es).
\ee

Indeed, in view of Theorem \ref{6.1}, to this end, it is sufficient to show that
\be\label{1.70}
\inf\{I_{nT_*}(\uu_\cdot); \, \uu_\cdot\in E_{n,T_*}\}> M.
\ee

We show that this inequality holds for any $n>(M+1)/a$. To see this, let us fix an integer $n$ satisfying this inequality and suppose that \eqref{1.70} is not true. Then there is an initial point $\vv\in B_R$ and an action $\ph$ defined on the interval $[0, nT_*]$ such that

$$
\f{1}{2}\int_{0}^{nT_*}|\ph(s)|_{H_\vartheta}^2\dd s<M+1
$$
and $S^\ph(jT_*;\vv)\in B_R\cap \aaa_\eta^c$ for all $j\in\{1,\ldots, n\}$. It follows from this inequality that there is $j\in\{1,\ldots, n-1\}$ such that
$$
\f{1}{2}\int_{jT_*}^{(j+1)T_*}|\ph(s)|_{H_\vartheta}^2\dd s<(M+1)/n<a.
$$
Therefore, the restriction of $\ph$ on the interval $[jT_*, (j+1)T_*]$ is an action whose energy is smaller than $a$ and that steers the point $\vv_1=S^\ph(jT_*;\vv)\in B_R$ to $\vv_2=S^\ph((j+1)T_*;\vv)\notin \aaa_\eta$. However, this is in contradiction with \eqref{2.16}. Inequality \ef{1.27} is thus established.

\bigskip
{\it Completion of the proof.} 
We now show that 
\be
\pP_3:=\int_{B_R}\pp(S^{\es}(t_*;\vv)\notin K_{\De}(M),\,S^{\es}(\cdot;\vv)\notin E_n)\mu^\es(d\vv)\leq \exp(-(M-2\De')/\es), \label{7.12}
\ee
where we set $t_*=(n+1)T_*$. Once this is proved, we will get \ef{1.99}. Indeed, by definition of the set $K_\De(M)$ and stationarity of $\mu^\es$, we have
\begin{align*}
\mu^\es(\uu\in\h: d_{\h}(\uu,\{V_\aaa\le M\})\geq\De)&=\mu^\es(\uu\in\h: \uu\notin K_\De(M))\\
&=\int_{\h}\pp(S^{\es}((n+1)T_*;\vv)\notin K_\De(M))\mu^\es(d\vv)\\
&\leq \pP_1+\pP_2+\pP_3.
\end{align*}
Combining inequalities \ef{1.30}, \eqref{1.27} and \eqref{7.12} we arrive at \ef{1.99}, where $\De'$
should be replaced by $3\De'$.

\medskip
To prove inequality \ef{7.12}, we first note that
\begin{align*}
\pP_3&\leq\int_{B_R}\pp\left(\bigcup_{j=1}^n \{S^{\es}(t_*;\vv)\notin K_{\De}(M),\,S^{\es}(jT_*;\vv)\in B_R^c\cup\aaa_\eta\}\right)\mu^\es(d\vv)\notag\\
&\leq \sum_{j=1}^n \int_{B_R}\pp(S^{\es}(t_*;\vv)\notin K_{\De}(M),\,S^{\es}(jT_*;\vv)\in B_R^c )\mu^\es(d\vv)\notag\\
&\q +\sum_{j=1}^n \int_{B_R}\pp(S^{\es}(t_*;\vv)\notin K_{\De}(M),\,S^{\es}(jT_*;\vv)\in\aaa_\eta)\mu^\es(d\vv):=\pP_3'+\pP_3''.
\end{align*}
By the stationarity of $\mu^\es$, the first term in the last inequality satisfies
\be\label{1.84}
\pP_3'\leq\sum_{j=1}^n\int_{\h}\pp(S^{\es}(jT_*;\vv)\in B_R^c )\mu^\es(d\vv)=n\,\pP_1\le n\exp(-M/\es).
\ee
Moreover, using inclusion \ef{1.69} and following \cite{sowers-1992b}, it is not difficult to prove (see Section \ref{9.40}) that
\be\label{9.8}
\pP_3''\leq\exp(-(M-\De')/\es)\q\text{ for }\es<<1
\ee
and thus to derive \ef{7.12}.

\bigskip
{\it Idea of the proof of \ef{1.69}-\ef{2.16}.} 
The proof of inequality \ef{1.30} is rather standard and relies on exponential estimates for solutions and a simple application of the Fatou lemma. 
The derivation of inclusion \eqref{1.69} is the most involved part in the proof. Without going into technicalities, we shall describe here the main ideas. We note that inclusion \eqref{1.69} clearly holds for $\eta=0$. Indeed, in this case $\uu(0)\in\aaa$, and since we have $I_t(\uu_\cdot)\leq M-\De'\leq M$, the point $\uu(t)$ is reached from the set $\aaa$ with action $\ph$ such that $J_t(\ph)\leq M$. It follows from the definition of $V_\aaa$ that $V_\aaa(\uu(t))\leq M$, so $d_\h(\uu(t),\{V_\aaa\le M\})=0$, and therefore we have \eqref{1.69}. So what we need to show is that if the initial point is sufficiently close to the attractor, then the inclusion \eqref{1.69} still holds. To prove this, we show that there is a flow $\hat \uu(t)$ issued from $\hat \uu(0)\in\aaa$ that remains in the $\De/2$-neighborhood of $\uu(t)$, and whose action function is $\De'$-close to that of $\uu(t)$. Once this is proved, the inclusion \eqref{1.69} will follow from the fact that $\hat \uu(t)\in \{V_\aaa\le M\}$, since it is reached from the set $\aaa$ at finite time $t$ with action whose energy is smaller than $M$. The construction of the flow $\hat \uu(t)$ relies on Proposition \ref{1.88}.

\medskip
As for the proof of inequality \eqref{2.16}, we first note that this inequality means the following: if we wait for sufficiently long time, then the energy needed to reach a  point outside $\eta$-neighborhood of the global attractor $\aaa$ is  positive uniformly with respect to the initial point in the ball $B_R$. The intuition behind this is that after sufficiently long time, the image of $B_R$ will be near the attractor $\aaa$, and the energy needed to steer the point close to the set $\aaa$ (say $\eta/2$-close) to a point outside its $\eta$-neighborhood, is positive. Let us finally mention that the fact that $V_\aaa$ vanishes only on the set $\aaa$ follows immediately from the definition of $V_\aaa$ and inequality \eqref{2.16}.

\subsection{Proof of inclusion \eqref{1.69}}\label{9.13}
\nt

{\it Step~1}. Let us suppose that \eqref{1.69} does not hold. Then there exist two sequences of positive numbers numbers $\eta_j\to 0$ and $T_j $, a sequence of initial points $(\uu^j_0)\subset\aaa_{\eta_j}$, and of action functions $(\ph^j)$ with $J_{T_j}(\ph^j)\leq M-\De'/2$, such that for each $j\geq 1$ the flow $\uu^j(t)=S^{\ph^j}(t;\uu_0^j)$ satisfies the inequality
\be\label{1.6.4}
d_\h(\uu^j(T_j),\{V_\aaa\le M\})\geq\De/2.
\ee

Let us also note that in view \eqref{1.38}, there is a positive constant $\mmm$ depending only on $\|h\|$ and $M$ such that for all $j\geq 1$ we have
\be\label{1.67}
\sup_{[0,T_j]} |\uu^j(t)|_\h\leq\mmm.
\ee

\bigskip
{\it Step~2}. For each $j\geq 1$, let us find $\vv^j_0\in\aaa$ such that $|\vv^j_0-\uu^j_0|_\h\leq\eta_j$ and introduce the intermediate flow $\vv^j(t)=[v(t),\dt v(t)]$ defined on the interval $[0,T_j]$ that solves 
\be\label{1.7}
\p_t^2 v+\gamma\p_t v-\de v+f(v)=h(x)+\ph^j+P_N[f(v)-f(u)], \q [v(0),\dt v(0)]=\vv^j_0,
\ee
where $N\geq 1$ is an integer to be chosen later and $u$ is the first component of $\uu^j(t)$. In view of Proposition \ref{1.88}, there is $N$ depending only on $\mmm$ such that for all $j\ge 1$ we have
\be\label{1.12}
|\vv^j(t)-\uu^j(t)|^2_\h\leq e^{-\al t}|\vv^j_0-\uu^j_0|_\h^2 \q\q\text{for all } t\in[0,T_j].
\ee

\bigskip
{\it Step~3}. Now let us fix $N=N(\mmm)$ such that we have \eqref{1.12}, and let us show that for $j>>1$  we have  
\be\label{1.16}
J_{T_j}(\hat\ph^j)\leq M, 
\ee
where we set 
\be\label{1.74}
\hat \ph^j=\ph^j+P_N[f(v)-f(u)].
\ee
By definition of $J$, we have
\be\label{1.14}
J_{T_j}(\hat\ph^j)=\f{1}{2}\int_0^{T_j}|\ph^j(s)+P_N[f(v(s))-f(u(s))]|_{H_\vartheta}^2\dd s.
\ee
We first note that
$$
|a+b|_{H_\vartheta}^2\leq p\,|a|_{H_\vartheta}^2+\f{p}{p-1}\,|b|_{H_\vartheta}^2,
$$
where $p>1$ is a constant to be chosen later. Therefore
\begin{align}
J_{T_j}(\hat\ph^j)&\leq \f{p}{2}\int_0^{T_j}|\ph^j(s)|_{H_\vartheta}^2\dd s+\f{p}{2(p-1)}\int_0^{T_j }|P_N[f(v(s))-f(u(s))]|_{H_\vartheta}^2\dd s\notag\\
&\leq p\,J_{T_j}(\ph^j)+C(N)\,\f{p}{p-1}\int_0^{T_j }|f(v(s))-f(u(s))|_{L^1}^2\dd s\label{1.15}.
\end{align}
By the H\"older and Sobolev inequalities, we have
$$
|f(v)-f(u)|_{L^1}^2\leq C_1\|u-v\|_1^2(\|u\|_1^2+\|v\|_1^2+1).
$$
Combining this with inequalities \eqref{1.67} and \eqref{1.12} we see that
$$
\int_0^{T_j }|f(v(s))-f(u(s))|_{L^1}^2\dd s\leq C_2 (\mmm^2+1)\int_0^{T_j}e^{-\al s}|\vv^j_0-\uu^j_0|_{\h}^2\dd s\leq C_3(\mmm^2+1)\eta_j^2.
$$
It follows from this inequality and \eqref{1.15}, and the fact that $N$ depends only on $\mmm$, that
$$
J_{T_j}(\hat\ph^j)\leq p\,J_{T_j}(\ph^j)+C(\mmm)\,\f{p}{p-1}\,\eta_j^2.
$$
Let us take 
$$
p=\f{M-\De'/4}{M-\De'/2}.
$$
Since $J_{T_j}(\ph^j)\leq M-\De'/2$, for $j$ large enough, we have $J_{T_j}(\hat\ph^j)\leq M$.

\bigskip
{\it Step~4}. We claim that for $j>>1$, we have
\be\label{1.17}
\vv^j(T_j)\in \{V_\aaa\le M\}.
\ee
Indeed, note that in view of \eqref{1.7}, we have $\vv^j(\cdot)=S^{\hat \ph^j}(\cdot; \vv_0^j)$. 
So the point $\vv^j(T_j)$ is reached from $\vv_0^j\in\aaa$ with action function $\hat \ph^j$ at finite time $T_j$. It follows from the definition of $V_\aaa$ and inequality \eqref{1.16} that $V_\aaa(\vv^j(T_j))\leq M$ for $j>>1$.

\bigskip
{\it Step~5}.
In view of inequality \eqref{1.12}, we have
$$
|\uu^j(T_j)-\vv ^j(T_j)|_\h\leq \eta_j \leq\De/4,
$$
provided $j\geq 1$ is large enough. Combining this with \eqref{1.17}, we see that
$$
d_\h(\uu^j(T_j), \{V_\aaa\le M\})\leq\De/4,
$$
which is in contradiction with \eqref{1.6.4}. The proof of  inclusion \eqref{1.69} is complete.

\subsection{Proof of inequality \eqref{2.16}}
Let us assume that \eqref{2.16} is not true, so for any $j\geq 1$ we have
$$
\inf\{I_{j}(\uu_\cdot); \, \uu_\cdot\in C(0,j;\h), \, \uu(0)\in B_R, \,\uu(j)\in \aaa_\eta^c\}=0.
$$
Then for each $j\geq 1$ there is an initial point $\uu^j_0\in B_R$ and an action $\ph^j$ defined on the interval $[0,j]$ with energy smaller than $e^{-j^2}$ such that the flow $\uu^j(t)=S^{\ph_j}(t;\uu^j_0)$ satisfies 
\be\label{2.4.4}
\uu^j(j)\notin\aaa_\eta.
\ee

For each $j\geq 1$, let $\vv^j(t)=S(t)\uu^j_0$. Using a priori bounds of the NLW equation it is not difficult to show (see Section \ref{9.29} for the proof) that
\be\label{9.41}
|\vv^j(t)-\uu^j(t)|_\h^2\leq C\int_0^t\|\ph^j(s)\|^2 \exp(C s)\dd s\q\text{ for } t\in [0, j].
\ee
Taking $t=j$ in this inequality and using $J_j(\ph^j)\le e^{-j^2}$, we get
\begin{align}
|\vv^j(j)-\uu^j(j)|_\h^2
&\leq C\exp(C j)\int_0^j \|\ph^j(s)\|^2ds\notag\\
&\leq C_3\exp(C j)J_j(\ph^j)\leq C_3\exp(-j^2+C j)\leq \eta^2/4\label{9.42},
\end{align}
provided $j$ is sufficiently large.
Combining this with \ef{2.4.4}, we see that for $j>>1$, we have
\be\label{9.30}
S(t)\uu^j_0=\vv^j(j)\notin\aaa_{\eta/2}.
\ee
Since $\aaa$ is the global attractor of the semigroup $S(t)$, we have
$$
\sup_{\uu_0\in B_R}d_\h(S(t)\uu_0, \aaa)\to 0\q\q\text{ as } t\to\iin.
$$
This is clearly in contradiction with \ef{9.30}. Inequality \ef{2.16} is established.
\subsection{Derivation of \ef{9.8}}\label{9.40}
We follow the argument presented in \cite{sowers-1992b}. Let us fix any $\vv\in B_R$ and $j\leq n$, and denote by $A$ the event $\{S^{\es}(t_*;\vv)\notin K_{\De}(M),\,S^{\es}(jT_*;\vv)\in\aaa_\eta\}$. Then, by the Markov property, we have
\begin{align*}
\pp(A)=\e[\e(\ch_A)|\fff^\es_{jT_*}]&=\e[\ch_{\bar \vv\in\aaa_\eta}\cdot \e(\ch_{S^{\es}(t_*-jT_*;\bar \vv)\notin K_{\De}(M)})]\notag\\
&\leq\sup_{\vv_0\in\aaa_\eta}\pp(S^{\es}(t_*-jT_*; \vv_0)\notin K_{\De}(M)),
\end{align*}
where $\fff^\es_t$ is the filtration corresponding to the process $S^{\es}(t;\vv)$ and we set $\bar \vv=S^{\es}(jT_*;\vv)$. It follows that
\be\label{3.2}
\pP_3''\leq \sum_{j=1}^n\sup_{\vv_0\in\aaa_\eta}\pp(S^{\es}(jT_*; \vv_0)\notin K_{\De}(M)).
\ee

\medskip
\nt
For any $l>0, M_1>0$ and $\vv\in\h$, introduce the level set
$$
K_{\vv,l}(M_1)=\{\uu(\cdot)\in C(0,l;\h); \uu(0)=\vv,\, I_l(\uu(\cdot))\leq M_1\}.
$$
Let us show that for any $l>0$ and $\vv_0\in\aaa_\eta$ we have
\be\label{3.3}
\{\om: S^{\es}(l;\vv_0)\notin K_{\De}(M)\}\subset \{\om: d_{C(0,l;\h)}(S^{\es}(\cdot;\vv_0), K_{\vv_0,l}(M-\De'))\geq \De/2\}.
\ee
Indeed, let us fix any $\om$ such that $S^{\es}(l;\vv_0,\om)\notin K_{\De}(M)$, and let $\uu_\cdot$ be any function that belongs to $K_{\vv_0,l}(M-\De')$. Then in view of inclusion \eqref{1.69} we have $\uu(l)\in K_{\De/2}(M)$, so that
$$
d_{C(0,l;\h)}(S^{\es}(\cdot;\vv_0,\om), \uu_\cdot)\geq |S^{\es}(l;\vv_0,\om)-z(l)|_\h\geq \De/2.
$$
Since $\uu_\cdot\in K_{\vv_0,l}(M-\De')$ was arbitrary, we conclude that inclusion \eqref{3.3} holds.
It follows from Theorem \ref{6.1} (applied to the time interval $[0, l]$ and the set $B=\aaa_\eta$) that there is $\es(l)=\es(l,\De,M,\eta)>0$ such that
\be\label{3.4}
\sup_{\vv_0\in\aaa_\eta}\pp(d_{C(0,l;\h)}(S^{\es}(\cdot;\vv_0), K_{\vv_0,l}(M-\De'))\geq \De/2)\leq \exp(-(M-2\De')/\es),\,\, \es\leq \es(l).
\ee
Let $\es_1(\De,M,T_*,n,\eta)=\min\{\es(T_*),\ldots \es(nT_*)\}$. Then in view of inequalities \eqref{3.2} and \eqref{3.4} we have
$$
\pP_3''\leq n\exp(-(M-2\De')/\es)\leq\exp(-(M-3\De')/\es) \q\text{ for }\es\leq\es_1.
$$
Inequality \ef{9.8} is established with $\De'$ replaced by $3\De'$.

\subsection{Proof of inequality \eqref{1.30}}
Let us show that for $R=R(M)$ sufficiently large and $\es_*=\es_*(M)>0$ small, we have \eqref{1.30}.  To this end, let us first show that the stationary solutions $\vv(t)$ of equation \eqref{0.1.4} satisfy 
\be\label{1.23}
\e\exp(\kp\,\ees (\vv(t))\leq Q(\es\,\BBB,\|h\|)\leq Q(\BBB,\|h\|),
\ee
for any $\kp\leq (\es\,\BBB)^{-1}\al/2$, where $Q$ and $\BBB$ are the quantities entering Proposition \ref{1.96}. Replacing $b_j$ by $b_j/\sqrt{\es}$, we see that it is sufficient to prove this inequality for $\es=1$.
Note that we cannot pass directly to the limit $t\to\iin$ in inequality \eqref{1.97}, since we first need to guarantee that $\e\exp(\kp\ees (\vv(0))$ is finite. This can be done by a simple application of the Fatou lemma. Indeed, for any $N\geq 1$, let $\psi_N(\uu)$ be the function that is equal to $\exp(\kp\ees (\uu))$ if $\ees(\uu)\leq N$, and to $\exp(\kp N)$ otherwise. Let us denote by $\mu$ the law of $\vv(t)$, and let $l$ be any positive number. Then using the stationarity of $\mu$ and inequality $\psi_N(\uu)\leq\exp(\kp\ees(\uu))$ we see that
\begin{align}
\int_\h \psi_N(\uu)\mu(d\uu)&=\int_\h\int_\h \psi_N(\uu')P_t(\uu, d\uu')\mu(d\uu)\notag\\
&\leq \int_{\ees(\uu)\leq l}\int_\h \exp(\kp\,\ees (\uu')) P_t(\uu, d\uu')\mu(d\uu)+\exp(\kp N)\mu(\ees(\uu)>l)\notag\\
&=\iI_1+\iI_2\label{1.80},
\end{align}
where $P_t$ stands for the transition function of the Markov process. Note that
$$
\iI_1\leq \sup_{\ees(\uu_0)\leq l}\e\exp(\kp\,\ees(\uu(t;\uu_0))\leq \exp(\kp \,l-\al\,t)+Q(\BBB,\|h\|),
$$
where $\uu(t;\uu_0)$ stands for the trajectory of \eqref{0.1.4} with $\es=1$ issued from $\uu_0$, and we used inequality \eqref{1.97}. Combining this with \eqref{1.80}, we obtain
$$
\int_\h \psi_N(\uu)\mu(d\uu)\leq \exp(\kp \, l-\al\,t)+\exp(\kp N)\mu(\ees(\uu)>l)+Q(\BBB,\|h\|).
$$
Passing to the limits $t\to\iin$ and then $l\to\iin$, and using the equivalence $\ees(\uu)\to\iin\Leftrightarrow |\uu|_\h\to\iin$, we get
$$
\int_\h \psi_N(\uu)\mu(d\uu)\leq Q(\BBB,\|h\|).
$$
Finally, letting $N$ go to infinity, and using Fatou's lemma, we derive \eqref{1.23}.

\medskip
We are now ready to establish \eqref{1.30}. Indeed, it follows from inequalities \eqref{1.56} that $\ees(\uu)\geq \f{1}{2}|\uu|_\h^2-2\,C> R^2/4$, provided $\uu\in B_R^c$ and $R^2\geq 8\,C$. Therefore, by the Chebyshev inequality, we have that
$$
\mu^\es(B_R^c)\leq \mu^\es(\ees(\uu)>R^2/4)\leq\exp(-\kp\, R^2/4)\int_{\h}\exp(\kp\,\ees(\uu))\,\mu^\es(d\uu).
$$
Now taking $\kp=(\es\,\BBB)^{-1}\,\al/2$ in this inequality, using \eqref{1.23} and supposing that $R$ is so large that $R\geq 16\,\BBB M\, \al^{-1}$, we obtain
$$
\mu^\es(B_R^c)\leq Q(\BBB,\|h\|) \exp(-2M/\es)\leq \exp(-M/\es),
$$
provided $\es>0$ is small. Inequality \eqref{1.30} is thus established.

\subsection{Lower bound with function $V_\aaa$ in the case of a unique equilibrium}\label{9.20}
The goal of this section is to show that in the case when equation \ef{1.51} possesses a unique equilibrium, the function $V_\aaa$ given by \ef{8.49}-\ef{9.6} provides also a lower bound for $(\mu^\es)$ and thus governs the LDP. The proof is almost direct and in this case there is no need to use the Freidlin-Wentzell theory.

\bigskip
So let $\hat\uu$ be the unique equilibrium of \ef{1.51}. It follows that the attractor $\aaa$ of the semigroup corresponding to \ef{1.51} is the singleton $\{\hat\uu\}$. Combining this with the fact that $(\mu^\es)$ is tight and any weak limit of this family is concentrated on $\aaa=\{\hat\uu\}$, we obtain 
\be\label{9.10}
\mu^\es\rightharpoonup\De_{\hat \uu}.
\ee 
We now use this convergence to establish the lower bound.
Due to the equivalence of \eqref{7.3} and \eqref{7.4}, we need to show that for any $\uu_*\in\h$ and any positive constants $\eta$ and $\eta'$, there is $\es_*>0$ such that we have \,\footnote{\,We write $V_{\hat\uu}$ instead of $V_\aaa$, since $\aaa=\{\hat\uu\}$.}
\be\label{5.2}
\mu^\es(B(\uu_*,\eta))\geq \exp(-(V_{\hat\uu}(\uu_*)+\eta')/\es) \q\text {for }\es\leq \es_*.
\ee
We assume $V_{\hat\uu}(\uu_*)<\iin$, since the opposite case is trivial. By definition of $V$, there is a finite time $T>0$ and an action $\ph\in L^2(0,T;H_\vartheta)$ such that
$$
J_T(\ph)\leq V_{\hat\uu}(\uu_*)+\eta'\q\text{ and }\q |S^\ph(T;\hat\uu)-\uu_*|_\h<\eta/4.
$$
Since the operator $S^\ph$ continuously depends on the initial point, there is $\kp>0$ such that 
$|S^\ph(T;\uu)-\uu_*|_\h<\eta/2$,
provided $|\uu-\hat\uu|_\h\leq\kp$. 
It follows that (see \ef{5.6}-\ef{9.17})
$$
\mu^\es(B(\uu_*,\eta))\geq \mu^\es(B(\hat\uu,\kp)) \exp(-(V_{\hat\uu}(\uu_*)+2\eta')/\es).
$$
Combining this inequality with convergence \ef{9.10} and using the portmanteau theorem, we infer
$$
\mu^\es(B(\uu_*,\eta))\geq C(\kp) \exp(-(V_{\hat\uu}(\uu_*)+2\eta')/\es)\geq  \exp(-(V_{\hat\uu}(\uu_*)+3\eta')/\es).
$$
Replacing $\eta'$ by $\eta'/3$, we arrive at \ef{5.2}.

\section{Appendix}

\subsection{Global attractor of the limiting equation}
In this section we recall some notions from the theory of attractors and an important result concerning the global attractor of the semigroup $S(t)$ generated by the flow of equation \eqref{1.51}.
\bi
\item {\it Equilibrium points}
\ei
We say that $\hat u\in \h$ is an equilibrium point for $S(t)$ if $S(t)\hat \uu=\hat \uu$ for all $t\geq 0$.
\bi
\item {\it Complete trajectory}
\ei
A curve $\uu(s)$ defined for $s\in\rr$ is called a complete trajectory of the semigroup $(S(t))_{t\geq 0}$ if 
\be\label{4.11}
S(t)\uu(s)=\uu(t+s)\q\text{ for all } s\in\rr \text{ and }t\in\rr_+.
\ee
\bi
\item {\it Heteroclinic orbits}
\ei
A heteroclinic orbit is a complete trajectory that joins two different equilibrium points, i.e., $\uu(t)$ is a heteroclinic orbit if it satisfies \eqref{4.11} and there exist two different equilibria $\hat \uu_1$ and $\hat \uu_2$, such that $\uu(-t)\to \hat \uu_1$ and $\uu(t)\to\hat \uu_2$ as $t\to\iin$.

\bi
\item {\it The global attractor}
\ei
\nt
The set $\aaa\subset \h$ is called the global attractor of the semigroup $(S(t))_{t\geq 0}$ if it has the following three properties:

1) $\aaa$ is compact in $\h$ ($\aaa\Subset \h$).

2) $\aaa$ is an attracting set for $(S(t))_{t\geq 0}$, that is
\be\label{4.9}
d_\h(S(t)B,\aaa)\to 0\q \text{ as } t\to\iin,
\ee
for any bounded set $B\subset\h$, where $d_\h(\cdot,\cdot)$ stands for the Hausdorff distance in $\h$.

3) $\aaa$ is strictly invariant under $(S(t))_{t\geq 0}$, that is
 \be\label{4.10}
S(t)\aaa=\aaa\q\text{ for all } t\geq 0.
\ee
The following result gives the description of the global attractor of the semigroup $S(t)$ corresponding to \eqref{1.51}. We assume that the nonlinear term $f$ satisfies \eqref{1.54}-\eqref{1.56}. We refer the reader to Theorem 2.1, Proposition 2.1 and Theorem 4.2 in Chapter 3 of \cite{BV1992} for the proof.
\begin{theorem}\label{Th-attractor}
The global attractor $\aaa$ of the semigroup $(S(t))_{t\geq 0}$ corresponding to \eqref{1.51} is a connected set that consists of equilibrium points of $(S(t))_{\geq 0}$ and joining them heteroclinic orbits. Moreover, the set $\aaa$ is bounded in the space $[H^2(D)\cap H^1_0(D)]\times H^1_0(D)$.
\end{theorem}

\subsection{Large deviations for solutions of the Cauchy problem}\label{6.0}
In this section we announce a version of large deviations principle for the family of Markov processes  generated by equation \eqref{0.1.4}. Its proof is rather standard, and relies on the contraction principle and the LDP for the Wiener processes.

\medskip

Let $B$ be any closed bounded subset of $\h$ and let $T$ be a positive number. We consider the Banach space $\mathsf{\yyy}_{B,T}$ of continuous functions $y(\cdot,\cdot):B\times[0,T]\to \h$ endowed with the norm of uniform convergence.

\begin{theorem}\label{6.1}
Let us assume that conditions \eqref{1.54}-\eqref{1.57} are fulfilled. Then $(S^{\es}(\cdot;\cdot), t\in [0, T], \vv\in B)_{\es>0}$ regarded as a family of random variables in $\yyy_{B,T}$ satisfies the LDP with rate function $I_T:\yyy_{B,T}\to [0,\iin]$ given by 
$$
I_T(y(\cdot,\cdot))=\f{1}{2}\int_0^T |\ph(s)|_{H_\vartheta}^2\dd s
$$
if there is $\ph\in L^2(0,T;H_\vartheta)$ such that $y(\cdot,\cdot)=S^{\ph}(\cdot;\cdot)$, and is equal to $\iin$ otherwise.
\end{theorem}
We refer the reader to the book \cite{DZ1992} and the paper \cite{CM-2010} for the proof of similar results.
Let us note that in the announced form, Theorem \ref{6.1} is slightly more general compared to the results from mentioned works, since they concern the case when the set $B$ is a singleton $B\equiv\{\uu_0\}$. However, recall that the LDP is derived by the application of a contraction principle to the continuous map
 $\gi: C(0,T; H^1_0(D))\to\yyy_{\uu_0,T}$ given by  
$$
\gi(q(\cdot))= y(\uu_0,\cdot),\q \text{ where }\, y(\uu_0,\cdot)=(S^{\dt q}(\uu_0;\cdot); t\in [0,T]).
$$
Using the boundedness of $B$, it is not difficult to show that the map $\tilde\gi$ from $ C(0,T; H^1_0(D))$ to $\yyy_{B,T}$ given by
$$
\tilde\gi(q(\cdot))= y(\cdot,\cdot),\q \text{ where }\, y(\cdot,\cdot)=(S^{\dt q}(\uu_0;\cdot); \uu_0\in B,  t\in [0,T])
$$
is also continuous. This allows to conclude.

\subsection{Lemma on large deviations}
\begin{lemma}\label{9.3}
Let $(\mathfrak{m}^\es)_{\es>0}$ be an exponentially tight family of probability measures on a Polish space $\zzz$ that possesses the following property: there is a good rate function $\mathfrak{I}$ on $\zzz$ such that for any $\beta>0$, $\rho_*>0$ and $z\in\zzz$ there are positive numbers $\tilde\rho<\rho_*$ and $\es_*$ such that
\begin{align*}
\mathfrak{m}^\es (B_\zzz(z,\tilde\rho))&\leq\exp(-(\mathfrak{I}(z)-\beta)/\es),\\
\mathfrak{m}^\es (\bar B_\zzz(z,\tilde\rho))&\ge\exp(-(\mathfrak{I}(z)+\beta)/\es) \q\text{ for }\es\leq\es_*.
\end{align*}
Then the family $(\mathfrak{m}^\es)_{\es>0}$ satisfies the LDP in $\zzz$ with rate function $\mathfrak{I}$.
\end{lemma}

\bp
We first note that in view of equivalence of \eqref{7.3} and \eqref{7.4}, we only need to establish the upper bound, that is inequality \eqref{7.1}. Moreover, since the family $(\mathfrak{m}^\es)_{\es>0}$ is exponentially tight, we can assume that $F\subset\zzz$ is compact (see Lemma 1.2.18 in \cite{DZ2000}). Now let us fix any $\beta>0$ and denote by $\tilde\rho(z)$ and $\es_*(z)$ the constants entering the hypotheses of the lemma. Clearly, we have
$$
F\subset\bigcup_{z\in F}B_\zzz(z,\tilde\rho(z)).
$$
Since $F\subset\zzz$ is compact, we can extract a finite cover 
$$
F\subset\bigcup_{i=1}^n B_\zzz(z_i,\tilde\rho(z_i)).
$$
It follows that
$$
\mathfrak{m}^\es(F)\leq n\exp(-(\inf_{z\in F}\mathfrak{I}(z)-\beta)/\es)\leq \exp(-(\inf_{z\in F}\mathfrak{I}(z)-2\beta)/\es),
$$
for $\es\leq\es_*(n, z_1,\ldots, z_n)$. We thus infer
$$
\limsup_{\es\to 0}\es\ln \mathfrak{m}^{\es}(F)\leq-\inf_{z\in F} \mathfrak{I}(z)+2\beta.
$$
Letting $\beta$ go to zero, we arrive at \eqref{7.1}.
\ep

\subsection{Proof of some assertions}\label{9.21}

\nt

\bigskip
{\it Genericity of finiteness of the set $\E$.} By genericity with respect to $h(x)$ we mean that $\E$ is finite for any $h(x)\in \cal{C}$, where $\cal{C}$ is a countable intersection of open dense sets (and therefore $\cal{C}$ is dense itself) in $H^1_0(D)$. This property is well known, and the proof relies on a simple application of the Morse-Smale theorem, see e.g., Chapter 9 in \cite{BV1992}. Here we would like to mention that there are also genericity results with respect to other parameters. Namely, it is known that in the case $h(x)\equiv 0$ and $f(0)=0$ the property of finiteness of the set $\E$ is generic with respect to the boundary $\p D$; we refer the reader to Theorem 3.1 in \cite{Saut1983}. Finally, let us mention that in the one-dimensional case, the genericity holds also with respect to the nonlinearity $f$, see \cite{Bru-Chow1984}.

\bigskip
{\it Exponential moments of Markov times $\sigma_k^\es$ and $\tau_k^\es$}. Let $R>0$ be so large that $\tilde g\subset B_R$.
 We claim that there is $\De(\es)>0$ such that we have
\be\label{9.22}
\sup_{\vv\in B_R}\e_\vv\exp(\De\sigma_0^\es)<\iin,\q \sup_{\vv\in B_R}\e_\vv\exp(\De\tau_1^\es)<\iin.
\ee
Indeed, in view of inequality \ef{9.27}, we have
$$
\sup_{\vv\in \tilde g}\e_\vv\exp(\De\sigma_0^\es)=\sup_{\vv\in \tilde g}\e_\vv\left(\sum_{n=0}^\iin \ch_{nt_*\le \sigma_0^\es< (n+1)t_*}e^{\De\sigma_0^\es}\right)\le e^{\De t_*}\sum_{n=0}^\iin (q e^{\De t_*})^n<\iin,
$$
where we set  $q=1-\exp(-\beta/\es)$ and choose $\De>0$ such that $q e^{\De t_*}< 1$. 
Thus, the first inequality in \ef{9.22} is established. By the strong Markov property, to prove the second one, it is sufficient to show that there is $\tilde\De\in (0, \De]$ such that 
$$
\sup_{\vv\in B_R}\e_\vv\exp(\tilde\De\tau^\es_g)<\iin, 
$$
where $\tau^\es_g(\vv)$ is the first instant when $S^\es(t;\vv)$ hits the set $\bar g$. The above relation follows \footnote{\,If the origin is among the equilibria, we use inequality (2.18) in the form announced in \cite{DM2014}. We note, however, that the latter is true for a neighborhood of any point, not only the origin, which allows to conclude in the general case.} from inequality (2.18) of  \cite{DM2014}, and we arrive at \ef{9.22} with $\De$ replaced by $\tilde \De$.

\subsection{Proof of inequality \ef{1.75}}

\nt

{\it Step~1.} Let $\uu_*\in\{V_\aaa\leq M\}$. By definition of $V_\aaa$, for any $j\geq 1$ there is an initial point $\uu_{0}^j\in\aaa$, a finite time $T_j>0$, and an action $\ph^j$ such that
\be\label{1.32}
J_{T_j}(\ph^j)\leq M+1\q \text{ and }\q |S^{\ph^j}(T_j;\uu_{0}^j)-\uu_*|_\h\leq 1/j.
\ee
In view of the second of these inequalities, in order to prove \eqref{1.75}, it is sufficient to show that
\be\label{1.44}
|S^{\ph^j}(T_j;\uu_{0}^j)|_{\h^s}\leq C(M) \q \text{ for all }j\geq 1,
\ee
where we set $\h^s=H^{s+1}(D)\times H^s(D)$.

\bigskip
{\it Step~2.} 
By definition of $S^\ph(t;\vv)$, we have $S^{\ph^j}(T_j;\uu_{0}^j)=\uu(T_j)$,
where $\uu(t)=[u(t),\dt u(t)]$ solves
\be\label{1.33}
\p_t^2 u+\gamma \p_t u-\de u+f(u)=h(x)+\ph^j(t,x),\q \uu(0)=\uu_{0}^j\q t\in[0,T_j].
\ee
We claim that
\be\label{1.45}
\ees(\uu(t))\leq C(\|h\|_1,M)\q\text{ for } t\in [0,T_j],
\ee
where $\ees(\uu)$ is given by \ef{1.95}.
Indeed, let us multiply equation \eqref{1.33} by $\dt u+\al u$ and integrate over $D$. Using some standard transformations and the dissipativity of $f$, we obtain 
$$
\p_t \ees(\uu(t))\leq -\al\ees(\uu(t))+C_1(\|h\|^2+\|\ph^j(t)\|^2) \q\q\q t\in[0,T_j].
$$
Applying the Gronwall lemma to this inequality, we get
\begin{align}
\ees(\uu(t))&\leq \ees(\uu(0))e^{-\al t}+C_1\int_0^t (\|h\|^2+\|\ph^j(s)\|^2)e^{-\al (t-s)}\dd s\notag\\
&\leq \ees(\uu(0))e^{-\al t}+C_1\al^{-1}\|h\|^2+C_1\int_0^{T_j}\|\ph^j(s)\|^2\dd s\notag\\
&\leq \ees(\uu(0))e^{-\al t}+C_1\al^{-1}\|h\|^2+C_2|\ph^j|_{L^2(0,T_j;H_\vartheta)}^2\leq C_3(\|h\|_1,M)\label{1.38},~~~~~~~
\end{align}
where we used first inequality of \eqref{1.32}, and the fact that since the initial point $\uu(0)=\uu_0^j$ belongs to the global attractor, its norm is bounded by constant depending on $\|h\|_1$ (see Theorem \ref{Th-attractor}).
Inequality \eqref{1.45} is thus established.

\bigskip
{\it Step~3.} 
We are now ready to prove \eqref{1.44}.
To this end, we split $u$ to the sum $u=v+z$, where $z$ solves
\be\label{1.37}
\p_t^2 z+\gamma \p_t z-\de z+f(u)=0,\q [z(0),\dt z(0)]=0\q t\in[0,T_j].
\ee
It is well known (e.g., see \cite{BV1992, Har85}) that inequality \eqref{1.45} implies  
\be\label{1.92}
|[z(T_j),\dt z(T_j)]|_{\h^s}\leq|[z(t),\dt z(t)]|_{C(0,T_j;\h^s)}\leq C_s(\|h\|_1,M)
\ee
for any $s<1-\rho/2$. 
Thus, it is sufficient to show that
\be\label{1.81}
|[v(t),\dt v(t)]|_{\h^{\f{1}{2}}}\leq C(\|h\|_1,M)\q\text{ for all } t\in[0,T_j].
\ee
Let us first note that in view of \eqref{1.33} and \eqref{1.37}, $v$ solves
\be\label{1.48}
\p_t^2 v+\gamma \p_t v-\de v=h(x)+\ph^j(t,x),\q [v(0),\dt v(0)]=\uu_0^j\q t\in[0,T_j].
\ee
Multiplying this equation with $-\de(\dt v+\al v)$, we obtain that for all $t\in [0,T_j]$
\begin{align}
\p_t |[v(t),\dt v(t)]|_{\h^{\f{1}{2}}}^2 &\leq -\al |[v(t),\dt v(t)]|_{\h^{\f{1}{2}}}^2+C_4(\|(-\de)^{\f{1}{2}} h\|^2+\|(-\de)^{\f{1}{2}}\ph^j(t)\|^2)\notag\\
&\leq -\al |[v(t),\dt v(t)]|_{\h^{\f{1}{2}}}^2+C_5(\|h\|_1^2+\|\ph^j(t)\|_{H_\vartheta}^2),\label{1.94}
\end{align}
where we used the fact that the space $H_{\vartheta}$ is continuously embedded in $H^1$, since 
$$
|\ph|_{\tilde H^1}^2=\sum_{j=1}^\iin\lm_j(\ph,e_j)^2=\sum_{j=1}^\iin\lm_j b_j^2(b_j^{-2}(\ph,e_j)^2)\leq \sup(\lm_j b_j^2)|\ph|_{H_{\vartheta}}^2
\leq\BBB_1 |\ph|_{H_{\vartheta}}^2.
$$
Applying the Gronwall lemma to inequality \eqref{1.94} and using first relation of \eqref{1.32} together with $\uu^j_0\in\aaa$, we derive \eqref{1.81}. Inequality \ef{1.75} is established.

\subsection{Some a priori estimates}\label{9.29}

\nt

\medskip
{\it Exponential moment of solutions.}
Let $\vv(t)$ be a solution of equation \eqref{0.1.4} with $\es=1$. We shall denote by $\ees:\h\to\rr$ the energy function given by
\be\label{1.95}
\ees(\uu)=|\uu|_\h^2+2\int_D F(u_1)\dd x \q\text { for } \uu=[u_1, u_2]\in\h.
\ee
The next result on the boundedness of exponential moment of $\ees(\vv(t))$ is taken from \cite{DM2014}.
\begin{proposition}\label{1.96}
Let conditions \ef{1.54}-\ef{1.57} be fulfilled. Then, we have
\be\label{1.97} 
\e\exp(\kp\,\ees(\vv(t))\leq\e\exp(\kp\,\ees (\vv(0))e^{-\al t}+Q(\BBB,\|h\|),
\ee
where inequality holds for any $\kp\leq (2\BBB)^{-1}\al$, and the function $Q(\cdot, \cdot)$ is increasing in both of its arguments. Here $\BBB$ stands for the sum $\sum b_j^2$ and $\al>0$ is the constant from \ef{9.11}.
\end{proposition}

\bigskip
{\it Feedback stabilization result.}
Let us consider functions $\uu(t)$ and $\vv(t)$ defined on the time interval $[0, T]$ that correspond, respectively, to the flows of equations \eqref{7.10} and
\be\label{1.98}
\p_t^2 v+\gamma \p_t v-\de v+f(v)+P_N [f(u)-f(v)]=h(x)+\ph(t,x).
\ee

We suppose that either $\ph(t,x)$ belongs to $L^2(0,T;L^2(D))$ or its primitive with respect to time belongs to $C(0,T;H^1_0(D))$. Let $\mmm$ be a positive constant such that 
\be\label{9.26}
|\uu(t)|_\h\vee |\vv(0)|_\h\le \mmm  \q\text{ for all }  t\in [0,T].
\ee
The following result is a variation of Proposition 4.1 in \cite{DM2014}.
\begin{proposition}\label{1.88}
Under the conditions \ef{1.54}-\ef{1.56}, there is an integer $N_*$ depending only on $\mmm$ such that for all $N\geq N_*$ we have
\be\label{1.87}
|\vv(t)-\uu(t)|_\h^2\leq e^{-\al t}|\vv(0)-\uu(0)|_\h^2 \q\text{ for all } t\in [0, T].
\ee
\end{proposition}
\bp
We first show that there is an integer $N_1$ depending only on $\mmm$ such that for all $N\geq N_1$ we have
\be\label{9.25}
|\vv(t)|_\h\le 4\mmm  \q\text{ for all }  t\in [0,T].
\ee
To this end, let us introduce
\be\label{1.9}
\tau=\inf\{t\in[0,T]: | \vv(t)|_{\h}>4\mmm\},
\ee
with convention that the infimum over the empty set is $\iin$. Inequality \eqref{9.25} will be proved if we show that there is $N_1=N_1(\mmm)$ such that $\tau=\iin$ for all $N\geq N_1$. Note that in view of \ef{9.26}, we have $\tau>0$. Moreover, by definition of $\tau$, we have
\be\label{1.11}
|\uu(t)|_\h \vee |\vv(t)|_\h\leq 4\mmm \q\q\q\text{ for all } t\in[0, \tau\wedge T].
\ee
It follows from Proposition 4.1 in \cite{DM2014} applied to the interval $[0,\tau\wedge T]$ that there is an integer $N_1$ depending only on $\mmm$ such that for all $N\geq N_1$ we have
\be 
|\vv(t)-\uu(t)|^2_\h\leq e^{-\al t}|\vv(0)-\uu(0)|_\h^2 \q\q\q\text{ for all } t\in[0,\tau\wedge T].
\ee
Therefore we have
\be
|\vv(\tau\wedge T)|_\h\leq |\uu(\tau\wedge T)|_\h+|\uu(0)|_\h+|\vv(0)|_\h\le 3\mmm.
\ee
Combining this with definition of $\tau$, we see that $\tau=\iin$, and thus inequality \eqref{9.25} is proved.
It follows that for $N\ge N_1$, we have
$$
|\uu(t)|_\h \vee |\vv(t)|_\h\leq 4\mmm \q\q\q\text{ for all } t\in[0, T].
$$ 
Once again using Proposition 4.1, but this time on the interval $[0,T]$, we see that there is $N_*\geq N_1$ such that for all $N\geq N_*$ we have \eqref{1.87}.
\ep

\bigskip
{\it Auxiliary estimates. Proof of \ef{9.41}.}
The standard argument shows (see the derivation of \eqref{1.38}) that there is a constant $\mmm$ depending only on $R$ and $\|h\|$ such that for all $j\geq 1$ we have
\be\label{2.6}
\sup_{[0, j]}|\uu^j(t)|_\h+\sup_{[0, j]}|\vv^j(t)|_\h\leq \mmm.
\ee
We shall write $\uu^j(t)=[u(t), \dt u(t)]$ and $\vv^j(t)=[v(t), \dt v(t)]$. Note that the difference $\uu^j(t)-\vv^j(t)$ corresponds to the flow of equation 
\be\label{2.3}
\p_t^2 z+\gamma\p_t z-\de z+f(v+z)-f(v)=\ph^j, \q [z(0),\dt z(0)]=0 \q\q  t\in [0,j].
\ee
Multiplying this equation by $\dt z+\al z$ and integrating over $D$, we obtain
\begin{align}
\p_t |[z(t),\dt z(t)]|_\h^2&\leq -\al |[z(t),\dt z(t)]|_\h^2+C(\|\ph^j(t)\|^2+\|f(v(t)+z(t))-f(v(t))\|^2)\notag\\
&\leq C(\|\ph^j(t)\|^2+\|f(v(t)+z(t))-f(v(t))\|^2).\label{2.7}
\end{align}
By the H\"older and Sobolev inequalities, we have
$$
\|f(v+z)-f(v)\|^2\leq C_1\|z\|_1^2(\|u\|_1^2+\|v\|_1^2+1)\leq C_2|[z,\dt z]|_\h^2(\|u\|_1^2+\|v\|_1^2+1).
$$
Combining this with inequalities \eqref{2.6} and \eqref{2.7}, we derive
$$
\p_t |[z(t),\dt z(t)]|_\h^2\leq C(\|\ph^j(t)\|^2+|[z(t),\dt z(t)]|_\h^2),
$$
where the constant $C$ depends only on $\mmm$.
Applying the comparison principle to this inequality, we see that for all $t\in[0,j]$, we have
$$
|[z(t),\dt z(t)]|_\h^2\leq C\int_0^t\|\ph^j(s)\|^2 \exp(C s)\dd s.
$$
Recalling the definition of $z$, we arrive at \ef{9.41}.

\subsection{Proof of Lemma \ref{1.201}}
We shall carry out the proof for the most involved case when $\uu_1$ and $\uu_2$ are the endpoints of the orbit $\ooo$. We need to show that for any positive constants $a$ and $\eta$, we have 
\be\label{9.28}
\mu^{\es_j}(B(\uu_2,\eta))\ge \exp(-a/\es)\q\q\text{ for } j>>1.
\ee
So, let $a$ and $\eta$ be fixed and let $\tilde\uu(\cdot)$ be a complete trajectory such that
$$
\tilde \uu(-t)\to\uu_1,\q \tilde \uu(t)\to\uu_2 \q\text{ as } t\to\iin.
$$
Let us find $T_1\geq 0$ and $T_2\geq 0$ so large that for $\tilde \uu_1=\tilde \uu(-T_1)$ and $\tilde \uu_2=\tilde \uu(T_2)$, we have
\be\label{5.8.4}
|\tilde \uu_1-\uu_1|_\h<\eta/8,\q |\tilde \uu_2-\uu_2|_\h<\eta/8.
\ee
Consider the flow $\uu'(t)=\tilde \uu(t-T_1)$. It corresponds to the flow of \eqref{1.51} issued from $\uu_1$.
Introduce the intermediate flow $\vv(t)$ corresponding to the solution of
$$
\p_t^2 v+\gamma \p_t v-\de v+f(v)=h(x)+\ph,\q [v(0),\dt v(0)]=\uu_1.
$$
where $\ph=P_N[f(v)-f(u')]$, and $u'$ is the first component of $\uu'$. In view of Proposition \ref{1.88}, there is an integer $N=N(\hat \uu_1,\|h\|)$ such that
\be\label{5.11.4}
|\vv(t)-\uu'(t)|_\h^2\leq e^{-\al t}\,|\uu_1-\tilde \uu_1|_\h^2\leq e^{-\al t}\,\eta^2/64,
\ee
where we used first inequality of \eqref{5.8.4}. In particular, for $T=T_1+T_2$, we have
\be\label{5.12.4}
|\vv(T)-\uu'(T)|_\h\leq\eta/8.
\ee
Now let us note that by construction we have
$$
\uu'(T)=\tilde \uu(T-T_1)=\tilde \uu(T_2)=\tilde \uu_2, \q \vv(T)=S^\ph(T;\uu_1).
$$
Combining this with second inequality of \eqref{5.8.4} and \eqref{5.12.4}, we obtain
$$
|S^\ph(T;\uu_1)- \uu_2|_\h<\eta/4.
$$
By continuity of $S^\ph$, there is $\kp>0$ such that
\be\label{5.13.4}
|S^\ph(T;\uu)-\uu_2|_\h<\eta/2 \q\text{ for } \,|\uu-\uu_1|_\h\leq\kp.
\ee
Moreover, it follows from inequality \eqref{5.11.4} that the action $\ph$ satisfies
\begin{align}
J_{T}(\ph)&=\f{1}{2}\int_0^{T}|\ph(s)|^2_{H_\vartheta}\dd s\leq C(N)\int_0^{T}|f(v)-f(u')|_{L^1}^2\dd s\notag\\
&\leq C_1\, C(N)\int_0^{T}(\|u'\|_1^2+\|v\|_1^2+1)\|v-u'\|_1^2\dd s\notag\\
&\leq C_2\,C(N) \int_0^{T}|\vv-\uu'|_\h^2 e^{-\al s}\dd s\leq C_3\, C(N)\, \eta^2\leq a\label{9.12},
\end{align}
provided $\eta$ is sufficiently small. Using inequality \eqref{5.13.4} and stationarity of $\mu^\es$ (see the derivation of \eqref{5.6}) we get
\begin{align*}
\mu^{\es_j}(B(\uu_2,\eta))&\geq \int_{|\uu-\uu_1|\leq\kp}\pp(|S^{\es_j}(T;\uu)-S^\ph(T;\uu)|<\eta/2)\,\mu^{\es_j}(d\uu)\\
&\ge \exp(-2a)\mu^{\es_j}(B(\uu_1,\kp))\q\q\text{ for } j>>1,
\end{align*}
where we used Theorem \ref{6.1} with inequality \eqref{9.12}. Moreover, since the point $\uu_1$ is stable with respect to $(\mu^{\es_j})$, we have 
$$
\mu^{\es_j}(B(\uu_1,\kp))\ge\exp(-a/\es_j)\q\q\text{ for } j>> 1.
$$
Combining last two inequalities, we arrive at \ef{9.28}, with $a$ replaced by $3a$. Lemma \ref{1.201} is established.

\selectlanguage{french}
 
% Chapter 3

\chapter{Grandes d\'eviations locales} % Main chapter title

\label{Chapter5} % For referencing the chapter elsewhere, use \ref{Chapter1} 
 
\author{D.~Martirosyan\footnote{Department of Mathematics, University of Cergy-Pontoise, CNRS UMR 8088, 2 avenue
Adolphe Chauvin, 95300 Cergy-Pontoise, France; e-mail: \href{mailto:Davit.Martirosyan@u-cergy.fr}{Davit.Martirosyan@u-cergy.fr}} \and V.~Nersesyan\footnote{Laboratoire de Math\'ematiques, UMR CNRS 8100, Universit\'e de Versailles-Saint-Quentin-en-Yvelines, F-78035 Versailles, France;  e-mail: \href{mailto:Vahagn.Nersesyan@math.uvsq.fr}{Vahagn.Nersesyan@math.uvsq.fr}} 
}

 \selectlanguage{english}
\section*{Local large deviations  principle for occupation measures  of the   damped
nonlinear wave equation perturbed by a   white~noise}
{\it joint work with V.~Nersesyan}\footnote{Laboratoire de Math\'ematiques, UMR CNRS 8100, Universit\'e de Versailles-Saint-Quentin-en-Yvelines, F-78035 Versailles, France;  e-mail: \href{mailto:Vahagn.Nersesyan@math.uvsq.fr}{Vahagn.Nersesyan@math.uvsq.fr}}

{\bf Abstract}.

   We consider the damped nonlinear wave (NLW) equation driven by a spatially regular white noise. Assuming that the     noise is non-degenerate in all Fourier modes, we establish a large deviations  principle (LDP) for the occupation measures of the trajectories. The lower bound in the LDP is of a local type, which is related to the weakly dissipative nature of the equation  and seems to be new  in the context of randomly forced PDE's. 
   The proof is based  on an extension   of methods developed in~\cite{JNPS-2012} and~\cite{JNPS-2014}  in the case of     kick forced dissipative PDE's with  parabolic  regularisation property such as, for example,  the Navier--Stokes system and the complex
Ginzburg--Landau equations.   We also show that a high concentration towards the stationary measure is impossible, by proving that the rate function that governs the LDP cannot have the trivial form (i.e., vanish on the stationary measure and be infinite elsewhere).

\setcounter{section}{-1}

\section{Introduction}
\label{S:0}

 This paper is devoted to the   study of the large deviations principle (LDP) for the occupation measures of the stochastic nonlinear wave (NLW) equation in a bounded domain $D\subset\R^3$ with a smooth boundary $\partial D$:
 \begin{align}
 \p_t^2u+\gamma \p_tu-\de u+f (u)&=h(x)+\vartheta(t,x), \q u|_{\partial D}=0,\label{0.1}\\ 
  \,\,\, [u(0),\dt u(0)]&=[u_0, u_1].\label{0.2}\
\end{align} 
   Here $\gamma>0$ is a damping parameter, $h$ is a   function in $  H^1_0(D)$, and 
 $f$ is a nonlinear term satisfying some standard dissipativity and growth conditions   (see~\eqref{1.8}-\eqref{1.6}).   These conditions are satisfied  
  for the classical examples~$f(u)=\sin u$ and $f(u)=|u|^{\rho}u-\lm u$, where~$\lm\in\R$ and~$\rho\in(0,2)$, coming from the damped sine--Gordon and
Klein--Gordon equations.
We assume that~$\vartheta(t,x)$ is a white noise of the form
\be\label{0.3}
\vartheta(t,x)=\p_t \xi(t,x), \q  \xi(t,x)=\sum_{j=1}^\infty b_j  \beta_j(t)e_j(x),
\ee where
$\{\beta_j\}$ is a sequence of independent standard Brownian motions, the set of   functions $\{e_j\}$ is
an   orthonormal basis in  $L^2(D)$ formed by    eigenfunctions of the Dirichlet
Laplacian with eigenvalues $\{\lm_j\}$, and
 $\{b_j\}$ is a sequence of real numbers  
  satisfying  
\be \label{a0.4}
\BBB_1:=\sum_{j=1}^\iin\lm_j b_j^2<\infty.
\ee

 We   denote by~$(\uu_t,\pp_\uu), \uu_t= [u_t,\dt u_t]$  the Markov family  associated with this stochastic NLW equation  and parametrised by the   initial condition  
 $ \uu=[u_0, u_1]$. 
 The exponential ergodicity for  this family     is established in \cite{DM2014}, this   result       is recalled  below in Theorem~\ref{T-Davit}.

     \smallskip 
    The LDP for   the occupation measures of randomly forced PDE's  has been previously established in   \cite{gourcy-2007a,gourcy-2007b} in the case of the Burgers equation and the Navier--Stokes system, based on some abstract results from~\cite{wu-2001}. In these papers,  the force is   assumed to be a rough white noise, i.e., it is  
     of the form~\eqref{0.3}   with the following    condition  on the coefficients:   
           $$
  cj^{-\alpha}\le b_j\le Cj^{-\frac12-\es}, \q  \frac12<\alpha<1 ,\,\, \es\in \left(0, \alpha-\frac12\right].
  $$    In the case of a    perturbation which is a    regular    random 
   kick force, the LDP is proved  in~\cite{JNPS-2012, JNPS-2014} for a family of PDE's with parabolic regularisation (such as the Navier--Stokes system or the complex
Ginzburg--Landau equation).  See also \cite{JNPS-2013} for the proof of the LDP and
     the Gallavotti--Cohen principle in the case of a rough      kick force.

     \smallskip 

   The  aim of the  present 
    paper is to extend the results and the methods of   these works under more 
    general   assumptions  
  on  both  stochastic and deterministic   parts of the equations. The random perturbation in our setting  is a spatially regular white noise,  and the NLW equation is only weakly dissipative and  lacks   a regularising property. In what follows, we shall denote by $\mu$ the stationary measure of the family $(\uu_t,\pp_\uu)$, and for any bounded continuous function  $\psi:H_0^1(D)\times L^2(D) \to \R   $, we shall write $\langle\psi,\mu\rangle$ for the integral of~$\psi$ with respect to~$\mu$. We   prove  the following level-1 LDP for the solutions of problem~\eqref{0.1},~\eqref{0.3}. 
      \begin{mt} Assume that conditions \eqref{a0.4} and \eqref{1.8}-\eqref{1.6} are verified and~$b_j>0$ for all~$j\ge1$. Then   for any non-constant bounded  H\"older-continuous function $\psi:H_0^1(D)\times L^2(D) \to \R   $,  there is $\es=\es(\psi)>0$ 
 and a convex function~$I^\psi:\R\to \R_+$
such that,  
  for any    $\uu\in H^{\sS+1}(D)\times H^\sS(D)$ and any open subset~$O$ of the interval $(\langle\psi, \mu\rangle-\es, \langle \psi, \mu\rangle+\es)$, we have
\be\label{Llimit}
\lim_{t\to\infty} \frac1t\log   \pp_\uu\left\{\f{1}{t}\int_0^t\psi(\uu(\tau))\dd\tau\in O\right\}= -\inf_{\alpha\in O} I^\psi(\alpha),
\ee where $ \sS>0$ is a    small number. Moreover,    limit \eqref{Llimit} is uniform with respect to~$\uu$ in a bounded set of  $H^{\sS+1}(D)\times H^\sS(D)$.
   \end{mt} 
 We also establish  a more general result of   level-2 type in Theorem~\ref{T:MT}. 
 These two theorems are slightly different
   from the standard  
  Donsker--Varadhan form    (e.g., see Theorem~3 in~\cite{DV-1975}), since here the LDP is proved to hold   locally on some part of the phase space.

      \smallskip

      The proof of the Main Theorem is obtained by  extending the techniques   and results introduced in~\cite{JNPS-2012, JNPS-2014}. According to a local version of the G\"artner--Ellis theorem, relation \eqref{Llimit} will be established if we show that,  for some $  \beta_0>0 $, the following limit exists
      $$
Q(\beta  )=\lim_{t\to+\ty} \frac{1}{t}
\log\e_\uu\exp \biggl(\,\int_0^t\beta \psi(\uu_\tau)\dd \tau\biggr), \q |\beta|< \beta_0
$$ and it is differentiable in $\beta$ on $(-\beta_0,\beta_0)$.
 We show that both properties   can be derived from a multiplicative ergodic theorem, which is a
 convergence result for
  the  Feynman--Kac semigroup of  the stochastic NLW equation. A continuous-time version of a criterion established in \cite{JNPS-2014} shows that a multiplicative ergodic theorem   holds provided that   the following four  conditions   are satisfied: \mpp{uniform irreducibility, exponential tightness, growth condition,} and \mpp{uniform Feller property}.  
  The smoothness of the noise and
the lack of a strong dissipation and of a regularising property in the equation   result in substantial differences in the  techniques used to verify these     conditions. While in the case of   kick-forced models the  first two of them are checked directly, 
  they have a rather non-trivial proof in our case, relying  on a feedback stabilisation result  and some subtle estimates for the Sobolev norms of the solutions. Nonetheless, the most involved and highly technical part of the paper remains the verification of the uniform Feller property. Based on the coupling method, its proof is more intricate here mainly  due to a more complicated  Foia\c{s}--Prodi   type estimate for the stochastic  NLW equation. We get a uniform Feller property only for potentials that have a sufficiently small oscillation, and  this is the main reason why the LDP   established in this paper is of a local type.

 \smallskip
 
 The paper is organised as follows. We formulate in Section~1 the second main result of this paper on  the level-2 LDP for the NLW equation and, by using a  local version of Kifer's criterion, we reduce its proof to a  multiplicative ergodic theorem.  Section~2 is devoted to the derivation of the Main Theorem. In Sections~3 and~4, we are checking the conditions of  an abstract result about the convergence of generalised Markov semigroups. In Section~5, we prove the exponential tightness property and provide some estimates for the growth of   Sobolev norms of the solutions.   The multiplicative ergodic theorem is established in Section~6. In the Appendix, we prove the local   version of Kifer's criterion, the abstract convergence result for the semigroups, and   some other technical results which are used throughout the paper.

\subsection*{Notation}
For a Banach space $X$, we denote by $B_X(a,R)$ the closed   ball in~$X$ of radius~$R$ centred at~$a$. In the case when~$a=0$, we write $B_X(R)$.   For any   function $V:X\to\R$, we set $\Osc_X(V):=\sup_X V-\inf_X V$.  
We   use the following~spaces:

\smallskip
\noindent
$L^\infty(X)$ is the space of bounded measurable functions $\psi:X\to\R$ endowed with the norm $\|\psi\|_\infty=\sup_{u\in X}|\psi(u)|$. 

\smallskip
\noindent
$C_b(X)$ is the space of continuous functions $\psi\in L^\infty(X)$, and  
$C_+(X)$ is the space of positive continuous functions $\psi:X\to\R$. 

\smallskip
\noindent
$C^q_b(X),$ $ q \in (0,1]$ is the space of functions $f\in C_b(X)$ for which the following
norm is finite 
$$
\|\psi\|_{C_b^ q }=\|\psi\|_\infty+\sup_{u\neq v} \frac{
{|\psi(u)-\psi(v)|}}{ \|u-v\|^ q }.
$$

\smallskip
\noindent 
$\MM(X)$ is the vector space of signed Borel
measures on $X$ with finite total mass   endowed with the topology of  the  weak convergence. $\MM_+(X)\subset \MM(X)$ is the cone of non-negative   measures. 

\smallskip
\noindent
$\ppp(X)$ is the set of probability Borel measures on~$X$. For $\mu\in\ppp(X)$ and  $\psi\in C_b(X)$,  we denote $\lag \psi,\mu\rag=\int_X\psi(u)\mu(\Dd u).$
If $\mu_1,\mu_2\in\ppp(X)$, we set
$$
|\mu_1-\mu_2|_{var}=\sup\{|\mu_1(\Gamma)-\mu_2(\Gamma)|:\Gamma\in\BBBBB(X)\},
$$where $\BBBBB(X)$ is the Borel $\sigma$-algebra of $X$.

\smallskip
\noindent
For any  measurable function ${\wwww}:X\to[1,+\infty]$, let   $C_\wwww(X)$ (respectively, $L_\wwww^\infty(X)$) be the space of continuous (measurable) functions $\psi:X\to\R$ such that $|\psi(u)|\le C\wwww(u)$ for all $u\in X$. We endow~$C_\wwww(X)$ and  $L_\wwww^\infty(X)$ with the seminorm
$$
\|\psi\|_{L_\wwww^\infty}=\sup_{u\in X}\frac{|\psi(u)|}{\wwww(u)}.
$$  

\smallskip
\noindent
$\ppp_\wwww(X)$ is  the space of measures $\mu\in\ppp(X)$ such that $\langle \wwww,\mu\rangle<\infty$.

\smallskip
\noindent
 For an open set $D$ of $\R^3$, we introduce the following function spaces:

\smallskip
\noindent
$L^p=L^p(D)$ is the Lebesgue space of measurable functions whose $p^{\text{th}}$ power is integrable. In the case $p=2$ the corresponding norm is denoted by $\|\cdot\|$.

\smallskip
\noindent
$H^\sS=H^\sS(D), \sS\ge0$ is  the domain of  definition of the operator $(-\Delta)^{\sS/2}$ endowed with the norm $\|\cdot\|_\sS$:
$$
H^\sS=\D\left((-\Delta)^{\sS/2}\right)=\left\{ u=\sum_{j=1}^\infty  u_j e_j \in L^2: \|u\|_\sS^2:=  \sum_{j=1}^\infty \lambda_j^s u_j^2<\infty\right\}.
$$  In particular, $H^1$ coincides with $H^1_0(D)$, the space of functions in the    Sobolev space of order $1$ that vanish at the boundary. We denote by $H^{-\sS}$   the dual of~$H^\sS$.

\section{Level-2   LDP for the NLW equation}
\label{S:1}

\subsection{Stochastic NLW equation and   its mixing properties}
\label{S:1.1}

In this subsection we give the precise hypotheses on the nonlinearity   and recall a result on the property of exponential mixing     for the Markov family associated with  the flow of \ef{0.1}.
We shall assume that $f $ belongs to~$C^2(\R)$, vanishes at zero,  satisfies the growth condition
\be\label{1.8}
|f ''(u)|\leq C(|u|^{\rho-1}+1),\q u\in\rr,
\ee  
for some  positive constants $C$ and $\rho<2$, and the dissipativity conditions 
\begin{align}
F(u)&\geq C^{-1} |f '(u)|^{\f{\rho+2}{\rho}}-\nu u^2-C, \label{1.5}\\
f (u)u- F(u)&\geq-\nu u^2-C, \label{1.6}
\end{align}
 where $F$ is a primitive of $f $,  $\nu$ is a positive number less than  $  (\lm_1\wedge\gamma)/8$. Let us note that inequality \ef{1.5} is slightly more restrictive than the one used in~\cite{DM2014};   this hypothesis allows us   to establish the exponential tightness property (see Section~\ref{S:4.1}). 
We   consider the NLW equation in the phase space   $\h=H^1\times L^2$   endowed with   the norm
\be\label{e40}
|\uu|_\h^2=\| u_1\|^2_1+\|u_2+\al u_1\|^2,\q  \uu=[u_1, u_2]\in\h,
\ee
where $\al=\al(\gamma)>0$ is a small parameter. 
Under the above conditions, for any   initial data $\uu_0=[u_0,u_1]\in \h$, there is a unique solution (or \emph{a flow})  $\uu_t= \uu(t; \uu_0)=[u_t,\dt u_t]$ of problem \ef{0.1}-\ef{0.3} in $\h$ (see Section~7.2 in~\cite{DZ1992}). 
For any $\sS\in \R$,  let~$\h^\sS$ denote the space $H^{\sS+1}\times H^\sS$   endowed   with the norm
$$|\uu|_{\h^{\sS}}^2=\|  u_1\|_{\sS+1}^2+\| u_2+\al u_1\|_\sS^2,\q  \uu=[u_1, u_2]\in\h^{\sS}
$$
with the same  $\al$ as in  \ef{e40}. If~$\uu_0\in \h^\sS$ and $0<\sS<1-\rho/2$, the solution $\uu(t; \uu_0)$ belongs\,\footnote{Some estimates for the  $\h^\sS$-norm of the solutions are given in Section~\ref{S:moments}.} to $ \h^\sS$ almost surely. Let us define a function  $\wwww:\h\to [0, \iin]$      by
\be\label{1.1}
  \wwww(\uu)=1+|\uu|_{\h^\sS}^2+ \ees^4(\uu),  
\ee  which will play the role of the  {\it weight function}. Here 
$$
\ees(\uu)=|\uu|_\h^2+2\int_D F(u_1)\,\dd x, \q \uu=[u_1,u_2]\in\h,
$$
 is the    energy functional of the NLW equation.     

\smallskip

We consider the  Markov family $(\uu_t,\pp_\uu)$ associated with   \ef{0.1} and define the corresponding Markov operators  
\begin{align*}
&\PPPP_t:C_b(\h)\to C_b(\h),  \quad\,\,\,\quad \quad \PPPP_t \psi(\uu)=\int_\h \psi(\vv) P_t(\uu,\Dd \vv),\\
&\PPPP_t^*:\ppp(\h)\to \ppp(\h), \quad \quad\quad\quad \PPPP_t^* \sigma(\Gamma)=\int_\h  P_t(\vv,\Gamma) \sigma(\Dd \vv),\quad t\ge0,
\end{align*}  
where $P_t(\uu,\Gamma)=\pp_\uu\{\uu_t\in\Gamma\}$ is the transition function. Recall that a measure~$\mu\in\ppp(\h)$ is said to be stationary   if $\PPPP_t^*\mu=\mu$ for any $t\ge0$. The following result is   Theorem~2.3 in~\cite{DM2014}.
\begin{theorem} \label{T-Davit} Let us assume that    conditions \eqref{a0.4} and \eqref{1.8}-\eqref{1.6} are verified and~$b_j>0$ for all~$j\ge1$.
Then   the   family~$(\uu_t,\pp_\uu)$  has a unique stationary measure $\mu\in\ppp(\h)$. Moreover, there are positive constants $C$ and $\kp$ such that, for any $\sigma\in\ppp(\h)$, we have
$$
|\PPPP^*_t\sigma-\mu|_L^*\leq C e^{-\kp t}\int_\h \exp\left(\kp|\uu|_\h^4\right)\,\sigma(\Dd \uu),
$$where we set 
$$
|\mu_1-\mu_2|_L^*=\sup_{\|\psi\|_{C_b^1}\leq 1}|\lag \psi,\mu_1\rag-\lag \psi,\mu_2\rag| 
$$ for any $ \mu_1,\mu_2\in\ppp(\h)$.
\end{theorem}
   
\subsection{The statement of the   result}
\label{S:1.2}

Before giving the formulation of  the main result of this section, let us introduce some notation and
   recall some basic definitions from the theory of   LDP (see~\cite{DZ2000}). 
For any $\uu \in \h$,  
we define the   following family of {\it occupation measures\/} 
 \begin{equation} \label{a0.5}
\zeta_t=\frac1t\int_{0}^{t}\delta_{\uu_\tau} \dd \tau,\quad t>0,
\end{equation}
   where~$\uu_\tau:=\uu(\tau;\uu)$ and~$\delta_\vv$ is the Dirac measure concentrated at $\vv\in \h$.
         For any~$V\in C_b(\h)$ and~$R>0$, we set
$$
Q_R(V)=\limsup_{t\to+\ty} \frac{1}{t}\log\sup_{\uu \in X_R}\e_\uu 
\exp\bigl(t\lag V, \zeta_t\rag\bigr),
$$
where $X_R:=B_{\h^\sS}(R)$,  $ \sS\in(0,1-\rho/2)$.
Then~$Q_R:C_b(\h)\to \R$ is  a convex    1-Lipschitz    function, and its    {\it Legendre transform\/} is given by    
\begin{equation}\label{I_R}  
I_R(\sigma):=\begin{cases} \sup_{V\in C_b(\h)}\bigl(\lag V, \sigma\rag-Q_R(V)\bigr) & \text{for $\sigma \in {\cal P}(\h)$},  \\ +\infty & \text{for   $\sigma \in \MM(\h)\setminus {\cal P}(\h).$}  \end{cases}
\end{equation}The function~$I_R: \MM(\h)\to [0,+\ty]$ is   convex   lower semicontinuous in the weak topology, and~$Q_R$ can be reconstructed from $I_R$  by the formula 
\be\label{Legtr}
Q_R(V)= \sup_{\sigma\in\ppp(\h)} \bigl(\lag V, \sigma\rag-I_R(\sigma)\bigr) \q \text{for any $V\in C_b(\h)$}.
\ee
We denote by~$\VV$   the set  of functions~$V\in C_b(\h)$ satisfying the following two properties. 
 \begin{description}
\item[\bf  Property~1.]   For any    $R>0$
  and~$\uu\in X_R$, the following limit exists   (called pressure function)
$$
Q(V)=\lim_{t\to+\ty} \frac{1}{t}
\log\e_\uu\exp \biggl(\,\int_0^tV(\uu_\tau)\dd \tau\biggr)
$$
 and does not depend on the initial condition $\uu$. Moreover, this limit is uniform with respect to $\uu\in X_R$.
 \end{description}
  \begin{description}
\item[\bf   Property~2.]
There is   a unique measure~$\sigma_V\in \ppp(\h)$ (called equilibrium state) satisfying the equality
$$
Q_R(V)=  \lag V, \sigma_V\rag-I_R(\sigma_V).
$$
\end{description}
  A mapping $I : \ppp(\h) \to [0, +\ty]$ is  a {\it good rate function\/} if  for any~$a \ge0$ the level set~$\{\sigma\in\ppp(\h) : I(\sigma) \le a \}$ is compact.  A good rate function~$I$ is {\it non-trivial} if  the  effective domain $D_I:=\{\sigma\in\ppp(\h) : I(\sigma) < \infty \}$ is    not a singleton.
 Finally,  we shall denote by $\UU$ the set of functions~$V\in C_b(\h)$   for which there is a number~$q\in (0,1]$, an  integer~$N \ge 1$,    and a function~$F \in C_b^q(\h_N)$
such that 
\be\label{repres1}
V(\uu)=F(P_N\uu),\q\uu\in \h,
\ee where $\h_N:= H_N\times H_N$, $H_N:=\text{span}\{e_1,\ldots, e_N\}$, and~$P_N$ is the orthogonal projection in~$\h$ onto $\h_N$.  Given a number   $\delta>0$,    $\UU_\delta$ is the subset of functions~$V\in \UU$
satisfying $\Osc(V)< \delta$.  
\begin{theorem}  \label{T:MT}
Under the conditions of the Main Theorem,
 for any  $R>0$, the function~$I_R: \MM(\h)\to [0,+\ty]$ defined by \eqref{I_R} is a non-trivial good rate function, and   the family~$\{\zeta_t,t>0\}$         satisfies the following local LDP. 
\begin{description}
\item[Upper bound.] 
For any closed set~$F\subset\ppp(\h)$, we have
\be\label{UB}
\limsup_{t\to\infty} \frac1t\log \sup_{\uu\in X_R} \pp_\uu\{\zeta_t\in F\}\le -I_R(F).
\ee
\item[Lower bound.]  For any open set~$G\subset\ppp(\h)$, we have
\be\label{LB}
\liminf_{t\to\infty} \frac1t\log \inf_{\uu\in X_R}\pp_\uu\{\zeta_t\in G\}\ge -I_R( \W \cap G).
\ee Here\,\footnote{The infimum over an empty set is equal to $+\ty$.}   $I_R(\Gamma):=\inf_{\sigma\in \Gamma} I (\sigma)$ for $\Gamma \subset \ppp(\h)$ and  
 $\W:=\{\sigma_V: V\in \VV\}$,  where~$\sigma_V$ is the equilibrium state\,\footnote{By the fact that $I_R$ is a good rate function, the set of equilibrium states  is non-empty for any  $V\in C_b(\h)$. In Property 2,  the important assumption is the uniqueness.} corresponding to $V$. 
 \end{description}
  Furthermore, there is a number $\delta>0$ such that $\UU_\delta \subset \VV$  and for any $V\in \UU_\delta$, the pressure function $Q_R(V)$ does not depend on $R$.
\end{theorem}
This theorem is proved in the next subsection,  using a multiplicative  ergodic theorem and a local version of Kifer's criterion for LDP. Then in Section~\ref{S:MTP}, we combine it with a local  version of the  G\"artner--Ellis theorem  to establish the Main Theorem. 
   
\subsection{Reduction to a  multiplicative  ergodic theorem}
\label{S:1.3}

 In this subsection    we reduce the proof of  Theorem~\ref{T:MT}   to some properties related to the large-time behavior of      the   {\it Feynman--Kac semigroup}     defined by
$$
\PPPP_t^V \psi(\uu)=\e_\uu \left\{ \psi(\uu_t)\exp \biggl(\,\int_0^tV(\uu_\tau)\dd \tau\biggr)\right\}.
$$ For any $V\in C_b(\h)$ and $t\ge0$, the application~$\PPPP_t^V$ maps $C_{b}(\h)$ into itself.  Let us denote by $\PPPP_t^{V*}:\MM_+(\h)\to \MM_+(\h)$ its dual semigroup, and recall that a measure~$\mu\in\ppp(\h)$ is   an {\it eigenvector\/}     if there is $\la\in\R$ such that $\PPPP^{V*}_t\mu=\la^t \mu$  for any~$t>0$.   Let $\wwww$ be the function defined by \eqref{1.1}.
  From~\eqref{e30} with $m=1$  it follows that~$\PPPP_t^V$ maps\,\footnote{When we write $C_\wwww(\h^\sS)$ or  $C(X_R)$, the sets $\h^\sS$ and $X_R$ are assumed to be endowed with the topology induced by $\h$.}
      $C_{\wwww}(\h^\sS)$ into itself  (note that $\wwww_1=\wwww$ in~\eqref{e30}).
We shall say that a function $h\in C_\wwww(\h^\sS)$ is an eigenvector for the semigroup $\PPPP_t^V$ if~$\PPPP^V_t h(\uu)=\la^t h(\uu)$ for    any~$\uu\in \h^\sS$ and~$t>0$.
Then  we have the following theorem.  \begin{theorem} \label{T:1.1}
Under the conditions  of  the Main Theorem,    there is    $\delta>0$    such that     the following assertions hold  for any $V\in \UU_\delta$.
\begin{description}
\item[Existence and uniqueness.] The semigroup  $\PPPP_t^{V*}$ admits a unique   eigenvector~$\mu_V \in\ppp_\wwww(\h)$ corresponding to an  eigenvalue $\la_V>0$.   Moreover,  for any $m\ge1$, we have 
\be
\int_\h \left[|\uu|^m_{\h^\sS}+ \exp(\kp\ees(\uu))\right] \mu_{ V}(\Dd\uu)<\iin,  \label{momentestimate}
\ee
 where $\kp:=(2\al)^{-1}\BBB$ and $\BBB:=\sum b_j^2$.
The semigroup $\PPPP_t^{V}$ admits a unique eigenvector~$h_V\in C_{\wwww}(\h^\sS)\cap C_+(\h^\sS)$ corresponding to  $\la_V$    normalised by the condition $\langle h_V,\mu_V\rangle=1$.
\item[Convergence.] 
For any $\psi\in C_{\wwww}(\h^\sS)$, $\nu\in \ppp_{\wwww}(\h)$, and $R>0$, we have 
\begin{align}
\lambda_V^{-t}\PPPP_t^V \psi&\to\lag \psi,\mu_V\rag h_V
\quad\mbox{in~$C_b(X_R)\cap L^1(\h,\mu_V)$ as~$t\to\infty$}, \label{a1.5}\\
\lambda^{-t}_V\PPPP_t^{V*}\nu&\to\lag h_V,\nu\rag\mu_V
 \quad\mbox{in~$\MM_+(\h)$ as~$t\to\infty$}. \label{a1.6}
\end{align} 
\end{description}
\end{theorem} 
This result is proved in Section~\ref{S:5}. Here we apply it to establish  Theorem~\ref{T:MT}. 
\begin{proof}[Proof of     Theorem \ref{T:MT}]
{\it Step~1: Upper and lower bounds}. We apply Theorem~\ref{T:2.3} to prove estimates~\eqref{UB} and \eqref{LB}. 
Let us consider the following totally ordered set $(\Theta, \prec)$, where  $\Theta= \R^*_+\times X_R$    and  $\prec$ is a relation defined by  $(t_1,\uu_1)\prec (t_2,\uu_2)$ if and only if $t_1\le t_2$. 
For any $\theta=(t,\uu)\in \Theta$,  we set $r_\theta:=t$ and~$\zeta_\theta:=\zeta_t$, where~$\zeta_t$  is  the random probability measure given by~\eqref{a0.5} defined on the probability space $(\Omega_\theta,\FF_\theta,\IP_\theta):=(\Omega,\FF,\IP_\uu)$.   The conditions of Theorem~\ref{T:2.3}    are satisfied   for the family~$\{\zeta_\theta\}_{\theta\in \Theta}$. Indeed,
    a family~$\{x_\theta\in \R,\theta\in\Theta\}$ converges  if and only if it converges uniformly with respect to $\uu\in X_R$ as~$t\to+\ty$.  Hence~\eqref{2.1} holds with $Q=Q_R$, and for any $V\in \VV$, Properties~1 and~2 imply limit \eqref{2.1a} and  the     uniqueness of  the   equilibrium state.  It remains to check the following condition, which we postpone to  Section~\ref{S:4}.
    \begin{description}
\item[\bf Exponential tightness.] There is  a function $\varPhi : \h  \to [0,+\ty]$   whose level sets $ \{\uu \in \h : \varPhi(\uu)\le a\}$ are compact for any $a\ge0$  and 
$$
\e_\uu\exp \biggl(\,\int_0^t\varPhi(\uu_\tau)\dd \tau\biggr)\le 
Ce^{c t},\quad      \uu\in X_R, \,t>0
$$for some positive constants $C$ and $c$.
\end{description}
Theorem~\ref{T:2.3}  implies that 
     $I_R$   is a   good rate function and  the following two inequalities hold  for any  closed set~$F\subset \ppp(\h)$ and    open set  $G\subset \ppp(\h)$ 
\begin{align*}
\limsup_{\theta\in\Theta} \frac{1}{r_\theta}\log\pp_\theta\{\zeta_\theta\in F\}&\le -  I_R (F),\\
\liminf_{\theta\in\Theta} \frac1{r_\theta}\log\pp_\theta\{\zeta_\theta\in G\}&\ge -  I_R (\W\cap G).
\end{align*}
These   inequalities imply~\eqref{UB} and \eqref{LB}, since we have the  equalities 
\begin{align*}
\limsup_{\theta\in\Theta} \frac{1}{r_\theta}\log\pp_\theta\{\zeta_\theta\in F\}&=\limsup_{t\to\infty} \frac1t\log \sup_{\uu\in X_R} \pp_\uu\{\zeta_t\in F\},\\
\liminf_{\theta\in\Theta} \frac1{r_\theta}\log\pp_\theta\{\zeta_\theta\in G\}&=\liminf_{t\to\infty} \frac1t\log \inf_{\uu\in X_R}\pp_\uu\{\zeta_t\in G\}.
\end{align*}

\medskip
  {\it Step~2: Proof of  the inclusion $\UU_\delta\subset\VV$}.  Let   $\delta>0$ be  the constant in  Theorem~\ref{T:1.1}.    Taking~$\psi={\mathbf1}$ in~\eqref{a1.5}, we get  Property~1 with~$Q_R(V):=\log \la_V$ for any $V\in \UU_\delta$  (in particular,  $Q(V):=Q_R(V)$  does not depend
on $R$).   

  Property~2 is deduced from limit~\eqref{a1.5}     in the same way as in~\cite{JNPS-2014}. Indeed, for any $V\in \UU_\delta$, we introduce  the semigroup  
\begin{equation} \label{6.67}
\SSSS_t^{V,F}\psi (\uu)=\lambda_V^{-t}h_V^{-1}\PPPP_t^{V+F}(h_V \psi)(\uu), \q \psi,F\in C_b(\h), \, t\ge0,
\end{equation}     the function  
\be\label{Qhav0}
Q_R^V(F):=\limsup_{t\to+\ty} \frac{1}{t}\log\sup_{\uu \in X_R}\log(\SSSS_t^{V,F}{\mathbf1})(\uu),  
\ee and the Legendre transform $I^V_R: \MM(\h)\to [0,+\infty] $ of $Q_R^V(\cdot)$. 
The arguments of  Section~5.7 of~\cite{JNPS-2014} show that $\sigma\in \ppp(\h)$ is an equilibrium state for $V$ if and only if  $I^V_R(\sigma)=0$. So the uniqueness follows from  the following result which   is a continuous-time version of Proposition 7.5 in \cite{JNPS-2014}. Its proof is given in the Appendix.
       \begin{proposition}\label{P:JNPS}For any $V\in \UU_\delta$ and $R>0$, the measure~$\sigma_V=h_V\mu_V$ is the unique 
        zero of $I_R^V$.  
      \end{proposition}

  {\it Step~3: Non-triviality of $I_R$}. We argue by contradiction. Let us assume that~$D_{I_R}$ is a singleton. 
  By Proposition~\ref{P:JNPS} with $V={\mathbf0}$, we have that the stationary measure~$\mu$ is the unique zero\,\footnote{Note that when $V={\mathbf0}$, we have $\la_V=1$, $h_V={\mathbf1}$, $I_R^V=I_R$, and $\mu_V=\mu$.} of $I_R$, so  $D_{I_R}=\{\mu\}$. Then~\eqref{Legtr} implies that $Q(V)=\lag V,\mu\rag$ for any $V\in C_b(\h)$. Let us choose   any non-constant~$V\in \UU_\delta$ such that $\lag V,\mu\rag=0$. Then $Q(V)=0$, and    limit~\eqref{a1.5} with $\psi={\mathbf1}$ implies that~$\la_V=e^{Q(V)}=1$ 
    and   
    \be\label{e52}
\sup_{t\ge0}\e_0 \exp\left(\int_0^t   V(\uu_\tau)\dd \tau\right)<\infty,
\ee
where $\e_0$ means that we consider the trajectory issued from the origin. Combining this with  the central limit theorem (see   Theorem 2.5 in \cite{DM2014} and Theorem~4.1.8 and Proposition~4.1.4  in~\cite{KS-book}), we get  $V= {\mathbf0}$. This contradicts the assumption that $V$ is non-constant and completes the proof of Theorem~\ref{T:MT}.
      \end{proof}

\section{Proof of the Main Theorem}\label{S:MTP} 
   {\it Step~1: Proof in the case    $\psi\in \UU$}. For any  $R>0$ and  non-constant   $\psi\in \UU$, we  
   denote   
$$
I_R^\psi(p)=
\inf   \{I_R(\sigma): \lag \psi, \sigma\rag=p, \,\sigma\in \ppp(\h) \}, \q  p\in \R,
$$ where $I_R $ is given  by \eqref{I_R}. Then  $Q_R(\beta \psi )$ is  convex    in $\beta \in \R$, and
using \eqref{Legtr},  it is straightforward to check that
$$
Q_R(\beta \psi )   =\sup_{  p \in \R}\left( \beta  p-I_R^\psi( p)\right) \q\text{for $\beta \in \R$}.  
$$
   By well-known properties of convex functions of a real variable    
   (e.g., see~\cite{RV_73}),  $Q_R(\beta \psi )$  is differentiable in $\beta \in \R $, except possibly on a countable set, the  right and left derivatives~$D^+Q_R(\beta \psi )$ and~$D^-Q_R(\beta \psi )$ exist  at any    $\beta $ and $D^-Q_R(\beta \psi ) \le D^+Q_R(\beta \psi )$. Moreover,     the following equality holds   for some $\beta, p \in \R$
 \be  \label{Q_R10}
Q_R(\beta \psi )  =  \beta  p-I_R^\psi( p)   
\ee if and only if $ p\in [D^-Q_R(\beta \psi ) , D^+Q_R(\beta \psi )]$.    Let us set  $\beta _0:= \delta/ (4\|\psi\|_{\ty})$, where~$\delta>0$ is  the constant in Theorem \ref{T:MT}. Then for any~$|\beta |\le \beta _0$, we have~$\beta \psi \in \UU_\delta\subset \VV$  and  $ Q_R(\beta \psi )$  does not depend on $R>0$; we set $Q(\beta \psi ):=Q_R(\beta \psi )$. Let us show that~$D^-Q(\beta \psi ) = D^+Q(\beta \psi )$ for any $|\beta |< \beta _0$, i.e.,  
$Q(\beta \psi )$ is differentiable at $\beta $.  Indeed, assume that $ p_1,  p_2 \in [D^-Q(\beta \psi ) , D^+Q(\beta \psi )] $. Then   equality~\eqref{Q_R10} holds  with $ p= p_i, i=1,2$. As $I_R$ is a good rate function,    there are measures $\sigma_i\in  \ppp(\h)$ such that~$\lag \psi,\sigma_i\rag=p_i$ and~$I_R(\sigma_i)=I_R^\psi( p_i), i=1,2.$ Thus
$$
Q(\beta \psi )  =  \beta  p_i-I_R^\psi( p_i)=\lag \beta \psi ,\sigma_i\rag -I_R(\sigma_i),
$$ i.e., $\sigma_1$ and $\sigma_2$ are equilibrium states corresponding to $V=\beta \psi $. 
As $\beta \psi \in \VV$, from Property~2 we derive that $\sigma_1=\sigma_2$, hence $ p_1= p_2$. Thus $Q(\beta  \psi)$ is differentiable at $\beta $ for any~$|\beta |< \beta _0$.
Let us define the convex function 
\begin{equation} \label{ratefu0}
Q^\psi(\beta ):=\begin{cases}  Q(\beta \psi ), & \text{for $|\beta |\le \beta _0  $},  \\ +\infty, & \text{for    $|\beta |> \beta _0  $ }  \end{cases}
\end{equation}and its Legendre transform
\be\label{ratefu}
I^\psi( p):=\sup_{\beta \in \R}\left( \beta  p-Q^\psi(\beta )\right) \q\text{for $ p\in \R$}.
\ee Then $I^\psi$  is a finite convex function  not depending on $R>0$. As    $Q^\psi(\beta )$ is differentiable    at  any $|\beta |< \beta _0$ and  \eqref{2.1a} holds with $Q=Q^\psi(\beta )$ (with respect to the    directed set   $(\Theta, \prec)$ defined in the proof of Theorem \ref{T:MT}), we see that the conditions of 
    Theorem~A.5 in \cite{JOPP} are satisfied\,\footnote{
Theorem~A.5 in \cite{JOPP} is stated in the case   $\Theta=\R_+$. However, the proof presented
there remains valid for  random variables indexed by a directed
set.}.  Hence, we have \eqref{Llimit} for any open subset~$O$ of the interval~$J^\psi:=  (D^+Q^\psi(-\beta _0) , D^-Q^\psi(\beta _0))$.

\medskip
  {\it Step~2: Proof in the case    $\psi\in C_b(\h)$}. Let us first define the rate function $I^\psi:\R\to \R_+$ in the case of a general function   $\psi\in C_b(\h)$.  To this end,   we  take  a sequence~$\psi_n\in \UU$ such that $\|\psi_n\|_\ty\le \|\psi\|_\ty$ and $\psi_n\to \psi$ in $C(K)$ for any compact~$K\subset \h$. The   argument  of the proof of property (a) in Section~5.6 in~\cite{JNPS-2014} implies that Property~1 holds with $V=\beta \psi$ for any $|\beta|\le \beta_0$, where~$\beta_0$ is defined  as in Step~1, and  for any compact set $\KK\subset \ppp(\h)$, we have
  \be\label{ratefu1}
\sup_{\sigma\in \KK}|\lag \psi_n-\psi,\sigma \rag|\to 0\quad \text{as $n\to\ty$}. 
  \ee
   Moreover, from the proof of Proposition~3.17 in~\cite{FK2006} it follows that  
    \be\label{karsahm}
   Q_R(\beta \psi_n )\to  Q_R(\beta \psi )\q \text{for  $|\beta|\le \beta_0$.}
  \ee
  This implies that $Q_R(\beta \psi )$ does not depend on $R$ when $|\beta|\le \beta_0$, so we can define the functions~$Q^\psi$ and $I^\psi$ by \eqref{ratefu0} and \eqref{ratefu}, respectively.
 
 \medskip

Let $J^\psi$ be the interval defined in Step 1.
   To establish limit \eqref{Llimit},
 it suffices to show that for any open subset~$O\subset J^\psi$   the following two inequalities hold
 \begin{align}
\limsup_{t\to\infty} \frac1t\log \sup_{\uu\in X_R} \pp_\uu\{\zeta_t^\psi\in O\}&\le -  I^\psi (O),\label{upx}\\
\liminf_{t\to\infty} \frac1t\log \inf_{\uu\in X_R} \pp_\uu\{\zeta_t^\psi\in O\}&\ge -  I^\psi (O),\label{lwx}
\end{align}where $\zeta_t^\psi:=\lag \psi,\zeta\rag$.
To prove \eqref{upx}, we first apply \eqref{UB}  for a closed subset  $F\subset \ppp(\h)$ defined by~$F=\{\sigma\in \ppp(\h): \lag \psi, \sigma\rag\in \overline O\}$, where $\overline O$ is the closure of~$O$ in $\R$: 
\begin{align}\label{upy}
\limsup_{t\to\infty} \frac1t\log \sup_{\uu\in X_R} \pp_\uu\{\zeta_t^\psi\in O\}&\le \limsup_{t\to\infty} \frac1t\log \sup_{\uu\in X_R} \pp_\uu\{\zeta_t^\psi\in \overline O\}\nonumber \\
&= \limsup_{t\to\infty} \frac1t\log \sup_{\uu\in X_R} \pp_\uu\{\zeta_t\in F\}  \nonumber \\
&\le -  I_R (F).
\end{align}As $Q_R(\beta \psi ) \le Q^\psi(\beta)$ for any $\beta\in \R$, we have 
\be\label{upy2}
   I^\psi (\overline O)\le I_R^\psi(\overline O). \ee It is straightforward to check that  
\be\label{upy1} 
     I_R^\psi(\overline O)=  I_R (F). \ee
 From  the continuity of $I^\psi$ it follows that  $  I^\psi (  O)= I^\psi (\overline O)$. Combining this with~\eqref{upy}-\eqref{upy1},   we get \eqref{upx}.

\medskip

To establish \eqref{lwx},  we first recall that   the exponential tightness property and Lemma~3.2  in~\cite{JNPS-2014} imply  that for any $a>0$ there is a compact $\KK_a\subset \ppp(\h)$ such that
  \be\label{zetap0}
 \limsup_{t\to\ty} \frac1t \log \sup_{\uu\in X_R}  \pp_\uu\{\zeta_t\in \KK_a^c\} \le-a.
 \ee Let us take any $p\in O$ and choose $\es>0$ so small that that $(p-2\es,p+2\es)\subset O$. Then for any $a>0$, we have
 \be\label{zetap}
 \pp_\uu\{\zeta_t^\psi\in O\} \ge \pp_\uu\{\zeta_t^\psi\in (p-2\es,p+2\es), \zeta_t\in \KK_a\}. 
 \ee By \eqref{ratefu1}, we can choose   $n\ge1$ so large that
$$
\sup_{\sigma\in \KK_a}|\lag \psi_n-\psi,\sigma \rag|\le \es.
$$Using    \eqref{zetap},  we get  
\begin{align}\label{zetap1}
\pp_\uu\{\zeta_t^\psi\in O\} &\ge \pp_\uu\{\zeta_t^{\psi_n}\in (p-\es,p+\es), \zeta_t\in   \KK_a\}\nonumber\\
&\ge \pp_\uu\{\zeta_t^{\psi_n}\in (p-\es,p+\es)\}- \pp_\uu\{ \zeta_t\in   \KK_a^c\}. 
\end{align}We need the following elementary property of convex functions;  see the Appendix for the proof.
 \begin{lemma}\label{e51}
 Let $J\subset \R$ be an open interval and $f_n:J\to\R$ be a sequence of convex functions   converging   pointwise to a  finite function $f$.  Then   we have
 \begin{align*}
\limsup_{n\to\infty}  D^+ f_n(x) &\le  D^+ f(x),  \\
\liminf_{n\to\infty}  D^- f_n(x) &\ge  D^- f(x) , \q x\in J. 
\end{align*}

 \end{lemma}
 This lemma implies that, for sufficiently large $n\ge1$, we have 
 $$
 (p-\es,p+\es) \subset J^{\psi_n}=  (D^+Q^{\psi_n}(-\beta ^n_0) , D^-Q^{\psi_n}(\beta ^n_0)),$$ where  $\beta _0^n:= \delta/ (4\|\psi_n\|_{\ty})$.  Hence
 the result of Step 1 implies that
$$
 \lim_{t\to\infty} \frac1t\log   \pp_\uu\{\zeta_t^{\psi_n}\in (p-\es,p+\es)\}= -I^{\psi_n}((p-\es,p+\es))
  $$ uniformly with respect to $\uu\in X_R$. As 
  $$\limsup_{n\to\ty}Q^{\psi_n}(\beta)\le Q^{\psi}(\beta), \q \beta\in \R, $$
  we have   
  $$\liminf_{n\to\ty}I^{\psi_n}(q)\ge I^{\psi}(q), \q q\in \R. $$
  This implies that 
 $$
\liminf_{n\to\ty}  I^{\psi_n}((p-\es,p+\es))\ge I^{\psi}((p-\es,p+\es)).
  $$ Thus we can choose  $n\ge1$   so large that 
$$
 \liminf_{t\to\infty} \frac1t\log  \inf_{\uu\in X_R}  \pp_\uu\{\zeta_t^{\psi_n}\in (p-\es,p+\es)\}\ge -I^{\psi}((p-\es,p+\es))-\es. 
  $$
  Combining    this  with \eqref{zetap1} and \eqref{zetap0} and choosing $a>I^{\psi}((p-\es,p+\es))+\es$, we obtain
  $$
   \liminf_{t\to\infty} \frac1t\log  \inf_{\uu\in X_R} \pp_\uu \{\zeta_t^{\psi}\in O \}\ge   -I^{\psi}((p-\es,p+\es))-\es .
  $$
Since $p\in O$ is arbitrary and  $\es>0$ can be chosen arbitrarily small, we get \eqref{lwx}.

\medskip
  {\it Step~3:  The interval $J^\psi$}.  Let us    show that if   $\psi\in C^q_b(\h),$ $ q \in (0,1]$ is non-constant, then  the interval $J^\psi= (D^+Q^\psi(-\beta_0 ) , D^-Q^\psi(\beta_0))$ is non-empty and contains the point $\lag \psi, \mu \rag$. Clearly we can assume that $\lag \psi, \mu \rag=0$. As~$Q^\psi(0)=0$, it is sufficient to show that $\beta=0$ is the only point of the interval $[-\beta_0, \beta_0]$, where~$Q^\psi(\beta)$ vanishes. Assume the opposite. Then, replacing $\psi$ by $-\psi$ if needed, we can suppose that there is $\beta\in (0, \beta_0]$ such that $Q^\psi(\beta)=0$. As in Step 3 of Theorem \ref{T:MT}, this implies
  $$
\sup_{t\ge0}\e_0 \exp\left(\beta \int_0^t   \psi(\uu_\tau)\dd \tau\right)<\infty
$$
and $\psi\equiv 0$. This contradicts our assumption that $\psi$ is non-constant and completes the proof of the Main Theorem.

\section{Checking conditions of  Theorem \ref{T:5.3}}\label{S:2}

The proof of Theorem~\ref{T:1.1} is based on an application   of Theorem \ref{T:5.3}.
In this  section,    we verify    the growth condition, the uniform irreducibility property, and   the   existence of an eigenvector    for the following generalised Markov family of     transition kernels (see Definition~\ref{D:5.2})  
$$
P_t^V(\uu,\Gamma)=(\PPPP_t^{V*} \delta_\uu ) (\Gamma),\quad  V\in C_b(\h),\,\, \Gamma\in \BBBBB(\h),\,\,  \uu\in \h,\,\,t\ge0 
$$ in the phase space
$X=\h $ endowed with  a sequence of compacts     $X_R= B_{\h^\sS}(R)$, $R\ge1$ and  a weight function~$\wwww$ defined by~\eqref{1.1}. The uniform Feller property is the most   delicate condition to check in Theorem~\ref{T:5.3},   it will be established   in Section~\ref{S:3}. 
In the rest of the paper,     we shall always assume  that the hypotheses of  Theorem  \ref{T:MT} are fulfilled.

\subsection{Growth condition}
\label{S:2.2}

Since we take  $X_R= B_{\h^\sS}(R)$,   the set $X_\ty$ in  the growth condition in  Theorem~\ref{T:5.3}    will be   equal to $\h^\sS$ which  is dense in $\h$.  For any $\uu\in \h^\sS$ and $t\ge0$, we have~$\uu(t;\uu)\in \h^\sS$, so      the measure $P_t^V(\uu,\cdot)$ is concentrated on $\h^\sS$.  As $V$ is a bounded function, condition \eqref{5.9} is verified. Let us show that         estimate \eqref{5.8} holds    for any       $V$ with a    sufficiently small oscillation. 
 \begin{proposition} \label{P:2.4} 
 There is a constant $\delta>0$ and an integer  $R_0\ge1$ such that, 
for any $V\in C_b(\h)$ satisfying $\Osc(V)< \delta$, 
we have 
\begin{align}
&\sup_{t\ge0}
\frac{\|\PPPP_t^V\wwww\|_{L_\wwww^\infty}}{\|\PPPP_t^V{\mathbf1}\|_{R_0}}<\infty,\label{a5.8} 
\end{align}    where $\mathbf 1$ is the function on~$\h$ identically equal to~$1$ and $\|\cdot\|_{R_0}$ is the $L^\ty$ norm on $X_{R_0}$.

\end{proposition}
 
\begin{proof} Without loss of generality, we can assume
that
  $V\ge0$ and $\Osc(V)=\|V\|_\infty$. Indeed, it suffices to replace  $V$ by $V-\inf_H V$. We   split    the proof of~\eqref{a5.8}   into two steps. 

\medskip

{\it Step 1}. Let us   show that there are  $\delta_0>0$ and   $R_0\ge1$  such that 
\begin{equation} \label{6.0015}
\sup_{t\ge0}
\frac{\|\PPPP_{t}^V{\mathbf1}\|_{L_\wwww^\infty}}{\|\PPPP_{t}^V{\mathbf1}\|_{R_0}}<\infty,
\end{equation} provided that     $\|V\|_\infty< \delta_0$. 
To prove this, we  introduce the  stopping time
$$
 \tau(R)=\inf\{t\ge0: |\uu_t|_{\h^\sS}\le R \}
$$and   use the following result.  
 \begin{lemma}\label{L:5.1}
 There are  positive  numbers $\delta_0$,   $C$, and  $R_0$ such that
 \be\label{8.1}
\e_\uu e^{\delta_0\tau(R_0)}\le C\wwww(\uu),  \q \uu\in \h^\sS.
\ee
 \end{lemma} We omit the proof of this lemma,  since it is carried out  by   standard arguments, using the  Lyapunov function~$\wwww$ and   estimate  \eqref{e30} for $m=1$ (see Lemma~3.6.1 in~\cite{KS-book}).  
Setting $G_t:=\{\tau(R_0)> t\}$ and 
\begin{equation}\label{S6ogt}
\Xi_V(t):=\exp \left(\int_0^tV(\uu_s)\dd s\right),
\end{equation} we get   
\begin{equation}
\PPPP_t^V{\mathbf1}(\uu)=\e_\uu\Xi_V(t)
=\e_\uu \bigl\{\I_{G_t}\Xi_V(t)\bigr\}+\e_\uu\bigl\{\I_{G_t^c}\Xi_V(t)\bigr\}=:I_1+I_2.
\label{6.16}
\end{equation}
  Since $V\ge0$, we have    $\PPPP_t^V{\mathbf1}(\uu)\ge1$. Combining this with    \eqref{8.1} and $\|V\|_\infty < \delta_0$, we obtain  for any $\uu\in \h^\sS$   
$$
I_1\le \e_\uu \Xi_V\bigl(\tau(R_0)\bigr)\le \e_\uu\exp\bigl(\delta_0\tau(R_0)\bigr)\le C\,\wwww(\uu)
\le C\,\wwww(\uu)\,\|\PPPP_t^V{\mathbf1}\|_{R_0}. 
$$
The strong Markov property and \eqref{8.1} imply 
\begin{align*}
I_2&\le \e_\uu\bigl\{\I_{G_t}\Xi_V(\tau(R_0))\,\e_{\uu(\tau(R_0))}\Xi_V(t)\bigr\}
 \\&\le\e_\uu  \{e^{\delta_0\tau(R_0)} \}\,\|\PPPP_t^V{\mathbf1}\|_{R_0} \le C\,\wwww(\uu)\,\|\PPPP_t^V{\mathbf1}\|_{R_0} ,
\end{align*}
where we write $\uu(\tau(R_0))$ instead of $\uu_{\tau(R_0)}$. Using \eqref{6.16} and the    estimates for $I_1$ and $I_2$, we get~\eqref{6.0015}.

\medskip

{\it Step~2}.  
To prove~\eqref{a5.8}, we set $\delta:= \delta_0\wedge (\alpha/2)$ and   assume that~$\|V\|_\infty < \delta$ and~$t=Tk$, where $k\ge1$ is an integer and $T>0$ is so large that $q:=2e^{-T\frac\alpha2}<1$. Then, using the Markov property and~\eqref{e30}, we get 
\begin{align*}
\PPPP_{Tk}^V\wwww(\uu) &\le  e^{T\delta} \e_\uu \left\{\Xi_V(T(k-1)) \wwww(\uu_{Tk}) \right\}\\&=  e^{T\delta} \e_\uu \left\{\Xi_V(T(k-1))  \e_{\uu(T(k-1))} \wwww(\uu_T)  \right\} \\&\le e^{T\delta} \e_\uu \left\{\Xi_V(T(k-1))   [2e^{-T\al   }\we(\uu_{T(k-1)})+C_1] \right\} \\&\le q \PPPP_{T(k-1)}^V\wwww(\uu)+ e^{T\delta}  C_1 \PPPP_{T(k-1)}^V{\mathbf1}(\uu).
\end{align*} Iterating this and using fact that $V\ge0$, we obtain
\begin{align*}
\PPPP_{Tk}^V\wwww(\uu)  \le q^k  \wwww(\uu)+ (1-q)^{-1}e^{T\delta}  C_1 \PPPP_{Tk}^V{\mathbf1}(\uu).
\end{align*} Combining this with \eqref{6.0015}, we  see that 
$$
A:=\sup_{k\ge0}
\frac{\|\PPPP_{Tk}^V\wwww\|_{L_\wwww^\infty}}{\|\PPPP_{Tk}^V{\mathbf1}\|_{R_0}}<\infty.  
$$ 
To derive \eqref{a5.8} from this, we use the semigroup property and the fact that $V$ is  non-negative and bounded:
\begin{align*}
\|\PPPP^V_t\we\|_{L_\we^\infty} &=\|\PPPP^V_{t-Tk}(\PPPP^V_{Tk}\we)\|_{L_\we^\infty} \le  C_2\|\PPPP^V_{Tk }\we\|_{L_\we^\infty},\\
\|\PPPP_t^V{\mathbf1}\|_{R_0}&\ge\|\PPPP_{Tk}^V{\mathbf1}\|_{R_0},
\end{align*}where $k\ge0$ is such that $Tk\le t< T(k+1)$  and 
$$
 C_2:=\sup_{s\in[0,T]}\|\PPPP^V_s\we\|_{L_\we^\infty}\le e^{T\|V\|_\ty}  \sup_{s\in[0,T]}\|\PPPP_s\we\|_{L_\we^\infty}<\ty.
$$ So we get   
$$
\sup_{t\ge0}
\frac{\|\PPPP_t^V\wwww\|_{L_\wwww^\infty}}{\|\PPPP_t^V{\mathbf1}\|_{R_0}}\le  C_2A<+\ty.
$$This completes the proof of the proposition.

 \end{proof}

 \subsection{Uniform irreducibility}
\label{S:2.1}

In this section, we show that the family $\{P_t^V\}$  satisfies the uniform irreducibility condition   with respect to the sequence of compacts~$\{X_R\}$.
Since~$V$ is bounded, we have
$$
P_t^V(\uu,\Dd\vv)\ge e^{-t\|V\|_\infty}P_t(\uu,\Dd\vv),\quad\mbox{$\uu\in \h$},
$$where $P_t(\uu,\cdot)$  stands for the transition function of~$(\uu_t,\IP_\uu)$. 
 So it suffices to establish  the uniform irreducibility for $\{P_t\} $. 
 \begin{proposition}\label{P:2.2}
For any  integers $\rho, R\ge1$   and any   $r>0$, there are  positive numbers~$l=l(\rho,r,R)$ and~$p=p(\rho,r)$ such that
\begin{equation} \label{aa0}
P_l(\uu,B_{\h}(\hat \uu,r))\ge p\quad\mbox{for all~$\uu\in X_R ,\,\hat \uu\in X_\rho$}.
\end{equation} 
 \end{proposition}
 \begin{proof}   
 
\medskip
  Let us show that,  
for sufficiently large      $d\ge1$  and any
    $R\ge1$,  there is a time $k=k(R)$    such that  
    \begin{equation}\label{aa1}
P_k(\uu, X_{d})\ge \frac12, \quad\mbox{$\uu\in X_R$}.
\end{equation} 
Indeed,  by         \eqref{e30} for $m=1$, we have   
$$
\e_\uu|\uu_t|_{\h^\sS}^2 \leq \e_\uu \wwww(\uu_t) \leq 2e^{-\al  t}\we(\uu)+C_1. 
$$  Combining this with the estimate 
\be\label{aa01}
|\ees(\uu)|\leq C_2(1+|\uu|_\h^4),
\ee  we get
$$
\e_\uu|\uu_t|_{\h^\sS}^2 \leq C_3 e^{-\al t} R^{16} +C_1,   \q \uu \in X_R.
$$ The Chebyshev inequality implies that 
$$P_{t}(\uu, X_d   )\ge 1-   d^{-2}( C_3 e^{-\al t} R^{16} +C_1 ).
$$Choosing $t=k$ and $d$   so large that $e^{-\al k}R^{16} \le1$ and $d^2>2( C_3 +C_1) $, we obtain~\eqref{aa1}.
 
\medskip
    Combining \eqref{aa1} with  Lemma~\ref{e39} and   the Kolmogorov--Chapman relation, we get \eqref{aa0}  for $l=k+m$ and $p=q/2$.  
\end{proof}

 \begin{lemma}\label{e39}
For any integers $d,\rho\ge1$ and any $r>0$,   there are positive   numbers~$m=m(d,\rho,r)$ and~$q=q(d, \rho,r)$ such that
\begin{equation} \label{e6}
P_m(\vv,B_{\h}(\hat \uu,r))\ge q\quad\text{for all $\vv\in X_d, \,\hat \uu\in X_\rho$}.
\end{equation} 
\end{lemma}
\bp
   It is sufficient to prove that there is $m\geq 1$ such that
\be\label{e3}
P_m(\vv, B_\h(\hat\uu,r/2))>0 \q\text{ for all }\vv\in X_d, \, \hat \uu\in  \tilde X_\rho,
\ee where   $\tilde X_\rho=\{\uu=[u_1,u_2]\in X_\rho: u_1,u_2\in C_0^\iin(D)\}$.
Indeed, let us take this inequality for granted and assume that~\ef{e6} is not true. Then there are   sequences $\vv_j\in X_d$ and  $\hat \uu_j\in X_\rho$  such that
\be\label{e4} 
P_m(\vv_j, B_\h(\hat\uu_j,r))\to 0.
\ee
Moreover, up to extracting a subsequence, we can suppose that $\vv_j$ and $\hat \uu_j$ converge  in~$\h$. Let us denote by $\vv_*$ and $\hat \uu_*$ their   limits. Clearly, $\vv_*\in X_d$ and~$\hat\uu_*\in X_\rho$. Choosing $j\ge1$ so large that $|\hat\uu_j-\hat\uu_*|_\h<r/2$ and applying the Chebyshev        inequality, we get
\begin{align*}
P_m(\vv_*,B_\h(\hat \uu_* ,r))&\leq P_m(\vv_j,B_\h(\hat\uu_j, r/2))+\pp\{|\uu(m; \vv_j)-\uu(m;\vv_*)|_\h\geq r/2\}\\
&\leq P_m(\vv_j,B_\h(\hat\uu_j, r/2))+4/r^2  \,\e|\uu(m; \uu_j)-\uu(m;\vv_*)|^2_\h.
\end{align*}
Combining this with \ef{e4} and using the  convergence $\vv_j\to\vv_*$ and a density property, we arrive at a contradiction with \ef{e3}. Thus, inequality \ef{e6} is reduced to the derivation of~\ef{e3}. We shall prove the latter in three steps.

\medskip
{\it Step~1: Exact controllability.}
 In what follows, given any $\ph\in C(0, T;H^1 )$, we shall denote by~$S_\ph(t;\vv)$ the solution at time $t$ of the problem
$$
\p_t^2u+\gamma \p_tu-\de u+f (u)=h+\dt\ph, \q u|_{\partial D}=0, \q t\in [0,T]
$$
issued from $\vv$.
Let $\hat\vv=[\hat v, 0]$, where~$\hat v\in H^1 $ is  a solution of  $$
-\de\hat v+f (\hat v)=h(x).
$$
 In this step we   prove that for  any $\hat \uu=[\hat u_1, \hat u_2]\in \tilde X_\rho$, there is $\ph_*$ satisfying
\be\label{e7}
\ph_*\in C(0, 1;H^1)\q\text{ and } \q S_{\ph_*}(1;\hat\vv)=\hat\uu.
\ee
First note that, since the function $f $ is continuous from $H^1$ to $L^2$, we have 
$$
-\de\hat v=-f (\hat v)+h\in L^2 ,
$$
so that $\hat v\in H^2   $. Moreover, since $f $ is also continuous from $H^2$ to~$H^1$ (recall that $f$ vanishes at the origin), we have $f (\hat v)\in H^1 $. As $h\in H^1 $, it follows that
\be\label{e2}
-\de\hat v\in H^1.
\ee
Let us introduce the functions
\begin{align}
u(t)&=a(t)\hat v+b(t)\hat u_1+c(t)\hat u_2,\label{e8}\\
\ph_{*}(t)&=\int_0^t (\p_t^2u+\gamma \p_tu-\de u+f (u)-h ) \dd \tau, \nonumber
\end{align}
where $a, b,c\in C^\ty( [0,1], \R)$  satisfy
\begin{align*}
a(0)&=1, \q a(1)=\dt a(0)=\dt a(1)=0, \q  b(1)=1, \q b(0)=\dt b(0)=\dt b(1)=0,\\
\dt c(1)&=1, \q c(0)=c(1)=\dt c(0)=0.
\end{align*}
Then, we have 
 $[u(0),\dt u(0)]=\hat\vv$,      $ [u(1),\dt u(1)]=\hat\uu$, and  
  $S_{\ph_*}(1;\hat\vv)=\hat\uu$.  Let us show the first    relation in \ef{e7}. In view of \ef{e8} and the smoothness of the functions~$a, b$ and~$c$, we have
$$
\p_t^2u+\gamma \p_tu-h\in C(0,1;H^1)
$$
and thus it is sufficient to prove that
\be\label{e1}
-\de u+f (u)\in C(0, 1;H^1).
\ee
Since $u\in C(0, 1; H^2  )$, we have $f (u)\in C(0, 1; H^1 )$. Moreover, in view of \ef{e2} and the  smoothness of   $\hat u_1$ and $\hat u_2$, we have 	    $-\de u\in C(0, 1; H^1)$. Thus, inclusion \ef{e1} is established and we arrive at \ef{e7}. Let us note that by continuity and compactness, there is  $\kp=\kp(\hat\vv,\rho, r)>0$, not depending on~$\hat \uu\in \tilde X_\rho$,  such that 
\be\label{e14}
S_{\ph_{*}}(1;\vv)\in B_\h(\hat\uu,r/4)   \q \text{ for any }\vv\in B_\h(\hat\vv, \kp).
\ee

\medskip
{\it Step~2: Feedback stabilisation.} 
We now show that there is $\tilde m\geq 1$ depending only on $d$ and $\kp$ such that for any $\vv\in X_d$ there is $\tilde\ph_{\vv}$ satisfying
\be\label{e11}
\tilde\ph_\vv\in C(0, \tilde m; H^1)\q\text{ and }\q S_{\tilde\ph_{\vv}}(\tilde m, \vv)\in B(\hat\vv, \kp).
\ee
To see this, let us consider the flow $\tilde\vv(t;\vv)$ associated with the solution of the equation
\be\label{e12}
\p_t^2\tilde v+\gamma \p_t \tilde v-\de \tilde v+f (\tilde v)=h +{\mathsf P}_N[f (\tilde v)-f (\hat v)], \q t\in [0,\tilde m]
\ee
issued from $\vv\in X_d$, where  ${\mathsf P}_N$ stands for the orthogonal projection in $L^2 $ onto the subspace spanned by the functions~$e_1,e_2,\ldots,e_N$. Then, in view of Proposition 6.5 in \cite{DM2015}, for $N\geq N(|\hat\vv|_\h, d)$, we have
$$
|\tilde \vv(\tilde m;\vv)-\hat\vv|_\h^2\leq |\vv-\hat\vv|_\h^2\,e^{-\al \tilde m}\leq C_d\, e^{-\al \tilde m}< \kp
$$
for $\tilde m$ sufficiently large. It follows that \ef{e11} holds with the function 
$$
\tilde\ph_\vv(t)=\int_0^t {\mathsf P}_N[f (\tilde v)-f (\hat v)] \dd \tau.
$$

\medskip
{\it Step~3: Proof of \eqref{e3}.}
 Let us take $m=\tilde m+1$ and,  for any $\vv\in X_d$, define a   function $\ph_\vv(t)$   on the interval $[0, m]$ by
\begin{equation*} 
\ph_\vv(t)=\begin{cases}  \tilde \ph_\vv(t)   & \text{for $t\in[0, m-1] $},  \\ \tilde\ph_\vv(m-1)+\ph_*(t-m+1)  & \text{for    $t\in [m-1, m].$  }  \end{cases}
\end{equation*}
In view of \ef{e7}, \ef{e14}, and \ef{e11}, we have $\ph_\vv(t)\in C(0, m;H^1)$  and $S_{\ph_\vv}(m;\vv)\in B_\h(\hat\uu, r/2).$ Hence
 there is $\De>0$ such that $S_{\ph}(m;\vv)\in B_\h(\hat\uu, r/2)$  provided  $\|\ph-\ph_\vv\|_{C(0,m;H^1)}<\De$.
It follows that
$$
P_m(\vv, B_\h(\hat\uu,r/2))\geq \pp\{\|\xi-\ph_\vv\|_{C(0,m;H^1)}<\De\}.
$$
To complete the proof, it remains to note that, due to the non-degeneracy of~$\xi$, the term on the right-hand side of this inequality is positive.
\ep

\subsection{Existence of an eigenvector}
\label{S:2.3}

 For any   $m\ge1$, let us define     functions $\we_m, \tilde  \we_m:\h\to [1, +\iin]$ by 
 \begin{align}
\we_m(\uu)&=1+|\uu|_{\h^\sS}^{2m}+\ees^{4m}(\uu),  \label{e27}\\
\tilde \we_m(\uu)&=\we_m(\uu)+\exp(\kp\ees(\uu)), \q \uu\in \h,  \label{e32}
\end{align}
where  $\kp$ is the constant in Theorem \ref{T:1.1}.  The following proposition proves the existence of an eigenvector $\mu=\mu(t,V,m)  $ for the operator~$\PPPP_t^{V*}$ for any $t>0$. We shall see in Section~\ref{S:5} that the measure~$\mu$ actually  does not depend on~$t$ and~$m$.
 \begin{proposition}\label{e23}
 For any $t>0$, $V\in C_b(\h)$ and $m\geq 1$, the operator $\PPPP_t^{V*}$ admits an eigenvector $\mu=\mu(t,V,m) \in \ppp(\h)$ with a positive eigenvalue~$\lambda=\lambda(t,V,m)$: 
$$
\PPPP_t^{V*}\mu =\lambda  \mu.
$$  Moreover,  we have  
 \be\label{e24}
 \int_\h \tilde \wwww_m(\uu)\mu (\Dd\uu)<\iin,  
 \ee  
 \be\label{e25}
 \|\PPPP_t^V\we_m\|_{X_R}\int_{X_R^c}\we_{m}(\uu)\mu (\Dd\uu)\to 0\q \text{ as }R\to\iin.
 \ee
 \end{proposition}
\bp
{\it Step~1}. We first establish the existence of an eigenvector $\mu $ for $\PPPP_t^{V*}$ with a positive eigenvalue and satisfying \ef{e24}.
Let $t>0$ and $V$ be fixed. For any~$A>0$ and $m\geq 1$, let us introduce the convex   set
$$
D_{A, m}=\{\sigma\in \ppp(\h): \lan \tilde \we_m, \sigma\ran\leq A\},
$$
and    consider the continuous mapping from $D_{A, m}$ to $\ppp(\h)$ given by
$$
G(\sigma)=\PPPP_t^{V*}\sigma/\PPPP_t^{V*}\sigma(\h).
$$
Thanks to inequality \ef{e34}, we have
\begin{align}
\lan\tilde\we_{m}, G(\sigma)\ran&\leq \exp\left(t \Osc_\h(V)\right)\lan\tilde\we_{m}, \PPPP_t^*\sigma\ran\notag\\
&\leq 2\exp\left(t (\Osc_\h(V)-\al m)\right)\lan\tilde\we_{m}, \sigma\ran+C_m \exp\left(t  \Osc_\h(V)\right)\label{e28}.
\end{align}
Assume that $m$ is so large that 
$$
\Osc_\h(V)\leq \al m/2\q\text{ and }\q \exp(-\al m t/2)\leq 1/4,
$$
and let $A:=2 C_m e^{\al m t}$. 
Then, in view \ef{e28}, we have $\lan\tilde\we_{m}, G(\sigma)\ran\leq A$ for any~$\sigma\in D_{A, m}$, i.e.,       $G(D_{A, m})\subset D_{A, m}$. Moreover, it is easy to see that the set~$D_{A, m}$ is compact in $\ppp(\h)$ (we   use the Prokhorov compactness criterion  %(see Theorem~11.5.4 in \cite{dudley2002}) 
to show that it is relatively compact and the Fatou lemma to prove that it is closed). Due to the Leray--Schauder theorem, %(e.g., see Chapter~14 in~\cite{taylor1996}), 
the map~$G$ has a fixed point~$\mu \in D_{A, m}$. Note that, by the definitions of~$D_{A, m}$ and~$G$, the measure $\mu $ is an eigenvector of $\PPPP_t^{V*}$ with    positive  eigenvalue~$\lambda:=\PPPP_t^{V*}\mu (\h)$ and satisfies~\ef{e24}.  

\medskip
{\it Step~2}.
We now establish \ef{e25}. Let us fix an integer $m\geq 1$ and let $n=17m$. In view of the previous step, there is an eigenvector $\mu $ satisfying~$\lan\we_n,\mu \ran<\iin$.  From the Cauchy--Schwarz and Chebyshev inequalities it follows  that
\be\label{dm1}
\int_{X_R^c}\we_{m}(\uu)\mu (\Dd\uu)\leq \lan\we^2_{m}, \mu \ran^{1/2}(\mu (X_R^c))^{1/2}\leq C_m\lan\we_{n}, \mu_{t,V}\ran R^{-n}.
\ee
 On the other hand, using   \ef{e30} and \eqref{aa01}, we get    
$$
\|\PPPP_t^V\we_m\|_{X_R}\leq\exp(t\|V\|_\iin)\sup_{\uu\in X_R}\e_\uu\we_m(\uu_t)\leq C_m' \exp(t\|V\|_\iin)(R^{16m}+1).
$$Combining this with \eqref{dm1}, we obtain
  \ef{e25}.
\ep

\section{Uniform Feller property}
\label{S:3}
   \subsection{Construction of coupling processes} \label{S:3.1}

As in the case of discrete-time models considered in~\cite{JNPS-2012,JNPS-2014},  the proof of the uniform Feller property is based on the coupling method. This method has proved to be an important tool for the study of the ergodicity of randomly forced PDE's (see    Chapter~3  in~\cite{KS-book}   and the papers \cite{KS-jmpa2002,mattingly-2002, odasso-2008, DM2014}).  
 In this section, we recall a construction of coupled trajectories from~\cite{DM2014}, which was used   to establish the exponential mixing   for problem   \eqref{0.1}, \eqref{0.3}.   This construction will play a central role in the proof of the uniform Feller property in the next section.
 
\smallskip

For any~$\z,\z' \in \h$, let us  denote by $\uu_t$ and $\uu_t'$   the flows   of~\eqref{0.1},~\eqref{0.3} issued from $\z$ and $\z'$, respectively. 
For   any integer $N\ge 1$,    let $\vv=[v,\p_t v]$ be the flow of the problem  
\be\label{interm}
 \p_t^2v+\gamma \p_tv-\de v+f (v)+{\mathsf P}_N(f (u) -f (v))=h +\vartheta(t,x), \q v|_{\partial D}=0, \q  \vv (0)= \z'.  
\ee
The laws of the processes $\{\vv_t, t\in [0,1]\}$ and $\{\uu'_t, t\in [0,1]\}$ are denoted by~$\la(\z,\z')$ and $\la(\z')$, respectively.  We have the following estimate for the   total variation distance between $\la(\z,\z')$ and $\la(\z')$. \begin{proposition}\label{P:TVE}
There is an integer  $N_1\ge1$ such that, for any $N\ge N_1$,  $\es>0$, and  $\z, \z'\in \h$, we have   
\begin{equation}\label{eoejtnvf} 
| \la(\z,\z')-\la(\z')|_{var}\le    C_*\es^a+C_*\left[\exp\left(C_{N} \es^{a-2}|\z-\z'|_\h^2 e^{(|\ees(\z)|+|\ees(\z')|)} \right)-1\right]^{1/2},
\end{equation} where    $a<2$, $C_*$, and $C_N$  are   positive numbers not depending on $\es,  \z,$ and $\z'$.
\end{proposition} See Section \ref{S:6.2} for the proof of this proposition. 
  By Proposition~1.2.28 in~\cite{KS-book}, there is a probability space $(\hat \Omega, \hat \FF, \hat\pp)$ and   measurable functions $\VV , \VV' :\h\times \h\times \hat \Omega\to  C([0,1],\h) $ such that $(\VV (\z,\z'), \VV'(\z,\z'))$  is a maximal coupling for~$(\la(\z,\z'), \la(\z'))$ for any $\z,\z'\in \h$. We denote by~$\tilde \vv=[\tilde v_t, \p_t \tilde v]$ and $\tilde \uu'_t=[\tilde u'_t, \p_t \tilde u']$ the restrictions of $\VV$ and  $ \VV' $ to time $t\in [0,1]$. Then $\tilde v_t$ is a solution of the problem
  $$
 \p_t^2 \tilde v+\gamma \p_t \tilde v-\de \tilde v+f (\tilde v)-{\mathsf P}_N f (\tilde v)=h +\psi(t), \q \tilde v|_{\partial D}=0, \q  \tilde \vv (0)= \z',
$$
 where the process $\{\int_0^t\psi(\tau)\dd \tau, t\in [0,1]\}$ has the same law as 
 $$\left\{\xi(t)- \int_0^t {\mathsf P}_N f (u_\tau)\dd \tau, t\in [0,1]\right\}.$$
Let $\tilde \uu_t= [\tilde u, \p_t \tilde u]$ be a solution of 
$$
 \p_t^2 \tilde u+\gamma \p_t \tilde u-\de \tilde u+f (\tilde u)-{\mathsf P}_N f (\tilde u)=h +\psi(t), \q \tilde u|_{\partial D}=0, \q  \tilde \uu (0)= \z.  
$$
 Then $\{\tilde \uu_t, t\in[0,1]\}$ has the same law as $\{  \uu_t, t\in[0,1]\}$ (see Section~6.1 in~\cite{DM2014} for the proof).   Now the coupling   operators  $\RRR$ and $\RRR'$ are defined by 
$$
\RRR_t(\z,\z',\omega)=\tilde \uu_t,\quad \RRR'_t(\z,\z',\omega)=\tilde \uu_t', \q \z, \z'\in \h, \omega\in\hat\Omega.
$$ 
  By Proposition \ref{P:TVE}, if $N\ge N_1$, then   for any $\es>0$, we have 
\begin{align}\label{sdsflklk}
\hat \pp\{\exists t\in [0&,1]\text{ s.t. } \tilde \vv_t\neq \tilde \uu_t' \} \nonumber\\&\le C_* \es^a+C_*\left[\exp\left(C_{N} \es^{a-2}|\z-\z'|_\h^2 e^{(|\ees(\z)|+|\ees(\z')|)} \right)-1\right]^{1/2}.
\end{align}
 Let $(\Omega^k,\FF^k,\pp^k)$, $k\ge0$ be a sequence of independent copies of the
probability space $(\hat \Omega, \hat \FF, \hat\pp)$. We denote by~$(\Omega,\FF,\pp)$ the   direct product of the  spaces $(\Omega^k,\FF^k,\pp^k)$, and  for any~$\z,\z'\in \h$, $\omega=(\omega^1,\omega^2,\ldots)\in\Omega$, and $k\ge0$, we set~$\tilde u_0=u$, $\tilde u_0'=u'$,  and 
\begin{align*}
\tilde \uu_t(\omega)&=\RRR_\tau(\tilde \uu_{k}(\omega),\tilde \uu'_{k}(\omega),\omega^k), &\quad
\tilde \uu'_t(\omega)&=\RRR_\tau'(\tilde \uu_{k}(\omega),\tilde \uu'_{k}(\omega),\omega^k), \\\tilde \vv_t(\om)&=\VV_\tau(\tilde \uu_{k}(\omega),\tilde \uu'_{k}(\omega),\omega^k),
\end{align*} 
where $t=\tau+k, \tau\in [0,1)$.         We shall say that $(\tilde \uu_t,\tilde \uu_t')$ is a {\it coupled trajectory at level~$N$\/}  issued from~$(\z,\z')$.
   \subsection{The result and its proof} \label{S:3.2}

The following theorem   establishes the  uniform Feller property for the  semigroup~$\PPPP_t^V$   for any function  $V\in \UU_\delta$  with    sufficiently small $\delta>0$.   The   property is proved with respect to the   space    $\CC=\UU$ which is a determining family  for~$\ppp(\h)$ and
 contains   the constant functions.
 \begin{theorem} \label{T:3.1} There are positive numbers $\delta$  and  $R_0$ such that,  for any function~$V\in \UU_\delta$, the family $\{\|\PPPP_t^V{\mathbf1}\|_R^{-1}\PPPP_t^V\psi,t\ge1\}$ is uniformly equicontinuous on~$X_R$ for any~$\psi\in \UU$  and~$R\ge R_0$. 
\end{theorem}

\begin{proof}    To prove this result, we   develop the  arguments of the proof of
Theorem~6.2 in \cite{JNPS-2014}.  For any  $\delta>0$, $V\in \UU_\delta $, and $\psi\in \UU$,    we have 
$$
\PPPP_t^V\psi(\uu)=\e_{\uu}\bigl\{(\Xi_V\psi)(\uu_t, t)\bigr\},
$$
where  
\begin{equation}\label{E:7.2}
(\Xi_V\psi)(\uu_t, t):=\exp\biggl(\int_{0}^tV(\uu_\tau) \dd \tau\biggr)\psi(\uu_{t}).
\end{equation}
We   prove the uniform equicontinuity of the family $\{g_t,t\ge 1\}$ on~$X_R$, where 
$$
g_t(\uu)= \|\PPPP_t^V{\mathbf1}\|_R^{-1}\PPPP_t^V\psi(\uu). 
$$  
Without    loss of generality, we can assume  that  $0\le \psi\le 1$   and $\inf_H V=0$, so that     $\Osc_H(V)=\|V\|_\infty$. We can assume also that  the integer  $N$   entering representation \eqref{repres1} is the same for $\psi$ and $V$ and it is denoted by   $N_0$.

\medskip
  {\it Step~1:~Stratification}. Let us take any   $N\ge N_0$ and $\z,\z' \in X_R$ such that $d:=|\z-\z'|_\h\le1$, and   denote by $(\Omega,\FF,\pp)$  the probability space constructed in the previous subsection. Let us  consider a coupled trajectory~$(\uu_t,\uu_t'):=(\tilde \uu_t,\tilde \uu_t')$ at level~$N$ issued from $(\z,\z')$  and   the associated process~$\vv_t:=\tilde \vv_t$.   For any integers~$r\ge0$ and~$\rho\ge1$, we set\footnote{The event $\bar G_r$ is well defined also for $r=+\ty$.} 
\begin{align*}
\bar G_r &=\bigcap_{j=0}^rG_j, \,\,\,  G_j=\{  \vv_t =\uu_t', \forall t\in (j,j+1]\},  \,\,\,  F_{r,0}=\varnothing,\\
F_{r,\rho}&=\biggl\{ \sup_{\tau \in [0,r]} \left(\int_0^\tau \left(\|\nabla u_\tau \|^2+\|\nabla u'_\tau\|^2\right)\dd \tau-L\tau \right) \le |\ees(\z)|+ |\ees(\z')|  + \rho;  \\&  \q\q\q\q\q\q\q|\ees(\uu_r)|+ |\ees(\uu_r')| \le \rho \biggr\}, \end{align*}
where     $L$ is the constant in \eqref{bb2.14}. We also define the pairwise disjoint  events        
 \begin{align*}
A_0=G_0^c,  \quad A_{r,\rho}=\bigl(\bar G_{r-1}\cap G_r^c\cap F_{r,\rho}\bigr)\setminus F_{r,\rho-1},\,\, r\ge1,\rho\ge1, \q \tilde A = \bar G_{+\ty}.
\end{align*}Then, for any $t\ge1$,   we have 
 \begin{align} 
\PPPP_t^V\psi(\z)-\PPPP_t^V\psi(\z')
&= \e \bigl\{\I_{A_{0}}\bigl[(\Xi_V\psi )(\uu_t,t ) -(\Xi_V\psi )(\uu_t',t)\bigr] \bigr\}\nonumber\\&\quad+ \sum_{r,\rho=1}^{\ty} \e \bigl\{\I_{A_{r,\rho}}\bigl[(\Xi_V\psi )(\uu_t,t ) -(\Xi_V\psi )(\uu_t',t)\bigr] \bigr\}\nonumber\\&\quad+  \e \bigl\{\I_{\tilde A}\bigl[(\Xi_V\psi )(\uu_t,t ) -(\Xi_V\psi )(\uu_t',t)\bigr] \bigr\}\nonumber\\&=I_0^t(\z,\z')+ \sum_{r,\rho=1}^{\ty} I_{r,\rho}^{t}(\z,\z')+ \tilde I^{t}(\z,\z'),
\label{E:7.3}
\end{align}
where  
\begin{align*}
I_{0}^{t}(\z,\z')&:=\e \bigl\{\I_{A_{0}}\bigl[(\Xi_V\psi )(\uu_t,t ) -(\Xi_V\psi )(\uu_t',t)\bigr] \bigr\},\\
I_{r,\rho}^{t}(\z,\z')&:=\e \bigl\{\I_{A_{r,\rho}}\bigl[(\Xi_V\psi )(\uu_t,t ) -(\Xi_V\psi )(\uu_t',t)\bigr] \bigr\},\\
\tilde I^{t}(\z,\z')&:=\e \bigl\{\I_{\tilde A}\bigl[(\Xi_V\psi )(\uu_t,t ) -(\Xi_V\psi )(\uu_t',t)\bigr] \bigr\}.  
\end{align*}
To prove the uniform equicontinuity of~$\{g_t, t\ge1\}$, we first estimate these three quantities.

\medskip
{\it Step~2:~Estimates for  $I_{0}^t$ and $I_{r,\rho}^t$}.  Let $\delta_1>0$ and $R_0\ge1$ be the numbers in Proposition~\ref{P:2.4}. Then, if $\Osc(V)< \delta_1$ and $R\ge R_0$, we have the following estimates
\begin{align}
|I_{0}^t(\z,\z')|&\le C_1(R,V)\|\PPPP_t^V{\mathbf1}\|_R\,\pp\{A_{0}\}^{1/2}, \label{6.5aaa}\\
|I_{r,\rho}^t(\z,\z')|&\le C_2(R,V)e^{r\|V\|_\infty}\|\PPPP_t^V{\mathbf1}\|_R\,\pp\{A_{r,\rho}\}^{1/2}   \label{6.5}
\end{align} for any integers $r,\rho\ge1$. Let us prove \eqref{6.5}, the other estimate is similar. First assume that $r\le t$.
 Using  the inequalities $0\le \psi \le 1$, the positivity of~$\Xi_V\psi $, and the Markov property, we derive
\begin{align*} 
I_{r,\rho}^t(\z,\z')
&\le\e\bigl\{I_{A_{r,\rho}}(\Xi_V\psi )(\uu_t,t)\bigr\}
\le \e\bigl\{I_{A_{r,\rho}}(\Xi_{V}{\mathbf1})(\uu_t,t)\bigr\}\\
&= \e\bigl\{I_{A_{r,\rho}}\e\bigl[(\Xi_{V}{\mathbf1})(\uu_t,t)\,\big|\,\FF_r\bigr]\bigr\}
\le e^{r\|V\|_\infty}\e\bigl\{I_{A_{r,\rho}}(\PPPP_{t-r}^V{\mathbf1})(\uu_r)\bigr\},
\end{align*}
where $\{\FF_t\}$ stands for the filtration generated by~$(\uu_t,\uu_t')$. Then from \eqref{a5.8} it follows that
$$
\PPPP_{t-r}^{V}{\mathbf1}(\z)\le M\|\PPPP_{t-r}^{V}{\mathbf1}\|_{R_0}\wwww (\z),
$$
 so we have
\begin{align*}
I_{r,\rho}^t(\z,\z')
&\le C_3 e^{r\|V\|_\infty}\|\PPPP_{t-r}^V{\mathbf1}\|_{R_0}\e\bigl\{I_{A_{r,\rho}}\wwww (\uu_r)\bigr\}\\
&\le C_3 e^{r\|V\|_\infty}\|\PPPP_{t-r}^V{\mathbf1}\|_{R_0}\bigl\{\pp(A_{r,\rho})\,\e\,\wwww ^2(\uu_r)\bigr\}^{1/2}. 
\end{align*}
 Using this,~\eqref{e30}, and the symmetry, we obtain  \eqref{6.5}. If $r> t$, then
 \begin{align*} 
I_{r,\rho}^t(\z,\z')
& 
\le e^{r\|V\|_\infty}\pp\{A_{r,\rho}\bigr\}\le e^{r\|V\|_\infty} \|\PPPP_t^V{\mathbf1}\|_R\,\pp\{A_{r,\rho}\}^{1/2},
\end{align*}which implies \eqref{6.5} by symmetry.

\medskip
{\it Step~3:~Estimates for $\pp\{A_{0}\}$ and $\pp\{A_{r,\rho}\}$}.  Let us show that,  for sufficiently large  $N\ge1$, we have   \begin{align}
\pp\{A_{0}\} &\le C_4(R,N)d^{a/2},\label{E:7.7ffd}\\
\pp\{A_{r,\rho}\} &\le C_5(R) \left\{ \!  \Bigg(\!    d^ae^{-a\al r/2} \! + \!\left[\exp \left(C_6(R,N)d^a e^{2 \rho-a\al r/2}    \!  \right) -1\right]^{1/2}  \! \Bigg) \!  \wedge \!    e^{-\beta\rho} \!\right\},\label{E:7.7}
\end{align}
where   $a, C_*$, and $\beta$ are the constants in~\eqref{eoejtnvf} and \eqref{bb2.14}. Indeed, taking $\es=d$ in~\eqref{sdsflklk},     using \eqref{aa01}, and recalling that $d\le1$, we get
$$
\IP\bigl\{ A_0 \bigr\}
 \le C_*d^a+C_*\left[\exp\left(C_{N} d^{a} e^{ {C_7R^4}} \right)-1\right]^{1/2}\le C_4(R,N)d^{a/2}, 
$$provided that $N$ is larger that the number $ N_1$ in Proposition~\ref{P:TVE}. This  gives~\eqref{E:7.7ffd}.
 To show      \eqref{E:7.7}, 
  we use the   estimates 
  \begin{gather}
\e_\uu  \exp\left( \beta |\ees(\uu_t )|\right) \le C    \exp(\beta |\ees(\uu)|), \q \uu\in \h, \label{cc2.14}\\
\pp_\uu \left\{\sup_{t\geq 0}\left( \int_0^t \| \nabla u_\tau\|^2 \dd \tau-Lt\right)\geq |\ees(\uu)| +\rho\right\}\leq  Ce^{-\beta \rho},\q \rho>0,  \label{bb2.14}
\end{gather}  
where $L, \beta$, and $C$ are some positive constants depending on~$\gamma, \|h\|$, and $\BBB$;  they follow immediately from 
  Propositions~3.1 and~3.2 in~\cite{DM2014}.   
From the inclusion  $A_{r,\rho}\subset F_{r,\rho-1}^c$ and inequalities~\eqref{cc2.14}, \eqref{bb2.14}, and \eqref{aa01} it follows  that 
\begin{equation} \label{6.7}
\pp\{A_{r,\rho}\}\le  C_8(R) e^{-\beta\rho}.
\end{equation}
  By the Foia\c{s}--Prodi type  estimate (see \eqref{4.16} in Proposition \ref{4.13}),    there is $N_2\ge1$ such that   for any $N\ge N_2$    on the event~$\bar G_{r-1}\cap F_{r,\rho}$ we have
\be
|\uu_{r}-\uu_{r}'|_\h^2
\le \exp(-\alpha r + \rho+|\ees(\z)|+|\ees(\z')| )d^2\le C_9(R)  e^{-\alpha r +\rho}d^2, \label{FP:est}
\ee where we used \eqref{aa01}. Recall that on the same event we have also
 \be \label{FP:est1}
|\ees(\uu_r)| +|\ees(\uu_r') |
\le   \rho . 
\ee  So using the Markov property, \eqref{sdsflklk} with $\es=de^{-\al r/2}$, \eqref{FP:est1} and \eqref{FP:est},
    we obtain
\begin{align*} 
\IP\{A_{r,\rho}\}&\le\IP\bigl\{\bar G_{r-1}) \cap G_r^c \cap F_{r,\rho}\bigr\}
=\e\bigl\{\I_{\bar G_{r-1}\cap F_{r,\rho}}\e\bigl(I_{G_r^c}\,\big|\,\FF_{r}\bigr)\bigr\}
\notag\\
&\le  C_*d^ae^{-a\al r/2} +C_*\e\Big\{\I_{\bar G_{r-1}\cap F_{r,\rho}}  \\&\quad  \times \left[\exp\left(C_{N}    d^{a-2} e^{-(a-2)\al r/2}|\uu_r-\uu_r'|_\h^2 e^{(|\ees(\uu_r)|+|\ees(\uu_r')|)} \right)-1\right]^{1/2}\Big\}\notag\\
&\le  C_*d^ae^{-a\al r/2}+C_*\left[\exp\left(C_6(R,N)d^a e^{2 \rho-a\al r/2}   \right)-1\right]^{1/2}. 
\end{align*}Combining this with~\eqref{6.7} and choosing $N\ge N_1\vee N_2 $,   we get the required inequality~\eqref{E:7.7}.

\medskip
{\it Step~4:~Estimate for   $\tilde I^t$}. Let us show that, for any   $N\ge N_0$, we have 
\begin{equation}\label{E:eerrtt}
|\tilde I_{\rho}^{t}(\z,\z')|\le C_{10}(\psi, V)   \|\PPPP_t^V{\mathbf1}\|_R  d^q.
\end{equation}  Indeed, 
 we write 
 \begin{align} \label{eeedfgfkjk}
\tilde I^t(\z,\z')
&=\e\bigl\{\I_{\tilde A}(\Xi_V{\mathbf1})(\uu_t,t)[ \psi (\uu_t)-\psi (\uu_t')] \bigr\}\nonumber\\&\quad+\e\bigl\{\I_{\tilde A}[(\Xi_V{\mathbf1})(\uu_t,t)-(\Xi_V{\mathbf1})(\uu_t',t)]\psi (\uu_t') \bigr\}.
\end{align} Let us denote by $J_{1,\rho}^t$ and $J_{2,\rho}^t$ the expectations in the right-hand side of this equality. Then  by   estimate \eqref{FPE1}, 
on the event $\tilde A$ we have
\begin{equation}\label{FPestima}
|P_N(\uu_{\tau}-\uu_{\tau}')|_\h^2
\le e^{-\alpha \tau }d^2, \quad \tau\in [0,t].
\end{equation} Since $\psi\in C^q_b(\h) $, we derive from~\eqref{FPestima}
\begin{align*}
|J_{1,\rho}^t|&\le   \e\bigl\{\I_{\tilde A} (\Xi_V{\mathbf1})(\uu_t,t) | \psi (\uu_t)-\psi (\uu_t')| \bigr\} \le  \| \psi\|_{C^q_b}  e^{-\alpha t /2}d^q  \|\PPPP_t^V{\mathbf1}\|_R\\&\le \| \psi\|_{C^q_b}    \|\PPPP_t^V{\mathbf1}\|_Rd^q.
\end{align*}
Similarly, as $V\in C^q_b(\h)$,
\begin{align*}
|J_{2,\rho}^t|&\le  \e\bigl\{\I_{\tilde A}|(\Xi_V{\mathbf1})(\uu_t,t)-(\Xi_V{\mathbf1})(\uu_t',t)| \bigr\}\\&\le   \e\left\{\I_{\tilde A}(\Xi_V{\mathbf1})(\uu_t,t)\left[ \exp\left(\int_{0}^t|V(\uu_\tau)-V(\uu_\tau')| \dd \tau\right)-1 \right] \right\}\\&\le     \left[ \exp\left( \|V\|_{C^q_b}   d^q  (1-e^{-\al q t/2})\right)-1 \right] \|\PPPP_t^V{\mathbf1}\|_R\\&\le     \left[ \exp\left( \|V\|_{C^q_b}    d^q  \right)-1 \right] \|\PPPP_t^V{\mathbf1}\|_R.
\end{align*}
Combining these   estimates for $J_{1,\rho}^t$ and $J_{2,\rho}^t$ with \eqref{eeedfgfkjk}, we get \eqref{E:eerrtt}.

\medskip
{\it Step~5.} From~\eqref{E:7.3}--\eqref{E:7.7} and \eqref{E:eerrtt} it follows  that, for any      $\z,\z'\in X_R$, $t\ge1$,   and $R\ge R_0$,   we have
\begin{align*}
&\bigl|g_t(\z)-g_t(\z')\bigr|
\le C_{11}(R,V,N, \psi)   \Bigg(d^{a/4} + d^q \\&
 + \sum_{r,\rho=1}^\infty \!\! e^{r\|V\|_\infty} \!  \left\{   \!   \left(    d^{a/2}e^{-a\al r/4} + \left[\exp \left(C_6d^a e^{2  \rho-a\al r/2}      \right) -1\right]^{1/4}   \right)   \wedge    e^{-\beta\rho/2}  \right\} \!    
\Bigg),
\end{align*}   provided that   $N \ge N_0\vee  N_1\vee N_2 $.
 When $d=0$,  the   series in the right-hand side  vanishes. So  to prove the uniform equicontinuity of~$\{g_t\}$, it suffices to show that    the series converges uniformly in~$d\in[0,1]$. Since its terms are   positive and monotone, it suffices to show   the converge for~$d=1$:
\begin{align} 
 \sum_{r,\rho=1}^\infty   e^{r\|V\|_\infty}   \left\{     \left(     e^{-a\al r/4} + \left[\exp \left(C_6 e^{2  \rho-a\al r/2}   \right) -1\right]^{1/4}   \right)   \wedge    e^{-\beta\rho/2} \right\} < \infty. \label{E:7.5}
\end{align}
 To prove this,   we  will assume that~$\Osc(V)$ is sufficiently small. Let us consider the sets
$$
S_1=\{(r,\rho)\in\N^2:  \rho\le a\al r/8  \}, \quad S_2=\N^2\setminus S_1.
$$
Then taking    $\delta <\delta_1 \vee(a\alpha/32)$ and $\Osc(V)< \delta$, we see that 
\begin{align*}
\sum_{(r,\rho)\in S_1}&e^{r\|V\|_\infty}         \left(     e^{-a\al r/4} + \left[\exp \left(C_6 e^{2 \rho-a\al r/2}   \right) -1\right]^{1/4}   \right)\\&\le  C_{12}(R,N) \sum_{(r,\rho)\in S_1} e^{r\|V\|_\infty}  e^{-a\al r/16}\le C_{13}(R,N) \sum_{r=1}^\ty e^{-a\al r/32}<\infty.
\end{align*}
Choosing $\delta<a\al \beta/32 $, we get 
$$
\sum_{(r,\rho)\in S_2} e^{r\|V\|_\ty}e^{-\beta\rho/2}\le C_{14} \sum_{\rho=1}^\infty   e^{-\beta\rho/4}<\infty.
$$
These two  inequalities show that~\eqref{E:7.5} holds.

 \end{proof}

\section{Estimates for regular solutions}\label{S:4}
In this section, we establish the exponential tightness property     and obtain some higher order     moment estimates for solutions  in $\h^\sS$.

\subsection{Exponential tightness}\label{S:4.1}
Here we show that the exponential tightness property in Section~\ref{S:1.3} is verified for the function~$\varPhi(\uu)= |\uu|_{\h^\sS}^\vk$,   if we choose~$\varkappa>0$ sufficiently small. Clearly, the level sets of $\varPhi$ are compact in $\h$.  
 \begin{theorem}\label{0.4} For any $\sS<1/2$, there is   $\kp\in (0,1)$ such that, for any $R\ge1$, we have
\be\label{0.5}
\e_\vv\exp \left(\int_0^t |\uu_\tau |^\kp_{\h^\sS}\dd \tau\right)\leq c\,e^{ct}\q\text{ for any }  \vv\in X_R, t\geq 0,
\ee 
where $c$ is a positive constant  depending   on $R$. 
\end{theorem}
\bp 
It is sufficient to prove that   there is $\kp\in (0, 1)$ such that, for any $R\ge 1$, we have
\be\label{e42}
\e_\vv\exp \left(\De\int_0^t |\uu_\tau|^\kp_{\h^\sS}\dd \tau\right)\leq \tilde c\,e^{\tilde c t}\q\text{ for any }  \vv\in X_R, t\geq 0,
\ee 
where $\De$ and $\tilde c$ are positive constants depending   on $R$.
Indeed, once this is proved, we can use the inequality 
$$
|\uu|^{\f{\kp}{2}}_{\h^\sS}\leq \De|\uu|^{\kp}_{\h^\sS}+\De^{-1}
$$
to derive \ef{0.5}, where $\kp$ should be replaced by $\kp/2$.
We divide the proof of \ef{e42} into several steps.

\medskip

{\it Step~1: Reduction.} Let us split the flow $\uu(t)$ to the sum $\uu=\vv_1+\vv_2+\Zz$, where $\vv_1(t)=[v_1(t), \dt v_1(t)]$ corresponds to the flow of \eqref{0.1} with $f=h=\vartheta=0$ issued from $\vv$ and $\vv_2(t)=[v_2(t), \dt v_2(t)]$ is the flow of \eqref{0.1} with $f=0$  issued from the origin. Some standard arguments show that   the following a priori estimates hold:
\begin{gather}
|\vv_1(t)|^2_{\h^\sS}\leq |\vv|^2_{\h^\sS}e^{-\al t}, \label{0.6a}\\
\e\exp \left(\delta_1\int_0^t |\vv_2(\tau)|^2_{\h^\sS}\dd \tau\right)\leq c_1\,e^{c_1t}\q\text{ for any }   t\geq 0,\label{0.6b}
\end{gather}
where $\delta_1$ and $c_1$ are positive constants depending only on $\al, \BBB_1$, and $\|h\|_1$. Now using the Cauchy--Schwarz  inequality and \eqref{0.6a}, we get, for any $\delta<\delta_1/2$,
\begin{align*}
\e_\vv\exp\left(\De\int_0^t |\uu(\tau)|^\kp_{\h^\sS}\dd \tau\right)&\leq \exp\left(\!\De\int_0^t |\vv_1(\tau)|^\kp_{\h^\sS}\dd \tau\!\right) \! \e\exp\left(\!2\De\int_0^t |\vv_2(\tau)|^\kp_{\h^\sS}\dd \tau\!\right)\notag\\
&\q\times \e\exp\left(2\De\int_0^t |\Zz(\tau)|^\kp_{\h^\sS}\dd \tau\right)\notag\\
&\leq \! \exp\!\left(2\De R^\kp (\al \kp)^{-1}\right)\! \e\exp\left(\!2\De\int_0^t (|\vv_2(\tau)|^2_{\h^\sS}+1)\dd \tau\!\right) \notag\\
&\q\times \e\exp\left(2\De\int_0^t |\Zz(\tau)|^\kp_{\h^\sS}\dd \tau\right).\label{0.25}
\end{align*}
Combining this with \eqref{0.6b}, we see that inequality \eqref{e42} will be established if we prove that
\be\label{0.27}
\e\exp\left(\De\int_0^t |\Zz(\tau)|^\kp_{\h^\sS}\dd \tau\right)\leq c\,e^{c t}\q\text{ for all } t\geq 0 
\ee for some $\delta>0$ and $c>0$.
The rest of the proof is devoted to the derivation of this inequality.

\medskip

{\it Step~2: Pointwise estimates.} Let us note that, by construction, $\Zz$ is the flow of equation
\be\label{0.7}
\p_t^2 z+\gamma \p_t z-\de z+f(u)=0, \q z|_{\partial D}=0,\q [z(0),\dt z(0)]=0.
\ee
Let us differentiate this equation in time, and set $a=\dt z(t)$. Then $a$ solves
\be\label{0.8} 
\p_t^2 a+\gamma \p_t a-\de a+f'(u)\p_t u=0, \q a|_{\partial D}=0,\q [a(0),\dt a(0)]=[0,-f(u(0))].
\ee
We   write $\aA(t)=[a(t),\dt a(t)]$. Multiplying equation \ef{0.8} by~$2(-\de)^{\sS-1}(\dt a+\al a)$ and integrating over $D$, we obtain
\be\label{0.9}
\f{\dd}{\dd t}|\aA|_{\h^{\sS-1}}^2+\f{3\al}{2}|\aA|_{\h^{\sS-1}}^2\leq 2\int_D |f'(u)||\dt u||(-\de)^{s-1}(\dt a+\al a)|\dd x=\elll.
\ee
Let $\kp<1$ be a positive constant that will be fixed later. Then, by the triangle inequality, we have
\begin{align}
\f{\elll}{2}&\leq \int_D |f'(u)||\dt v_1|^{1-\kp}|\dt u|^\kp|(-\de)^{\sS-1}(\dt a+\al a)|\dd x\notag\\
&\q+ \int_D |f'(u)||\dt v_2|^{1-\kp}|\dt u|^\kp|(-\de)^{\sS-1}(\dt a+\al a)|\dd x\notag\\
&\q\q  +\int_D |f'(u)||a|^{1-\kp}|\dt u|^\kp|(-\de)^{\sS-1}(\dt a+\al a)|\dd x=\elll_1+\elll_2+\elll_3.\label{0.18}
\end{align}
Using the H\"older inequality, we derive
\begin{align}
\elll_1&\leq |f'(u)|_{L^{p_1}}|\dt v_1|_{L^{(1-\kp)p_2}}^{1-\kp}|\dt u|_{L^{\kp p_3}}^\kp |(-\de)^{\sS-1}(\dt a+\al a)|_{L^{p_4}},\label{0.10}\\
\elll_2 &\leq |f'(u)|_{L^{q_1}}|\dt v_2|_{L^{(1-\kp)q_2}}^{1-\kp}|\dt u|_{L^{\kp q_3}}^\kp |(-\de)^{\sS-1}(\dt a+\al a)|_{L^{q_4}}\label{0.11},\\
\elll_3 &\leq |f'(u)|_{L^{p_1}}|a|_{L^{(1-\kp)p_2}}^{1-\kp}|\dt u|_{L^{\kp p_3}}^\kp |(-\de)^{\sS-1}(\dt a+\al a)|_{L^{p_4}}\label{0.12},
\end{align}
where the exponents $p_i, q_i$ are H\"older admissible. We now need the following lemma, which is established in the appendix.
\begin{lemma}\label{0.14}
Let us take $p_1=6/\rho, p_3=2/\kp, q_1=(\rho+2)/\rho$ and $q_3=2/\kp$. Then, for $\kp>0$ sufficiently small, the exponents $p_2, p_4, q_2$ and $q_4$ can be chosen in such a way that we have the following embeddings:
\begin{alignat}{2}
H^\sS &\hookrightarrow L^{(1-\kp)p_2} , &\qquad H^{1-\sS} &\hookrightarrow L^{p_4} ,\label{0.31}\\ 
 H^1 &\hookrightarrow L^{(1-\kp)q_2} , &\qquad H^{1-\sS} &\hookrightarrow L^{q_4} .\label{0.32}
\end{alignat}
\end{lemma}

\medskip

{\it Step~3: Estimation of $\elll_1$ and $\elll_3$}. In view of Lemma~\ref{0.14} and inequalities~\eqref{1.8} and \eqref{0.10}, we have
\begin{align*}
\elll_1&\leq C_0 |f'(u)|_{L^{6/\rho}} \|\dt v_1\|_\sS^{1-\kp} \|\dt u\|^\kp \|(-\de)^{\sS-1}(\dt a+\al a)\|_{1-\sS}\notag\\
&\leq C_1  \|\dt v_1\|_\sS^{1-\kp}(\|u\|_1^\rho+1) \|\dt u\|^\kp \|\dt a+\al a\|_{\sS-1}.
\end{align*}
Now let us suppose that $\kp<2-\rho$. Then using   \eqref{0.6a} together with the Young inequality, we derive
\be\label{0.15}
\elll_1\leq  C_2|\vv|_{\h^\sS}^{1-\kp}(\|u\|_1^2+\|\dt u\|^2+C_\kp)  \|\dt a+\al a\|_{\sS-1}\leq C_3 \,R (\ees(\uu) +C_3) |\aA|_{\h^{\sS-1}}.
\ee
To  estimate $\elll_3$, we   again apply Lemma \ref{0.14} and inequalities \eqref{1.8} and \eqref{0.12} 
$$
\elll_3\leq C_4(\|u\|_1^\rho+1)\|a\|_{\sS}^{1-\kp}\|\dt u\|^\kp   \|\dt a+\al a\|_{\sS-1}\leq C_4(\|u\|_1^\rho+1)\|\dt u\|^\kp |\aA|_{\h^{\sS-1}}^{2-\kp}.
$$
Applying the Young inequality, we get
\be\label{0.17}
\elll_3\leq C_5(\ees(\uu)+C_{5})|\aA|_{\h^{\sS-1}}^{2-\kp}.
\ee

\medskip

{\it Step~4: Estimation of $\elll_2$}. It follows from Lemma \ref{0.14} and inequalities \eqref{1.5} and  \eqref{0.11} that
\begin{align*}
\elll_2 &\leq C_6 |f'(u)|_{L^{(\rho+2)/\rho}} \|\dt v_2\|_1^{1-\kp} \|\dt u\|^\kp \|(-\de)^{\sS-1}(\dt a+\al a)\|_{1-\sS}\label{0.11}\notag\\
&\leq C_7 \|\dt v_2\|_1^{1-\kp} \left(\int_D (F(u)+\nu u^2+C)\dd x\right)^{\rho/{\rho+2}} \|\dt u\|^\kp \|\dt a+\al a\|_{\sS-1}\\
&\leq C_8 \|\dt v_2\|_1^{1-\kp} \left(\ees(\uu)+C_8\right)^{\rho/{\rho+2}} \|\dt u\|^\kp |\aA|_{\h^{\sS-1}}.
\end{align*}
Finally, applying the Young inequality, we obtain
\be\label{0.16}
\elll_2\leq C_9 (\ees(\uu)+|\vv_2|_{\h^s}^2+C_{9})|\aA|_{\h^{\sS-1}}.
\ee

\medskip

{\it Step~5: Estimation of $|\aA|_{\h^{\sS-1}}$}. Combining inequalities \eqref{0.9}, \eqref{0.18} and~\eqref{0.15}-\eqref{0.16}, we see that
\be\label{0.19}
\f{\dd}{\dd t}|\aA(t)|_{\h^{\sS-1}}^2+\alpha|\aA(t)|_{\h^{\sS-1}}^2\leq C_{10}\, R \left(\ees(\uu(t))+|\vv_2(t)|_{\h^s}^2+C_{10}\right)\left(|\aA(t)|_{\h^{\sS-1}}^{2-\kp}+1\right).
\ee
We now need an auxiliary result, whose proof is presented in the appendix.
\begin{lemma}\label{0.20}
Let $x(t)$ be an absolutely continuous nonnegative function   satisfying the differential inequality
\be\label{0.21}
\dt x(t)+\al x(t)\leq g(t)x^{1-\beta}(t)+b(t)\q\text{ for all } t\in [0, T],
\ee
where   $\al, T$, and $\beta<1$ are positive constants and $g(t)$ and $b(t)$ are nonnegative functions integrable on $[0, T]$. Then we have
\be\label{0.28}
\f{\al}{2}\int_0^t x^\beta(\tau)\dd \tau\leq \beta^{-1}(1+x(0))^{\beta}+\int_0^t (\al+g(\tau)+b(\tau))\dd \tau\q\text{ for } t\in [0, T].
\ee
\end{lemma}
Applying this lemma to inequality \eqref{0.19}, we obtain
\begin{align}
\f{\al}{2} \int_0^t |\aA(\tau)|_{\h^{\sS-1}}^\kp\dd\tau&\leq 2\kp^{-1}(1+|\aA(0)|_{\h^{\sS-1}}^2)^{\kp/2}+\al t\notag\\
&\q+2\,C_{10}\, R \int_0^t \left(\ees(\uu(\tau))+|\vv_2(\tau)|_{\h^s}^2+C_{10}\right)\Dd\tau.\label{0.22}
\end{align}

\medskip

{\it Step~6: Completion of the proof}. Note that
$$
|\Zz |_{\h^\sS}^2=\|z \|_{\sS+1}^2+\|\dt z +\alpha z \|_\sS^2=\|\de z \|_{\sS-1}^2+\|a+\alpha z \|_\sS^2.
$$
On the other, in view of \eqref{0.7}, we have
$$
\|\de z \|_{\sS-1}^2=\|\dt a +\gamma a +f(u )\|_{\sS-1}^2\leq C_{11}(|  \aA   |_{\h^{\sS-1}}^2+\|f(u )\|^2),
$$
whence we get
\be\label{e35}
|\Zz |_{\h^\sS}^2\leq C_{12}\left(|\aA |_{\h^{\sS-1}}^2+\ees^3(\uu )+C_{12}\right).
\ee
It follows that
$$
|\Zz |_{\h^\sS}^\kp\leq C_{13}\left(|\aA |_{\h^{\sS-1}}^{\kp}+ \ees (\uu )+C_{13}\right),
$$
provided $\kp<2/3$. Multiplying this inequality by $\al/2$, integrating over $[0, t]$ and using \eqref{0.22} together with the fact that 
\be\label{e38}
|\aA(0)|_{\h^{\sS-1}}^2=\|f(u(0))\|_{\sS-1}^2\leq \|f(u(0))\|^2\leq C_{14}(\|\vv\|_1^6+1),
\ee
we derive
$$
\f{\al}{2} \int_0^t |\Zz(\tau)|_{\h^\sS}^{\kp}\dd\tau \leq C_{15} \left(1+   \int_0^t \left[\ees(\uu(\tau))+|\vv_2(\tau)|_{\h^\sS}^2+C_{15}\right]\Dd\tau \right),
$$where $C_{15}$ depends on  $R$.
 Multiplying this  inequality   by a small constant $\De(R)>0$, taking the exponent  and then the expectation, and using  
 \eqref{0.6b} together with Proposition 3.2 in~\cite{DM2014}, we   
 derive \eqref{0.27}.
\ep

\subsection{Higher moments of regular solutions}\label{S:moments}

 For any $m\geq 1$, let   $\we_m$ and $\tilde \we_m$ be the functions given by \ef{e27} and \eqref{e32}. The following result shows that  they are both   Lyapunov functions   for the   trajectories of problem   \eqref{0.1},~\eqref{0.3}.
\begin{proposition}
For any $\vv\in\h^\sS$, $m\ge1$, and $t\ge0$, we have
\begin{align}
\e_\vv\we_m(\uu_t)&\leq 2e^{-\al m t}\we_m(\vv)+C_m \label{e30},\\
\e_\vv\tilde\we_m(\uu_t )&\leq 2e^{-\al m t}\tilde\we_m(\vv)+C_m.\label{e34}
\end{align}
\end{proposition}
\bp
{\it Step~1: Proof of \eqref{e30}}.   
We split the flow $\uu(t;\vv)$ to the sum $\uu(t;\vv)=\tilde \uu(t)+\Zz(t)$, where $\tilde\uu$ is the flow issued from $\vv$ corresponding to the solution of~\ef{0.1} with $f=0$. Let us note that here $\Zz=[z, \dt z]$ is the same as in   Section \ref{S:4.1}. A standard argument shows that
\be\label{e18}
\e|\tilde\uu(t)|_{\h^\sS}^{2m}\leq e^{-\al m t}|\vv|_{\h^\sS}^{2m}+C(m, \|h\|_1, \BBB_1).
\ee
As in Section \ref{S:4.1}, we set $a =\dt z $ and write $\aA =[a , \dt a ]$. Notice that thanks to the H\"older inequality, the Sobolev   embeddings $H^{1}\hookrightarrow L^{6}$ and~$H^{1-\sS}\hookrightarrow L^{6/(3-\rho)}$ for $\sS<1-\rho/2$, and inequality \eqref{7.9}, we can  estimate the right-hand side of inequality \ef{0.9} by 
\begin{align*}
\elll&\leq C_1(|u|^\rho_{L^6}+1)\|\dt u\| |(-\de)^{\sS-1}(\dt a+\al a)|_{L^{6/(3-\rho)}}\\
&\leq C_2(\|u\|^2_{1}+1)\|\dt u\| \|(-\de)^{\sS-1}(\dt a+\al a)\|_{ {1-\sS}}\leq C_3(|\uu|_\h^3+1)\|\dt a+\al a\|_{\sS-1}\\
&\leq \f{\al}{4}|\aA |_{\h^{\sS-1}}^2+C_4\left(\ees^3(\uu )+C_4\right).
\end{align*}
   Combining this with \ef{0.9}, we infer
$$
\f{\dd}{\dd t}|\aA |_{\h^{\sS-1}}^2\leq -\f{5\al}{4}|\aA |_{\h^{\sS-1}}^2+C_4\left(\ees^3(\uu )+C_4\right).
$$
It follows that\footnote{All the constants $C_i, i\ge 5$ depend on $m$. }
$$
\f{\dd}{\dd t}|\aA |_{\h^{\sS-1}}^{2m}=m|\aA |_{\h^{\sS-1}}^{2m-2}\f{\dd}{\dd t}|\aA |_{\h^{\sS-1}}^2
\leq-\al m|\aA |_{\h^{\sS-1}}^{2m}+C_5 \left(\ees^{3m}(\uu )+C_5\right),
$$
where we used the Young inequality. Taking the mean value in this inequality and applying the comparison principle, we derive
$$
\e|\aA(t)|_{\h^{\sS-1}}^{2m}\leq e^{-\al m t}|\aA(0)|_{\h^{\sS-1}}^{2m}+ C_6  \int_0^te^{\al m (\tau-t)}  \left(\e\ees^{3m}(\uu(\tau))+ C_6\right) \dd \tau .
$$
Combining this with \ef{e35} and \ef{e38}, we get
\begin{align*}
\e|\Zz(t)|_{\h^\sS}^{2m}&\leq C_7 \left(e^{-\al m t} \ees^{3m}(\vv)  +  \int_0^te^{\al m(\tau-t)}\e \ees^{3m}(\uu(\tau))\dd \tau+C_7\right) .
\end{align*}
Using the It\^o formula, it is not difficult to show (cf. Proposition 3.1 in \cite{DM2014}) that 
\be\label{e31}
\e \ees^{k}(\uu(t))\leq \exp(-\al k t)\ees^{k}(\vv)+C(k, \|h\|,\BBB)\q\text{ for any }k\geq 1.
\ee
It follows from the last two inequalities that
$$
\e|\Zz(t)|_{\h^\sS}^{2m}\leq C_8 (e^{-\al m t} \ees^{3m}(\vv)+C_8).
$$
Combining this with the inequality 
$$
(A+B)^{2m}\leq 2 A^{2m}+  C_9 B^{2m}\q\text{ for any } A, B\ge 0.
$$ and  \ef{e18}, we infer
\begin{align*}
\e |\uu(t)|_{\h^\sS}^{2m}&\leq\e (|\tilde\uu(t)|_{\h^\sS}+|\Zz(t)|_{\h^\sS})^{2m}\leq 2 \e|\tilde\uu(t)|_{\h^\sS}^{2m}+  C_{9} \e|\Zz(t)|_{\h^\sS}^{2m}\notag\\
&\le 2e^{-\al m t}|\vv|_{\h^\sS}^{2m}+C_{10}  (e^{-\al m t}\ees^{3m}(\vv)+C_{10}).
\end{align*}
So that we have
\begin{align*}
\e \we_m(\uu(t))&\leq 2e^{-\al m t}|\vv|_{\h^\sS}^{2m}+C_{10} (e^{-\al m t} \ees^{3m}(\vv)+C_{10})  +\e \ees^{4m}(\uu(t))\notag\\
&\leq 2e^{-\al m t}\left(|\vv|_{\h^\sS}^{2m}+\ees^{4m}(\vv)\right)+C_{11} =2e^{-\al m t}\we_m(\vv)+C_{11} ,
\end{align*}
where we used the Young inequality together with \ef{e31}. 

\medskip
{\it Step~2: Proof of~\eqref{e34}}.
   It was shown in Section~3.2 of~\cite{DM2014}, that for any~$\kp\leq (2\al)^{-1}\BBB$, we have
\begin{align*}
\e_\vv\exp[\kp\ees(\uu(t))]&\leq\exp(\kp\ees(\vv))\\
&\q+\kp\int_0^t \e_\vv\exp[\kp\ees(\uu(\tau))](-\al\ees(\uu(\tau))+C(\BBB, \|h\|))\dd \tau.
\end{align*}
Using this with inequality
$$
 e^r(-\al r+C)\leq -\al m e^{r}+C_{12} \q\text{ for any }r\ge-C
$$
and applying the Gronwall lemma, we see that
$$
\e_\vv\exp[\kp\ees(\uu(t))]\leq e^{-\al m t}\exp(\kp\ees(\vv))+C_{13} .
$$
Finally, combining this inequality with \ef{e30}, we arrive at \ef{e34}.
\ep

\section{Proof of  Theorem \ref{T:1.1}}\label{S:5}

The results of Sections~\ref{S:2}-\ref{S:4} imply that    the growth conditions, the uniform irreducibility and uniform Feller properties  in     Theorem~\ref{T:5.3} are satisfied if we take  
\begin{gather*}
X=\h, \,\,\, X_R= B_{\h^\sS}(R) ,  \,\,\, P_t^V(\uu,\Gamma)=(\PPPP_t^{V*} \delta_\uu ) (\Gamma),\,\,\, \\    \wwww(\uu)=1+|\uu|_{\h^\sS}^2+ \ees^4(\uu) 
,\,\,\, \CC= \UU, \,\,\, V\in \UU_\delta
\end{gather*}
for sufficiently large integer  $R_0\ge1$, small $\delta>0$, and     any $\sS\in(0,1-\rho/2)$. Let us  show that the time-continuity property is also verified.

\medskip
{\it Step~1:    Time-continuity property}.
 We need to show that the function $t\mapsto\PPPP^V_tg(\uu) $ is continuous from~$\R_+$ to $\R$ for any 
  $g\in C_\wwww(\h^\sS)$ and $\uu\in\h^\sS$ (recall that $X_\infty=\h^\sS$). For any $T,t\ge0$ and $\uu\in \h^\sS$, we have
  \begin{align}\label{S6:1}
  \PPPP^V_Tg(\uu)-\PPPP^V_tg(\uu)&=\e_\uu \left\{\left[\Xi_V(T)-\Xi_V(t)  \right] g(\uu_t)\right\}+ \e_\uu\left\{\left[ g(\uu_T)- g(\uu_t)  \right] \Xi_V(T)\right\}\nonumber\\&=: S_1+S_2,
  \end{align}
where $\Xi_V$ is defined by \eqref{S6ogt}. As $V$ is bounded and $g\in C_\wwww(\h^\sS) $, we see that 
\begin{align*} 
|S_1|&\le  \e_\uu \left\{\left|\exp\left(\int_t^TV(\uu_\tau)\dd \tau\right)- 1 \right| \Xi_V(t)|g(\uu_t)|\right\}\nonumber\\&\le C_1 \left(e^{|T-t|\|V\|_\infty}-1\right)e^{T\|V\|_\infty} \e_\uu \wwww(\uu_t).
\end{align*}
Combining this with \eqref{e30}, we get   $S_1\to0$ as $t\to T$. To  estimate $S_2$, let us take any $R>0$ and write
\begin{align*} 
e^{-T\|V\|_\infty} | S_2|&\le \e_\uu \left| g(\uu_T)- g(\uu_t)  \right| \\&= \e_\uu \left\{\I_{G_R^c} \left| g(\uu_T)- g(\uu_t)  \right| \right\}+ \e_\uu \left\{\I_{G_R} \left| g(\uu_T)- g(\uu_t)  \right| \right\}\\&=: S_3+S_4, 
\end{align*}where $G_R:=\{u_t, u_T\in X_R\}$. From the Chebyshev inequality, the fact that $g\in C_\wwww(\h^\sS) $, and inequality \eqref{e30}  we derive 
\begin{align*} 
S_3 &\le C_1 \e_\uu \left\{\I_{G_R^c}  (\wwww(\uu_T)+ \wwww(\uu_t)  ) \right\} \\&\le C_1 R^{-2} \e_\uu \left\{ \wwww^2(\uu_T)+ \wwww^2(\uu_t)   \right\}  \le C_2 R^{-2}   \wwww^2(\uu).   
\end{align*} 
On the other hand,  by the Lebesgue theorem on dominated convergence,  for any $R>0$,  we have
$S_4 \to 0$ as $t\to T$. Choosing $R>0$ sufficiently large and $t$ sufficiently close to $T$, we see that $S_3+S_4$ can be made arbitrarily small. This shows that $S_2\to0$ as $t\to T$ and proves the time-continuity property.

\medskip
{\it Step~2: Application of Theorem~\ref{T:5.3}}.
     We   conclude from Theorem~\ref{T:5.3} that there is an eigenvector $\mu_V\in \ppp(\h)$ for the semigroup~$\PPPP_t^{V*}$ corresponding to some positive eigenvalue $\la_V$, i.e.,  $\PPPP_t^{V*}\mu_V=\la_V^t\mu_V$ for any~$t>0$. Moreover, the semigroup~$\PPPP_t^{V}$
      has an eigenvector $h_V\in C_\wwww(\h^\sS)\cap C_+(\h^\sS)$ corresponding to~$\lambda_V$ such that $\lag h_V, \mu_V\rag=1$. 
     The uniqueness of $\mu_V$ and $h_V$ follows immediately from~\eqref{a1.5} and~\eqref{a1.6}.  The uniqueness of $\mu_V$  implies that it does not depend on~$m$ and  \eqref{momentestimate} holds for any $m\ge1$.
It remains to prove limits~\eqref{a1.5} and~\eqref{a1.6}.

\medskip
{\it Step~3: Proof of~\eqref{a1.5}}.
By~\eqref{5.12}, we have \eqref{a1.5} for any $\psi\in \UU$.  To establish the   limit   for any $\psi \in C_\wwww(\h^\sS)$, we apply an approximation argument similar to the one used in Step 4 of the proof of Theorem~5.5 in~\cite{JNPS-2014}. Let us   take a sequence~$\psi _n  \in \UU$ such that $\|\psi _n\|_\ty\le \|\psi \|_\ty$ and~$\psi _n \to \psi $ as $n\to\ty$, uniformly on bounded
subsets of $\h^\sS$.  If we define
$$
\Delta_t(g)=\sup_{\uu\in X_R}\bigl|\lambda_V^{-t}\PPPP_t^Vg(\uu)-\langle g,\mu_V\rangle h_V(\uu)\bigr|,
\quad \|g\|_{_R}=\sup_{\uu\in X_R}|g(\uu)|,
$$then 
$$
\Delta_t(\psi )\le \Delta_t(\psi _n)+\|h_V\|_{R}\,|\langle \psi -\psi _n,\mu_V\rangle|+\lambda_V^{-t}\|\PPPP_t^V(\psi -\psi _n)\|_{R}
$$for any $t\ge 0$ and $n\ge1$.
 In view of~\eqref{a1.5} for $\psi _n$ and the Lebesgue theorem on dominated convergence,  
\begin{gather*}
\Delta_t(\psi _n) \to0\quad\mbox{as $t\to\infty$ for any fixed $n\ge1$},\\
|\langle \psi -\psi _n,\mu_V\rangle|\to0\quad\mbox{as $n\to\infty$}. 
\end{gather*}
Thus, it suffices to show that
\begin{equation} \label{6.57}
\sup_{t\ge 0}\lambda_V^{-t}\|\PPPP_t^V(\psi -\psi _n)\|_{R}\to0\quad\mbox{as $n\to\infty$}.
\end{equation}
 To   this end, for any $\rho>0$, we write
$$
\|\PPPP_t^V(\psi -\psi _n)\|_{R}\le J_1(t,n,\rho)+J_2(t,n,\rho),
$$
where 
$$
J_1(t,n,\rho)=\|\PPPP_t^V\bigl((\psi -\psi _n)\I_{X_\rho}\bigr)\bigr\|_{R}, \quad 
J_2(t,n,\rho)=\|\PPPP_t^V\bigl((\psi -\psi _n)\I_{X_\rho^c}\bigr)\|_{R}. 
$$
Since $\psi _n\to \psi $ uniformly on $X_\rho$, we have 
$$
J_1(t,n,\rho)\le \es(n,\rho)\,\|\PPPP_t^V\mathbf1\|_{R},
$$
where $\es(n,\rho)\to0$ as $n\to\infty$. 
Using convergence \eqref{a1.5} for $\psi=\mathbf1$, we see that
\be\label{a234}
\lambda_V^{-t}\|\PPPP_t^V{\mathbf1}\|_R \le C_3(R)\quad\mbox{for all $t\ge0$}.
\ee
Hence,  
$$
\sup_{t\ge 0}\lambda_V^{-t}J_1(t,n,\rho)\le C_3(R)\,\es(n,\rho)\to0\quad\mbox{as $n\to\infty$}. 
$$We use~\eqref{a5.8} and \eqref{a234},  to estimate $J_2$:
\begin{align*}
\la_V^{-t}J_2(t,n,\rho)
&\le 2\|\psi \|_\infty \rho^{-2} \la_V^{-t} \|\PPPP_t^V \wwww \|_{R}\le C_{4}(R) \|\psi \|_\infty \rho^{-2} \la_V^{-t} \|\PPPP_t^V \mathbf1 \|_{R_0} \\
&\le  C_{4}(R) \|\psi \|_\infty \rho^{-2} C_3({R_0}).  
\end{align*}
Taking  first $\rho$ and then~$n$ sufficiently large, we see that $\sup_{t\ge 0}\lambda_V^{-t}\|\PPPP_t^V(\psi -\psi_n)\|_{R}$
   can be made arbitrarily small.  This proves~\eqref{6.57} 
and  completes the proof of~\eqref{a1.5}.

\medskip
{\it Step~4: Proof of~\eqref{a1.6}}. Let us   show that
$$
\lambda_V^{-t}\lag\PPPP_t^V \psi, \nu\rag\to\lag \psi,\mu_V\rag \lag h_V, \nu\rag\q\text{as $t\to\infty$}
$$ for any $\psi\in C_b(\h)$.
 In view of \eqref{a1.5}, it suffices to show that
\begin{equation} \label{6.63}
 \sup_{t\ge0}\left\{\int_H\I_{X_R^c}\bigl|\lambda_V^{-t}\PPPP_t^V \psi (\uu)-  \langle \psi,\mu_V\rangle h_V(\uu)\bigr|\, \nu(\Dd\uu) \right\} \to0\,\,\,
 \mbox{as $R\to\infty$}.
\end{equation}
From~\eqref{6.0015} and~\eqref{a234}  we derive that
$$
\|\PPPP_t^V \psi\|_{L_\wwww^\infty}\le \| \psi \|_\infty \|\PPPP_t^V \mathbf1\|_{L_\wwww^\infty}\le C_5\|\PPPP_t^V{\mathbf1}\|_{R_0}\le C_6(R_0)\lambda_V^t,\quad\mbox{  $t\ge0$},
$$
hence
$$
\bigl|\lambda_V^{-k}\PPPP_t^V \psi(\uu)\bigr|\le C_6(R_0)\wwww(\uu), \quad \uu\in \h^\sS, \quad t\ge0. 
$$
Since $h_V\in C_\wwww(\h^\sS)$ and 
$$
 \int_\h\I_{X_R^c}(\uu)\,\wwww(\uu)\,\nu(\Dd \uu) \to0\quad\mbox{as $R\to\infty$}, 
$$
we obtain~\eqref{6.63}.
This completes the proof of Theorem~\ref{T:1.1}.

\section{Appendix}
\label{S:6}

 \subsection{Local version of Kifer's theorem}
\label{s4}

In \cite{kifer-1990},  Kifer  established a sufficient condition for the validity of the LDP for a family  of  random probability measures on a   compact metric space. This result was  extended by     Jak$\check{\rm s}$i\'c et al.~\cite{JNPS-2014} to the case of a general Polish   space.  
 In this section, we  obtain  a local version of these results. Roughly speaking,  we assume  the existence of a pressure function (i.e., limit~\eqref{2.1a}) and the uniqueness of the equilibrium state for functions $V$ in a set $\VV$, which is not necessarily dense in the space of bounded continuous functions.  We prove the LDP with a     lower bound in which the     infimum of the rate function is   taken  over  a subset of the equilibrium states.
    To give the exact formulation of  the result, we first introduce some notation and  definitions. Assume that~$X$ is a Polish space,  and~$\zeta_\theta$  is a  random probability measure on~$X$   defined on some probability space~$(\Omega_\theta, \FF_\theta, \pp_\theta)$, where the index $\theta$ belongs to some       directed set\footnote{i.e.,  a partially ordered set  whose every finite subset  has an upper bound.}~$\Theta$. Let $r:\Theta\to \R$ be a   positive function such that $\lim_{\theta\in\Theta} r_\theta=+\ty$. For any~$V\in C_b(X)$, let us set
 \begin{equation}\label{2.1}
Q(V):=\limsup_{\theta\in\Theta} \frac{1}{r_\theta}\log\e_\theta 
\exp\bigl(r_\theta\lag V, \zeta_\theta\rag\bigr),
\end{equation} 
where~$\e_\theta$ is the expectation with respect to $\pp_\theta$. 
The function~$Q:C_b(X)\to \R$ is  convex,~$Q(V)\ge0$ for any~$V\in C_+(X)$, and~$Q(C)=C$ for any~$C\in\R$. Moreover,~$Q$ is 1-Lipschitz. Indeed, for any $V_1,V_2\in C_b(X)$ and $\theta\in \Theta$, we have
$$
\frac{1}{r_\theta}\log\e_\theta 
\exp\bigl(r_\theta\lag V_1, \zeta_\theta\rag\bigr) \le \|V_1-V_2\|_\infty+\frac{1}{r_\theta}\log\e_\theta 
\exp\bigl(r_\theta\lag V_2, \zeta_\theta\rag\bigr),
$$which implies that 
$$
Q(V_1)\le \|V_1-V_2\|_\infty+Q(V_2).
$$By symmetry    we get
$$
|Q(V_1)-Q(V_2)|\le \|V_1-V_2\|_\infty.
$$

The {\it Legendre transform\/} of~$Q$ is given  by 
\begin{equation} \label{2.2}
I(\sigma)=\begin{cases} \sup_{V\in C_b(X)}\bigl(\lag V, \sigma\rag-Q(V)\bigr)  & \text{for $\sigma \in {\cal P}(X)$},  \\ +\infty  & \text{for   $\sigma \in \MM(X)\setminus {\cal P}(X)$}  \end{cases}
\ee
  (see Lemma 2.2 in \cite{BD99}). Then~$I$ is  convex  and  lower semicontinuous function, and 
$$
Q(V)= \sup_{\sigma\in\ppp(X)} \bigl(\lag V, \sigma\rag-I(\sigma)\bigr).
$$A
  measure~$\sigma_V\in\ppp(X)$ is said to be an {\it equilibrium state\/} for~$V$ if
$$
Q(V)= \lag V,\sigma_V\rag-I(\sigma_V).
$$
 We shall denote by~$\VV$ the set of functions~$V\in C_b(X)$  admitting a unique equilibrium state  $\sigma_V$ and for which  the following limit exists
 \begin{equation}\label{2.1a}
Q(V)=\lim_{\theta\in\Theta} \frac{1}{r_\theta}\log\e_\theta 
\exp\bigl(r_\theta\lag V, \zeta_\theta\rag\bigr).
\end{equation}   We have the following version of
Theorem~2.1 in~\cite{kifer-1990} and    Theorem~3.3 in~\cite{JNPS-2014}. 
\begin{theorem}\label{T:2.3}
Suppose that  there is a function $\varPhi : X  \to [0,+\ty]$   whose level sets $ \{u \in X : \varPhi(u)\le a\}$ are   compact for all $a\ge0$  and 
\begin{equation}\label{2.4}
\e_\theta\exp\bigl(r_\theta \lag \varPhi,\zeta_\theta\rag \bigr)\le 
Ce^{c r_\theta}\quad \text{for $\theta\in \Theta$},
\end{equation} for some positive constants  $C$ and $c$.
 Then $I$ defined by~\eqref{2.2} is  a good rate function,   for any closed set $F\subset\ppp(X) $, 
\be 
\limsup_{\theta\in\Theta} \frac{1}{r_\theta}\log\pp_\theta\{\zeta_\theta\in   F\}\le - I (  F),\label{2.6U}
\ee
 and for any open set   $G\subset\ppp(X) $,
 \be
\liminf_{  \theta\in\Theta} \frac1{r_\theta}\log\pp_\theta\{\zeta_\theta\in  G \}\ge-I (\W \cap   G ),   \label{2.6L}
\ee where $\W := \{\sigma_V: V\in \VV\}$ and $I(\Gamma):=\inf_{\sigma\in \Gamma} I (\sigma)$, $\Gamma\subset \ppp(X)$.
\end{theorem}

 \begin{proof}  The fact that $I$ is a good rate function is shown  in Step 1  of the proof of Theorem~3.3 in~\cite{JNPS-2014}.   In Step 2 of the same proof,  the upper bound~\eqref{2.6U} is established,  under the condition that  the limit $Q(V)$ in~\eqref{2.1a} exists for any $V\in C_b(X)$. The latter condition can be removed, using literally the same proof, if one defines    $Q(V)$   by \eqref{2.1} for any $V\in C_b(X)$ (see Theorem~2.1 in~\cite{Ac85}). 

\smallskip
    To prove the lower bound,  
following  the ideas of~\cite{kifer-1990}, for   any integer~$n\ge1$ and any functions $V_1, \ldots, V_n\in  C_b(X)$, we 
  define an  auxiliary family    of finite-dimensional   random variables    $\zeta_\theta^n:= f_n(\zeta_\theta)$, where $f_n: \ppp(X)\to \R^n$ is given by      
$$
f_n(\mu):=\bigl(\lag V_1,\mu\rag, \ldots, \lag V_n,\mu\rag\bigr).
$$ Let us set 
$$
\W_n:=\{\sigma_V: V\in \VV\cap \textup{span}\{V_1, \ldots, V_n\}\}.
$$
 The following result is a local version   of Lemma~2.1 in~\cite{kifer-1990} and Proposition~3.4 in~\cite{JNPS-2014}; its proof is   sketched  at the end of this section.

\begin{proposition} \label{L:2.4} 
Assume that the hypotheses of Theorem~\ref{T:2.3} are satisfied and set $J_n(\Gamma)=\inf_{\sigma\in f_n^{-1}(\Gamma)} I(\sigma), \Gamma \subset \R^n$. Then 
  for any closed set $M\subset\R^n $ and open set $U\subset \R^n $, we have 
\begin{align} 
   \limsup_{\theta\in\Theta}\frac1{r_\theta}\log\pp\{ \zeta_\theta^n \in   M\}&\le -J_n( M  ),\label{9.01}\\
 \liminf_{\theta\in\Theta} \frac1{r_\theta}\log\pp\{  \zeta_\theta^n \in   U\} &\ge -J_n(f_n(\W_n)\cap   U ).\label{9.32}
\end{align}
\end{proposition}To derive \eqref{2.6L} from Proposition~\ref{L:2.4}, we follow the arguments of Step~4 of the proof of Theorem~3.3 in~\cite{JNPS-2014}.
The case $I(\W\cap G)=+\ty$ is trivial, so we  assume that $I(\W\cap G)<+\ty$. Then for  any $\es>0$,   there is $\nu_\es\in \W\cap G$ such that 
\begin{equation}\label{2.20}
I(\nu_\es)\le I(\W\cap G)+\es,
\end{equation} and there is a function      $V_1\in \VV$ such that    $\nu_\es=\sigma_{V_1}$. By Lemma~3.2 in~\cite{JNPS-2014}, the family $\{\zeta_\theta\}$ is exponentially tight, hence there is a compact set
 $\KK \subset\ppp(X)$ such that $\nu_\es\in \KK $ and 
\begin{equation}\label{zzerfs}
\limsup_{\theta\in \Theta}\frac{1}{r_\theta}\log \pp\{\zeta_\theta \in \KK ^c\}\le -(I(\W\cap G)+1+\es).
\end{equation}
We  choose functions $V_k\in C_b(X), k\ge2$,   $\|V_k\|_\infty=1$ such that  
$$
d(\mu,\nu):=\sum_{k=1}^\ty 2^{-k}|\lag V_k,\mu\rag-\lag V_k,\nu\rag|
$$ defines a metric   on~$\KK $ compatible with the weak topology. 
 As~$G$ is open, there are~$\delta>0$ and~$n\ge1$ such that if 
$$
\sum_{k=1}^n2^{-k}|\lag V_k,\nu\rag-\lag V_k,\nu_\es\rag|<\delta
$$ for some    $\nu\in\KK $, then  
  $\nu\in G$. Let  $x_\es:=f_n(\nu_\es)$, and denote by $\mathring{B}_{\R^n}(x_\es,\delta)$  the open ball in $\R^n$ of radius $\delta > 0$ centered at $x_\es$, with respect to the norm 
  $$
\|x\|_n:=\sum_{k=1}^n 2^{-k}|x_k|, \quad x=(x_1, \ldots, x_n).
$$
  Then we have   $f_n^{-1}\bigl(\mathring{B}_{\R^n}(x_\es,\delta)\bigl)\cap\KK \subset G$, hence \begin{align*} 
\pp\{\zeta_\theta\in  G\}&\ge\pp\{\zeta_\theta\in  G\cap \KK \}
\ge \pp\bigl\{\zeta_\theta\in  f_n^{-1}\bigl(\mathring{B}_{\R^n}(x_\es,\delta)\bigl)\cap\KK \bigr\}\\ 
&=\pp\{\zeta_\theta^n\in  \mathring{B}_{\R^n}(x_\es,\delta)\}
-\IP\{\zeta_\theta\in \KK ^c\}.
\end{align*}
Using the inequality  
$$\log(u-v)\ge \log u-\log 2, \q 0<v\le u/2
$$ and inequalities~\eqref{9.32}-\eqref{zzerfs},   we obtain   
\begin{align*}
\liminf_{\theta\in\Theta}  \frac1{r_\theta}\log\pp\{\zeta_\theta\in  G\}
&\ge \liminf_{\theta\in\Theta} \frac1{r_\theta}
\bigl(\log\pp\{\zeta_\theta^n\in  \mathring{B}_{\R^n}(x_\es,\delta)\}-\log 2\bigr)\\
&\ge -J_n(f_n(\W_n) \cap\mathring{B}_{\R^n}(x_\es,\delta))\ge -I_n(x_\es)\\
&\ge -I(\nu_\es) \ge -I(\W\cap G)-\es,
\end{align*} 
which proves~\eqref{2.6L}.  
\end{proof}

\begin{proof}[Sketch of the proof of Proposition~\ref{L:2.4}]
Inequality~\eqref{9.01} follows  from~\eqref{2.6U}.  To show \eqref{9.32}, for any $\beta=(\beta_1, \ldots, \beta_n)\in \R^n$ and $\alpha=(\alpha_1, \ldots, \alpha_n)\in \R^n$,  we set $V_\beta:=\sum_{j=1}^n\beta_jV_j$, $Q_n(\beta):=Q(V_\beta)$, and~$I_n(\alpha):=
\inf_{\sigma \in f_n^{-1}(\alpha)} I(\sigma)$.      One can verify that 
\begin{align*}
Q_n(\beta)&=\sup_{\alpha\in\R^n} \Bigl(\sum_{j=1}^n\beta_j\alpha_j-I_n(\alpha)\Bigr),\\
J_n(U)&=\inf_{\alpha\in  U}I_n(\alpha).
\end{align*}  
  Assume that   $J_n(f_n(\W_n)\cap U)<+\ty $, and  for any~$\es>0$,  choose~$\alpha_\es\in f_n(\W_n)\cap U$ such that 
$$I_n(\alpha_\es)<J_n(f_n(\W_n)\cap U)+\es.$$  Then
$\alpha_\es=f_n(\sigma_{V_{\beta_\es}})$ for some~$\beta_\es\in \R^n$   such that $V_{\beta_\es}\in \VV$. 
 It is easy to verify that  the following equality holds
$$
Q_n(\beta_\es)=\sum_{j=1}^n\beta_{\es j}\alpha_{\es j}-I_n(\alpha_\es).
$$
Literally repeating the proof of Proposition~3.4 in~\cite{JNPS-2014} (starting from equality~(3.16)) and using   the   
      uniqueness of the equilibrium state for $V=V_{\beta_\es}$ and the existence of limit~\eqref{2.1a},  one obtains 
$$
-J_n(f_n(\W_n)\cap U )-\es\le-I_n(\alpha_\es) \le\liminf_{\theta\in\Theta} \frac1{r_\theta}\log\pp\{  \zeta_\theta^n \in U  \}
$$for any $\es>0$. This implies \eqref{9.32}.
\end{proof}

\subsection{Large-time asymptotics for generalised Markov semigroups}
\label{S:5.2}

In this section, we give   a  continuous-time version of   Theorem~4.1  in~\cite{JNPS-2014}  with some modifications,  due to the fact that    the  generalised Markov family associated with the  stochastic NLW equation does not have a regularising property.   See also~\cite{KS-mpag2001,LS-2006, JNPS-2012} for some related results.  

 We start by recalling some terminology from \cite{JNPS-2014}.
 \begin{definition}\label{D:5.2} Let~$X$ be a Polish space.
 We shall say that $\{P_t(\uu,\cdot),\uu\in X, t\ge0\} $ is a {\it generalised Markov family of     transition kernels} if the following two  properties are satisfied. 
\begin{description}
\item[Feller property.] 
For any $t\ge0$,  the function $\uu\mapsto P_t(\uu,\cdot)$ is continuous from~$X$ to~$\MM_+(X)$   and does not vanish.
\item[Kolmogorov--Chapman relation.]  For any $t,s\ge0, \uu\in X$, and Borel set~$\Gamma\subset X$,  the following  relation  holds
$$
P_{t+s}(\uu,\Gamma)=\int_X P_s(\vv,\Gamma) P_t(\uu,\Dd \vv).
$$
\end{description} 
\end{definition}
To any such family we associate two semigroups by the following relations:
\begin{align*}
&\PPPP_t:C_b(X)\to C_b(X),  \quad\,\,\,\quad \quad \PPPP_t \psi(\uu)=\int_X \psi(\vv) P_t(\uu,\Dd \vv),\\
&\PPPP_t^*:\MM_+(X)\to \MM_+(X), \quad \quad \PPPP_t^* \mu(\Gamma)=\int_X  P_t(\vv,\Gamma) \mu(\Dd \vv),\quad t\ge0.
\end{align*}  
  For a   measurable function ${\wwww}:X\to[1,+\infty]$ and a family $\CC\subset C_b(X)$, we denote by~$\CC^\wwww$ the set of functions $\psi\in L_\wwww^\infty(X)$ that can be approximated  with respect to~$\|\cdot\|_{L_\wwww^\infty}$   by finite linear combinations of functions from~$\CC$.   We shall say that a family $\CC\subset C_b(X)$ is   {\it determining\/} if for any~$\mu,\nu\in\MM_+(X)$ satisfying   $\lag \psi,\mu\rag=\lag \psi,\nu\rag$ for all~$\psi\in\CC$, we have $\mu=\nu$. Finally, a family of functions~$\psi_t: X\to\R$ is   {\it uniformly equicontinuous\/} on a subset~$K\subset X$ if for any~$\es>0$ there is~$\delta>0$ such that $|\psi_t(\uu)-\psi_t(\vv)|<\es$ for any~$\uu\in K$, $\vv\in B_X(\uu,\delta)\cap K$, and~$t\ge1$. We have the following version   of  Theorem~4.1 in~\cite{JNPS-2014}.
\begin{theorem} \label{T:5.3}
Let~$\{P_t(\uu,\cdot),\uu\in X, t\ge0\} $ be  a generalised Markov family of     transition kernels  satisfying  the following four properties. 
\begin{description}
\item[Growth conditions.] There is   an increasing  sequence~$\{X_R\}_{R=1}^\infty$ of compact subsets of~$X$ such that~$X_\infty:=\cup_{R=1}^\ty X_R$ is dense in~$X$.  The measures $P_{t}(\uu,\cdot)$ are concentrated on~$X_\infty$ for any $\uu\in X_\ty$ and  $t>0$, and    there is a measurable function $\wwww:X\to[1,+\infty]$ and an integer~$R_0\ge1$ such that\,\footnote{The expression $(\PPPP_t\wwww)(\uu)$ is understood as an integral of a positive function~$\wwww$ against a positive measure $P_t(\uu,\cdot)$.}
\begin{align}
&\sup_{t\ge0}
\frac{\|\PPPP_t\wwww\|_{L_\wwww^\infty}}{\|\PPPP_t{\mathbf1}\|_{R_0}}<\infty,\label{5.8}\\
&\sup_{t\in [0,1]} \|\PPPP_t{\mathbf1}\|_{\infty}<\ty,\label{5.9}
\end{align}
where   $\|\cdot\|_R$ and $\|\cdot\|_\ty$ denote  the~$L^\infty$ norm on~$X_R$ and $X$, respectively, and we set $\infty/\infty=0$.

\item[Time-continuity.] 
For any  function $g\in L^\ty_\wwww(X_\ty)$ whose    restriction to~$X_R$ belongs to $C(X_R)$ and any $\uu\in X_\ty$,
 the function~$t\mapsto\PPPP_tg(\uu) $ is continuous from~$\R_+$ to $\R$.

\item[Uniform irreducibility.] 
For sufficiently large $\rho\ge1$, any   $R\ge 1$ and   $r>0$, there are positive numbers  $l=l(\rho,r,R)$ and~$p=p(\rho,r)$  such that
$$
P_l(\uu,B_{X}(\hat \uu,r))\ge p\quad\mbox{for all~$\uu\in X_R ,\,\hat \uu\in X_\rho$}. 
$$

\item[Uniform Feller property.]
There is a number~$R_0\ge1$ and a
 determining family~$\CC\subset  C_b(X)$      such that~${\mathbf1}\in\CC$ and the family $\{\|\PPPP_t{\mathbf1}\|_R^{-1}\PPPP_t\psi,t\ge1\}$ is uniformly equicontinuous on~$X_R$ for any~$\psi\in \CC$ and~$R\ge R_0$. 
\end{description}
%Suppose, in addition, that the vector span of~$\CC$ is dense in~$C_{<\wwww}(X)$ with respect to the norm~$\|\cdot\|_{L_\wwww^\infty}$. 
Then~ for any $t>0$, there is at most one measure $\mu_t\in\ppp_\wwww(X)$ such that  $\mu_t(X_\ty)=1$ and
\begin{equation}\label{5.10}
\PPPP^*_t \mu_t=\la(t)\mu_t \quad \text{for some $\la(t)\in \R$} 
\end{equation}satisfying the following condition: 
\begin{equation} \label{5.11}
\|\PPPP_t\wwww\|_R\int_{X\setminus X_R}\wwww\dd\mu_t\to0
\quad\mbox{as $R\to\infty$}.
\end{equation}
Moreover, if such a measure~$\mu_t$ exists for all $t>0$, then it is independent of~$t$ (we set $\mu:=\mu_t$),  the corresponding eigenvalue is of the form~$\lambda(t)=\la^t$,   $\la>0$,~$\supp \mu=X$, 
  and     there is a  non-negative function~$h\in L^\ty_\wwww(X_\ty)$  such that  $\lag h,\mu\rag=1$,
  \be\label{eigenf}
(\PPPP_t h)(\uu)=\lambda^t h(\uu)\quad\mbox{for $\uu\in X_\ty$,  $t>0$},
\ee  the restriction of~$h$ to~$X_R$ belongs to $C_+(X_R)$,  
  and for any $\psi\in \CC^\wwww$ and $R\ge1$, we have 
\begin{equation}
\lambda^{-t}\PPPP_t \psi\to\lag \psi,\mu\rag h
\quad\mbox{in~$C(X_R)\cap L^1(X,\mu)$ as~$t\to\infty$}. \label{5.12}
%\lambda^{-k}\PPPP_t^*\nu\to\lag h,\nu\rag\mu
%\quad\mbox{in~$\MM_+(X)$ as~$k\to\infty$}. \label{E:2.6}
\end{equation}
Finally, if a Borel set $B\subset X$ is such that
\begin{align} 
 \sup_{\uu\in B}\biggl(\int_{X\setminus X_R}\wwww(\vv)\,P_s(\uu,\Dd \vv)\biggr) &\to0
\quad\mbox{as $R\to\infty$} \label{5.13}
\end{align} for some   $s>0$, 
then for any $\psi\in \CC^\wwww$, we have
\begin{equation}
\lambda^{-t}\PPPP_t \psi\to\lag \psi,\mu\rag h
\quad\mbox{in~$L^\ty(B) $ as~$t\to\infty$}. \label{5.14}
%\lambda^{-k}\PPPP_t^*\nu\to\lag h,\nu\rag\mu
%\quad\mbox{in~$\MM_+(X)$ as~$k\to\infty$}. \label{E:2.6}
\end{equation}
\end{theorem}
\begin{proof}[Sketch of the proof] {\it Step~1: Existence of eigenvectors $\mu$ and $h$.}
 For any~$t>0$,
the conditions of  Theorem~4.1 in~\cite{JNPS-2014} are  satisfied\footnote{Let us note that in Theorem~4.1 in~\cite{JNPS-2014} it is assumed that the measures~$P_{t}(\uu,\cdot)$ are concentrated on~$X_\infty$ for any $\uu\in X$. Here this  is replaced  by the 
   condition that the measures~$P_{t}(\uu,\cdot)$ and $\mu_t$ are concentrated on~$X_\infty$ for any $u\in X_\ty$.  The
   uniform irreducibility property is slightly different from the one assumed in~\cite{JNPS-2014}. Both  modifications  are due to the  lack of a   regularising property   for  the   stochastic NLW equation.  
          These changes   do not affect the proof given  in~\cite{JNPS-2014}, one only needs   to 
 replace inequality~(4.16) in the proof by    the inequality 
\be\label{modif}
\sup_{k\ge0}\|\PPPP_k\psi\|_{L_\wwww^\infty(X)}\le M_1\,\|\psi\|_{L_\wwww^\infty(X)}
\quad\mbox{for any $\psi\in L_\wwww^\infty(X)$},
\ee  and literally repeat all the arguments.  The  proof of \eqref{modif} is similar to the one of~(4.16). Under these modified  conditions, the 
concept of eigenfunction for $\PPPP_t$ is understood in a weaker 
sense; namely, relation \eqref{eigenf}   needs to hold only for $\uu\in X_\infty$.}  for the      discrete-time semigroup~$\{\tilde\PPPP_k=\PPPP_{tk}, k\ge1\}$ generated by~$\tilde P=P_t$.        So that theorem implies the existence of    at most one measure~$\mu_t\in\ppp_\wwww(X)$ satisfying~$\mu_t(X_\ty)=1$,~\eqref{5.10}, and~\eqref{5.11}.
 Moreover, if such a measure $\mu_t$ exists for any $t>0$,   it   follows from the Kolmogorov--Chapman relation that           $\mu_t=\mu_1=:\mu$ and $\la(t)=(\la(1))^t=:\la^t$ for any  $t$ in the set $\Q_+^*$ of  positive rational numbers, i.e.,    
\begin{equation}\label{5.15}
\PPPP^*_t \mu=\la^t\mu \quad \text{for }t\in \Q_+^*. 
\end{equation}      
 Using the time-continuity property and density,   we get  that   \eqref{5.15} holds   for any~$t>0$. So we have $\mu_t=\mu$ and $\la(t)=\la^t$ for any $t>0$, by uniqueness of the eigenvector.  

Theorem 4.1 in~\cite{JNPS-2014} also implies  that $\supp \mu= X, \la>0$, and          there is a  non-negative function $h_t\in L^\ty_\wwww(X_\ty)$  such that $\lag h_t,\mu\rag=1$,  the restriction of~$h_t$ to~$X_R$ belongs to $C_+(X_R)$,  and
\begin{align}
& (\PPPP_{t} h_t)(\uu)=\lambda^{t} h_t(\uu)\quad\mbox{for $\uu\in X_\ty$}, \label{sahmana1}\\
&\lambda^{-tk}\PPPP_{tk} \psi\to\lag \psi,\mu\rag h_t
\quad\mbox{in~$C(X_R)\cap L^1(X,\mu)$ as~$k\to\infty$} \label{sahmana2}
\end{align}
     for any $\psi\in \CC^\wwww, R\ge1$, and $ t>0$. Taking $\psi={\mathbf1}$ in \eqref{sahmana2}, we see that
      $h_t=h_1=:h$ for any~$t\in \Q_+^*$. The continuity of   the function~$t\mapsto\PPPP_th(\uu) $ and~\eqref{sahmana1} imply that   $h_t=h$ for any $t>0$ and 
      \begin{equation}
      \lambda^{-tk}\PPPP_{tk} \psi \to\lag \psi ,\mu\rag h
\quad\mbox{in~$C(X_R)\cap L^1(X,\mu)$ as~$k\to\infty$}. \label{5.16}
      \end{equation}

\medskip
  {\it Step~2: Proof of  \eqref{5.12}.}
 First let us prove \eqref{5.12} for any $\psi \in\CC$. 
Replacing~$P_t(\uu, \Gamma)$ by $\la^{-t}P_t(\uu, \Gamma)$, we may
assume that~$\la = 1$.  Taking~$\psi =\mathbf1$ and $t=1$ in \eqref{5.16}, we obtain    $ \sup_{k\ge0} \|\PPPP_{k} \mathbf1    \|_R<\ty$. So using~\eqref{5.9}, we  get $ \sup_{t\ge0} \|\PPPP_{t} \mathbf1    \|_R<\ty$.
 This implies that 
  $\{  \PPPP_t \psi, t\ge1\}$ is uniformly equicontinuous on $X_R$ for any  $R\ge R_0$. 
   Setting $g=\psi-\lag \psi,\mu\rag h$, we need to prove that $ \PPPP_tg\to0$ in~$C(X_R)$ for any~$R\ge1$. Since~$\{ \PPPP_tg,t\ge1\}$ is uniformly equicontinuous on~$X_R$, the required assertion will be established if we prove that 
\begin{equation} \label{5.17}
\bbar  \PPPP_tg \bbar_\mu:=\lag \bbar\PPPP_tg\bbar, \mu\rag \to0\quad\mbox{as~$t\to\infty$}.
\end{equation}
For any~$\varphi \in L^\ty_\wwww(X)$, we have
$$
\bbar \PPPP_t \varphi\bbar_\mu \le\langle \PPPP_t |\varphi|,\mu\rangle
=\langle|\varphi|,\mu\rangle=\bbar\varphi\bbar_\mu,
$$
thus~$\bbar  \PPPP_tg\bbar_\mu$ is a non-increasing function in $t$. By~\eqref{5.16}, we have   $\bbar  \PPPP_{tk}g\bbar_\mu\to 0$  as~$k\to\ty$. This  proves  \eqref{5.17}, hence also  \eqref{5.12} for any $\psi\in\CC$. 

 An easy approximation argument shows that   \eqref{5.12} holds for any $\psi\in\CC^\wwww$ (see Step 4 of the proof of Theorem 4.1 in~\cite{JNPS-2014}). Finally,
the proof of~\eqref{5.14} under condition \eqref{5.13} is exactly the same as in Step 7 of the proof of the discrete-time case.
 \end{proof}

 \subsection{Proofs of some auxiliary assertions}\label{S:6.2}

 \subsubsection*{The Foia\c{s}-Prodi estimate} 
Here we briefly recall an a priori   estimate established in Proposition~4.1 in~\cite{DM2014}.
Let $\uu_t=[u,\dot u]$ and~$\vv_t=[v,\dot v]$ be some flows of the equations
\begin{align}
\p_t^2 u+\gamma \p_t u-\de u+f(u)&=h(x)+\p_t\ph(t,x),\label{FP"IN"1}\\
\p_t^2 v+\gamma \p_t v-\de v+f(v)+{\mathsf P}_N[f(u)-f(v)]&=h(x)+\p_t\ph(t,x)\label{FP"IN"2},
\end{align}
       where $\varphi$ is a function belonging  to $L^2_{loc}(\rr_+,L^2(D))$. We recall that~${\mathsf P}_N$ stands for the orthogonal projection in $L^2(D)$ onto the vector span $H_N$ of the functions~$e_1,e_2,\ldots,e_N$ and $P_N $ is the projection in $\h$ onto~$\h_N:=H_N\times H_N$.
\begin{proposition}\label{4.13}
Assume that, for some non-negative numbers $s$  and $T$, we have $\uu,\vv\in C(s,s+T;\h)$. Then  
\be\label{FPE1}
|P_N(\vv_t-\uu_t)|_\h^2\le e^{-\alpha (t-s)}|\vv_s-\uu_s |_\h^2 \q \text{ for } s\leq t\leq s+T,
\ee
where $\al>0$ is the constant entering \ef{e40}.
If we suppose  that 
  the inequality holds
\be\label{4.17}
\int_s^t\|\g z\|^2\,d\tau\leq l+K(t-s) \q \text{ for } s\leq t\leq s+T
\ee
 for $z=u$ and $z=v$ and some positive numbers $K$  and $l$, then, for any $\es>0$, there is an integer $N_*=N_*(\es, K)\geq 1$   such that
\be \label{4.16}
|\vv_t-\uu_t |^2_\h\leq e^{-\al(t-s)+\es l}|\vv_s-\uu_s|^2_\h \q \text{ for } s\leq t\leq s+T
\ee   for all~$N\geq N_*$ and $ s\leq t\leq s+T$.
\end{proposition}
\bp  Estimate \eqref{4.16} is proved in Proposition~4.1 in~\cite{DM2014}. To prove \eqref{FPE1}, let us note that $\Zz=[z,\dot z]=P_N(\vv-\uu)$ is a solution of the linear equation
$$
\p_t^2 z+\gamma \p_t z-\de z=0.
$$ So we have 
$$|P_N(\vv_t-\uu_t)|_\h^2=|  \Zz_t|_\h^2\le e^{-\al(t-s)}|\Zz_s|_\h^2 \le e^{-\alpha (t-s)}|\vv_s-\uu_s |_\h^2.
$$
\ep

\subsubsection*{Proof of Proposition \ref{P:TVE}}
 
This proposition is essentially proved in Section 4.2 in \cite{DM2014} in a different form. However, since it plays a central role in the proof of   the uniform Feller property,   we find it worthwhile to give here a   detailed proof of it. As in~\cite{DM2014}, we follow the arguments presented in Section~3.3 of~\cite{KS-book} and Section~4 of~\cite{KN-2013}. As inequality \eqref{eoejtnvf} concerns only the laws of solutions, we can assume that    the underlying probability space~$(\omm,\fff,\pp)$ is
of a particular form.  We assume that   $\omm= C(\R_+, \R)$ is  endowed with the topology of uniform convergence on bounded intervals, $\pp$ is the law of the  Wiener process $\hat \xi=[0,\xi]$, where  $\xi$ is defined in \eqref{0.3}, and $\fff$ is the completion of the Borel $\sigma$-algebra of $\Omega$ with respect to $\pp$.

\smallskip
We introduce some notation. Let $\hat\h_N$ be the $N$-dimensional subspace of~$\h$ spanned by the vectors $\hat e_1,\hat e_2,\ldots,\hat e_N$, where $\hat e_j=[0,e_j]$. Then  $
 \Omega=\Omega_N\dt+\Omega_N^{\perp},
 $
 where $\Omega_N=C(\rr_+, \hat \h_N)$ and $\Omega_N^\perp=C(\rr_+, {\hat\h^\perp_N})$.
For $\om=\om_{1}\dt+\om_{2}$, we write $\om=(\om_{1},\om_{2})$.
 For any continuous process $\uu_t$ with range in $\h$, we introduce the functional
$$
\fff^\uu(t)=|\ees(\uu_t)|+\al\int_0^t|\ees (\uu_s)|\dd s,
$$
and the stopping time
$$
\tau^\uu=\inf\{t\geq 0:\fff^\uu(t)\geq \fff^\uu(0)+Lt+\rho\},
$$
where $L$ and $\rho$ are some positive constants to be chosen later. Now let us fix initial points $\Zz$ and $\Zz'$ in $\h$. We shall denote by $\uu_t$ and $\uu'_t$ the flows of \ef{0.1} issued from $\Zz$ and $\Zz'$, respectively, and by $\vv(t)$   the flow of \eqref{interm}. We
define a stopping time
$
\tilde\tau=\tau^{\uu}\wedge\tau^{\uu'}\wedge\tau^{\vv}
$
and a transformation $\Lambda:\Omega\to\Omega$ given by
$$
 \Lambda(\om)(t)=\om(t)-\int_0^t \ph(s)\dd s,\q \ph(t)=\ch_{t\leq\tilde\tau}\hat P_N(0,[f(u_t)-f(v_t)]),
$$
 where $\ch_{t\leq\tilde\tau}$ stands
for the indicator function of the interval $[0,\tilde\tau]$,
 $\hat P_N$ is the orthogonal projection in $\h$ onto $\hat \h_N$, and $u$ is the first component of $\uu$. Let us prove the following result, which is a global version of Lemma~4.3 in~\cite{DM2014}. 
\begin{lemma}\label{4.21}
For any initial points
 $\Zz$ and $\Zz'$ in $\h$, we have
\be\label{6.46}
 |\Lambda_*\pp-\pp|_{var}\leq \left[\exp\left(C_{N}|\Zz-\Zz'|_\h^2 e^{(|\ees(\Zz)|+|\ees(\Zz')|)+\rho} \right)-1\right]^{1/2},
 \ee
where $\Lambda_*\pp$ stands for the image of $\pp$ under $\Lambda$.
\end{lemma}
\begin{proof}[Proof of Lemma \ref{4.21}]  {\it Step~1.} 
By the definition of $\tilde\tau$, we have
\be\label{6.54}
\fff^\uu(t)\leq \fff^\uu(0)+Lt+\rho,\q
\fff^\vv(t)\leq\fff^{\uu'}(0)+Lt+\rho
\ee
for all $t\leq\tilde\tau$.
Let us show that there is an integer $N_1=N_1(L)$ such that, for all~$N\ge N_1$ and $t\leq\tilde\tau$, we have
\be\label{7.8}
|\vv(t)-\uu(t)|_\h^2\leq e^{-\al t+\theta}|\Zz'-\Zz|_\h^2,\q \theta=\f{|\ees(\Zz)|\vee|\ees(\Zz')|+\rho}{2}.
\ee
Indeed, thanks to \eqref{1.5} and the   Poincar\'e inequality, we have 
\begin{align*}
|\uu|_{\h}^2&\leq \left| |\uu|_{\h}^2+2\int_D F(u_1)\dd x\right|-2\int_D F(u_1)\dd x
\leq |\ees(\uu)|+2\nu \|u_1\|^2+2C\\
&\leq |\ees(\uu)|+\f{\lm_1}{4}\|u_1\|^2+2C\leq |\ees(\uu)|+\f{1}{4}|\uu|_\h^2+2C,
\end{align*}
for any $\uu=[u_1,u_2]$ in $\h$. Therefore
\be\label{7.9}
|\uu|_{\h}^2\leq 2 |\ees(\uu)|+3C.
\ee
Combining this inequality with \eqref{6.54}, we see that for all $t\leq\tilde\tau$
$$
\al\int_0^t\|\g w(s)\|^2\dd s \leq 2(|\ees(\Zz)|\vee |\ees(\Zz')|+\rho)+2(L+3C)t
$$
for $w=u$ and $w=v$. Using the above inequality and applying Proposition \ref{4.13} with~$\es=\al/4$, we  infer \eqref{7.8}.

\medskip
{\it Step 2.} 
Note that the transformation  $\Lambda$ can be written as
$
\Lambda(\om)=(\Upsilon(\om),\om_{2}),
$
where $\Upsilon:\omm\to\omm_N$ is given by
$$
\Upsilon(\om)(t)=\om_{1}(t)+\int_0^t \ph(s;\om)\dd s.
$$
It is not difficult to see that 
 \be \label{4.20}
 |\Lambda_*\pp-\pp|_{var}\leq\int_{\omm^\perp_N}|\Upsilon_*(\pp_N,\om_{2})-\pp_N|_{var}\pp^\perp_N(\Dd\om_{2}),
 \ee
 where $\pp_N$ and $\pp^\perp_N$ stand for the images of $\pp$ under the projections $\hat P_N:\omm\to\omm_N$ and $\hat Q_N:\omm\to\omm^\perp_N$, respectively.
Introduce
$$
 X=\om_{1}(t),\q
 \hat X=\om_{1}(t)+\int_0^t \ph(s;\om)\dd s.
$$
 Then $\pp_N$ coincides with the distribution $\DD(X)$ of the random variable $X$ and $\Upsilon_*(\pp_N,\om_{2})$ coincides with that of $\hat X$.
 By the Girsanov theorem (see Theorem A.10.1 in \cite{KS-book}), we have
 \be\label{Dz-D tilde z}
 |\DD(\hat X)-\DD(X)|_{var}\leq\f{1}{2}\left(\left(\e\exp\left[6\max_{1\leq j\leq N} b_j^{-1}\int_0^\infty |\ph(t)|^2\dd t\right]\right)^{\f{1}{2}}-1\right)^{\f{1}{2}},
 \ee
if we assume that the Novikov condition
 $$
 \e\exp\left(c\int_0^\infty |\ph(t)|^2\dd t\right)<\infty\q \text{ for any } c>0,
 $$
 is satisfied. To check this condition, first note that 
 \begin{align}\label{naxord}
\e\exp\left(c\int_0^\infty |\ph(t)|^2\dd t\right)&=\e\exp\left(c\int_0^{\tilde\tau} |\ph(t)|^2\dd t\right)\notag\\
&\leq\e\exp\left(c\int_0^{\tilde\tau} \|f(v_t)-f(u_t)\|^2\dd t\right).
 \end{align}Using \eqref{1.8}, the H\"older inequality, and  the  Sobolev embedding $H^{1} \hookrightarrow L^{6} $, we see that 
 $$
  \|f(v)-f(u)\|^2\le C_1 \|v-u\|_1^2(1+\|u\|_1^4+\|v\|_1^4).
 $$ Joining this together with inequalities \eqref{6.54}-\eqref{7.9} and \eqref{naxord}, we get
 
 \begin{align*}
\e\exp&\left(c\int_0^\infty |\ph(t)|^2\dd t\right)\\&\leq\e\exp\left(C_2 |\Zz'-\Zz|_\h^2\int_0^\infty e^{-\al t+\theta}(1+|\ees(\Zz)|\vee|\ees(\Zz')|+Lt+\rho)^2\dd t\right)\\&\le\exp\left(C_{3}|\Zz-\Zz'|_\h^2 e^{(|\ees(\Zz)|+|\ees(\Zz')|)+\rho} \right)<\infty.
 \end{align*}
 Finally, combining this with \ef{4.20} and \ef{Dz-D tilde z}, we get \eqref{6.46}.
\end{proof}

 \medskip
Now we can prove \ef{eoejtnvf}. Indeed, for each $\om$ belonging to the event~$\{\tilde\tau<\iin\}$,  let  us introduce auxiliary $\h$-continuous processes $ y_{\uu'}$ and $y_\vv$ defined as follows: for $t\leq \tilde\tau$ they coincide with the  processes $ \uu'$ and $\vv$, respectively, while for $t\geq\tilde\tau$ they solve 
$
\dt y=-m y. 
$ We choose $m\ge1$ so large that 
\be\label{e44}
\{\tau^{y_{\uu'}}<\iin\}\subset \{\tau^{{\uu'}}<\iin\}.
\ee
This construction implies that, with probability 1, we have
\be\label{4.29}
y_\vv(t,\om)=y_{\uu'}(t,\Lambda(\om)) \q\text{ for all } t\geq 0.
\ee
Let us denote by $\uu_1'$ and $\vv_1$ the restrictions of $\uu'(t)$ and $\vv(t)$ to the time interval~$[0, 1]$. Then
\begin{align*}
&|\lambda(\Zz, \Zz')-\lambda(\Zz')|_{var}=\sup_{\Gamma}|\pp\{\vv_1\in\Gamma\}-\pp\{\uu'_1\in\Gamma\}\notag|\\
&\q\leq \pp\{\tilde\tau<\iin\}+ \sup_{\Gamma}|\pp\{\vv_1\in\Gamma, \tilde\tau=\iin\}
-\pp\{\uu'_1\in\Gamma, \tilde\tau=\iin\}|=\elll_1+\elll_2,
\end{align*}
where the supremum is taken over all Borel subsets of $C(0, 1;\h)$.
Note that
$$
\elll_2\leq|\Lambda_*\pp-\pp|_{var}.
$$
Further, we have
$$
\elll_1\leq \pp\{\tau^\vv<\iin, \tau^\uu\wedge\tau^{\uu'}=\iin\}+\pp\{\tau^\uu<\iin\}+\pp\{\tau^{\uu'}<\iin\}.
$$
Moreover, thanks to \ef{4.29} and \eqref{e44}, we have
\begin{align*}
\pp\{\tau^\vv<\iin, \tau^\uu\wedge\tau^{\uu'}=\iin\}&\leq\pp\{\tau^{y_\vv}<\iin\}=\Lambda_*\pp\{\tau^{y_{\uu'}}<\iin\}\notag\\
&\leq \pp\{\tau^{y_{\uu'}}<\iin\}+|\Lambda_*\pp-\pp|_{var}\notag\\
&\leq \pp\{\tau^{{\uu'}}<\iin\}+|\Lambda_*\pp-\pp|_{var}.
\end{align*}
 Combining last four inequalities, we infer
$$
|\lambda(\Zz, \Zz')-\lambda(\Zz')|_{var}\leq 2\left(\pp\{\tau^{\uu}<\iin\}+\pp\{\tau^{u'}<\iin\}+|\Lambda_*\pp-\pp|_{var}\right).
$$
Finally using this with inequality \ef{6.46} and Corollary 3.3 from \cite{DM2014}, we get
$$
|\lambda(\Zz, \Zz')-\lambda(\Zz')|_{var}\leq 2e^{4\beta C-\beta \rho}+2\left[\exp\left(C_{N}|\Zz-\Zz'|_\h^2 e^{(|\ees(\Zz)|+|\ees(\Zz')|)+\rho} \right)-1\right]^{1/2},
$$
where $\beta=\al/8\,(\sup b_j^2)^{-1}$ and $C$ is the constant entering inequalities \eqref{1.8}-\eqref{1.6}. Denoting $a=2\beta/(\beta+1)$ and $C_*=2\exp(4\beta C)$, and making a change of variable~$\rho=-\beta^{-1}a\ln\es$, we derive \ef{eoejtnvf}.

 \subsubsection*{Proof of Proposition \ref{P:JNPS}}\label{S:6.2.1}

{\it Step~1: Preliminaries}.
 We denote by  $ \SSSS_t^{V,F}$ the semigroup   defined by \eqref{6.67}, and  write~$ \SSSS_t^{V}$ instead of   $ \SSSS_t^{V,0}$ (i.e.,   $F=\mathbf0$). Let
$\D({\mathbb L_V})$   be the space of functions~$\psi \in C_b(\h^\sS)$ such that  
\be\label{repQ1}
\SSSS_t^{V} \psi(\uu)=\psi(\uu)+\int_0^t \SSSS_\tau^{V} g(\uu) \dd\tau, \q t\ge0, \, \uu\in \h^\sS
\ee for some $g\in C_b(\h^\sS)$. Then the continuity of the mapping   $t\mapsto \SSSS_t^{V}g(\uu)$ from~$\R_+$ to $\R$ implies  
the following limit   
$$
g (\uu)=\lim_{t\to 0} \frac{ \SSSS_t^{V}\psi (\uu)- \psi (\uu)}{t},
$$ and proves the uniqueness of $g$ in representation \eqref{repQ1}. 
   We set   ${\mathbb L_V} \psi :=g$.   The proof is based on the following two lemmas.
   \begin{lemma}\label{Lem1} For any $F\in C_b(\h^\sS)$, the following properties hold 
 \begin{enumerate}
  % \item[i)] {\color{red}We have $\D(\mathbb L_V^F)= \D(\mathbb L)$  and   $$ \mathbb L_V^F \psi= (V+F-Q(V)) \psi+ h_V^{-1}  \mathbb L( h_V \psi),\q \psi\in \D({\mathbb L}).$$}
 \item[i)] For any $\psi \in \D({\mathbb L_V})$,  we have  $ \varphi_t:=\SSSS_t^{V,F}\psi \in \D({\mathbb L_V})$ and  
 $$
  \partial_t\varphi_t  =  ({\mathbb L_V}+ F)\varphi_t, \q   t>0. 
 $$
  \item[ii)] The set $\D_+:=\left\{\psi\in \D(\mathbb L_V): \inf_{\uu\in\h^\sS}\psi(\uu)>0\right\}$ is  determining for $\ppp(\h^\sS)$, i.e., if $\lag \psi, \sigma_1\rag =\lag \psi, \sigma_2\rag $ for some $\sigma_1,\sigma_2\in \ppp(\h^\sS)$ and any $\psi\in \D_+$, then~$\sigma_1=\sigma_2$.
 \end{enumerate}
    \end{lemma}
    This lemma is proved at the end of this subsection.  The next result is established exactly in the same way as Lemma~5.9 in~\cite{JNPS-2014}, by using limit~\eqref{a1.5};    we omit its proof.
    \begin{lemma}\label{Lem2} 
    The Markov semigroup  $\SSSS_t^{V}$
has a unique stationary measure, which is
given by $\nu_V=h_V\mu_V$.
    \end{lemma}
   \medskip

{\it Step~2}.  
  Let us show that, for any $\psi \in \D_+$, we have
    \be\label{Qhav1}
Q^V_R(F_\psi)=0,
\ee  where $F_\psi:= -\mathbb L_V\psi/ \psi \in C_b(\h^s)$ and $Q^V_R(F_\psi)$ is defined by \eqref{Qhav0}. Indeed, by property i) in Lemma~\ref{Lem1}, the function
   $\varphi_t= \SSSS_t^{V,F_\psi}\psi$ satisfies 
   $$
     \partial_t\varphi_t  = \left(\mathbb L_V  - \frac{\mathbb L_V \psi}{\psi}\right)\varphi_t, \q \varphi_0=\psi.
   $$
From the uniqueness of the solution we derive that $\psi=\varphi_t$ for any $t\ge0$, hence
\be\label{Qhav3}
\lim_{t\to+\ty} \frac{1}{t}\log\sup_{\uu \in X_R}\log(\SSSS_t^{V,F_\psi}{\psi })(\uu)=0.
\ee As   $c\le\psi(\uu)\le C$ for any $\uu\in \h^\sS$ and some  constants $C,c>0$, we have
$$
Q^V_R(F_\psi)\le \limsup_{t\to+\ty} \frac{1}{t}\log\sup_{\uu \in X_R}\log(\SSSS_t^{V,F_\psi}{\psi })(\uu)\le Q^V_R(F_\psi).
$$Combining this with \eqref{Qhav3}, we obtain  \eqref{Qhav1}.

\medskip
{\it Step~3}. 
Let us assume\,\footnote{As $I_R$ defined by \eqref{I_R} is a good rate function, the set of equilibrium measures for  $V$ is non-empty. So 
  the set of zeros of $ I^V_R$ is  also non-empty, by the remark made at the end of Step 2 of the proof of Theorem~\ref{T:MT}.} that $ I^V_R(\sigma)=0$. Then $\sigma\in \ppp(\h^\sS)$ and 
$$
 0=I^V_R(\sigma)=\sup_{F\in C_b(\h^\sS)}\bigl(\lag F, \sigma\rag-Q^V_R(F)\bigr).
$$So taking here $F=F_\psi$ for any $\psi\in \D_+$ and using the result of Step 2, we get 
$$
0\le  \inf_{\psi\in \D_+} \int_{\h^\sS}  \frac{\mathbb L_V \psi}{\psi}\sigma(\Dd \uu).
$$ Since  $\SSSS_t^{V}$ is a   Markov semigroup, we have  
 $\mathbb L_V{\mathbf1}={\mathbf0} $.  We see that $\theta=0$ is a local minimum of the function
$$
f(\theta):= \int_{\h^\sS}  \frac{\mathbb L_V (1+\theta\psi)}{1+\theta \psi}\sigma(\Dd \uu)
$$for any $\psi\in \D_+$, so
$$
0=f'(0)= \int_{\h^\sS}    \mathbb L_V  \psi  \,  \sigma(\Dd \uu).
$$Combining this with property i) in Lemma~\ref{Lem1}, we obtain 
$$
  \int_{\h^\sS}    \SSSS_t^{V}  \psi  \,  \sigma(\Dd \uu)= \int_{\h^\sS}      \psi  \,  \sigma(\Dd\uu), \q t>0.
$$From  ii) in Lemma~\ref{Lem1}, we derive that  $\sigma$ is a stationary measure for    $ \SSSS_t^{V}$, and 
Lemma~\ref{Lem2} implies that $\sigma=h_V\mu_V$. This completes the proof of Proposition~\ref{P:JNPS}.
\begin{proof}[Proof of Lemma~\ref{Lem1}]
{\it Step~1: Property  i)}. %The proof of this property is similar to that of Theorem~2.4 in~\cite{Pazy}. In that reference the   case of a $C_0$-semigroup is considered, 
Let us  show that, for any~$\psi\in C_b({\h^\sS})$,   the function $ \varphi_t=\SSSS_t^{V,F}\psi $ satisfies the equation in the Duhamel form 
\be\label{Duhamel}
\varphi_t =  \SSSS_t^{V}\psi+\int_0^t \SSSS_{t-s}^{V}(F \varphi_s) \dd s.
\ee Indeed, we have
\begin{align*} 
\varphi_t  &- \SSSS_t^{V}\psi= \la_V^{-t}h_V^{-1}\nonumber\\
&\times\e_\uu \left\{ \exp\left(\int_0^t V(\uu_\tau) \dd \tau\right) \!\left[ \exp\left(\int_0^t  F(\uu_\tau)\dd \tau\right) -1\right]  \! h_V(\uu_t)\psi(\uu_t)\right\}. 
\end{align*} Integrating by parts and using the
  the Markov property, we get   
\begin{align*} 
&\varphi_t   - \SSSS_t^{V}\psi= \la_V^{-t}h_V^{-1}\nonumber\\
&\q\times\int_0^t\e_\uu \left\{ \exp\left(\int_0^t V(\uu_\tau) \dd \tau\right) \!\left[F(\uu_s) \exp\left(\int_s^tF(\uu_\tau)\dd \tau\right)\right]  \! h_V(\uu_t)\psi(\uu_t)\right\}\dd s\\
&= \int_0^t  \la_V^{-s}h_V^{-1} \e_\uu \left\{ \exp\left(\int_0^s V(\uu_\tau) \dd \tau\right) h_V(\uu_s) F(\uu_s)  \varphi_{t-s} (\uu_s) \right\} \dd s \\&
=    \int_0^t   \SSSS_s^{V}(F \varphi_{t-s}) \dd s=\int_0^t \SSSS^{V}_{t-s}(F \varphi_s) \dd s.
\end{align*} This proves \eqref{Duhamel}. The identity
$$
\SSSS_t^{V} (\varphi_r)(\uu)=\varphi_{r+t}(\uu)=\varphi_r(\uu)+\int_0^t \SSSS_{\tau}^{V}(  \SSSS_{r}^{V}g)(\uu) \dd\tau, \q t\ge0, \, \uu\in \h^\sS
$$
 shows that $\varphi_r\in \D({\mathbb L_V})$ for $\psi\in \D({\mathbb L_V})$ and $r>0$.     
  \medskip

{\it Step~2: Property   ii)}. Assume  that,   for some       $\sigma_1,\sigma_2\in \ppp(\h^\sS)$, we have
   \be\label{Qhav6}
\lag \psi, \sigma_1\rag =\lag \psi, \sigma_2\rag, \q \psi\in \D_+.
\ee    
Let us take any  $\psi\in C_b(\h^\sS)$ such that $c\le \psi(\uu)\le C$ for  any $\uu\in \h^\sS$ and some constants  $c,C>0$. Then $\tilde \varphi_r:=\frac1r\int_0^r \SSSS_\tau^{V} \psi \dd \tau$ belongs to~$\D_+$ for any $r>0$. Indeed, the inequality
   $c\le \tilde \varphi_r(\uu)\le C $ follows immediately from  the definition of~$\SSSS_r^{V}$, and the fact that $\tilde \varphi_r\in \D({\mathbb L_V})$ follows from the identity   
   \begin{align*}
 \SSSS^V_t\tilde \varphi_r-\tilde \varphi_r&= \f1r\int_0^r (\SSSS^V_{\tau+t}\psi-\SSSS^V_\tau\psi)\dd\tau= \f1r\int_r^{r+t}\SSSS^V_\tau\psi\dd\tau- \f1r\int_0^{t}\SSSS^V_\tau\psi\dd\tau\\
&= \int_0^t \SSSS^V_\tau \left( \frac{\SSSS^V_r\psi-\psi}{r} \right)\dd \tau .
\end{align*}
   Then, by \eqref{Qhav6},   we have 
   \be\label{Qhav5}
   \lag \tilde\varphi_r, \sigma_1\rag =\lag \tilde\varphi_r, \sigma_2\rag, \q r>0.
   \ee Using the continuity of the mapping $r\mapsto \SSSS_r^{V}\psi(\uu)$ from $\R_+$ to $\R$, we see that~$\tilde\varphi_r(\uu)\to \psi(\uu)$ as $r\to0$. Passing to the limit in \eqref{Qhav5} and using the Lebesgue theorem on dominated convergence, we obtain 
   $
   \lag \psi, \sigma_1\rag =\lag \psi, \sigma_2\rag.  
   $ It is easy to verify that the set $\{\psi\in C_b(\h^\sS): \inf_{\uu\in \h^\sS} \psi(\uu)>0 \}$ is   determining, so we get $\sigma_1=\sigma_2$.

 %{\color{red} For any $\psi\in C_b(\h^\sS)$ and $t>0$, we have $\psi_t:=\frac1t\int_0^t \SSSS_\tau^{V,F}\psi \dd \tau \in D({\mathbb L_V^F})$ and $\psi_t(\uu)\to \psi(\uu)$ as $t\to0$ for any $\uu\in \h^\sS$.}
  %  The proof of this lemma is similar to   that of Theorem 2.4 in \cite{Pazy}.  One only needs to replace the convergence in the norm of $C_b(\h^\sS)$ in the proof of~\cite{Pazy} by a pointwise convergence, since our semigroup $\SSSS_t^{F} $ is not of class~$C_0$ in general. We omit the details of the proof. 
\end{proof}

 \subsubsection*{Proof of Lemma \ref{e51}} 
 The function $f:J\to \R$ is convex, so the derivatives $D^\pm f(x)$ exist for any~$x\in J$.
We confine ourselves to the derivation of the first inequality in the lemma. Assume the opposite, and let $x_0\in J$, $(n_k)\subset\nn$, and $\eta>0$  be such that
\be\label{e50}
  D^+f_{n_k}(x_0)\ge D^+f(x_0)+\eta\q\text{ for }k\ge 1.
\ee
 Let us fix $x_1\in J$, $x_1> x_0$ such that
 $$
  D^+f(x_0)\ge \f{f(x_1)-f(x_0)}{x_1-x_0}-\eta/4.
 $$
 Since $f_{n_k}$ is a convex function, we have
 $$
  D^+f_{n_k}(x_0)\le\f{f_{n_k}(x_1)-f_{n_k}(x_0)}{x_1-x_0}.
 $$
Assume that $k\ge 1$ is so large that we have
$$
|f_{n_k}(x_1)-f(x_1)|+|f_{n_k}(x_0)-f(x_0)|\le \eta(x_1-x_0)/4.
$$
Then, combining last three inequalities, we derive
$$
  D^+f_{n_k}(x_0)\le   D^+f(x_0)+\eta/2,
  $$
  which contradicts \ef{e50} and proves the lemma.

 \subsubsection*{Proof of Lemma \ref{0.14}} 
 
Let us first prove \eqref{0.31}. We take   $p_4=6/{(1+2\sS)}$ the maximal   exponent for which the    Sobolev embedding  $H^{1-\sS} \hookrightarrow L^{p_4} $ holds.
  We choose $p_2$ in such a way that exponents $(p_i)$ are H\"older admissible. It follows that $p_2=6/(5-\rho-2\sS-3\kp)$. Now let $\kp>0$ be so small that $\rho+2 \sS \kp\leq 2$. Then a simple calculation shows that $(1-\kp)p_2\leq 6/(3-2\sS)$, so the  Sobolev embedding implies the first inclusion in     \eqref{0.31}. 

\smallskip
\nt
We now prove \eqref{0.32}. Proceeding   as above, we take $q_4= 6/(1+2\sS)$ and choose $q_2$ such that the exponents $(q_i)$ are H\"older admissible,  i.e., $q_2=6(\rho+2)/(12-(\rho+2)(1+2\sS+3\kp))$. It is easy to check that for~$\kp<1/2-\sS$, we have $(1-\kp)q_2\leq 6$. The Sobolev embedding allows to conclude.

\subsubsection*{Proof of Lemma \ref{0.20}}
In view of inequality \eqref{0.21}, we have
\begin{align*}
\beta^{-1}\f{\Dd}{\Dd t}(1+x)^{\beta}&=(1+x)^{\beta-1}\dt x\leq (1+x)^{\beta-1} (-\al x+gx^{1-\beta}+b)\\
&\leq -\al\,x(1+x)^{\beta-1}+g+b\leq -\f{\al}{2} x^{\beta}+\al+g+b.
\end{align*}
Fixing $t\in [0,T]$ and integrating this inequality over $[0, t]$, we obtain
$$
\beta^{-1}(1+x(t))^{\beta}+\f{\al}{2}\int_0^t x^\beta(\tau)\dd \tau\leq  \beta^{-1}(1+x(0))^{\beta}+\int_0^t (\al+g(\tau)+b(\tau))\dd \tau,
$$
which implies \eqref{0.28}.

\bigskip
\bigskip
{\bf Acknowledgments}. I am grateful to my supervisor Armen Shirikyan for attracting my attention to these problems, and for many fruitful discussions. This research was carried out within the MME-DII Center of Excellence (ANR 11 LABX 0023 01) and partially supported by the ANR grant STOSYMAP (ANR 2011 BS01 015 01)
 
\selectlanguage{french}

%\input{Chapters/Chapter6} 
%\input{Chapters/Chapter7} 

%----------------------------------------------------------------------------------------
%	THESIS CONTENT - APPENDICES
%----------------------------------------------------------------------------------------

%\addtocontents{toc}{\vspace{2em}} % Add a gap in the Contents, for aesthetics

\appendix % Cue to tell LaTeX that the following 'chapters' are Appendices

% Include the appendices of the thesis as separate files from the Appendices folder
% Uncomment the lines as you write the Appendices

%\input{Appendices/AppendixA}
%\input{Appendices/AppendixB}
%\input{Appendices/AppendixC}

%\addtocontents{toc}{\vspace{2em}} % Add a gap in the Contents, for aesthetics

%\backmatter

%----------------------------------------------------------------------------------------
%	BIBLIOGRAPHY
%----------------------------------------------------------------------------------------

%\label{Bibliography}

%\lhead{\emph{Bibliography}} % Change the page header to say "Bibliography"

%\bibliographystyle{plain} % Use the "unsrtnat" BibTeX style for formatting the Bibliography

%\bibliography{Bibliography} % The references (bibliography) information are stored in the file named "Bibliography.bib"

%\newpage
%\addcontentsline{toc}{section}{Bibliographie}

\bibliography{Bibliography}
\bibliographystyle{plain}
\afterpage{
\null
\vfill
\thispagestyle{empty}
\clearpage}

\end{document}